\documentclass[11pt,reqno]{amsart}

\usepackage[a4paper,margin=30mm]{geometry}
\usepackage[T1]{fontenc}
\usepackage[utf8]{inputenc}
\usepackage{lmodern}
\usepackage{amsmath,amssymb,amsthm,mathtools}
\usepackage{microtype}
\usepackage{booktabs,array,tabularx,longtable}
\usepackage{enumitem}
\usepackage{xspace}
\usepackage{graphicx}
\usepackage{float}
\usepackage[section]{placeins}
\usepackage{subfig}
\usepackage[dvipsnames,svgnames,table]{xcolor}
\usepackage{tikz}

\definecolor{TileDarkBlue}{HTML}{0B3B75}
\definecolor{TileOrange}{HTML}{F28E2B}
\definecolor{TileGreen}{HTML}{8BCB72}
\definecolor{TileTan}{HTML}{E7B26D}
\definecolor{TileLightBlue}{HTML}{A6CEE3}
\definecolor{TilePink}{HTML}{E78AC3}
\definecolor{TileGold}{HTML}{B89D2D}
\definecolor{TileTeal}{HTML}{66C2A5}
\definecolor{TileRose}{HTML}{F4A6B0}
\definecolor{TilePaleGreen}{HTML}{A9D18E}
\definecolor{TileCyan}{HTML}{9CCED5}
\definecolor{TileMauve}{HTML}{E1B7C8}
\definecolor{TileYellow}{HTML}{F1C27D}
\definecolor{TileBlue}{HTML}{9BC2E6}
\definecolor{TilePurple}{HTML}{C27BA0}
\definecolor{TileTurquoise}{HTML}{76B7B2}
\definecolor{TileBrightYellow}{HTML}{EDC948}
\definecolor{TileBrown}{HTML}{9C755F}
\definecolor{TileFreshGreen}{HTML}{8FD17F}
\definecolor{TileSteel}{HTML}{4E79A7}
\definecolor{TilePaleRose}{HTML}{F6C1D0}

\newcommand{\PickTileColor}[1]{%
  \pgfmathtruncatemacro{\TileColorIndex}{mod((#1)+2000,20)}%
  \def\fillcol{TileOrange}%
  \ifcase\TileColorIndex
    \def\fillcol{TileOrange}\or
    \def\fillcol{TileGreen}\or
    \def\fillcol{TileTan}\or
    \def\fillcol{TileLightBlue}\or
    \def\fillcol{TilePink}\or
    \def\fillcol{TileGold}\or
    \def\fillcol{TileTeal}\or
    \def\fillcol{TileRose}\or
    \def\fillcol{TilePaleGreen}\or
    \def\fillcol{TileCyan}\or
    \def\fillcol{TileMauve}\or
    \def\fillcol{TileYellow}\or
    \def\fillcol{TileBlue}\or
    \def\fillcol{TilePurple}\or
    \def\fillcol{TileTurquoise}\or
    \def\fillcol{TileBrightYellow}\or
    \def\fillcol{TileBrown}\or
    \def\fillcol{TileFreshGreen}\or
    \def\fillcol{TileSteel}\or
    \def\fillcol{TilePaleRose}%
  \fi
}

\tikzset{
  every picture/.append style={baseline=(current bounding box.south)},
  origin marker/.style={circle,fill=red,draw=red!55!black,line width=.55pt,
                        inner sep=2.05pt,outer sep=0pt}
}
\newcommand{\MarkOrigin}{}

\newcommand{\AxisNumberFont}{\fontsize{5.4}{5.8}\selectfont}
\tikzset{
  axis x tick/.style={fill=black,inner sep=0pt,minimum width=.35pt,minimum height=3pt},
  axis y tick/.style={fill=black,inner sep=0pt,minimum width=3pt,minimum height=.35pt},
  axis number/.style={font=\AxisNumberFont,fill=white,inner sep=.55pt,text=black},
  lattice cross/.style={draw=red!88!black,line width=.45pt,overlay}
}
\newcommand{\LatticeCross}[2]{%
  \draw[lattice cross]
    ([xshift=-1.55pt,yshift=-1.55pt]#1,#2)--
    ([xshift= 1.55pt,yshift= 1.55pt]#1,#2);
  \draw[lattice cross]
    ([xshift=-1.55pt,yshift= 1.55pt]#1,#2)--
    ([xshift= 1.55pt,yshift=-1.55pt]#1,#2);%
}
\newcommand{\DrawIntegerAxes}[5]{%
  \foreach \xx in {#1,...,#2}{%
    \node[axis x tick,anchor=north,overlay] at (\xx,#3) {};
  }
  \foreach \yy in {#3,...,#4}{%
    \node[axis y tick,anchor=east,overlay] at (#1,\yy) {};
  }
  \foreach \kk in {0,1,2,3}{%
    \pgfmathtruncatemacro{\xxlab}{\kk*(#5)}%
    \node[axis number,anchor=north,yshift=-3.2pt,overlay] at (\xxlab,#3) {$\xxlab$};
  }
  \pgfmathtruncatemacro{\ymaxlab}{floor((#4)/(#5))}%
  \foreach \kk in {0,...,\ymaxlab}{%
    \pgfmathtruncatemacro{\yylab}{\kk*(#5)}%
    \node[axis number,anchor=east,xshift=-3.2pt,overlay] at (#1,\yylab) {$\yylab$};
  }
  \path (0,#3) ++(0,-7pt);
}
\newcommand{\RasterLatticeOverlay}{}
\newcommand{\RasterWithAxes}[9]{%
  \begin{tikzpicture}[baseline=(image.south)]
    \node[anchor=south west,inner sep=0pt] (image) at (0,0)
      {\includegraphics[width=\linewidth]{#1}};
    \begin{scope}[shift={(image.south west)},
      x={($(image.south east)-(image.south west)$)},
      y={($(image.north west)-(image.south west)$)}]
      \pgfmathsetmacro{\axisleft}{#2+(#6)*(#4)}
      \pgfmathsetmacro{\axisright}{#2+(#7)*(#4)}
      \pgfmathsetmacro{\axisbottom}{#3+(#8)*(#5)}
      \pgfmathsetmacro{\axistop}{#3+(#9)*(#5)}
      \def\RasterPoint##1##2{%
        \pgfmathsetmacro{\pointx}{#2+(##1)*(#4)}%
        \pgfmathsetmacro{\pointy}{#3+(##2)*(#5)}%
        \LatticeCross{\pointx}{\pointy}%
      }%
      \begin{scope}
        \clip (\axisleft,\axisbottom) rectangle (\axisright,\axistop);
        \RasterLatticeOverlay
      \end{scope}
      \foreach \ii in {#6,...,#7}{%
        \pgfmathsetmacro{\tickx}{#2+(\ii)*(#4)}%
        \node[axis x tick,anchor=north,overlay] at (\tickx,\axisbottom) {};
      }
      \foreach \jj in {#8,...,#9}{%
        \pgfmathsetmacro{\ticky}{#3+(\jj)*(#5)}%
        \node[axis y tick,anchor=east,overlay] at (\axisleft,\ticky) {};
      }
      \pgfmathtruncatemacro{\axisspan}{#7-#6}%
      \def\axislabelstep{1}%
      \ifnum\axisspan>15 \def\axislabelstep{6}%
      \else\ifnum\axisspan>9 \def\axislabelstep{3}%
      \else\ifnum\axisspan>6 \def\axislabelstep{2}%
      \fi\fi\fi
      \foreach \kk in {0,1,2,3}{%
        \pgfmathtruncatemacro{\xxlab}{\kk*\axislabelstep}%
        \pgfmathsetmacro{\labx}{#2+(\xxlab)*(#4)}%
        \node[axis number,anchor=north,yshift=-3.2pt,overlay]
          at (\labx,\axisbottom) {$\xxlab$};
      }
      \pgfmathtruncatemacro{\ymaxlab}{floor((#9)/\axislabelstep)}%
      \foreach \kk in {0,...,\ymaxlab}{%
        \pgfmathtruncatemacro{\yylab}{\kk*\axislabelstep}%
        \pgfmathsetmacro{\laby}{#3+(\yylab)*(#5)}%
        \node[axis number,anchor=east,xshift=-3.2pt,overlay]
          at (\axisleft,\laby) {$\yylab$};
      }
      \path (#2,\axisbottom) ++(0,-7pt);
    \end{scope}
  \end{tikzpicture}%
}

\newcommand{\RasterWithZTwo}[9]{%
  \begingroup
  \def\RasterLatticeOverlay{%
    \foreach \xx in {#6,...,#7}{\foreach \yy in {#8,...,#9}{\RasterPoint{\xx}{\yy}}}%
  }%
  \RasterWithAxes{#1}{#2}{#3}{#4}{#5}{#6}{#7}{#8}{#9}%
  \endgroup
}
\newcommand{\RasterWithThreeZxZ}[9]{%
  \begingroup
  \def\RasterLatticeOverlay{%
    \foreach \mm in {-2,...,7}{\foreach \yy in {#8,...,#9}{\RasterPoint{3*\mm}{\yy}}}%
  }%
  \RasterWithAxes{#1}{#2}{#3}{#4}{#5}{#6}{#7}{#8}{#9}%
  \endgroup
}
\newcommand{\RasterWithFiveZxZ}[9]{%
  \begingroup
  \def\RasterLatticeOverlay{%
    \foreach \mm in {-1,...,2}{\foreach \yy in {#8,...,#9}{\RasterPoint{5*\mm}{\yy}}}%
  }%
  \RasterWithAxes{#1}{#2}{#3}{#4}{#5}{#6}{#7}{#8}{#9}%
  \endgroup
}
\newcommand{\RasterWithSqrtTwoZxZ}[9]{%
  \begingroup
  \def\RasterLatticeOverlay{%
    \foreach \mm in {-1,...,4}{\foreach \yy in {#8,...,#9}{\RasterPoint{sqrt(2)*\mm}{\yy}}}%
  }%
  \RasterWithAxes{#1}{#2}{#3}{#4}{#5}{#6}{#7}{#8}{#9}%
  \endgroup
}
\newcommand{\RasterWithZxFourZ}[9]{%
  \begingroup
  \def\RasterLatticeOverlay{%
    \foreach \xx in {#6,...,#7}{\foreach \nn in {-1,...,3}{\RasterPoint{\xx}{4*\nn}}}%
  }%
  \RasterWithAxes{#1}{#2}{#3}{#4}{#5}{#6}{#7}{#8}{#9}%
  \endgroup
}
\newcommand{\RasterWithSixZxZ}[9]{%
  \begingroup
  \def\RasterLatticeOverlay{%
    \foreach \mm in {-1,...,2}{\foreach \yy in {#8,...,#9}{\RasterPoint{6*\mm}{\yy}}}%
  }%
  \RasterWithAxes{#1}{#2}{#3}{#4}{#5}{#6}{#7}{#8}{#9}%
  \endgroup
}
\newcommand{\RasterWithSevenZxZ}[9]{%
  \begingroup
  \def\RasterLatticeOverlay{%
    \foreach \mm in {-1,...,3}{\foreach \yy in {#8,...,#9}{\RasterPoint{7*\mm}{\yy}}}%
  }%
  \RasterWithAxes{#1}{#2}{#3}{#4}{#5}{#6}{#7}{#8}{#9}%
  \endgroup
}
\newcommand{\RasterWithThreeZxThreeZ}[9]{%
  \begingroup
  \def\RasterLatticeOverlay{%
    \foreach \mm in {-2,...,7}{\foreach \nn in {-2,...,3}{\RasterPoint{3*\mm}{3*\nn}}}%
  }%
  \RasterWithAxes{#1}{#2}{#3}{#4}{#5}{#6}{#7}{#8}{#9}%
  \endgroup
}

\usetikzlibrary{matrix,patterns,arrows,arrows.meta,calc,positioning,fit,backgrounds,decorations.pathreplacing,decorations.pathmorphing,shapes.geometric}
\usepackage[numbers,sort&compress]{natbib}
\usepackage[colorlinks=true,linkcolor=MidnightBlue,citecolor=MidnightBlue,urlcolor=MidnightBlue]{hyperref}
\hypersetup{%
  pdftitle={Tilings, Packings, and the existence of Schwartz-class Gabor windows},%
  pdfauthor={Andrei Caragea, Mihail N. Kolountzakis, and Goetz E. Pfander},%
  pdfsubject={Sharp lattice criteria for Schwartz Parseval Gabor frames and multiwindow extensions},%
  pdfkeywords={lattice tiling, epsilon-packing, Parseval Gabor frames, Schwartz windows, Feichtinger algebra, Balian--Low theorem, multiwindow Gabor frames, continuous Zak transform, symplectically rational lattices, common zeros, Zibulski--Zeevi transform, rectangular Zak matrix}%
}

\setlist{itemsep=.2em,topsep=.4em}
\allowdisplaybreaks
\numberwithin{equation}{section}
\renewcommand{\thepart}{\Alph{part}}

\makeatletter
\renewcommand\subsection{\@startsection{subsection}{2}{\z@}%
  {0.85\baselineskip \@plus 0.25\baselineskip \@minus 0.10\baselineskip}%
  {0.42\baselineskip}%
  {\normalfont\bfseries}}
\renewcommand\subsubsection{\@startsection{subsubsection}{3}{\z@}%
  {0.70\baselineskip \@plus 0.20\baselineskip \@minus 0.10\baselineskip}%
  {0.30\baselineskip}%
  {\normalfont\bfseries}}
\makeatother

\newcommand{\displaypart}[1]{%
  \clearpage
  \vspace*{1cm}%
  \refstepcounter{part}%
  \addcontentsline{toc}{part}{Part~\thepart.\ #1}%
  \noindent{\fontsize{13}{16}\selectfont\normalfont\scshape Part~\thepart\hspace{1em}#1\par}
  \vspace{2.2\baselineskip}%
}

\newtheorem{theorem}{Theorem}[section]
\newtheorem{proposition}[theorem]{Proposition}
\newtheorem{lemma}[theorem]{Lemma}
\newtheorem*{unnumberedlemma}{Lemma}
\newtheorem*{unnumberedtheorem}{Theorem}
\newtheorem{corollary}[theorem]{Corollary}

\theoremstyle{definition}
\newtheorem{definition}[theorem]{Definition}

\newtheorem{examples}[theorem]{Examples}
\theoremstyle{remark}
\newtheorem{remark}[theorem]{Remark}

\newcommand{\R}{\mathbb R}
\newcommand{\Z}{\mathbb Z}
\newcommand{\N}{\mathbb N}
\newcommand{\Q}{\mathbb Q}
\newcommand{\C}{\mathbb C}

\newcommand{\eps}{\varepsilon}
\newcommand{\covol}{\operatorname{covol}}

\newcommand{\dist}{\operatorname{dist}}
\newcommand{\supp}{\operatorname{supp}}
\newcommand{\cl}{\overline}
\newcommand{\rank}{\operatorname{rank}}
\newcommand{\Pf}{\operatorname{Pf}}

\newcommand{\G}{\mathcal G}
\newcommand{\Sclass}{\mathcal S}

\newcommand{\ip}[2]{\langle #1,#2\rangle}

\newcommand{\norm}[1]{\lVert #1\rVert}

\newcommand{\TeP}{\ensuremath{(\mathrm{T}\epsilon\mathrm{P})}\xspace}
\newcommand{\SG}{\ensuremath{(\mathrm{SG})}\xspace}
\newcommand{\mSG}{\mathrm{mSG}}
\newcommand{\Sp}{\operatorname{Sp}}

\newcommand{\Id}{\mathrm I}

\newcommand{\CZ}{C_Z}
\newcommand{\Szero}{S_0}
\newcommand{\Wiener}{W(C,\ell^1)}
\newcommand{\Wzero}{W_0}

\newcommand{\ConfigurationTable}[1]{%
  \begin{tikzpicture}[x=.50cm,y=.50cm,baseline=(current bounding box.south)]
    \useasboundingbox (-1.15,-.20) rectangle (12.15,8.85);
    \fill[yellow!28] (0,0) rectangle (1,1);
    \draw[line width=.75pt] (0,0) rectangle (12,8);
    \foreach \x in {2,4,6,8,10}{\draw[line width=.75pt] (\x,0)--(\x,8);}
    \foreach \y in {2,4,6}{\draw[line width=.75pt] (0,\y)--(12,\y);}
    \foreach \x in {1,3,5,7,9,11}{\draw[gray!60,line width=.28pt] (\x,0)--(\x,8);}
    \foreach \y in {1,3,5,7}{\draw[gray!60,line width=.28pt] (0,\y)--(12,\y);}
    \foreach \cc in {1,...,6}{
      \pgfmathsetmacro{\configlabelx}{2*\cc-1}
      \pgfmathtruncatemacro{\configlabelindex}{\cc-1}
      \node[font=\scriptsize] at (\configlabelx,8.48) {$Q_{\configlabelindex}$};
    }
    \foreach \rr in {1,...,4}{
      \pgfmathsetmacro{\configlabely}{2*\rr-1}
      \pgfmathtruncatemacro{\configrowindex}{\rr-1}
      \node[left,font=\scriptsize] at (-.24,\configlabely) {$R_{\configrowindex}$};
    }
    #1
  \end{tikzpicture}%
}

\newcommand{\MiniConfigurationTable}[2]{%
  \begin{tikzpicture}[x=.248cm,y=.233cm,baseline=(current bounding box.south west)]
    \path[use as bounding box] (0,0) rectangle (12,8);
    \fill[yellow!28] (0,0) rectangle (1,1);
    \draw[#1,line width=.62pt] (0,0) rectangle (12,8);
    \foreach \x in {2,4,6,8,10}{\draw[#1,line width=.62pt] (\x,0)--(\x,8);}
    \foreach \y in {2,4,6}{\draw[#1,line width=.62pt] (0,\y)--(12,\y);}
    \foreach \x in {1,3,5,7,9,11}{\draw[#1!55,line width=.24pt] (\x,0)--(\x,8);}
    \foreach \y in {1,3,5,7}{\draw[#1!55,line width=.24pt] (0,\y)--(12,\y);}
    \foreach \cc in {1,...,6}{
      \pgfmathsetmacro{\configlabelx}{2*\cc-1}
      \pgfmathtruncatemacro{\configlabelindex}{\cc-1}
      \node[overlay,font=\fontsize{3.8}{4.2}\selectfont] at (\configlabelx,8.52) {$Q_{\configlabelindex}$};
    }
    \foreach \rr in {1,...,4}{
      \pgfmathsetmacro{\configlabely}{2*\rr-1}
      \pgfmathtruncatemacro{\configrowindex}{\rr-1}
      \node[overlay,left,font=\fontsize{3.8}{4.2}\selectfont] at (-.18,\configlabely) {$R_{\configrowindex}$};
    }
    #2
  \end{tikzpicture}%
}
\newcommand{\MiniConfigurationEntry}[6]{%
  \pgfmathsetmacro{\miniconfigentryx}{2*(#2-1)+#3}%
  \pgfmathsetmacro{\miniconfigentryy}{2*(#1-1)+#4}%
  \node[anchor=center,inner sep=0pt,text=#6,font=\fontsize{6.6}{6.9}\selectfont] at
    (\miniconfigentryx,\miniconfigentryy) {$#5$};%
}
\newcommand{\MiniConfigTL}[4]{\MiniConfigurationEntry{#1}{#2}{.5}{.5}{#3}{#4}}
\newcommand{\MiniConfigTR}[4]{\MiniConfigurationEntry{#1}{#2}{1.5}{.5}{#3}{#4}}
\newcommand{\MiniConfigBL}[4]{\MiniConfigurationEntry{#1}{#2}{.5}{1.5}{#3}{#4}}

\newcommand{\ConfigurationEntry}[5]{%
  \pgfmathsetmacro{\configentryx}{2*(#2-1)+#3}%
  \pgfmathsetmacro{\configentryy}{2*(#1-1)+#4}%
  \node[text=MidnightBlue,font=\footnotesize] at (\configentryx,\configentryy) {$#5$};%
}
\newcommand{\ConfigTL}[3]{\ConfigurationEntry{#1}{#2}{.5}{.5}{#3}}
\newcommand{\ConfigTR}[3]{\ConfigurationEntry{#1}{#2}{1.5}{.5}{#3}}
\newcommand{\ConfigBL}[3]{\ConfigurationEntry{#1}{#2}{.5}{1.5}{#3}}
\newcommand{\ConfigBR}[3]{\ConfigurationEntry{#1}{#2}{1.5}{1.5}{#3}}

\title[Tilings, packings, and Schwartz-class Gabor windows]{Tilings, Packings, and the existence of Schwartz-class Gabor windows}

\author[A. Caragea]{Andrei Caragea}
\address{Mathematical Institute for Machine Learning and Data Science (MIDS), Catholic University of Eichst\"att--Ingolstadt, Auf der Schanz 49, 85049 Ingolstadt, Germany}

\author[M. N. Kolountzakis]{Mihail N. Kolountzakis}
\address{Department of Mathematics and Applied Mathematics, University of Crete, Voutes Campus, 70013 Heraklion, Greece}
\email{kolount@gmail.com}

\author[G. E. Pfander]{G\"otz E. Pfander}
\address{Chair of Mathematics -- Scientific Computing, Catholic University of Eichst\"att--Ingolstadt, MIDS, Auf der Schanz 49, 85049 Ingolstadt, Germany}
\email{pfander@ku.de}

\date{}

\begin{document}
\raggedbottom

\begin{abstract}
The existence and construction of Schwartz-class windows for lattice Gabor
frames is a central problem in time--frequency analysis.  For general
time--frequency lattices, we characterize exactly those lattices that admit
Schwartz-class windows and, additionally, windows in the Feichtinger algebra.
Thereby we arrive at sharp classical and amalgam Balian--Low theorems for lattices in $\R^d$.
For separable lattices and specified block-zero forms, we construct smooth
windows that are compactly supported either in time or in frequency.  These constructions are based on a
characterization of pairs of lattices for which there exists a single set
that tiles with one lattice and whose $\varepsilon$-neighborhood packs with
the other.
\end{abstract}

\maketitle
\tableofcontents

\section{Introduction and main results}\label{sec:introduction}

The search for measurable sets that tile Euclidean space by translations
from two different lattices, or that tile by one lattice and pack by another,
is a well-studied problem at the intersection of harmonic analysis and
discrete geometry.  In this paper we sharpen the second question.  We ask
not merely that the packing translates be disjoint, but that distinct
packing translates remain a positive Euclidean distance apart; equivalently,
a positive Euclidean neighborhood of the tile must still pack by the second
lattice.

Gabor frames, on the other hand, play a central role in time--frequency analysis and signal
processing because they provide stable expansions into translated and
modulated copies of a window function.  A long-standing problem is to
decide, for a prescribed time--frequency lattice, whether one can choose a
frame window with strong localization and regularity.  We settle this
question for Schwartz class windows by characterizing exactly the
phase-space lattices that admit a Gabor frame generated by such a window.

The two problems are intimately connected.  For a separable phase-space
lattice, the tiling--packing criterion for a pair of Euclidean lattices is
exactly the arithmetic condition that governs the existence of a Schwartz
class Gabor frame.  Moreover, a positively separated tiling--packing
configuration can be mollified and normalized to produce a compactly
supported smooth Parseval Gabor window.  Thus the geometric theorem supplies
both the intuition and the constructive core of the Gabor-frame existence
result.

\medskip

A \emph{full-rank lattice} in $\R^d$ is a subgroup of the form
$L=M\Z^d$ with $M\in GL_d(\R)$.  A measurable set $F\subset\R^d$ is a
\emph{fundamental domain} for $L$ if the translates
$\{F+\ell:\ell\in L\}$ form a pairwise disjoint partition of $\R^d$;
equivalently, every $x\in\R^d$ has a unique representation $x=f+\ell$ with
$f\in F$ and $\ell\in L$.  The \emph{covolume} and \emph{density} of $L$ are
\[
 \covol(L)=|\det M|=|F|,
 \qquad
 \operatorname{density}(L)=\frac{1}{\covol(L)}=\frac{1}{|F|}.
\]
Thus the covolume is the measure of a fundamental domain and the density is
its reciprocal.  Throughout, $\norm{\cdot}_2$ denotes the Euclidean norm on
$\R^d$.

\subsection{Simultaneous tilings and $\varepsilon$-packings}
\label{subsec:intro-tiling-packing}

Let $T,P\subset\R^d$ be full-rank lattices, and let $\Omega\subset\R^d$ be a
bounded Borel $T$-fundamental domain.  For $\eps>0$, set
\[
 B_\eps=\{x\in\R^d:\norm{x}_2<\eps\},
 \qquad \Omega_\eps=\Omega+B_\eps.
\]
We say that $\Omega_\eps$ \emph{packs with $P$} if
\[
 (\Omega_\eps+p)\cap\Omega_\eps=\varnothing
 ,\qquad p\in P\setminus\{0\}.
\]
The first problem of the paper is to determine exactly when such a pair
$(\Omega,\eps)$ exists.

The closure of $T+P$ is a closed subgroup of $\R^d$ containing a full-rank
lattice.  Hence there is a unique $k\in\{0,\ldots,d\}$ such that
\[
 \overline{T+P}\cong \Z^k\times\R^{d-k}.
\]
We call $k$ the \emph{commensurability rank} of $(T,P)$.  The case $k=d$ is
exactly the commensurable case, in which $T+P$ is a lattice.

\begin{theorem}\label{thm:main-geometric}
Let $T,P\subset\R^d$ be full-rank lattices.  The following are equivalent.
\begin{enumerate}[label=(\alph*)]
\item There exist a bounded Borel $T$-fundamental domain $\Omega$ and
$\eps>0$ such that $\Omega+B_\eps$ packs with $P$.
\item The condition \TeP holds, namely,
\begin{equation}
 \covol(T)<\covol(P)
 \quad\text{and, if $T+P$ is discrete,}\quad
 [T+P:P]\ge [T+P:T]+d.
 \tag{\ensuremath{\mathrm{T}\epsilon\mathrm{P}}}\label{eq:TeP}
\end{equation}
\stepcounter{equation}
\end{enumerate}
Whenever these conditions hold, $\Omega$ may be chosen as a finite union of
bounded half-open polyhedra.
\end{theorem}

The proof yields a stronger functional formulation.  Whenever
\eqref{eq:TeP} holds, there exist both a nonnegative compactly supported
continuous piecewise-affine $T$-partition of unity $\varphi_{\rm pa}$ and a
nonnegative compactly supported smooth $T$-partition of unity
$\varphi_{\infty}$.  Each may be chosen with a positive support buffer:
for suitable $\eta_{\rm pa},\eta_{\infty}>0$,
\[
 (\supp\varphi_{\rm pa}+B_{\eta_{\rm pa}})+P
 \quad\text{and}\quad
 (\supp\varphi_{\infty}+B_{\eta_{\infty}})+P
\]
are packings (see Corollary~\ref{cor:global-partitions}).

The polyhedral conclusion is obtained in three slightly different ways.  In
the discrete case, Section~\ref{sec:soft-configurations} constructs the required compactly supported
piecewise-affine partition of unity directly.  If $T+P$ is dense in
$\R^d$, Section~\ref{sec:fully-dense-case} gives a direct piecewise-affine
partition-of-unity construction.  In the remaining non-discrete cases, the partition-of-unity/slot
argument first produces a continuous compactly supported partition of unity,
and Lemma~\ref{lem:pou-replacement} replaces it by piecewise-affine and
smooth partitions of unity without changing the partition-of-unity identity
or the packing buffer.

The strict covolume inequality in \eqref{eq:TeP} is necessary by volume
comparison: if such an $\Omega$ and $\eps$ exist, then
$\covol(T)=|\Omega|<|\Omega+B_\eps|\le\covol(P)$.

\begin{examples}\label{ex:intro-TeP}
In dimension $d=2$, the criterion \eqref{eq:TeP} gives the following.
\begin{enumerate}[label=(\roman*)]
\item Let $T=\Z^2$ and $P=\Z\times\sqrt2\,\Z$.  Then
$T+P=\Z\times(\Z+\sqrt2\,\Z)$ is not discrete and
$\covol(T)=1<\sqrt2=\covol(P)$, so \eqref{eq:TeP} holds.

\item Let $T=\Z^2$ and $P=\Z\times2\Z$.  Here
$T+P=\Z^2$, $[T+P:P]=2$, and $[T+P:T]=1$; hence
$2<1+2$ and \eqref{eq:TeP} fails.

\item Let $T=7\Z\times\Z$ and $P=3\Z\times3\Z$.  Here
$T+P=\Z^2$, $[T+P:P]=9$, and $[T+P:T]=7$; thus
$9=7+2$ and \eqref{eq:TeP} holds at equality.

\item Let $K=\begin{psmallmatrix}2&1\\0&2\end{psmallmatrix}$,
$T=\Z^2$, and $P=K\Z^2$.  Then $P$ is non-diagonal,
$T+P=\Z^2$, $[T+P:P]=4$, and $[T+P:T]=1$, so
$4\ge1+2$ and \eqref{eq:TeP} holds.
\end{enumerate}
\end{examples}

The two cases in \TeP, according to whether $T+P$ is discrete or not,
are proved by closely related constructions.  In the commensurable case, a
finite row--column table leads to a soft configuration space, and an elementary
sphere-filling argument supplies the continuously varying choices.  When
$T+P$ is dense in all of $\R^d$, there is no residual parameter torus, and a direct piecewise-affine
partition-of-unity assignment suffices.  In the intermediate non-discrete
case, the connected directions are refined into compactly supported elements
of a partition of unity and separated target slots; the resulting larger
finite table is handled by
the same sphere-filling argument over a lower-dimensional torus and then by the
regularity replacement lemma of Section~\ref{sec:pou-replacement}.

\subsection{Sharp Balian--Low theorems}
\label{subsec:intro-gabor}

For $x,\xi,t\in\R^d$, the \emph{time shift},
\emph{frequency shift}, and \emph{time--frequency shift} are
\[
 (T_xf)(t)=f(t-x),
 \qquad
 (M_\xi f)(t)=e^{2\pi i\ip{\xi}{t}}f(t),
 \qquad
 \pi(x,\xi)=M_\xi T_x.
\]
A family $(f_j)_{j\in J}\subset L^2(\R^d)$ is a \emph{frame} if there are
constants $0<A\le B<\infty$ such that
\[
 A\norm{f}_2^2
 \le \sum_{j\in J}|\ip{f}{f_j}|^2
 \le B\norm{f}_2^2
 ,\qquad f\in L^2(\R^d).
\]
It is \emph{tight} if $A=B$ and \emph{Parseval} if $A=B=1$.  A full-rank
lattice $\Lambda\subset\R^{2d}$ indexes the \emph{Gabor system}
\[
 \G(g,\Lambda)=\{\pi(\lambda)g:\lambda\in\Lambda\};
\]
when this family is a frame, it is called a \emph{Gabor frame}.  Standard
background on Gabor systems, frame operators, dual windows, and lattice
density may be found in \cite[Chapters~5--7]{Grochenig2001}.

A Parseval frame gives a painless nonorthogonal expansion: for every
$f\in L^2(\R^d)$,
\[
 f=\sum_{\lambda\in\Lambda}
     \ip{f}{\pi(\lambda)g}\,\pi(\lambda)g,
\]
with unconditional convergence in $L^2(\R^d)$.  For windows and signals with additional
decay, the localization of both $g$ and $\widehat g$ also controls the decay
of the Gabor coefficients and hence the convergence of natural finite
truncations; compare \cite{DaubechiesGrossmannMeyer1986} and
\cite[Chapters~5--7]{Grochenig2001}.

For a full-rank Euclidean lattice $L\subset\R^d$, its dual lattice is
\[
 L^*=\{\xi\in\R^d:\ip{\xi}{\ell}\in\Z
       \text{ for every }\ell\in L\},
\]
where $\ip{\cdot}{\cdot}$ denotes the standard Euclidean inner product on
$\R^d$.  At times we shall also write $\ip{x}{\xi}=x\cdot\xi$.  If
$L=M\Z^d$, then $L^*=M^{-T}\Z^d$ and
$\covol(L^*)=\covol(L)^{-1}$.

On $\R^{2d}=\R^d\times\R^d$, write
\[
 \sigma((x,\xi),(y,\eta))=\ip{\eta}{x}-\ip{\xi}{y}.
\]
The adjoint lattice of $\Lambda$ is
\[
 \Lambda^\circ=\{z\in\R^{2d}:\sigma(z,\lambda)\in\Z
 \text{ for all }\lambda\in\Lambda\}.
\]
The integral symplectic subgroup is
\[
 \Lambda_{\rm int}=\Lambda\cap\Lambda^\circ.
\]
We call $\Lambda$ \emph{symplectically rational} if
$[\Lambda:\Lambda_{\rm int}]<\infty$, and in this case write
$\nu(\Lambda)=[\Lambda:\Lambda_{\rm int}]^{1/2}\in\N$.

As a last preparation for stating the sharp Balian--Low theorems obtained
here, we define the Feichtinger algebra.  Choose a nonnegative
$\varphi\in C_c^\infty(\R^d)$ such that
$\sum_{k\in\Z^d}\varphi(x-k)=1$.  Then
\begin{equation}\label{eq:S0-definition}
 \Szero(\R^d)
 =\left\{f\in L^2(\R^d):
   \sum_{k\in\Z^d}
   \left\|\mathcal F\bigl(f\,\varphi(\,\cdot-k)\bigr)\right\|_{L^1(\R^d)}<\infty
  \right\}.
\end{equation}
The Schwartz class $\Sclass(\R^d)$ consists of all $f\in C^\infty(\R^d)$
for which $\sup_{x\in\R^d}|x^\alpha\partial^\beta f(x)|<\infty$ for every
pair of multiindices $\alpha,\beta$; one has
$\Sclass(\R^d)\subseteq\Szero(\R^d)$.

\begin{theorem}[Sharp Amalgam Balian--Low theorem]
\label{thm:main-gabor}\label{thm:amalgam-balian-low}
Let $\Lambda\subset\R^{2d}$ be a full-rank phase-space lattice.  The following
are equivalent.
\begin{enumerate}[label=(\alph*)]
\item There exists $g\in\Sclass(\R^d)$ such that $\G(g,\Lambda)$ is a Gabor
frame.
\item There exists $g\in\Szero(\R^d)$ such that $\G(g,\Lambda)$ is a Gabor
frame.
\item The lattice $\Lambda$ satisfies the condition \SG, namely,
\begin{equation}
 \covol(\Lambda)<1
 \quad\text{and, if $\Lambda$ is symplectically rational,}\quad
 \nu(\Lambda)\ge \nu(\Lambda^\circ)+d.
 \tag{\ensuremath{\mathrm{SG}}}\label{eq:SG}
\end{equation}
\stepcounter{equation}
\end{enumerate}
\end{theorem}

The implication \textup{(c)}$\Rightarrow$\textup{(a)} was proved by
Enstad--Thiel--Vilalta~\cite{EnstadThielVilalta2025}; Jakobsen--Luef had
already covered the non-symplectically-rational case
\cite{JakobsenLuef2020}.  The necessity implication
\textup{(b)}$\Rightarrow$\textup{(c)} is new; it is based on results of
de Dios Pont--Liehr--Taylor~\cite{deDiosLiehrTaylor2026}, who found lattices
admitting no Gabor-frame window with continuous Zak transform (see
Subsection~\ref{sec:rectangular-common-zero}).

Similarly, we can classify the lattices for which there exist Gabor-frame
windows with finite uncertainty product.

\begin{theorem}[Sharp Balian--Low theorem]
\label{thm:full-multivariate-balian-low}
Let $\Lambda\subset\R^{2d}$ be a full-rank phase-space lattice.  The following
are equivalent.
\begin{enumerate}[label=(\alph*)]
\item There exists $g\in L^2(\R^d)$ such that $\G(g,\Lambda)$ is a Gabor
frame and
\[
 \left(\int_{\R^d}|x|^2|g(x)|^2\,dx\right)
 \left(\int_{\R^d}|\xi|^2|\widehat g(\xi)|^2\,d\xi\right)<\infty.
\]
\item The strict density condition $\covol(\Lambda)<1$ holds.
\end{enumerate}
\end{theorem}

The implication \textup{(a)}$\Rightarrow$\textup{(b)} follows from the
Gabor density theorem and, at critical density, by combining the weak
arbitrary-lattice Balian--Low theorem of Gr\"ochenig--Han--Heil--Kutyniok
with the canonical-dual regularity theorem of
Lee--Philipp--Voigtlaender
\cite{GrochenigHanHeilKutyniok2002,LeePhilippVoigtlaender2023}.  The converse
is constructed in Sections~\ref{sec:full-multivariate-balian-low}
and~\ref{sec:gabor-necessity}.

In Theorems~\ref{thm:main-gabor} and
\ref{thm:full-multivariate-balian-low}, the term ``Gabor frame'' may be
replaced by ``tight Gabor frame'' or ``Parseval Gabor frame,'' as discussed
in Subsection~\ref{subsec:frames-duality-density}.

\stepcounter{theorem}

\begin{examples}\label{ex:intro-SG}
Use the coordinate order $(x_1,x_2,\omega_1,\omega_2)$ and write
$\Lambda_j=M_j\Z^4$.  The first four matrices correspond to the four pairs
in Examples~\ref{ex:intro-TeP}, with $\Gamma=T$ and $\Phi=P^*$.
\begin{enumerate}[label=(\roman*)]
\item For $M_1=I_3\oplus(1/\sqrt2)$, the lattice is not symplectically rational and $\covol(\Lambda_1)=1/\sqrt2<1$, so~\eqref{eq:SG} holds.
\item For $M_2=I_3\oplus(1/2)$, one has $\covol(\Lambda_2)=1/2$, $\nu(\Lambda_2)=2$, and $\nu(\Lambda_2^\circ)=1$; hence $2<1+2$ and~\eqref{eq:SG} fails.
\item For $M_3=\operatorname{diag}(7,1,1/3,1/3)$, one has $\covol(\Lambda_3)=7/9$, $\nu(\Lambda_3)=9$, and $\nu(\Lambda_3^\circ)=7$; thus~\eqref{eq:SG} holds at equality.
\item Let $K=\begin{psmallmatrix}2&1\\0&2\end{psmallmatrix}$ and $M_4=\operatorname{diag}(I_2,K^{-T})$, where $K^{-T}=\begin{psmallmatrix}1/2&0\\-1/4&1/2\end{psmallmatrix}$.  Then $\covol(\Lambda_4)=1/4$, $\nu(\Lambda_4)=4$, and $\nu(\Lambda_4^\circ)=1$, so~\eqref{eq:SG} holds.
\item The nonseparable shear $M_5=\begin{psmallmatrix}I_2&0\\I_2&I_2\end{psmallmatrix}\begin{psmallmatrix}(1/3)I_2&0\\0&I_2\end{psmallmatrix}=\begin{psmallmatrix}(1/3)I_2&0\\(1/3)I_2&I_2\end{psmallmatrix}$ has $\covol(\Lambda_5)=1/9$, $\nu(\Lambda_5)=9$, and $\nu(\Lambda_5^\circ)=1$, so~\eqref{eq:SG} holds.
\item Put $D_0=\operatorname{diag}(1,1/2)$ and $M_6=2^{-1/2}\begin{psmallmatrix}I_2&I_2\\-I_2&I_2\end{psmallmatrix}\begin{psmallmatrix}D_0&0\\0&I_2\end{psmallmatrix}$.  This nonseparable symplectic rotation has $\covol(\Lambda_6)=1/2$, $\nu(\Lambda_6)=2$, and $\nu(\Lambda_6^\circ)=1$, so~\eqref{eq:SG} fails.
\end{enumerate}
\end{examples}

For a symplectically rational lattice, the normal-form calculation in
Section~\ref{sec:lattice-theory} gives
\begin{equation}\label{eq:covol-nu-ratio-intro}
 \covol(\Lambda)=\frac{\nu(\Lambda^\circ)}{\nu(\Lambda)}.
\end{equation}
Consequently,
\[
 \covol(\Lambda)<1
 \quad\Longleftrightarrow\quad
 \nu(\Lambda)>\nu(\Lambda^\circ)
 \quad\Longleftrightarrow\quad
 \nu(\Lambda)\ge\nu(\Lambda^\circ)+1.
\]
Thus the algebraic gap in~\eqref{eq:SG} already implies the strict density
inequality in the symplectically rational case.

For separable lattices, the geometric condition \eqref{eq:TeP} and the
Gabor-frame condition \eqref{eq:SG} are the same arithmetic requirement.
This connection is rooted in the use of dual lattices and fundamental
domains in Gabor analysis; see \cite{HanWang2001,PfanderRashkovWang2012}
and \cite[Chapters~7--8]{Grochenig2001}.

\begin{proposition}
\label{prop:TeP-SG-relation}
Let $\Lambda=\Gamma\times\Phi$, where $\Gamma,\Phi\subset\R^d$ are full-rank
lattices.  Then
\[
 \Lambda\text{ satisfies \eqref{eq:SG}}
 \quad\Longleftrightarrow\quad
 (\Gamma,\Phi^*)\text{ satisfies \eqref{eq:TeP}}.
\]
If $\Gamma+\Phi^*$ is not discrete, the common condition is
$\covol(\Gamma)<\covol(\Phi^*)$, equivalently
$\covol(\Lambda)<1$.  If $\Gamma+\Phi^*$ is a lattice, then
\[
 \nu(\Lambda)=[\Gamma+\Phi^*:\Phi^*],
 \qquad
 \nu(\Lambda^\circ)=[\Gamma+\Phi^*:\Gamma],
\]
and the common condition is
\[
 [\Gamma+\Phi^*:\Phi^*]\ge[\Gamma+\Phi^*:\Gamma]+d.
\]
\end{proposition}

At the level of arbitrary square-integrable windows, there is no obstruction
beyond the ordinary density condition.  Bekka proved that a full-rank
phase-space lattice $\Lambda$ admits a window $g\in L^2(\R^d)$ for which
$\G(g,\Lambda)$ is a Gabor frame if and only if
\[
                         \covol(\Lambda)\le 1;
\]
see \cite[Theorem~4]{Bekka2004}.  The corresponding Parseval statement follows
from the frame-operator discussion in Section~\ref{sec:gabor-background}.
Thus every lattice satisfying the necessary density inequality admits an
$L^2$ Gabor frame; the purpose of \SG is to characterize when the window
can be required to have strong time--frequency regularity.

We write $C_c^\infty(\R^d)$ for the smooth compactly supported functions.
The classical Wiener space $\Wiener(\R^d)$ consists of the continuous
functions $g$ for which
\[
 \norm{g}_{\Wiener}
 :=\sum_{k\in\Z^d}\sup_{x\in[0,1]^d}|g(x+k)|<\infty,
\]
and
\[
 \Wzero(\R^d)
 :=\{g\in\Wiener(\R^d):\widehat g\in\Wiener(\R^d)\}.
\]
We write $\CZ(\R^d)$ for the set of all $g\in L^2(\R^d)$ whose Zak
transform has a continuous representative on $\R^d\times\R^d$.  The
relevant inclusions are
\[
 C_c^\infty(\R^d)\subset \Sclass(\R^d)\subset \Szero(\R^d)
 \subset \Wzero(\R^d)\subset \Wiener(\R^d)
 \subset \CZ(\R^d)\subset L^2(\R^d).
\]
The Feichtinger algebra $\Szero(\R^d)$ is invariant under every metaplectic
operator; see Section~\ref{sec:lattice-theory},
\cite[Section~9.4]{Grochenig2001}, and \cite{GjertsenLuef2024}.

The following theorem records the sharper continuous-Zak obstruction in the
case of rectangular lattices.

\begin{theorem}[Continuous-Zak obstruction for diagonal lattices]
\label{thm:intro-diagonal-zak-obstruction}
Let
\[
 D=\operatorname{diag}(b_1/a_1,\ldots,b_d/a_d),
 \qquad (a_i,b_i)=1,
\]
where the $a_i,b_i$ are positive integers, let
$\Lambda_D=D\Z^d\times\Z^d$, and put
\[
 N=\prod_{i=1}^d a_i,
 \qquad
 R=\prod_{i=1}^d b_i.
\]
If $\Lambda_D$ does not satisfy \eqref{eq:SG}, then there is no $g\in\CZ(\R^d)$ for which
$\G(g,\Lambda_D)$ is a frame.  In particular, no window in the
Fourier-invariant Wiener space $\Wzero(\R^d)$, or even in the larger
classical Wiener amalgam $\Wiener(\R^d)$, can generate such a frame.
\end{theorem}

We finish the discussion of sharp Balian--Low theorems with the corresponding
multiple-window result.
For $\mathbf g=(g_1,\ldots,g_q)$, write
\[
 \G(\mathbf g,\Lambda)
   =\{\pi(\lambda)g_j:\lambda\in\Lambda,\ 1\le j\le q\}.
\]

\begin{theorem}[Sharp multi-window amalgam Balian--Low theorem]
\label{thm:main-multiwindow}
Let $q\ge1$ and let $\Lambda\subset\R^{2d}$ be a full-rank phase-space
lattice.  For $\mathbf g=(g_1,\ldots,g_q)$, the following are equivalent.
\begin{enumerate}[label=(\alph*)]
\item There exist $g_1,\ldots,g_q\in\Sclass(\R^d)$ such that
$\G(\mathbf g,\Lambda)$ is a Gabor frame.
\item There exist $g_1,\ldots,g_q\in\Szero(\R^d)$ such that
$\G(\mathbf g,\Lambda)$ is a Gabor frame.
\item The lattice satisfies the condition $(\mathrm{mSG}_q)$, namely,
\begin{align}
 \covol(\Lambda)<q
 \quad\text{and, if $\Lambda$ is symplectically rational,}\quad
 q\,\nu(\Lambda)\ge\nu(\Lambda^\circ)+d.
 \tag{\ensuremath{\mathrm{mSG}_q}}\label{eq:mSG}
\end{align}
\stepcounter{equation}
\end{enumerate}
\end{theorem}

In parts~\textup{(a)} and~\textup{(b)}, the term ``Gabor frame'' may be
replaced by ``tight Gabor frame'' or ``Parseval Gabor frame.''  On a
diagonal representative of a symplectically rational lattice, the necessity
of~\eqref{eq:mSG} remains valid under the weaker assumptions
$g_1,\ldots,g_q\in\CZ(\R^d)$.  On an arbitrary symplectically rational
lattice it remains valid for $g_1,\ldots,g_q\in\Szero(\R^d)$, by
metaplectic invariance.

The minimum number of Schwartz windows, equivalently of
Feichtinger-algebra windows, is
\[
 q_{\min}^{\Sclass}(\Lambda)
 =\left\lceil
    \frac{\nu(\Lambda^\circ)+d}{\nu(\Lambda)}
  \right\rceil
\]
when $\Lambda$ is symplectically rational, and
\[
 q_{\min}^{\Sclass}(\Lambda)=\lfloor\covol(\Lambda)\rfloor+1
\]
otherwise.

\subsection{Smooth Gabor windows with compact support}
\label{subsec:smooth-compact-support}
The next theorem collects the explicit compact-support conclusions based on
Part~A of this paper.

\begin{theorem}
\label{thm:strengthened-one-window}
Let $\Lambda\subset\R^{2d}$ be a full-rank phase-space lattice.
\begin{enumerate}[label=(\alph*)]
\item If $\Lambda=\Gamma\times\Phi$ is separable and satisfies
\eqref{eq:SG}, then there exists a nonnegative $g\in C_c^\infty(\R^d)$ such that
$\G(g,\Gamma\times\Phi)$ is a Parseval frame and
\begin{align}
 \supp g\cap(\supp g+\xi^*)&=\varnothing,
 &&\xi^*\in\Phi^*\setminus\{0\},
 \label{eq:intro-support-separation}\\
 \sum_{\gamma\in\Gamma}|g(x+\gamma)|^2&=\covol(\Phi),
 &&x\in\R^d.
 \label{eq:intro-parseval-periodization}
\end{align}
There is also an unsmoothed polyhedral alternative: one may choose a
nonnegative Parseval window $g_{\rm pa}\in C_c(\R^d)$ such that
$\covol(\Phi)^{-1}g_{\rm pa}^2$ is a continuous piecewise-affine
$\Gamma$-partition of unity and
\eqref{eq:intro-support-separation}--\eqref{eq:intro-parseval-periodization}
hold with $g_{\rm pa}$ in place of $g$.

Alternatively, one may choose $g\in\Sclass(\R^d)$ so that
$\widehat g\in C_c^\infty(\R^d)$ is nonnegative and
\begin{align}
 \supp\widehat g\cap(\supp\widehat g+\gamma^*)&=\varnothing,
 &&\gamma^*\in\Gamma^*\setminus\{0\},
 \label{eq:intro-fourier-support-separation}\\
 \sum_{\phi\in\Phi}|\widehat g(\omega+\phi)|^2&=\covol(\Gamma),
 &&\omega\in\R^d.
 \label{eq:intro-fourier-periodization}
\end{align}
There is likewise an unsmoothed Fourier-side alternative: one may choose
a Parseval window $g_{\rm fpa}\in L^2(\R^d)$ such that
$\widehat g_{\rm fpa}\in C_c(\R^d)$ and
$\covol(\Gamma)^{-1}|\widehat g_{\rm fpa}|^2$ is a continuous
piecewise-affine $\Phi$-partition of unity, and
\eqref{eq:intro-fourier-support-separation}--\eqref{eq:intro-fourier-periodization}
hold with $g_{\rm fpa}$ in place of $g$.

\item Let $a>0$ and
\[
 \Lambda=aS\Z^{2d},
 \qquad
 S=\begin{pmatrix}A&B\\ C&D\end{pmatrix}\in\Sp(2d,\R),
\]
and assume that $\Lambda$ satisfies \eqref{eq:SG}.  If $A=0$ or $B=0$, a Parseval frame window
$g\in C_c^\infty(\R^d)$ can be chosen.  Alternatively, one may choose a
compactly supported continuous Parseval window for which $|g|^2$ is
piecewise affine.  If $C=0$ or $D=0$, a Parseval frame window can be chosen
so that $\widehat g\in C_c^\infty(\R^d)$; alternatively,
$\widehat g$ may be chosen compactly supported and continuous with
$|\widehat g|^2$ piecewise affine.
\end{enumerate}
\end{theorem}

The diagonal obstruction is proved in Section~\ref{sec:diagonal-model}.
The detailed proofs of parts~(a) and~(b) are completed in
Sections~\ref{sec:gabor-sufficiency} and~\ref{sec:gabor-necessity}.

The block conditions in part~(b) are sufficient, but they are not necessary
for a compactly supported smooth Parseval window.  Let
\[
 a=\frac{71}{100},
 \qquad
 \Gamma_0=
 a\begin{pmatrix}1&1\\-1&1\end{pmatrix}\Z^2.
\]
Appendix~\ref{app:compact-parseval-example} constructs
$g_0\in C_c^\infty(\R)$ such that $\G(g_0,\Gamma_0)$ is an orthonormal
system.  Its adjoint lattice is
\[
 \Lambda_0=\Gamma_0^\circ
 =\frac1{2a}\begin{pmatrix}1&-1\\1&1\end{pmatrix}\Z^2
 =cS\Z^2,
 \qquad
 c=\frac1{\sqrt2a},
 \qquad
 S=\frac1{\sqrt2}\begin{pmatrix}1&-1\\1&1\end{pmatrix}.
\]
All four scalar blocks of $S$ are nonzero.  Nevertheless,
$g=\sqrt{\covol(\Lambda_0)}\,g_0$ generates a Parseval frame on
$\Lambda_0$.  The relevant arithmetic values are
\[
 \covol(\Lambda_0)=\frac{5000}{5041},
 \qquad
 \nu(\Lambda_0)=5041,
 \qquad
 \nu(\Lambda_0^\circ)=5000.
\]
The full elementary construction and verification are given in
Appendix~\ref{app:compact-parseval-example}.

\subsection*{Outline of the paper}
Section~\ref{sec:literature} gives a brief summary of the most relevant
literature.  Part~A proves Theorem~\ref{thm:main-geometric} in the
commensurable, fully dense, and partially dense cases and then proves
necessity.  Part~B first develops the rational Zak matrix model and then
treats its two regularity branches: continuous matrices yield the amalgam
arithmetic gap, while Sobolev matrices yield finite-uncertainty windows below
critical density.  It then transfers both conclusions metaplectically and
proves the one-window and multiwindow Gabor results.

\section{Earlier results and relation to the literature}\label{sec:literature}

This section records the strands of the literature that motivate the geometric and
Gabor-theoretic arguments used later.

\subsection{Tiling and packing by lattices}\label{sec:lit-tiling}

The relation between translational tilings, packings, and orthogonal
exponentials has a long history.  Kolountzakis developed the viewpoint that
orthogonality and completeness can be encoded as packing and tiling
identities for suitable nonnegative functions
\cite{Kolountzakis2000Packing}.  Fourier methods for lattice tilings by
cubes and related polyhedral sets appear in
\cite{Kolountzakis1998Cubes}.  These works are part of the broader
connection between translational tiling and spectral questions.

Han and Wang \cite{HanWang2001} proved that any two full-rank lattices of
equal covolume have a common measurable fundamental domain and, more
generally, that if $\covol(L)\le\covol(M)$ then there is a measurable set
which tiles by $L$ and packs by $M$.  Essential predecessors and extensions
for simultaneous tiling include the multi-lattice problem
\cite{Kolountzakis1997MultiLattice}, compactly supported functions tiling
with several lattices \cite{KolountzakisPapageorgiou2022}, and common
domains for more general group actions
\cite{DutkayHanJorgensenPicioroaga2013}.

Grepstad and
Kolountzakis \cite{GrepstadKolountzakis2026} proved that two Euclidean
lattices of equal covolume have a bounded common fundamental domain, which
can be chosen as a finite union of polytopes.  A direct bounded measurable
construction was subsequently given by Spyridakis
\cite{Spyridakis2026BoundedCommon}.  Grepstad, Kolountzakis, and Spyridakis
\cite{GrepstadKolountzakisSpyridakis2025} established the corresponding
bounded tile/packing statement under the strict covolume inequality.  The
present paper asks for the stronger metric conclusion that a positive
neighborhood of the tile packs.  This is not a null-set refinement: it
requires a uniform gap between every pair of nontrivial packing translates.

\subsection{Gabor frames, regular windows, and Balian--Low phenomena}
\label{sec:lit-gabor}

A standard general reference for Gabor frames, duality, density, the Zak
transform, modulation spaces, and the Balian--Low theorem is
Gr\"ochenig's monograph~\cite{Grochenig2001}.  The geometric use of
fundamental domains goes back to Han and Wang~\cite{HanWang2001}; related
constructive results include the painless expansions of Daubechies,
Grossmann, and Meyer~\cite{DaubechiesGrossmannMeyer1986} and the smooth
common-domain construction of Pfander, Rashkov, and
Wang~\cite{PfanderRashkovWang2012}.  The duality and finite Zak-matrix tools
used later originate in work of Janssen, Ron--Shen, and Zibulski--Zeevi
\cite{Janssen1995,RonShen1997,ZibulskiZeevi1997}.

For classical and modern forms of the Balian--Low phenomenon, see
\cite{Balian1981,Low1985,Battle1988,BenedettoHeilWalnut1995,GrochenigHanHeilKutyniok2002,CabrelliMolterPfander2016,CarageaLeePfanderPhilipp2019,CarageaLeePhilippVoigtlaender2021,CarageaLeePhilippVoigtlaender2023}.
Gabor expansions on modulation and amalgam spaces, and stability under
changes of lattice, are treated in
\cite{FeichtingerGrochenig1997,GrochenigHeilOkoudjou2002,FeichtingerKaiblinger2004,GrochenigOrtegaRomero2015}.
For related work on localized or continuous frames, frame operators, and
regular or compactly supported dual Gabor windows, see
\cite{FornasierRauhut2005,BalazsFeichtingerHampejsKracher2006,Strohmer2001,Christensen2006}.
Structured Gabor windows and sharp frame-set questions for Hermite or totally
positive windows are developed, among other places, in
\cite{AbreuGrochenig2012,GrochenigRomeroStoeckler2018}.
For general phase-space lattices, operator-algebraic methods interpret Gabor
frames through Heisenberg modules and noncommutative tori
\cite{Luef2009Projections,JakobsenLuef2020}.  The Schwartz existence result
of Enstad--Thiel--Vilalta~\cite{EnstadThielVilalta2025}, already summarized
in Section~\ref{sec:introduction}, is the main external positive result used
in Part~B.  Symplectic normal forms and metaplectic covariance provide the
reduction from rational lattices to diagonal models; see
\cite{Newman1972,Folland1989,GjertsenLuef2024}.  The present paper uses the continuous-Zak obstruction and common-zero
theorem of de Dios Pont--Liehr--Taylor~\cite{deDiosLiehrTaylor2026} in
Subsection~\ref{sec:rectangular-common-zero}, and combines them with the
rational Zak matrix model to obtain the sharp arithmetic gap.

\displaypart{Tilings and $\varepsilon$-packings}
\section{Planar examples for the Tiling--$\varepsilon$--Packing theorem}
\label{sec:planar-examples}

This section gives a graphical progression from common fundamental domains,
through ordinary tiling--packing pairs, to the buffered packing problem of
Theorem~\ref{thm:main-geometric}.  In every tiling or packing picture, the
unshifted fundamental domain is the unique dark blue copy.

\subsection{The simultaneous tiling problem}
The classical simultaneous tiling problem asks whether two lattices of the same
covolume admit a single measurable fundamental domain.  It appears in
\cite{HanWang2001,PfanderRashkovWang2012,DutkayHanJorgensenPicioroaga2013,GrepstadKolountzakis2026}.
Figure~\ref{fig:simultaneous-tiling-four-panels} begins with separate fundamental
domains for
\[
 L_1=2\Z\times2\Z,
 \qquad
 L_2=4\Z\times\Z,
\]
and then shows that the parallelogram spanned by $(0,1)$ and $(4,2)$ is a
fundamental domain for both lattices.

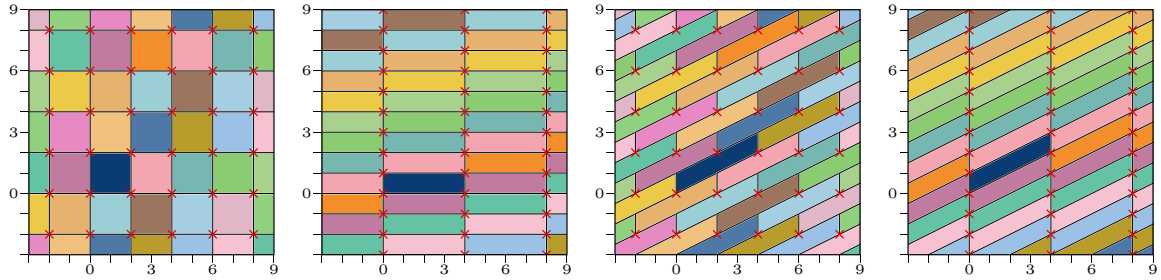
\begin{figure}[htbp]
\centering
\begin{minipage}[b]{\dimexpr(\textwidth-1.5cm)/4\relax}
\centering
\begin{tikzpicture}[x=.27cm,y=.27cm]
  \begin{scope}
  \clip (-3,-3) rectangle (9,9);
  \foreach \m in {-6,-5,...,6}{\foreach \n in {-8,-7,...,8}{
    \PickTileColor{7*\m+11*\n}
    \ifnum\m=0\ifnum\n=0\def\fillcol{TileDarkBlue}\fi\fi
    \path[fill=\fillcol,draw=white,line width=.55pt] (2*\m,2*\n) rectangle ++(2,2);
    \draw[black,line width=.18pt] (2*\m,2*\n) rectangle ++(2,2);
  }}
  \foreach \m in {-1,...,4}{\foreach \n in {-1,...,4}{\LatticeCross{2*\m}{2*\n}}}
  \end{scope}
  \draw[black,line width=.45pt] (-3,-3) rectangle (9,9);
  \DrawIntegerAxes{-3}{9}{-3}{9}{3}
  \MarkOrigin
\end{tikzpicture}
\end{minipage}\hfill%
\begin{minipage}[b]{\dimexpr(\textwidth-1.5cm)/4\relax}
\centering
\begin{tikzpicture}[x=.27cm,y=.27cm]
  \begin{scope}
  \clip (-3,-3) rectangle (9,9);
  \foreach \m in {-4,-3,...,4}{\foreach \n in {-12,-11,...,12}{
    \PickTileColor{13*\m+7*\n}
    \ifnum\m=0\ifnum\n=0\def\fillcol{TileDarkBlue}\fi\fi
    \path[fill=\fillcol,draw=white,line width=.55pt] (4*\m,\n) rectangle ++(4,1);
    \draw[black,line width=.18pt] (4*\m,\n) rectangle ++(4,1);
  }}
  \foreach \m in {0,...,2}{\foreach \n in {-3,...,9}{\LatticeCross{4*\m}{\n}}}
  \end{scope}
  \draw[black,line width=.45pt] (-3,-3) rectangle (9,9);
  \DrawIntegerAxes{-3}{9}{-3}{9}{3}
  \MarkOrigin
\end{tikzpicture}
\end{minipage}\hfill%
\begin{minipage}[b]{\dimexpr(\textwidth-1.5cm)/4\relax}
\centering
\begin{tikzpicture}[x=.27cm,y=.27cm]
  \begin{scope}
  \clip (-3,-3) rectangle (9,9);
  \foreach \m in {-6,-5,...,6}{\foreach \n in {-8,-7,...,8}{
    \PickTileColor{7*\m+11*\n}
    \ifnum\m=0\ifnum\n=0\def\fillcol{TileDarkBlue}\fi\fi
    \begin{scope}[shift={(2*\m,2*\n)}]
      \path[fill=\fillcol,draw=white,line width=.55pt] (0,0)--(0,1)--(4,3)--(4,2)--cycle;
      \draw[black,line width=.18pt] (0,0)--(0,1)--(4,3)--(4,2)--cycle;
    \end{scope}
  }}
  \foreach \m in {-1,...,4}{\foreach \n in {-1,...,4}{\LatticeCross{2*\m}{2*\n}}}
  \end{scope}
  \draw[black,line width=.45pt] (-3,-3) rectangle (9,9);
  \DrawIntegerAxes{-3}{9}{-3}{9}{3}
  \MarkOrigin
\end{tikzpicture}
\end{minipage}\hfill%
\begin{minipage}[b]{\dimexpr(\textwidth-1.5cm)/4\relax}
\centering
\begin{tikzpicture}[x=.27cm,y=.27cm]
  \begin{scope}
  \clip (-3,-3) rectangle (9,9);
  \foreach \m in {-4,-3,...,4}{\foreach \n in {-12,-11,...,12}{
    \PickTileColor{13*\m+7*\n}
    \ifnum\m=0\ifnum\n=0\def\fillcol{TileDarkBlue}\fi\fi
    \begin{scope}[shift={(4*\m,\n)}]
      \path[fill=\fillcol,draw=white,line width=.55pt] (0,0)--(0,1)--(4,3)--(4,2)--cycle;
      \draw[black,line width=.18pt] (0,0)--(0,1)--(4,3)--(4,2)--cycle;
    \end{scope}
  }}
  \foreach \m in {0,...,2}{\foreach \n in {-3,...,9}{\LatticeCross{4*\m}{\n}}}
  \end{scope}
  \draw[black,line width=.45pt] (-3,-3) rectangle (9,9);
  \DrawIntegerAxes{-3}{9}{-3}{9}{3}
  \MarkOrigin
\end{tikzpicture}
\end{minipage}
\caption{The simultaneous tiling problem for two lattices of covolume $4$.  From left to right: the square $[0,2)^2$ tiled by $2\Z\times2\Z$; the rectangle $[0,4)\times[0,1)$ tiled by $4\Z\times\Z$; and the common parallelogram spanned by $(0,1)$ and $(4,2)$ tiled by the two respective lattices.  Here and in the following, the elements of lattices are marked by red crosses.}
\label{fig:simultaneous-tiling-four-panels}
\end{figure}

We next describe a finite cube procedure that captures the common-domain
problem for integer lattices.  Let $T,P\subset\Z^d$ satisfy
\begin{equation}\label{eq:commensurable-normalization}
 T,P\subset\Z^d,
 \qquad
 T+P=\Z^d.
\end{equation}
Put
\[
 m=[\Z^d:T],
 \qquad
 n=[\Z^d:P],
\]
and enumerate the quotient groups by
\begin{equation}\label{eq:quotient-enumeration}
 \Z^d/T=\{R_0,R_1,\ldots,R_{m-1}\},
 \qquad R_0=T,
\end{equation}
\begin{equation}\label{eq:column-enumeration}
 \Z^d/P=\{Q_0,Q_1,\ldots,Q_{n-1}\},
 \qquad Q_0=P.
\end{equation}
Assume first that $T$ and $P$ have the same covolume.  For sublattices of
$\Z^d$, this is equivalent to $m=n$, because
$\covol(T)=[\Z^d:T]\covol(\Z^d)=m$ and similarly $\covol(P)=n$.  Let $F=[0,1)^d$.  To obtain a common fundamental domain
built from unit cubes, choose addresses $z_0,\ldots,z_{m-1}\in\Z^d$ so that
the classes $z_i+T$ run through $\Z^d/T$ exactly once and the classes
$z_i+P$ run through $\Z^d/P$ exactly once.  Then
\begin{equation}\label{eq:section3-hard-domain}
 \Omega(Z)=\bigcup_{i=0}^{m-1}(F+z_i)
\end{equation}
is a fundamental domain for both lattices.  Indeed, the half-open cubes
$F+z$, $z\in\Z^d$, partition $\R^d$, and the selected addresses give exactly
one representative in every coset modulo either lattice.  Only the two
residue classes of an address matter: if $h\in T\cap P$, then replacing
$z_i$ by $z_i+h$ changes neither class and therefore preserves both tiling
properties.

A configuration table records these data.  The sets $R_i$ index its rows,
the sets $Q_j$ index its columns, and the cell in row $R_i$ and column $Q_j$
is the address class
\begin{equation}\label{eq:group-address-class}
 \mathcal A_{i,j}=R_i\cap Q_j
 =\{z\in\Z^d:z+T=R_i,\ z+P=Q_j\}.
\end{equation}
Every address class is nonempty.  Indeed, if $R_i=a+T$ and $Q_j=b+P$, write
$b-a=t+p$ with $t\in T$ and $p\in P$; then $a+t=b-p$ belongs to
$R_i\cap Q_j$.  Moreover, for any $z_{i,j}\in\mathcal A_{i,j}$,
\begin{equation}\label{eq:address-class-coset}
 \mathcal A_{i,j}=z_{i,j}+(T\cap P).
\end{equation}
Thus every cell contains infinitely many actual addresses.  In the figures, a
cell may be divided into four subcells to display four chosen representatives;
subcell positions are read in the order lower left, lower right, upper left,
upper right.  The representatives used for each guiding lattice pair are
specified when that pair is introduced.

\begin{definition}\label{def:hard-configuration}
A \emph{hard configuration} is a finitely supported function
$\lambda:\Z^d\to\{0,1\}$ such that every row contains exactly one occupied
address,
\begin{equation}\label{eq:hard-row}
 \sum_{z\in R_i}\lambda(z)=1,
 \qquad 0\le i\le m-1,
\end{equation}
and every column contains at most one occupied address,
\begin{equation}\label{eq:hard-column}
 \sum_{z\in Q_j}\lambda(z)\le1,
 \qquad 0\le j\le n-1.
\end{equation}
Equivalently, the occupied addresses meet every row once and every column at
most once.  When $m=n$, every column is occupied exactly once.
\end{definition}

For a concrete example, take
\[
 T=4\Z\times\Z,
 \qquad
 P=\Z\times4\Z.
\]
The table in Figure~\ref{fig:four-four-hard-common} uses the non-diagonal
matching
\[
 R_0\leftrightarrow Q_1,
 \quad R_1\leftrightarrow Q_3,
 \quad R_2\leftrightarrow Q_0,
 \quad R_3\leftrightarrow Q_2.
\]
In the subcell order fixed above, the four displayed representatives are
$(0,1)$, $(5,3)$, $(2,4)$, and $(7,6)$.  They occur, respectively, in the lower-left,
lower-right, upper-left, and upper-right positions of their occupied main
cells.  Subtracting suitable vectors of $T\cap P=4\Z\times4\Z$ gives the
compact representatives
\begin{equation}\label{eq:four-four-addresses}
 Z_{4,4}=\{(0,1),(1,3),(2,0),(3,2)\}.
\end{equation}
Thus
\[
 \Omega_{4,4}=\bigcup_{z\in Z_{4,4}}\bigl([0,1)^2+z\bigr)
\]
is a common fundamental domain.

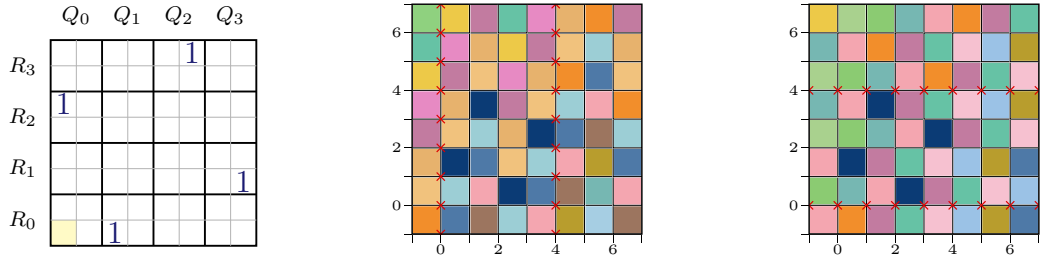
\begin{figure}[htbp]
\centering
\begin{minipage}[b]{.30\textwidth}
\centering
\begin{tikzpicture}[x=.34cm,y=.34cm]
  \fill[yellow!28] (0,0) rectangle (1,1);
  \draw[line width=.72pt] (0,0) rectangle (8,8);
  \foreach \x in {2,4,6}{\draw[line width=.72pt] (\x,0)--(\x,8);}
  \foreach \y in {2,4,6}{\draw[line width=.72pt] (0,\y)--(8,\y);}
  \foreach \x in {1,3,5,7}{\draw[gray!60,line width=.27pt] (\x,0)--(\x,8);}
  \foreach \y in {1,3,5,7}{\draw[gray!60,line width=.27pt] (0,\y)--(8,\y);}
  \foreach \j in {0,...,3}{\node[above,font=\scriptsize] at (2*\j+1,8.18) {$Q_{\j}$};}
  \foreach \i in {0,...,3}{\node[left,font=\scriptsize] at (-.14,2*\i+1) {$R_{\i}$};}
  \node[text=MidnightBlue,font=\large] at (2.5,.5) {$1$};
  \node[text=MidnightBlue,font=\large] at (7.5,2.5) {$1$};
  \node[text=MidnightBlue,font=\large] at (.5,5.5) {$1$};
  \node[text=MidnightBlue,font=\large] at (5.5,7.5) {$1$};
\end{tikzpicture}

\end{minipage}\hfill
\begin{minipage}[b]{.30\textwidth}
\centering
\begin{tikzpicture}[x=.38cm,y=.38cm]
  \begin{scope}
  \clip (-1,-1) rectangle (7,7);
  \foreach \a in {-3,-2,...,3}{
    \foreach \b in {-8,-7,...,8}{
      \PickTileColor{7*\a+11*\b}
      \ifnum\a=0\ifnum\b=0\def\fillcol{TileDarkBlue}\fi\fi
      \foreach \x/\y in {0/1,1/3,2/0,3/2}{
        \path[fill=\fillcol,draw=white,line width=.62pt]
          ({\x+4*\a},{\y+\b}) rectangle ++(1,1);
        \draw[black,line width=.16pt]
          ({\x+4*\a},{\y+\b}) rectangle ++(1,1);
      }
    }
  }
  \foreach \a in {0,1}{\foreach \b in {-1,...,7}{\LatticeCross{4*\a}{\b}}}
  \end{scope}
  \draw[black,line width=.5pt] (-1,-1) rectangle (7,7);
  \DrawIntegerAxes{-1}{7}{-1}{7}{2}
  \MarkOrigin
\end{tikzpicture}

\end{minipage}\hfill
\begin{minipage}[b]{.30\textwidth}
\centering
\begin{tikzpicture}[x=.38cm,y=.38cm]
  \begin{scope}
  \clip (-1,-1) rectangle (7,7);
  \foreach \a in {-8,-7,...,8}{
    \foreach \b in {-3,-2,...,3}{
      \PickTileColor{13*\a+7*\b}
      \ifnum\a=0\ifnum\b=0\def\fillcol{TileDarkBlue}\fi\fi
      \foreach \x/\y in {0/1,1/3,2/0,3/2}{
        \path[fill=\fillcol,draw=white,line width=.62pt]
          ({\x+\a},{\y+4*\b}) rectangle ++(1,1);
        \draw[black,line width=.16pt]
          ({\x+\a},{\y+4*\b}) rectangle ++(1,1);
      }
    }
  }
  \foreach \a in {-1,...,7}{\foreach \b in {0,1}{\LatticeCross{\a}{4*\b}}}
  \end{scope}
  \draw[black,line width=.5pt] (-1,-1) rectangle (7,7);
  \DrawIntegerAxes{-1}{7}{-1}{7}{2}
  \MarkOrigin
\end{tikzpicture}

\end{minipage}
\caption{A cube-built common fundamental domain for $T=4\Z\times\Z$ and $P=\Z\times4\Z$.  From left to right: a non-diagonal $4\times4$ hard configuration, with four sample addresses displayed in four different subcells; the $T$-tiling; and the $P$-tiling.  The displayed addresses reduce modulo $T\cap P$ to the compact representatives in~\eqref{eq:four-four-addresses}, whose unit-cube union is the dark-blue set.}
\label{fig:four-four-hard-common}
\end{figure}

\subsection{The tiling and packing problem}
Suppose now that $T,P\subset\Z^d$, $T+P=\Z^d$, and
$\covol(T)<\covol(P)$, equivalently $m<n$.  A hard configuration uses every
row and only $m$ of the $n$ columns.  The cube union~\eqref{eq:section3-hard-domain} is a
$T$-fundamental domain and packs with $P$.

For a planar example take $T=4\Z\times\Z$ and
$P=\begin{psmallmatrix}1&0\\1&5\end{psmallmatrix}\Z^2$.  Then
$T+P=\Z^2$, $m=4$, and $n=5$.  The rows are the first-coordinate classes
modulo $4$, while the columns are given by $y-x$ modulo $5$.  Choose the
addresses
\begin{equation}\label{eq:four-five-addresses}
 Z_{4,5}=\{(0,0),(1,0),(2,0),(3,0)\},
 \qquad \Omega_{4,5}=[0,4)\times[0,1).
\end{equation}
They occupy $R_0\cap Q_0$, $R_1\cap Q_4$, $R_2\cap Q_3$, and
$R_3\cap Q_2$.  All four entries are displayed in the lower-left address
subcell, and one column remains unoccupied.  The set $\Omega_{4,5}$ tiles by $T$ and packs
by $P$; the shear in $P$ makes the vacant regions stagger rather than line
up in one horizontal strip.

\begin{figure}[htbp]
\centering
\begin{minipage}[b]{.30\textwidth}
\centering
\begin{tikzpicture}[x=.31cm,y=.31cm]
  \fill[yellow!28] (0,0) rectangle (1,1);
  \draw[line width=.72pt] (0,0) rectangle (10,8);
  \foreach \x in {2,4,6,8}{\draw[line width=.72pt] (\x,0)--(\x,8);}
  \foreach \y in {2,4,6}{\draw[line width=.72pt] (0,\y)--(10,\y);}
  \foreach \x in {1,3,5,7,9}{\draw[gray!60,line width=.27pt] (\x,0)--(\x,8);}
  \foreach \y in {1,3,5,7}{\draw[gray!60,line width=.27pt] (0,\y)--(10,\y);}
  \foreach \j in {0,...,4}{\node[above,font=\scriptsize] at (2*\j+1,8.18) {$Q_{\j}$};}
  \foreach \i in {0,...,3}{\node[left,font=\scriptsize] at (-.14,2*\i+1) {$R_{\i}$};}
  \node[text=MidnightBlue,font=\large] at (.5,.5) {$1$};
  \node[text=MidnightBlue,font=\large] at (8.5,2.5) {$1$};
  \node[text=MidnightBlue,font=\large] at (6.5,4.5) {$1$};
  \node[text=MidnightBlue,font=\large] at (4.5,6.5) {$1$};
\end{tikzpicture}
\end{minipage}\hfill
\begin{minipage}[b]{.30\textwidth}
\centering
\begin{tikzpicture}[x=.37cm,y=.37cm]
  \begin{scope}
  \clip (-1,-1) rectangle (7,7);
  \foreach \a in {-3,-2,...,3}{\foreach \b in {-8,-7,...,8}{
    \PickTileColor{7*\a+11*\b}
    \ifnum\a=0\ifnum\b=0\def\fillcol{TileDarkBlue}\fi\fi
    \path[fill=\fillcol,draw=white,line width=.62pt] (4*\a,\b) rectangle ++(4,1);
    \draw[black,line width=.16pt] (4*\a,\b) rectangle ++(4,1);
  }}
  \foreach \a in {0,1}{\foreach \b in {-1,...,7}{\LatticeCross{4*\a}{\b}}}
  \end{scope}
  \draw[black,line width=.5pt] (-1,-1) rectangle (7,7);
  \DrawIntegerAxes{-1}{7}{-1}{7}{2}
  \MarkOrigin
\end{tikzpicture}
\end{minipage}\hfill
\begin{minipage}[b]{.30\textwidth}
\centering
\begin{tikzpicture}[x=.37cm,y=.37cm]
  \begin{scope}
  \clip (-1,-1) rectangle (7,7);
  \foreach \a in {-9,-8,...,9}{\foreach \b in {-4,-3,...,4}{
    \PickTileColor{13*\a+7*\b}
    \ifnum\a=0\ifnum\b=0\def\fillcol{TileDarkBlue}\fi\fi
    \path[fill=\fillcol,draw=white,line width=.62pt]
      (\a,{\a+5*\b}) rectangle ++(4,1);
    \draw[black,line width=.16pt] (\a,{\a+5*\b}) rectangle ++(4,1);
  }}
  \foreach \a in {-9,...,9}{\foreach \b in {-4,...,4}{\LatticeCross{\a}{\a+5*\b}}}
  \end{scope}
  \draw[black,line width=.5pt] (-1,-1) rectangle (7,7);
  \DrawIntegerAxes{-1}{7}{-1}{7}{2}
  \MarkOrigin
\end{tikzpicture}
\end{minipage}
\caption{The tiling and packing problem for $T=4\Z\times\Z$ and $P=\begin{psmallmatrix}1&0\\1&5\end{psmallmatrix}\Z^2$.  From left to right: the $4\times5$ hard configuration determined by the addresses in~\eqref{eq:four-five-addresses}; the resulting $T$-tiling by the dark-blue set $\Omega_{4,5}=[0,4)\times[0,1)$; and its staggered $P$-packing.}
\label{fig:four-five-hard-packing}
\end{figure}
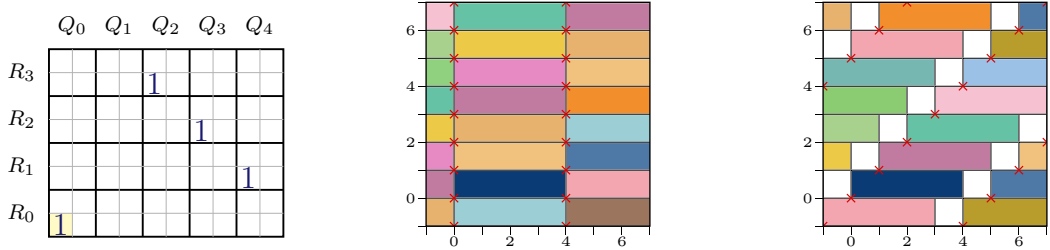

\subsection{The tiling and $\varepsilon$-packing problem}
The stronger goal is to find a $T$-fundamental domain $\Omega$ and
$\varepsilon>0$ such that $\Omega+B_\varepsilon$ packs with $P$, so that
nontrivial packing translates are separated by a positive gap.

\begin{figure}[htbp]
\centering
\begin{minipage}[b]{.48\textwidth}
  \centering
  \RasterWithZTwo{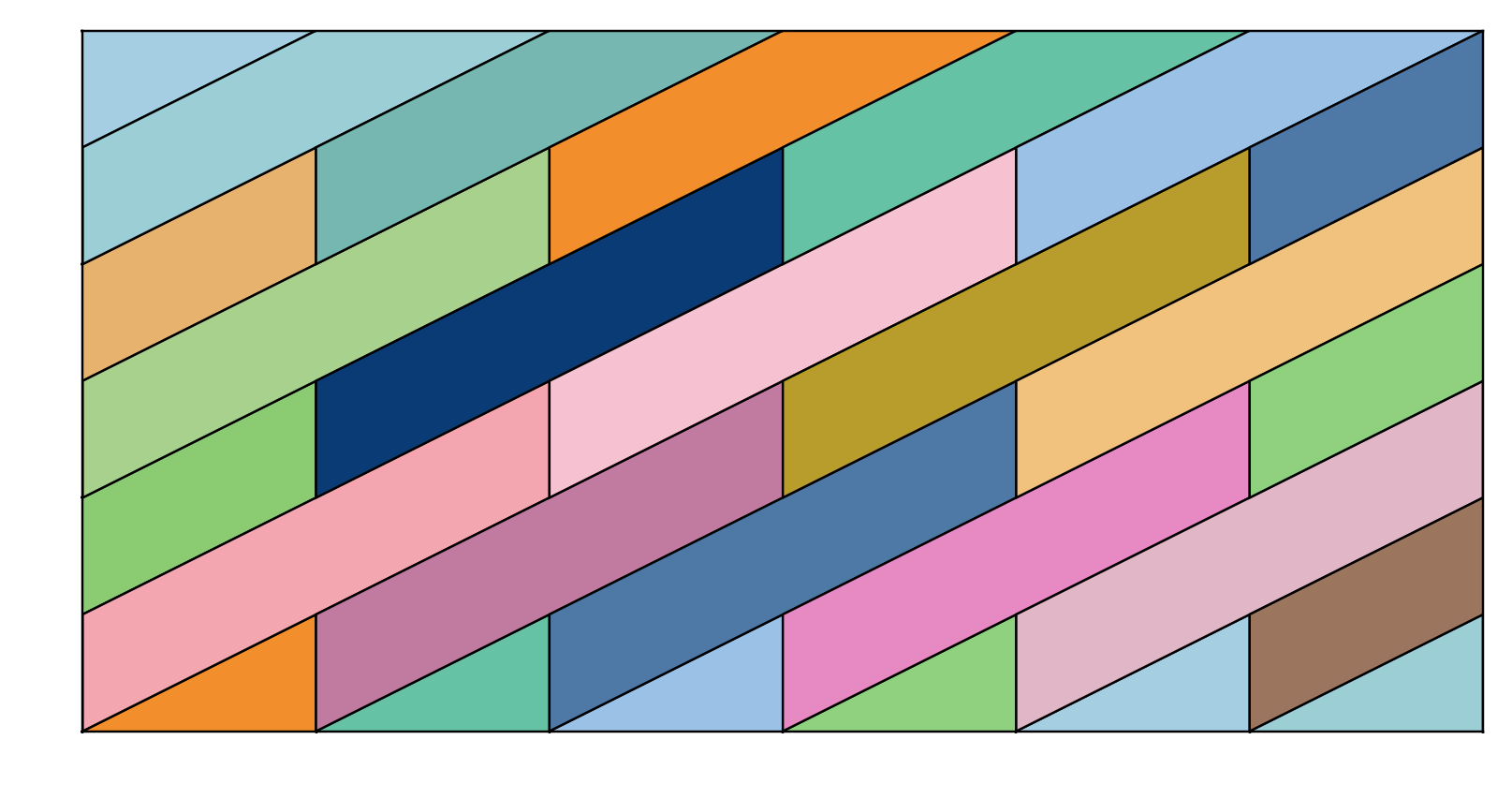}{.208566}{.377296}{.154386}{.292302}{-1}{5}{-1}{2}
\end{minipage}\hfill
\begin{minipage}[b]{.48\textwidth}
  \centering
  \RasterWithThreeZxZ{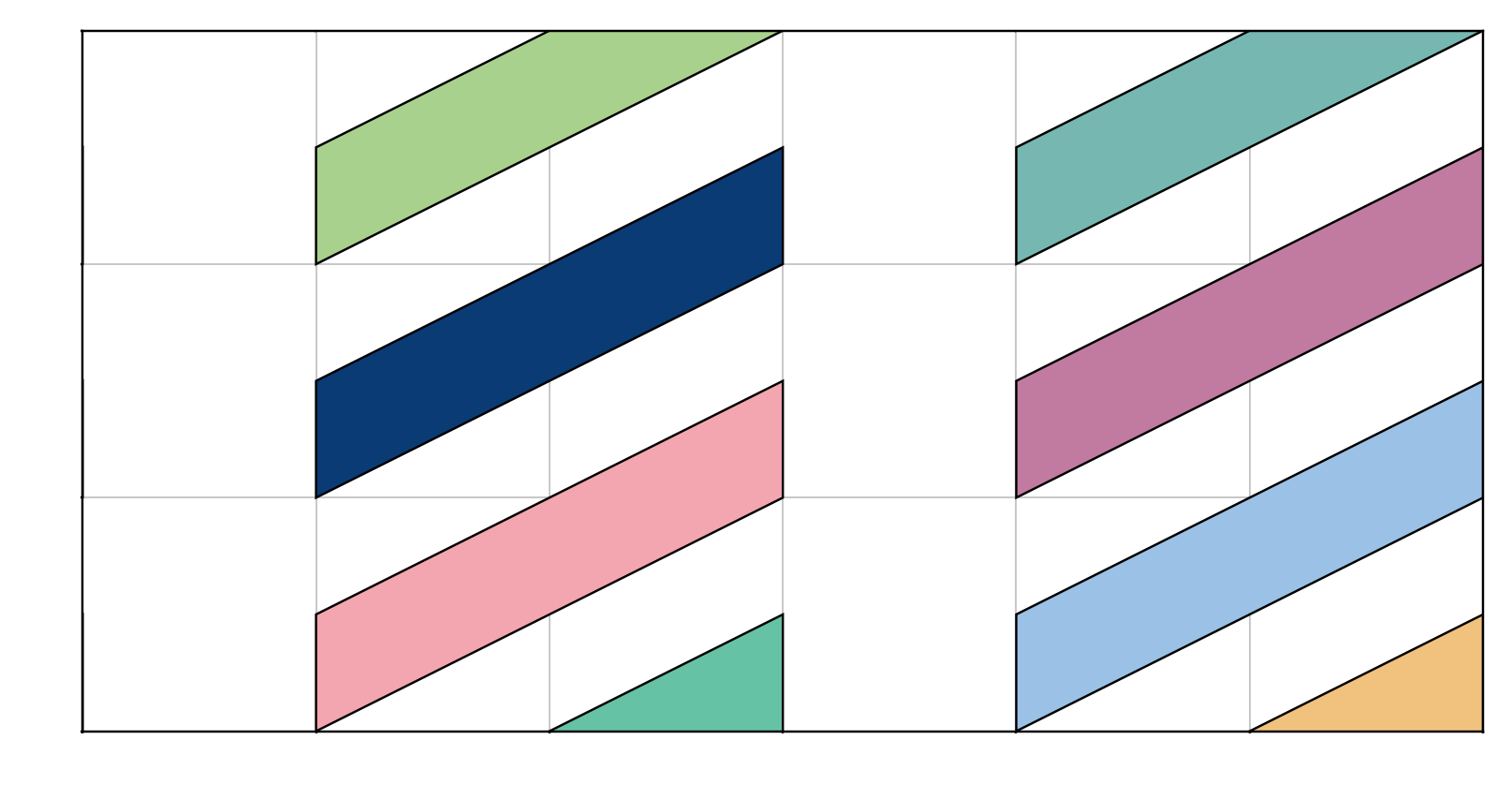}{.208566}{.377296}{.154386}{.292302}{-1}{5}{-1}{2}
\end{minipage}
\caption{A direct buffered construction for $T=\Z^2$ and
$P=3\Z\times\Z$.  The parallelogram is spanned by $(0,1/2)$ and $(2,1)$.
Left: the $T$-tiling.  Right: the buffered $P$-packing.}
\label{fig:parallelogram-3zxz}
\end{figure}

A second solution to the same lattice pair is obtained from the soft field.
Start from the unit square $F=[0,1)^2$, partition it into the finitely many
half-open polygonal pieces selected by the first-hit construction in
Subsection~\ref{subsec:worked-example-Z2-3ZxZ}, and translate those pieces by
integer vectors.  Figure~\ref{fig:one-three-two-solutions} displays this
cube-derived domain for the alternative first-hit enumeration
$((1,0),(0,-1),(-1,-1),(0,0),(-1,0))$.

\begin{figure}[htbp]
\centering
\begin{minipage}[b]{.48\textwidth}
  \centering
  \RasterWithZTwo{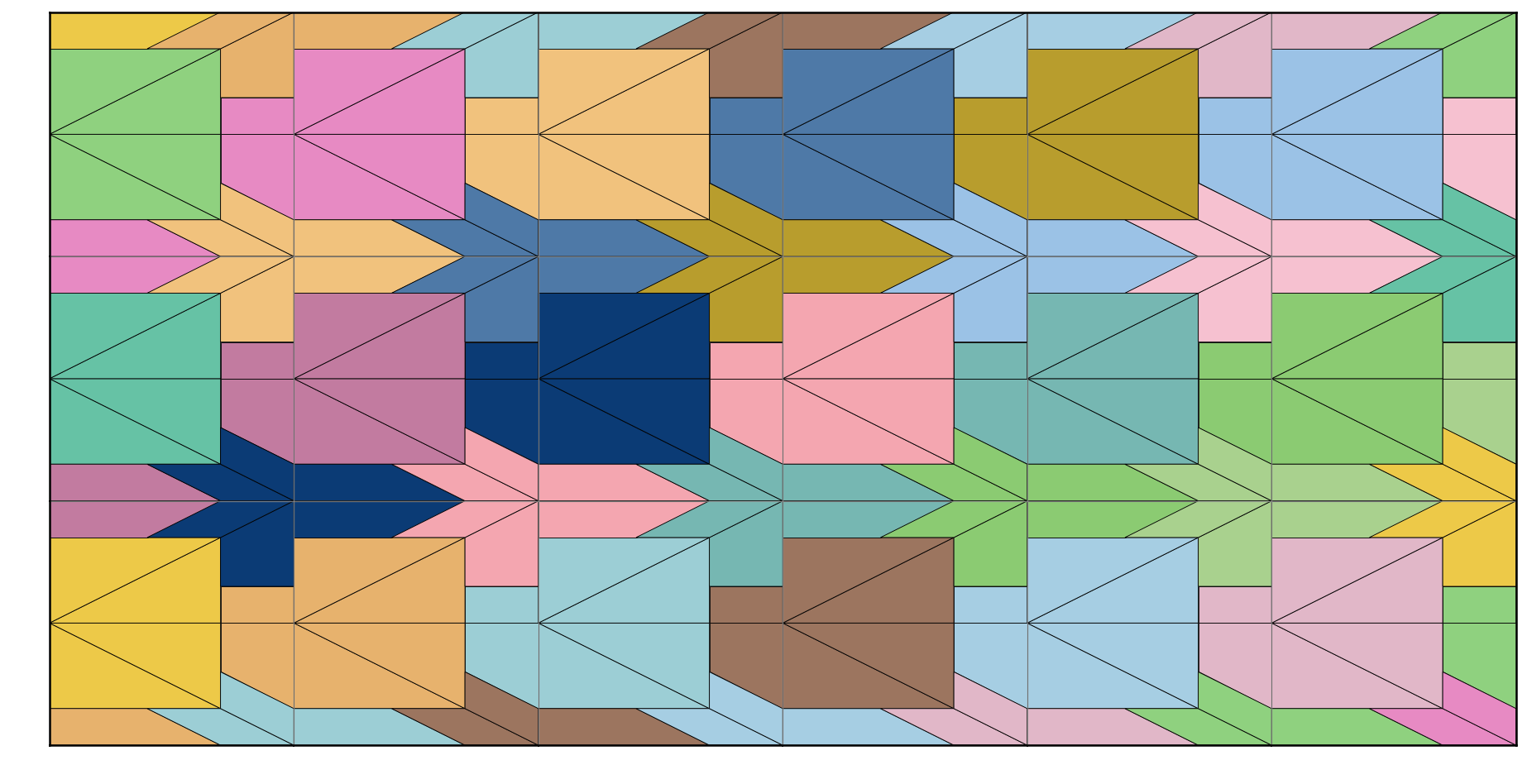}{.192579}{.361111}{.160076}{.311533}{-1}{5}{-1}{2}
\end{minipage}\hfill
\begin{minipage}[b]{.48\textwidth}
  \centering
  \RasterWithThreeZxZ{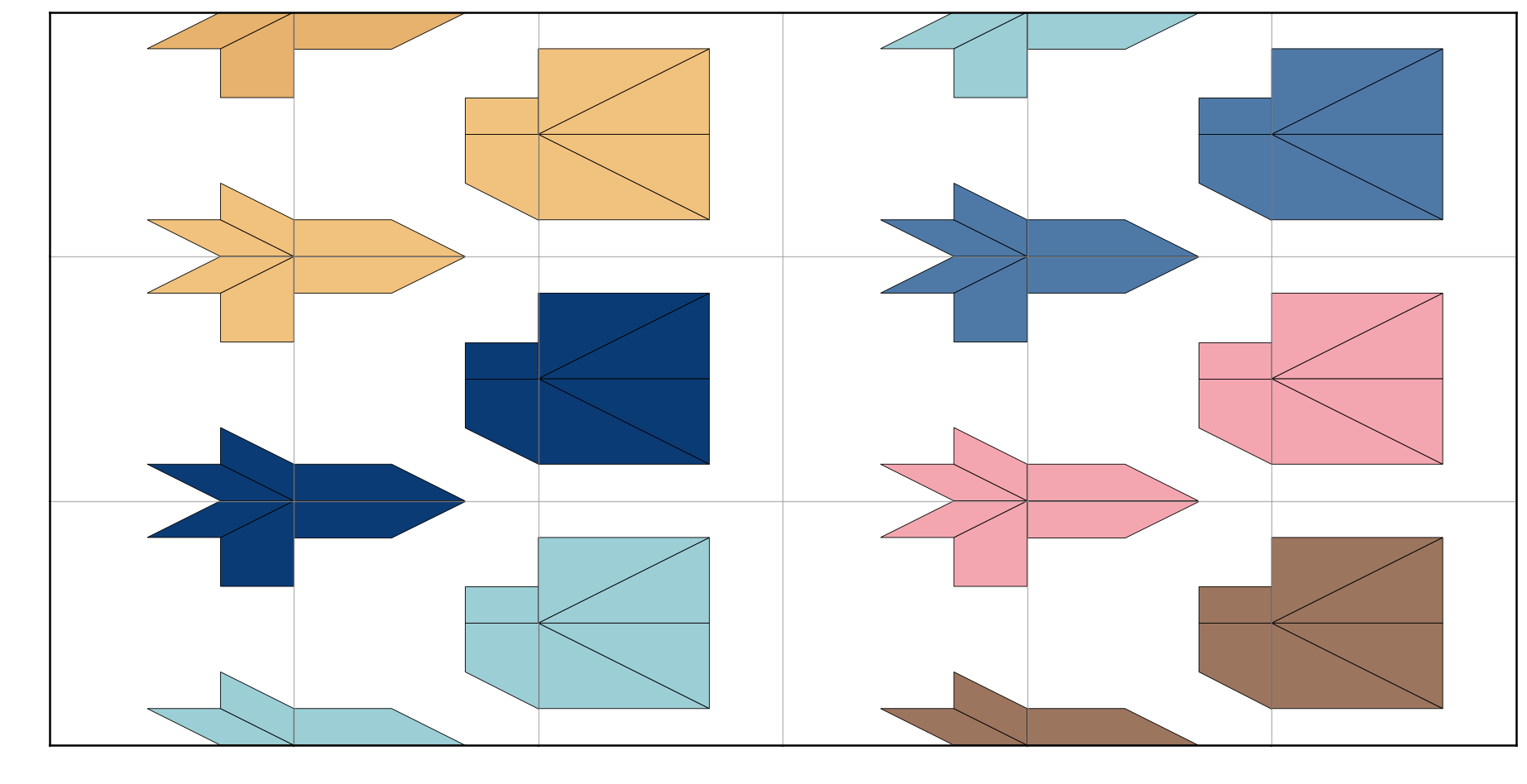}{.192579}{.361111}{.160076}{.311533}{-1}{5}{-1}{2}
\end{minipage}
\caption{The cube-derived solution of the $\Z^2$--$3\Z\times\Z$ problem obtained from the alternative first-hit enumeration above.  From left to right: the exact tiling by $\Z^2$ and the buffered packing by $3\Z\times\Z$.  The detailed construction in Subsection~\ref{subsec:worked-example-Z2-3ZxZ} retains the standard enumeration.}
\label{fig:one-three-two-solutions}
\end{figure}

The analogous buffered problem cannot be solved for
\[
 T=\Z^2,
 \qquad
 P=\Z\times2\Z.
\]
This obstruction appears in the geometric Gabor-frame work of Pfander,
Rashkov, and Wang~\cite{PfanderRashkovWang2012}, and it also follows from
the necessity theorem proved later in this paper.  The best corresponding
parallelogram picture gives only ordinary packing, with boundary contact.

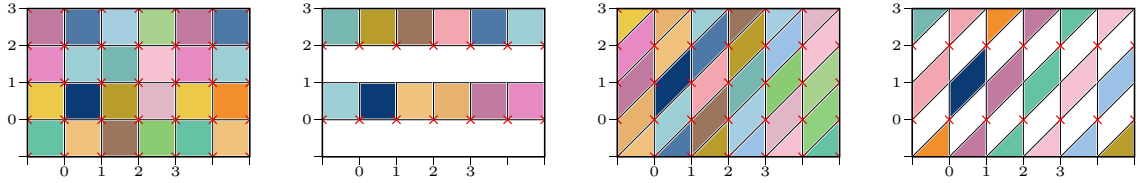
\begin{figure}[htbp]
\centering
\begin{minipage}[b]{.22\textwidth}
\vspace{0pt}\centering
\begin{tikzpicture}[x=.49cm,y=.49cm]
  \begin{scope}
  \clip (-1,-1) rectangle (5,3);
  \foreach \m in {-1,0,...,5}{\foreach \n in {-1,0,...,3}{
    \PickTileColor{5*\m+9*\n}
    \ifnum\m=0\ifnum\n=0\def\fillcol{TileDarkBlue}\fi\fi
    \path[fill=\fillcol,draw=white,line width=.9pt] (\m,\n) rectangle ++(1,1);
    \draw[black,line width=.25pt] (\m,\n) rectangle ++(1,1);
  }}
  \foreach \m in {-1,...,5}{\foreach \n in {-1,...,3}{\LatticeCross{\m}{\n}}}
  \end{scope}
  \draw[black,line width=.6pt] (-1,-1) rectangle (5,3);
  \DrawIntegerAxes{-1}{5}{-1}{3}{1}
  \MarkOrigin
\end{tikzpicture}
\end{minipage}\hfill
\begin{minipage}[b]{.22\textwidth}
\vspace{0pt}\centering
\begin{tikzpicture}[x=.49cm,y=.49cm]
  \begin{scope}
  \clip (-1,-1) rectangle (5,3);
  \foreach \m in {-1,0,...,5}{\foreach \n in {-1,0,1}{
    \PickTileColor{11*\m+5*\n}
    \ifnum\m=0\ifnum\n=0\def\fillcol{TileDarkBlue}\fi\fi
    \path[fill=\fillcol,draw=white,line width=.9pt] (\m,{2*\n}) rectangle ++(1,1);
    \draw[black,line width=.25pt] (\m,{2*\n}) rectangle ++(1,1);
  }}
  \foreach \m in {-1,...,5}{\foreach \n in {0,1}{\LatticeCross{\m}{2*\n}}}
  \end{scope}
  \draw[black,line width=.6pt] (-1,-1) rectangle (5,3);
  \DrawIntegerAxes{-1}{5}{-1}{3}{1}
  \MarkOrigin
\end{tikzpicture}
\end{minipage}\hfill
\begin{minipage}[b]{.22\textwidth}
\vspace{0pt}\centering
\begin{tikzpicture}[x=.49cm,y=.49cm]
  \begin{scope}
  \clip (-1,-1) rectangle (5,3);
  \foreach \m in {-1,0,...,5}{\foreach \n in {-2,-1,...,3}{
    \PickTileColor{7*\m+11*\n}
    \ifnum\m=0\ifnum\n=0\def\fillcol{TileDarkBlue}\fi\fi
    \begin{scope}[shift={(\m,\n)}]
      \path[fill=\fillcol,draw=white,line width=.9pt] (0,0)--(0,1)--(1,2)--(1,1)--cycle;
      \draw[black,line width=.25pt] (0,0)--(0,1)--(1,2)--(1,1)--cycle;
    \end{scope}
  }}
  \foreach \m in {-1,...,5}{\foreach \n in {-1,...,3}{\LatticeCross{\m}{\n}}}
  \end{scope}
  \draw[black,line width=.6pt] (-1,-1) rectangle (5,3);
  \DrawIntegerAxes{-1}{5}{-1}{3}{1}
  \MarkOrigin
\end{tikzpicture}
\end{minipage}\hfill
\begin{minipage}[b]{.22\textwidth}
\vspace{0pt}\centering
\begin{tikzpicture}[x=.49cm,y=.49cm]
  \begin{scope}
  \clip (-1,-1) rectangle (5,3);
  \foreach \m in {-1,0,...,5}{\foreach \n in {-1,0,1}{
    \PickTileColor{13*\m+7*\n}
    \ifnum\m=0\ifnum\n=0\def\fillcol{TileDarkBlue}\fi\fi
    \begin{scope}[shift={(\m,2*\n)}]
      \path[fill=\fillcol,draw=white,line width=.9pt] (0,0)--(0,1)--(1,2)--(1,1)--cycle;
      \draw[black,line width=.25pt] (0,0)--(0,1)--(1,2)--(1,1)--cycle;
    \end{scope}
  }}
  \foreach \m in {-1,...,5}{\foreach \n in {0,1}{\LatticeCross{\m}{2*\n}}}
  \end{scope}
  \draw[black,line width=.6pt] (-1,-1) rectangle (5,3);
  \DrawIntegerAxes{-1}{5}{-1}{3}{1}
  \MarkOrigin
\end{tikzpicture}
\end{minipage}
\caption{Two ordinary tiling--packing solutions for $T=\Z^2$ and $P=\Z\times2\Z$.  From left to right: the unit-square tiling and packing; and the tiling and packing by the parallelogram spanned by $(0,1)$ and $(1,1)$.  Both packings have boundary contact, so neither has a positive buffer.}
\label{fig:basic-tiling-packing}
\end{figure}

For other pairs at the sharp planar index gap, a geometric construction may
succeed without being described by one fixed hard table.  In the next
example the fundamental domain is cut into polyhedral pieces and different
pieces are translated by different lattice vectors.  Equivalently, the
associated hard choice varies piecewise with the fractional position.

\begin{figure}[H]
\centering
\begin{minipage}[b]{.48\textwidth}
  \centering
  \RasterWithThreeZxZ{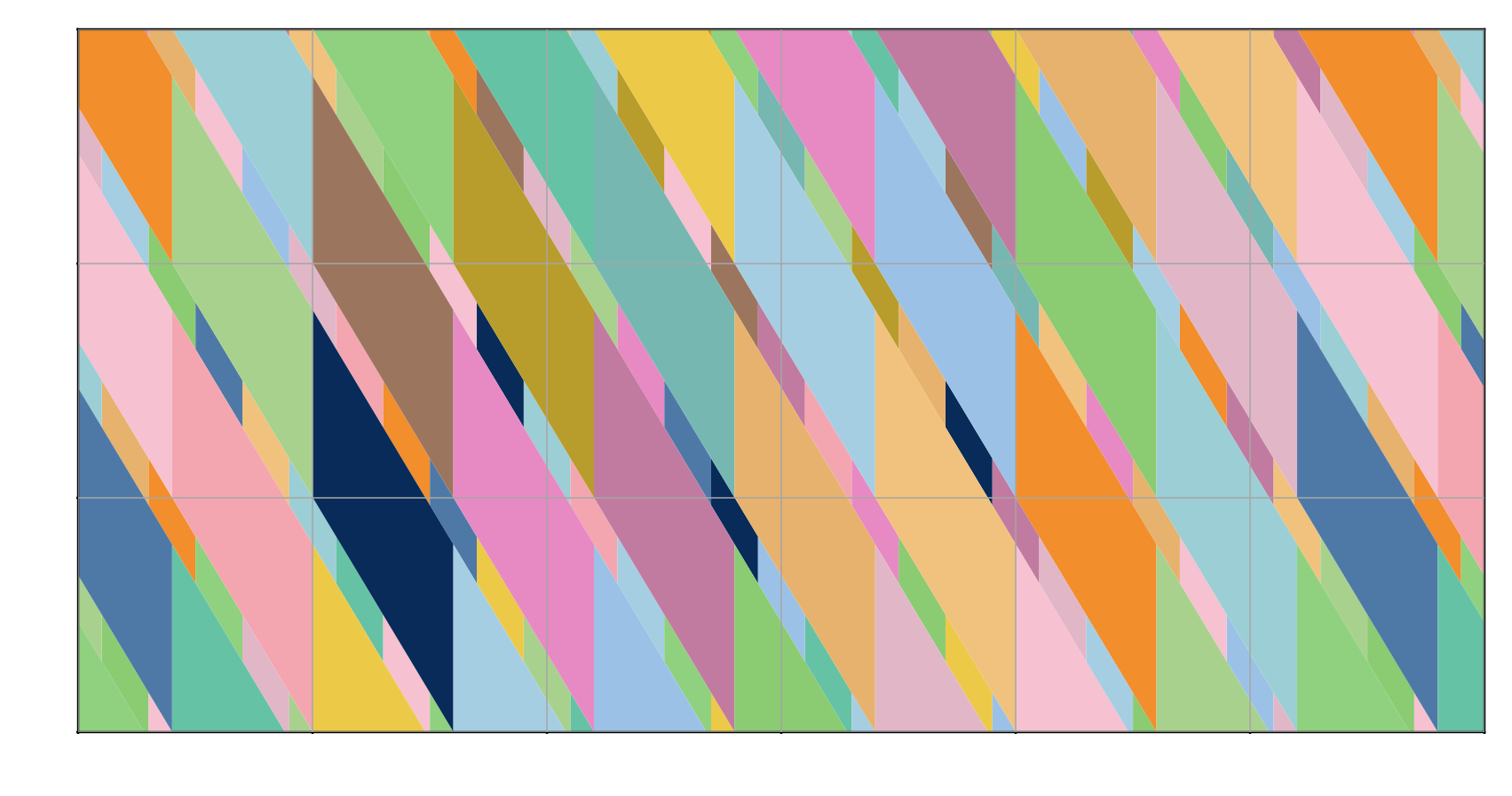}{.206205}{.374954}{.155118}{.294686}{-1}{5}{-1}{2}
\end{minipage}\hfill
\begin{minipage}[b]{.48\textwidth}
  \centering
  \RasterWithFiveZxZ{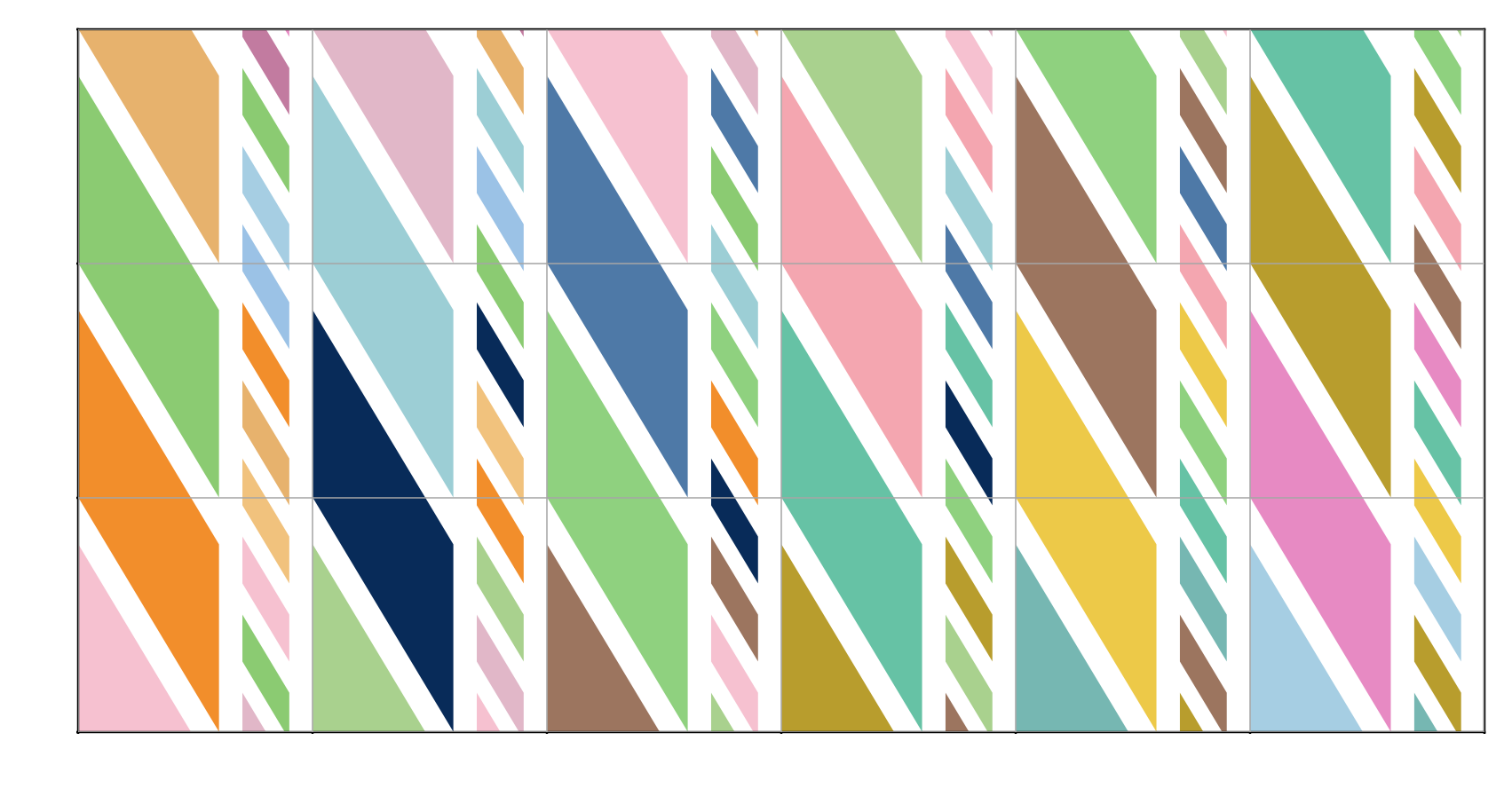}{.206205}{.374954}{.155118}{.294686}{-1}{5}{-1}{2}
\end{minipage}
\caption{The sharp planar integer example $T=3\Z\times\Z$ and
$P=5\Z\times\Z$, for which $m=3$ and $n=5=m+2$.  The displayed
cut-and-translate construction is not based on one fixed hard configuration.
Left: the $T$-tiling.  Right: the buffered $P$-packing.}
\label{fig:sharp-three-fifths}
\end{figure}

The main theorem classifies all lattice pairs for which the buffered problem
can be solved.  In the normalized commensurable case the criterion is
\[
 n\ge m+d,
\]
which becomes $n\ge m+2$ in the plane.  The remaining sections prove
sufficiency by replacing a fixed hard table with a continuous field of soft
configurations, then applying a first-hit selector to recover the actual
polyhedral fundamental domain; see Sections~\ref{sec:hard-configurations}
and~\ref{sec:soft-configurations}.

\subsection{The noncommensurable case}
The noncommensurable case has a different source of flexibility.  The
connected component of $\overline{T+P}$ can be subdivided into arbitrarily
many small source packets and target slots.  These labels refine the finite
data from the discrete quotient, so the configuration table can be enlarged
until the required surplus is available.  This is made precise in
Sections~\ref{sec:fully-dense-case} and~\ref{sec:nondiscrete-construction}.

As a concrete irrational example, let
\[
 T=\Z^2,
 \qquad
 P=\sqrt2\Z\times\Z.
\]
Then $\overline{T+P}=\R\times\Z$ and
$\covol(T)=1<\sqrt2=\covol(P)$.  Start from the $T$-fundamental
parallelogram spanned by $(0,1)$ and $(1,-1)$.  Keep the smaller main
parallelogram spanned by $(0,9/10)$ and $(1,-1)$, divide the remaining thin
strip into five vertical slices, and translate these slices by
$(4,0),(1,0),(5,0),(2,0),(6,0)\in T$.  The tiling property is unchanged,
and periodization by $P$ places the slices in separated gaps.

\begin{figure}[htbp]
\centering
\begin{minipage}[b]{.48\textwidth}
  \centering
  \RasterWithZTwo{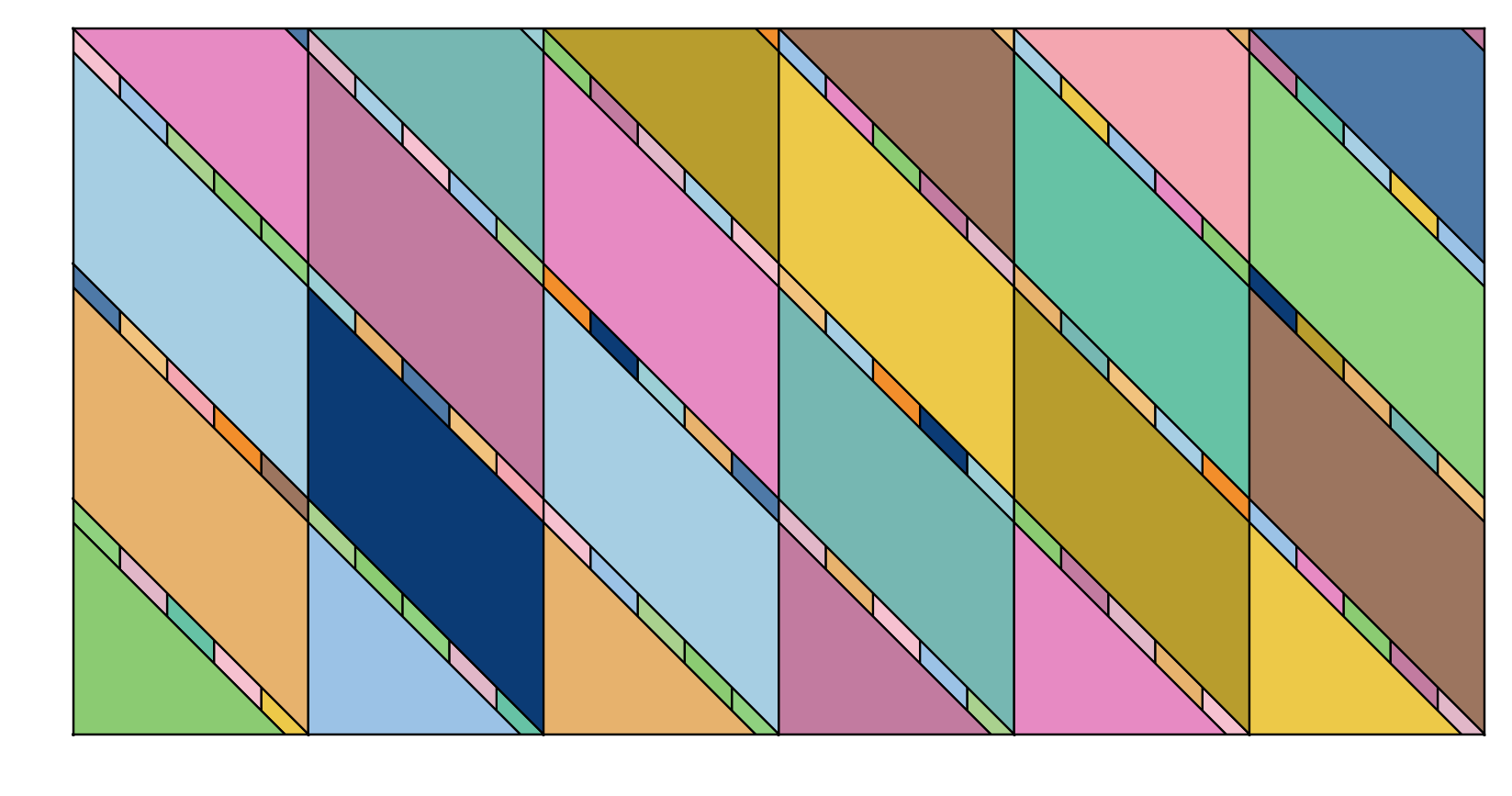}{.203805}{.373510}{.155535}{.295656}{-1}{5}{-1}{2}
\end{minipage}\hfill
\begin{minipage}[b]{.48\textwidth}
  \centering
  \RasterWithSqrtTwoZxZ{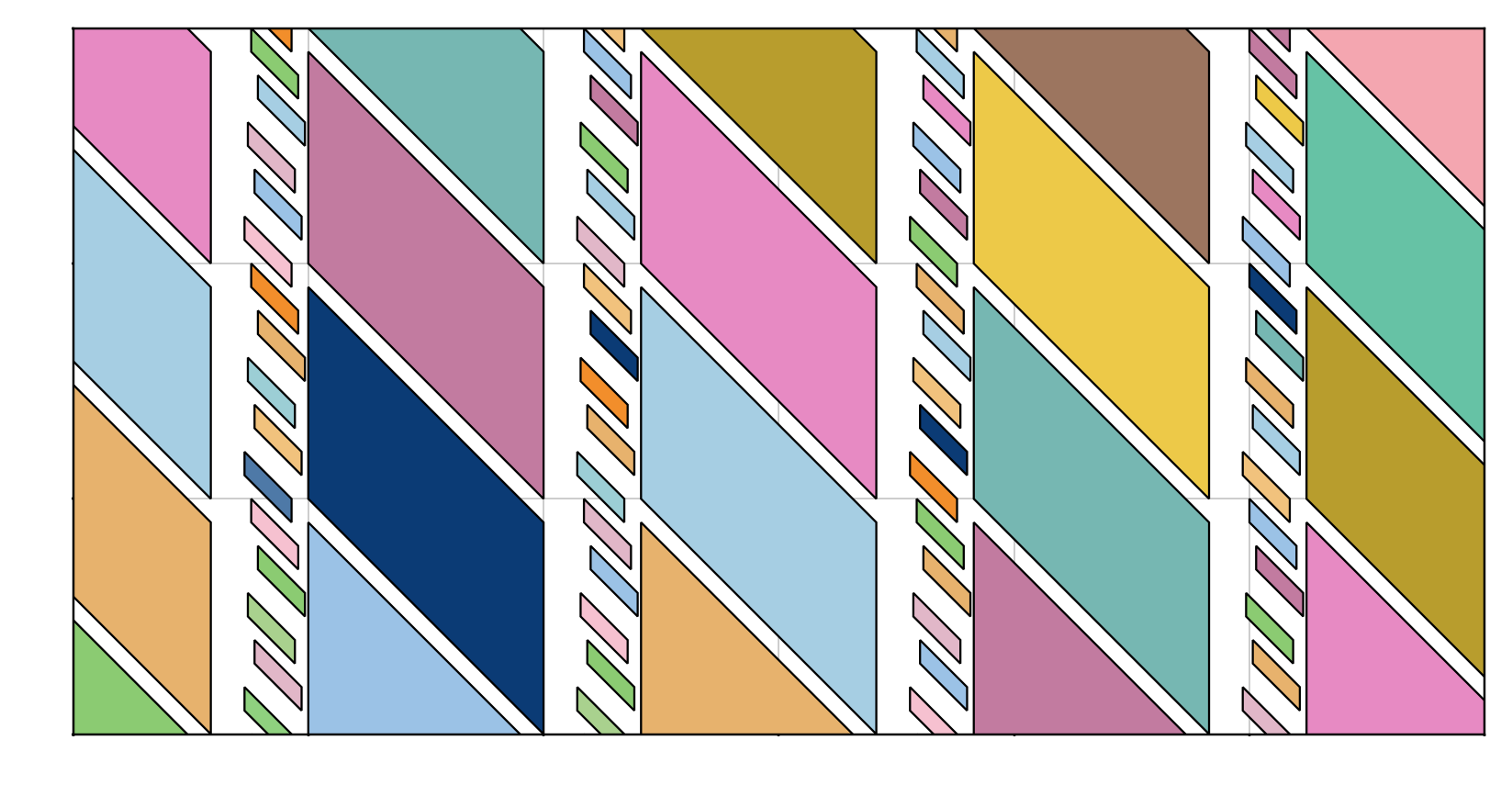}{.203805}{.373510}{.155535}{.295656}{-1}{5}{-1}{2}
\end{minipage}
\caption{A direct cut-and-translate construction for
$T=\Z^2$ and $P=\sqrt2\Z\times\Z$.  Left: the $T$-tiling.  Right: the
buffered $P$-packing.  The unshifted fundamental domain is dark blue.}
\label{fig:noncommensurable-sqrt2}
\end{figure}

In the commensurable case, the finite table is fixed by the quotient groups.
In the noncommensurable case, the connected factor supplies as many packet
and slot labels as needed; the resulting enlarged table is then filled by
the same equivariant configuration-space method used later.

\section{Soft configurations}\label{sec:hard-configurations}

We now overcome the limitations of hard configurations by allowing the
selected addresses to vary continuously and by replacing a single hard table
with a field of weighted tables.

\subsection{Partitions of unity generating tiling and
$\varepsilon$-packing sets $\Omega$}
Throughout this subsection, $\varphi$ is assumed nonnegative, continuous,
and compactly supported.  Continuity makes its positive superlevel sets
closed, while compact support makes them compact and ensures that only
finitely many translates meet a fixed fundamental parallelepiped.  We call
$\varphi$ a $T$-partition of unity when its $T$-translates are locally finite and
\begin{equation}\label{eq:roadmap-partition-separation}
 \sum_{t\in T}\varphi(x+t)=1
 \quad (x\in\R^d).
\end{equation}
The goal of Section~\ref{sec:soft-configurations} is to construct such a
function that is compactly supported and continuous piecewise affine, and
also satisfies
\begin{equation}\label{eq:roadmap-pointwise-separation}
 \varphi(x)\varphi(x-p)=0,
 \qquad p\in P\setminus\{0\}.
\end{equation}
The partition-of-unity identity implies $\{\varphi>0\}+T=\R^d$.  The second
identity gives pointwise disjointness of nontrivial $P$-translates.  To obtain
a positive packing buffer we pass to a compact positive superlevel set.

Write $T=A\Z^d$ and let
\[
 F=A[0,1)^d
\]
be the standard half-open $T$-fundamental parallelepiped.  Since
$\supp\varphi$ is compact and the sets $F+t$, $t\in T$, partition $\R^d$,
there are $t_1,\ldots,t_N\in T$ such that
\[
 \supp\varphi\subset\bigcup_{j=1}^N(F+t_j).
\]
Consequently, for every $x\in F$,
\begin{equation}\label{eq:finite-roadmap-partition}
 1=\sum_{t\in T}\varphi(x+t)
  =\sum_{j=1}^N\varphi(x+t_j).
\end{equation}
Therefore, at least one term contributes at least $1/N$ to the sum.  Put
\[
 K=\{y\in\R^d:\varphi(y)\ge 1/N\},
 \qquad
 D_j=\{x\in F:x+t_j\in K\}.
\]
The sets $D_1,\ldots,D_N$ cover $F$.  Since $\varphi$ is continuous and
piecewise affine on a finite polyhedral subdivision of its compact support,
$K$ is a compact finite union of polyhedra.

Choose the first admissible translate in the fixed order by setting
\begin{equation}\label{eq:finite-roadmap-first-hit}
 E_1=D_1,
 \qquad
 E_j=D_j\setminus\bigcup_{k<j}D_k\quad (2\le j\le N),
 \qquad
 \Omega=\bigcup_{j=1}^N(E_j+t_j).
\end{equation}
The sets $E_j$ form a disjoint partition of $F$.  Hence every $T$-coset has
exactly one representative in $\Omega$, so $\Omega+T$ is an exact tiling.
After refining the finitely many polyhedral subdivisions on which the
functions $x\mapsto\varphi(x+t_j)$ are affine, each $E_j$ is a finite union
of half-open polyhedral pieces.  Thus $\Omega$ is a finite union of bounded
half-open polyhedra.

Equations~\eqref{eq:roadmap-partition-separation} and
\eqref{eq:roadmap-pointwise-separation} imply
\[
 K\cap(K+p)=\varnothing,
 \qquad p\in P\setminus\{0\}.
\]
Because $K$ is compact and $P$ is discrete,
\[
 \delta:=\inf_{p\in P\setminus\{0\}}\dist(K,K+p)>0.
\]
Since $\Omega\subset K$, for every $0<\varepsilon<\delta/2$, the set
$\Omega+B_\varepsilon$ packs with $P$.  Section~\ref{sec:soft-configurations}
uses the sharper row--column threshold $1/(n-m+1)$ instead of the coarser
$1/N$ used here.

We retain the normalization and notation of
Section~\ref{sec:planar-examples}.  For all configuration pictures below, the representatives in the four subcell positions are fixed by the guiding pair $T=\Z\times4\Z$, $P=6\Z\times\Z$; thus
$R_i=(0,i)+T$, $Q_j=(j,0)+P$,
$\mathcal A_{i,j}=(j,i)+(6\Z\times4\Z)$, and the four subcells of
$\mathcal A_{i,j}$ represent $(j,i)$, $(j+6,i)$, $(j,i+4)$, and
$(j+6,i+4)$.

\subsection{Why hard configurations are insufficient}

The preceding examples show that fixed hard tables can work for some pairs.  Nevertheless, an arbitrarily large surplus $n-m$ does not force the unit-square construction to have a positive buffer for a given pair.  The following elementary family makes this limitation explicit.

\begin{remark}\label{prop:hard-not-enough}
An arbitrarily large column surplus does not guarantee a positive buffer for
the unit-square hard construction.
For any $1\le m<n$, take $T=m\Z\times\Z$ and
$P=\Z\times n\Z$.  A hard configuration gives a union of unit cubes that
tiles by $T$ and packs by $P$.  Since $(1,0)\in P$, the closure of every
selected cube meets the closure of its translate by $(1,0)$.  The distance
to a nonzero $P$-translate is therefore zero, independently of $n-m$.
\end{remark}

\subsection{Smith normal form reduction}

The ordinary Smith normal form recalled in
Section~\ref{sec:lattice-theory} implies that, after a further common
invertible linear change of variables, the commensurable pair may be put in
the diagonal form
\[
 T=\operatorname{diag}(b_1,\ldots,b_d)\Z^d,
 \qquad
 P=\operatorname{diag}(a_1,\ldots,a_d)\Z^d,
 \qquad
 \gcd(a_\ell,b_\ell)=1.
\]
In these coordinates, $\Z^d/T\cong\prod_{\ell=1}^d\Z/b_\ell\Z$, $m=\prod_{\ell=1}^d b_\ell$, $\Z^d/P\cong\prod_{\ell=1}^d\Z/a_\ell\Z$, and $n=\prod_{\ell=1}^d a_\ell$.
If $R_i$ corresponds to $(\alpha_\ell)$ and $Q_j$ corresponds to
$(q_\ell)$, the coordinatewise Chinese remainder theorem gives a unique
representative $c(i,j)$ modulo
$\operatorname{diag}(a_1b_1,\ldots,a_db_d)\Z^d$, and
\[
 \mathcal A_{i,j}=c(i,j)+
 \operatorname{diag}(a_1b_1,\ldots,a_db_d)\Z^d.
\]
The reader may keep this diagonal model in mind.  We do not impose it in the
proof, because the row--column construction works directly under
\eqref{eq:commensurable-normalization}.

\subsection{Soft configurations}

The hard configurations of Definition~\ref{def:hard-configuration} are the
vertices of the continuous configuration space introduced below.  They
choose one actual address in every row, with the chosen addresses in
pairwise distinct columns.  Since every address class is nonempty, a hard
configuration exists whenever $n\ge m$.

\begin{figure}[htbp]
\centering
\begin{minipage}[b]{.48\textwidth}
\centering
\ConfigurationTable{%
  \ConfigTL{1}{1}{1}%
  \ConfigTL{2}{3}{1}%
  \ConfigTL{3}{5}{1}%
  \ConfigTL{4}{6}{1}%
}
\end{minipage}\hfill
\begin{minipage}[b]{.48\textwidth}
\centering
\ConfigurationTable{%
  \ConfigTL{1}{2}{1}%
  \ConfigTL{2}{4}{1}%
  \ConfigTL{3}{1}{1}%
  \ConfigTL{4}{5}{1}%
}
\end{minipage}
\caption{Two hard configurations with four rows and six columns.  Every main cell is an address class $\mathcal A_{i,j}=R_i\cap Q_j$, and its four subcells represent four possible addresses.  In each panel, exactly one subcell in every
row has value $1$, and the four occupied main cells lie in distinct
columns.  The two unused columns provide the vacancies for elementary hard
moves.}
\label{fig:hard-configurations}
\end{figure}

With the guiding address convention above, the supports
of the two hard configurations are, respectively,
\begin{equation}\label{eq:hard-picture-supports}
 \{(0,0),(2,1),(4,2),(5,3)\}
 \quad\text{and}\quad
 \{(1,0),(3,1),(0,2),(4,3)\}.
\end{equation}
Each support meets every $R_i$ once and every $Q_j$ at most once.

At the level of row--column choices, the finite cut-and-translate viewpoint
is closely related to the lattice packing and tiling arguments in
\cite{Kolountzakis1998Cubes,Kolountzakis2000Packing}.  If a row occupies one
column and another column is vacant, the row can be moved to an address in
the vacant column.  Its previous column then becomes vacant.  Consequently,
if $n\ge m+1$, any two hard configurations can be joined by a sequence of
such elementary moves through a vacant column.

Hard configurations are discrete.  To let configurations vary continuously,
we allow each row to distribute its unit mass among finitely many addresses.

\begin{definition}\label{def:soft-config}
A \emph{soft configuration} is a finitely supported function
\[
 \lambda:\Z^d\longrightarrow[0,1]
\]
such that every row has total mass one,
\begin{equation}\label{eq:soft-row}
 \sum_{z\in R_i}\lambda(z)=1,
 \qquad i=0,\ldots,m-1,
\end{equation}
and distinct positive addresses never lie in the same column:
\begin{equation}\label{eq:soft-column}
 z\ne z',\quad z+P=z'+P
 \quad\Longrightarrow\quad
 \lambda(z)\lambda(z')=0.
\end{equation}
We denote the resulting space by $\mathcal C(T,P)$.
\end{definition}

A hard configuration is precisely a soft configuration whose values lie in
$\{0,1\}$.  The row condition and the column-exclusion condition are now
expressed directly in the ambient group $\Z^d$.

\begin{figure}[htbp]
\centering
\begin{minipage}[b]{.48\textwidth}
\centering
\ConfigurationTable{%
  \ConfigTL{1}{1}{.7}%
  \ConfigBL{1}{2}{.3}%
  \ConfigTL{2}{3}{.6}%
  \ConfigBR{2}{4}{.4}%
  \ConfigTR{3}{5}{1}%
  \ConfigBR{4}{6}{1}%
}
\end{minipage}\hfill
\begin{minipage}[b]{.48\textwidth}
\centering
\ConfigurationTable{%
  \ConfigTL{1}{1}{.2}%
  \ConfigBL{1}{2}{.3}%
  \ConfigTR{1}{3}{.5}%
  \ConfigTL{2}{4}{1}%
  \ConfigBR{3}{5}{1}%
  \ConfigBL{4}{6}{1}%
}
\end{minipage}
\caption{Two soft configurations.  Empty subcells have weight zero.  In
every row the displayed values add to $1$, while no main column contains
positive values belonging to two different rows.  In this schematic each main cell is an address class $\mathcal A_{i,j}$ and displays four possible addresses.  The column-exclusion rule allows at most one positive subcell anywhere in each column.}
\label{fig:soft-configurations}
\end{figure}

The same guiding pair also provides the first concrete illustration of why
weights are useful.  The following legal weighted table occurs in the later
equivariant field and produces the buffered construction that a fixed hard
table cannot supply.

\begin{figure}[htbp]
\centering
\begin{minipage}[b]{.29\textwidth}
\centering
\resizebox{\linewidth}{!}{%
\ConfigurationTable{%
  \ConfigTL{1}{6}{1}%
  \ConfigTL{2}{1}{1}%
  \ConfigTL{3}{2}{.5}%
  \ConfigTL{3}{3}{.5}%
  \ConfigTL{4}{4}{1}%
}}
\end{minipage}\hfill
\begin{minipage}[b]{.335\textwidth}
  \centering
  \RasterWithZxFourZ{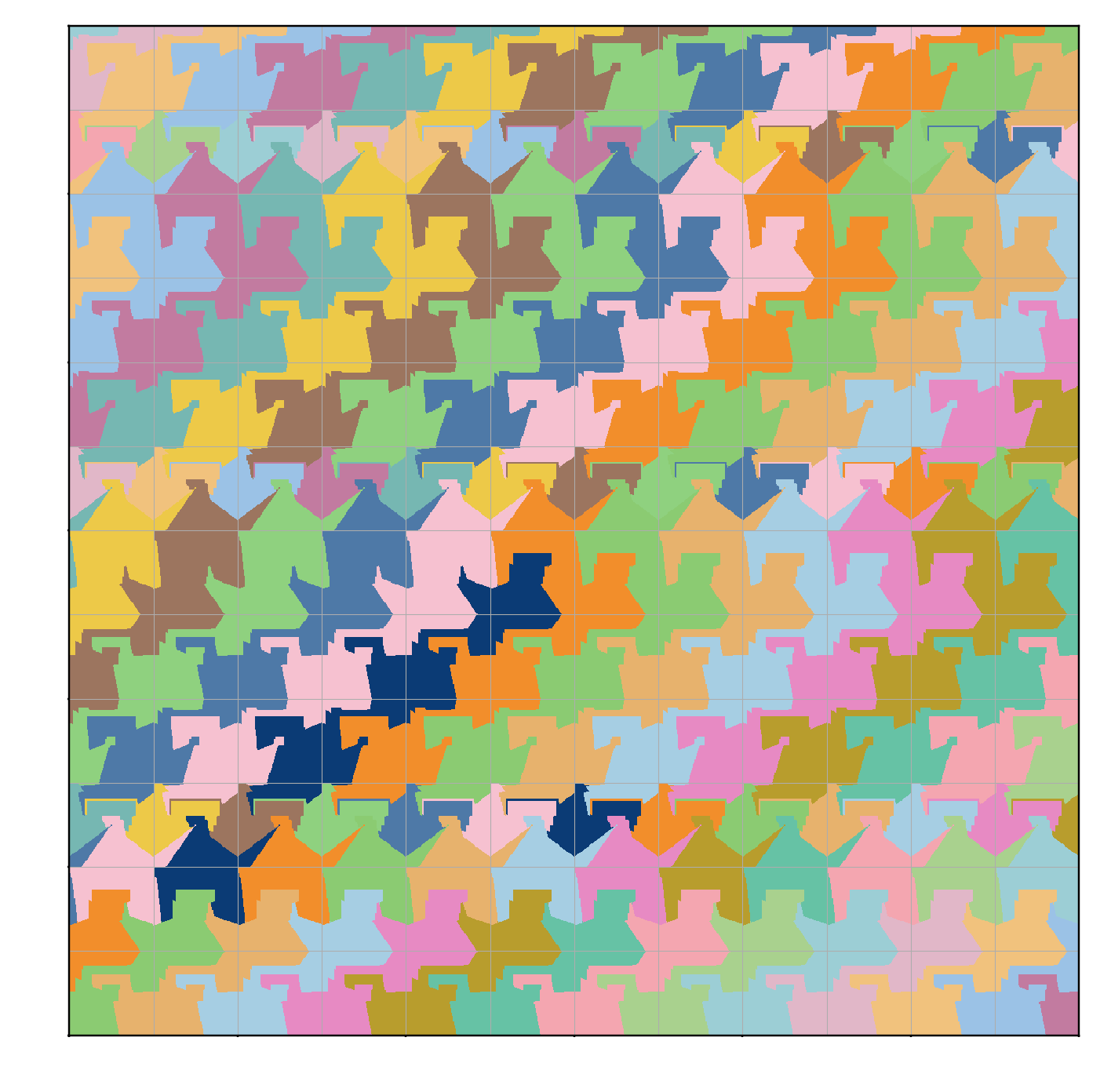}{.213277}{.206178}{.075742}{.077047}{-2}{10}{-2}{10}
\end{minipage}\hfill
\begin{minipage}[b]{.335\textwidth}
  \centering
  \RasterWithSixZxZ{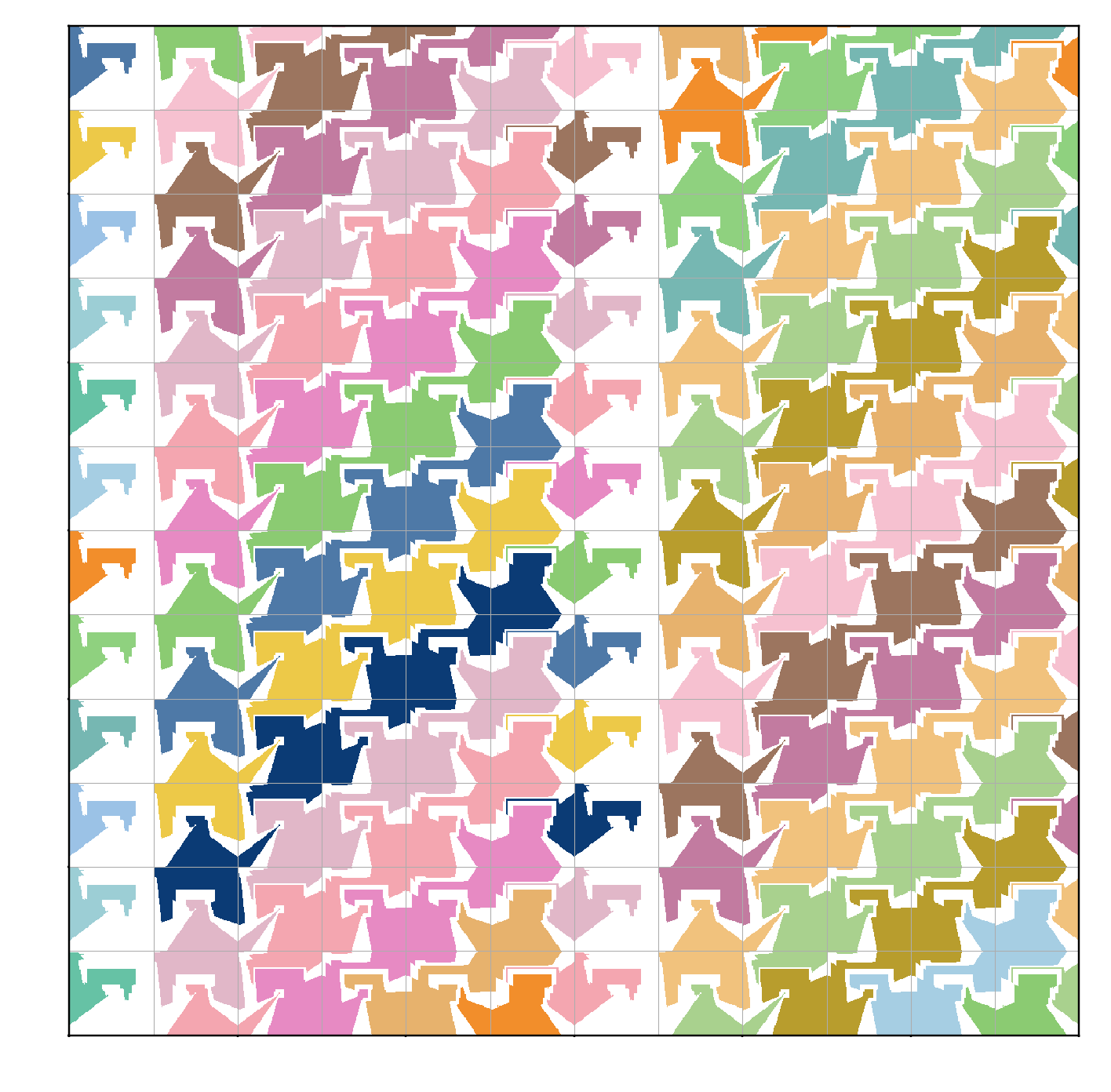}{.213277}{.206178}{.075742}{.077047}{-2}{10}{-2}{10}
\end{minipage}
\caption{A guiding soft-configuration example for
$T=\Z\times4\Z$ and $P=6\Z\times\Z$.  From left to right: one legal
weighted configuration occurring in the configuration field; the resulting
exact $T$-tiling; and the buffered $P$-packing.  The complete construction
is given in the second example in Subsection~\ref{subsec:two-further-equality-examples}.}
\label{fig:four-six-preview-success}
\end{figure}

An admissible support pattern is a tuple
\[
 S=(S_0,\ldots,S_{m-1}),
\]
where every $S_i\subset R_i$ is finite and nonempty and no two addresses in
$\bigcup_iS_i$ lie in the same column $Q_j$.  The configurations supported
in $S$ form the convex cell
\begin{equation}\label{eq:soft-cell}
 C_S=\prod_{i=0}^{m-1}\Delta(S_i),
\end{equation}
where $\Delta(S_i)$ is the probability simplex on $S_i$.

The first configuration in Figure~\ref{fig:soft-configurations} has support
pattern $S^{A}=(S^{A}_0,S^{A}_1,S^{A}_2,S^{A}_3)$ with
\begin{equation}\label{eq:first-soft-support}
 \begin{aligned}
 S^{A}_0&=\{(0,0),(1,4)\},&
 S^{A}_1&=\{(2,1),(9,5)\},\\
 S^{A}_2&=\{(10,2)\},&
 S^{A}_3&=\{(11,7)\}.
 \end{aligned}
\end{equation}
Thus
\[
 C_{S^{A}}
 =\Delta\{(0,0),(1,4)\}
  \times\Delta\{(2,1),(9,5)\}
  \times\Delta^0\times\Delta^0
 \cong\Delta^1\times\Delta^1\times\Delta^0\times\Delta^0.
\]
Here $\Delta^0=\{1\}$ is the probability simplex on a singleton.

The second configuration has support pattern
$S^{B}=(S^{B}_0,S^{B}_1,S^{B}_2,S^{B}_3)$ with
\begin{equation}\label{eq:second-soft-support}
 \begin{aligned}
 S^{B}_0&=\{(0,0),(1,4),(8,0)\},&
 S^{B}_1&=\{(3,1)\},\\
 S^{B}_2&=\{(10,6)\},&
 S^{B}_3&=\{(5,7)\}.
 \end{aligned}
\end{equation}
Hence
\[
 C_{S^{B}}
 =\Delta\{(0,0),(1,4),(8,0)\}
  \times\Delta^0\times\Delta^0\times\Delta^0
 \cong\Delta^2\times\Delta^0\times\Delta^0\times\Delta^0.
\]
The point, segment, triangle, and square cell types that occur in this
four-row, six-column setting are displayed in
Figure~\ref{fig:soft-cell-types}.

We regard every cell as a subset of the vector space
$\R^{(\Z^d)}$ of finitely supported real functions on $\Z^d$, by extending
its coordinates by zero outside $\bigcup_iS_i$.  With this convention,
\begin{equation}\label{eq:group-soft-union}
 \mathcal C(T,P)=\bigcup_{S\text{ admissible}}C_S.
\end{equation}
For $\lambda\in\mathcal C(T,P)$, the support in row $R_i$ is
\[
 S_i(\lambda)=\{z\in R_i:\lambda(z)>0\}.
\]
The tuple $S(\lambda)$ is admissible, and $C_{S(\lambda)}$ is the unique
smallest cell containing $\lambda$.  Conversely, every point of every
$C_S$ satisfies the row-normalization and column-exclusion conditions.
Faces arise by setting weights equal to zero.

If $S$ uses $s$ addresses in total, then
\[
 \dim C_S=\sum_{i=0}^{m-1}(|S_i|-1)=s-m.
\]
Since an admissible pattern uses at most one address in each of the $n$
columns, $s\le n$, and hence
\begin{equation}\label{eq:group-cell-dimension}
 \dim C_S\le n-m.
\end{equation}
The full space is generally an infinite polyhedral complex, but every map
from a finite triangulation that is defined by finitely many vertex values
uses only finitely many addresses and therefore lies in a finite
subcomplex.

\begin{figure}[htbp]
\centering
\begin{tikzpicture}[>=Latex,scale=.87]
  \path[use as bounding box] (-.1,-.42) rectangle (15.0,3.62);
  \def\base{.72}

  \begin{scope}[shift={(0,0)}]
    \node[font=\bfseries] at (1.25,3.38) {point};
    \fill[MidnightBlue] (1.25,\base) circle (3.2pt);
    \node[font=\large] at (1.25,.18) {$\Delta^0$};
  \end{scope}

  \begin{scope}[shift={(3.35,0)}]
    \node[font=\bfseries] at (1.35,3.38) {segment};
    \coordinate (a) at (.20,\base);
    \coordinate (b) at (2.50,\base);
    \draw[very thick,MidnightBlue] (a)--(b);
    \fill[MidnightBlue] (a) circle (2.8pt)
      node[below=5pt,font=\scriptsize] {$(0,0)$};
    \fill[MidnightBlue] (b) circle (2.8pt)
      node[below=5pt,font=\scriptsize] {$(4,0)$};
    \node[font=\large] at (1.35,.18) {$\Delta^1$};
  \end{scope}

  \begin{scope}[shift={(7.05,0)}]
    \node[font=\bfseries] at (1.35,3.38) {triangle};
    \coordinate (a) at (.25,\base);
    \coordinate (b) at (2.45,\base);
    \coordinate (c) at (1.35,{\base+1.905});
    \draw[very thick,fill=MidnightBlue!8] (a)--(b)--(c)--cycle;
    \fill[MidnightBlue] (a) circle (2.6pt)
      node[below left,font=\scriptsize] {$(0,0)$};
    \fill[MidnightBlue] (b) circle (2.6pt)
      node[below right,font=\scriptsize] {$(4,0)$};
    \fill[MidnightBlue] (c) circle (2.6pt)
      node[above right=1pt,font=\scriptsize] {$(2,0)$};
    \node[font=\large] at (1.35,.18) {$\Delta^2$};
  \end{scope}

  \begin{scope}[shift={(11.05,0)}]
    \node[font=\bfseries] at (1.40,3.38) {square};
    \draw[very thick,fill=ForestGreen!10] (.50,\base) rectangle ++(1.80,1.80);
    \fill[ForestGreen] (.50,\base) circle (2.6pt)
      node[below left,font=\scriptsize] {$ac$};
    \fill[ForestGreen] (2.30,\base) circle (2.6pt)
      node[below right,font=\scriptsize] {$bc$};
    \fill[ForestGreen] (.50,{\base+1.80}) circle (2.6pt)
      node[above left,font=\scriptsize] {$ad$};
    \fill[ForestGreen] (2.30,{\base+1.80}) circle (2.6pt)
      node[above right,font=\scriptsize] {$bd$};
    \node[font=\large] at (1.40,{\base+.90}) {$\Delta^1\!\times\!\Delta^1$};
    \node[font=\fontsize{5.8}{6.2}\selectfont] at (1.40,.20)
      {$a=(0,0),\ b=(4,0)$};
    \node[font=\fontsize{5.8}{6.2}\selectfont] at (1.40,-.08)
      {$c=(1,1),\ d=(5,1)$};
  \end{scope}
\end{tikzpicture}
\caption{The four elementary cell types in the concrete four-row,
six-column model.  Their lower edges are vertically aligned; the triangle is
equilateral and the product cell is drawn as a square.  The fixed
$\Delta^0$ factors are suppressed.  The displayed labels are actual
addresses in the guiding classes $\mathcal A_{i,j}$, chosen so that
the active addresses lie in pairwise distinct columns.  In the square, each
vertex records the pair selected in rows $R_0$ and $R_1$.}
\label{fig:soft-cell-types}
\end{figure}
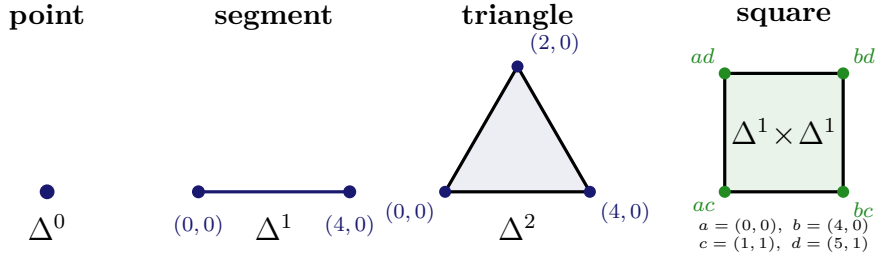

\section{Commensurable sufficiency in the Tiling--$\varepsilon$--Packing theorem}
\label{sec:soft-configurations}

We retain without reintroduction the normalization, quotient cosets,
address classes, and hard configurations of
Section~\ref{sec:planar-examples}, together with the soft configurations of
Section~\ref{sec:hard-configurations}.
Assume throughout this section that
\begin{equation}\label{eq:commensurable-gap}
 n\ge m+d.
\end{equation}

\subsection{The fundamental topological result}

The commensurable construction repeatedly fills the boundary of a parameter cell.  For $k\ge1$, write $B^k=\{x\in\R^k:\norm{x}_2\le1\}$ and $S^{k-1}=\partial B^k$.

\begin{definition}[Simplexwise affine maps]\label{def:simplexwise-affine}
Let $K$ be a finite simplicial complex, with geometric realization $|K|$.  A map $f:|K|\to\mathcal C(T,P)$ is \emph{simplexwise affine with respect to $K$} if, for every simplex $\sigma=[v_0,\ldots,v_p]$, the vertex values lie in one common cell $C_{S_\sigma}$ and
\[
 f\!\left(\sum_{i=0}^pt_iv_i\right)=\sum_{i=0}^pt_if(v_i)
 \quad\text{whenever }t_i\ge0\text{ and }\sum_{i=0}^pt_i=1.
\]
It is \emph{simplexwise affine} if this holds after a finite subdivision of $K$.
\end{definition}
The common-cell condition keeps the affine interpolation inside the soft-configuration space.  The formulas agree on common faces, so the map is continuous.

\begin{theorem}\label{thm:soft-filling}
Assume $n\ge m+d$.  For every $1\le k\le d$, every simplexwise affine map from a finitely
triangulated $S^{k-1}$ into $\mathcal C(T,P)$ extends simplexwise affinely
over a finitely triangulated $B^k$ without changing the boundary
triangulation or its vertex values.  No finiteness assumption is imposed on
the address classes $\mathcal A_{i,j}$; as in Definition~\ref{def:soft-config},
each individual configuration has finite support.  More precisely, for every
boundary map there exist finite nonempty sets
\[
 B_{i,j}\subset\mathcal A_{i,j},
 \qquad 0\le i\le m-1,\quad 0\le j\le n-1,
\]
containing every address that occurs in a boundary vertex value, such that
the extension may be chosen to use only addresses in
$\bigcup_{i,j}B_{i,j}$.
\end{theorem}

The proof is deferred to Section~\ref{sec:proof-soft-filling}.  There the
lattice origin of the address classes is forgotten and a stronger theorem is
proved for arbitrary nonempty row--column sets, finite or infinite.  The
proof reduces every simplexwise affine boundary problem to a finite complete
address subsystem before applying the finite polyhedral argument.

\subsection{The translation action}\label{sec:commensurable-construction}

For $a\in\Z^d$ and $\lambda\in\mathcal C(T,P)$, define
\begin{equation}\label{eq:rho-action}
 (\rho_a\lambda)(z)=\lambda(z-a),
 \qquad z\in\Z^d.
\end{equation}
Translation by $a$ permutes the rows $R_i$, the columns $Q_j$, and their
address classes, so $\rho_a\lambda$ is again a soft configuration.  Moreover,
$\rho_a\rho_b=\rho_{a+b}$.

The action is needed because the configuration attached to a base point
records masses at integer addresses relative to that base point.  We read
the coordinate $\lambda(z)$ at a base point $u$ as the mass attached to the
physical point $u-z$.  If the base point is shifted from $u$ to $u+a$, then
the same physical point is written as
\[
 u-z=(u+a)-(z+a).
\]
Thus the mass formerly stored at address $z$ must, at the new base point
$u+a$, be stored at address $z+a$.  Equivalently,
$\lambda_{\rm new}(z+a)=\lambda_{\rm old}(z)$, which is precisely the rule
\[
 (\rho_a\lambda)(z)=\lambda(z-a).
\]
We need this relabeling because opposite faces of neighboring cubes represent
the same physical points.  The action $\rho_a$ is what allows the local
fillings on adjacent cubes to match along those shared boundary points and to
assemble into one globally defined equivariant field.  It will also yield the
key identity $\varphi(x-z)=\Lambda(x)(z)$, from which the
partition-of-unity and packing relations follow.

\subsection{Constructing the equivariant configuration field $\Lambda$}

We first construct the field on the closed fundamental cube
\[
 F=[0,1]^d
\]
and only afterwards extend it equivariantly to all of $\R^d$.  For
$i=1,\ldots,d$, write
\[
 F_i^- =\{x\in F:x_i=0\},
 \qquad
 F_i^+ =\{x\in F:x_i=1\}=F_i^-+e_i.
\]
The map on the cube will satisfy the opposite-face compatibility relations
\begin{equation}\label{eq:cube-face-compatibility}
 \Lambda_F(x+e_i)=\rho_{e_i}\Lambda_F(x),
 \qquad x\in F_i^-.
\end{equation}
These relations are precisely what is needed for the copies on neighboring
integer cubes to fit together continuously.

We triangulate $F$ by the standard staircase triangulation.  For every
permutation $\pi$ of $1,\ldots,d$, the corresponding $d$-simplex is
\begin{equation}\label{eq:staircase-simplex}
 \operatorname{conv}\bigl\{0,\ e_{\pi(1)},\
 e_{\pi(1)}+e_{\pi(2)},\ \ldots,\
 e_{\pi(1)}+\cdots+e_{\pi(d)}\bigr\}.
\end{equation}
The triangulations induced on opposite faces of the cube agree after
translation by the corresponding coordinate vector.  In other words, when one
translates a simplex lying in a face $x_i=0$ by $e_i$, one obtains exactly the
corresponding simplex in the opposite face $x_i=1$.

\begin{figure}[H]
\centering
\begin{tikzpicture}[>=Latex]
  \newcommand{\stairtet}[6]{%
    \begin{scope}[shift={(#1,#2)},scale=.61]
      \coordinate (O)   at (0,0);
      \coordinate (X)   at (2.45,0);
      \coordinate (Y)   at (.88,.70);
      \coordinate (XY)  at (3.33,.70);
      \coordinate (Z)   at (0,2.45);
      \coordinate (XZ)  at (2.45,2.45);
      \coordinate (YZ)  at (.88,3.15);
      \coordinate (U)   at (3.33,3.15);
      \filldraw[fill=#5!52,fill opacity=.34,draw=#5!85!black,line width=.55pt]
        (O)--(#3)--(#4)--cycle;
      \filldraw[fill=#5!46,fill opacity=.30,draw=#5!85!black,line width=.55pt]
        (O)--(#3)--(U)--cycle;
      \filldraw[fill=#5!38,fill opacity=.27,draw=#5!85!black,line width=.55pt]
        (O)--(#4)--(U)--cycle;
      \filldraw[fill=#5!58,fill opacity=.30,draw=#5!85!black,line width=.55pt]
        (#3)--(#4)--(U)--cycle;
      \draw[black!75,line width=.62pt] (O)--(X)--(XY)--(Y)--cycle;
      \draw[black!75,line width=.62pt] (Z)--(XZ)--(U)--(YZ)--cycle;
      \draw[black!75,line width=.62pt] (O)--(Z) (X)--(XZ)
        (Y)--(YZ) (XY)--(U);
      \draw[black!45,dashed,line width=.48pt] (O)--(U);
      \foreach \p in {O,X,Y,XY,Z,XZ,YZ,U}
        \fill[black] (\p) circle (1.15pt);
      \node[font=\small] at (1.67,-.55) {$T_{#6}$};
    \end{scope}%
  }
  \stairtet{0}{4.05}{X}{XY}{MidnightBlue}{123}
  \stairtet{4.25}{4.05}{X}{XZ}{OrangeRed}{132}
  \stairtet{8.50}{4.05}{Y}{XY}{ForestGreen}{213}
  \stairtet{0}{0}{Y}{YZ}{Magenta}{231}
  \stairtet{4.25}{0}{Z}{XZ}{Purple}{312}
  \stairtet{8.50}{0}{Z}{YZ}{Cyan!70!black}{321}
\end{tikzpicture}
\caption{The standard staircase triangulation in dimension three.
Each panel shows the unit cube from the same perspective and shades one of
the six tetrahedra $T_\pi$.  Every tetrahedron contains the long diagonal
from $(0,0,0)$ to $(1,1,1)$; together the six shaded tetrahedra fill the
cube and meet along common faces.}
\label{fig:staircase-triangulation}
\end{figure}
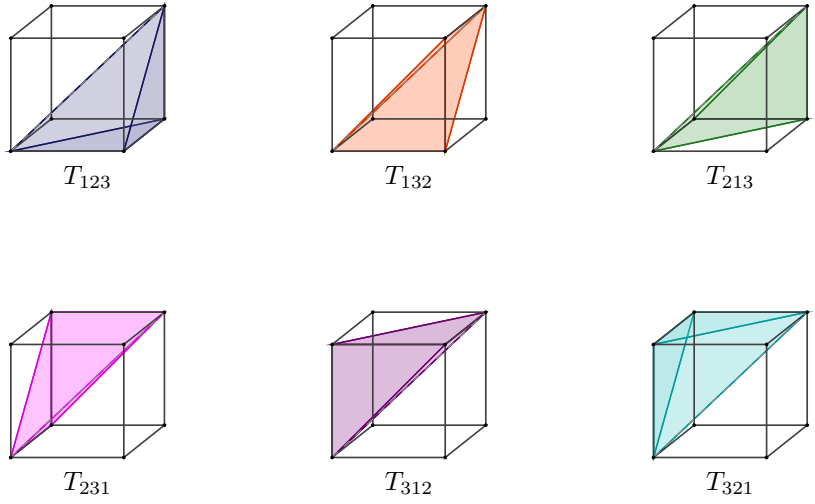

For a triangulated polyhedron, the \emph{$j$-skeleton} is the union of all
simplices of dimension at most $j$.  In particular, the $0$-skeleton of the
triangulated cube consists precisely of its corners.

We begin on the $0$-skeleton, that is, on the corners of the cube.
Choose one hard configuration $\lambda^0$.  For every cube vertex
$\varepsilon\in\{0,1\}^d$, set
\begin{equation}\label{eq:field-vertices}
 \Lambda_F(\varepsilon)=\rho_\varepsilon\lambda^0.
\end{equation}
The commuting relations $\rho_a\rho_b=\rho_{a+b}$ show that these values
already satisfy~\eqref{eq:cube-face-compatibility} at the corners, that is, on the $0$-skeleton.
For the concrete four-row, six-column illustration in dimension two, we take
\[
 \supp\lambda^0=\{(0,0),(1,1),(2,2),(3,3)\}.
\]
Since $(\rho_a\lambda)(z)=\lambda(z-a)$, the support is translated by $+a$.
Consequently,
\begin{align*}
 \supp(\rho_{e_1}\lambda^0)
   &=\{(1,0),(2,1),(3,2),(4,3)\},\\
 \supp(\rho_{e_2}\lambda^0)
   &=\{(0,1),(1,2),(2,3),(3,4)\},\\
 \supp(\rho_{e_1+e_2}\lambda^0)
   &=\{(1,1),(2,2),(3,3),(4,4)\}.
\end{align*}

For the explicit filling used in the later four-row, six-column example,
the two non-corner values displayed below are
\begin{align}
 \Lambda(0,0.5)
 &=\delta_{(5,0)}+\delta_{(0,1)}
   +\tfrac12\delta_{(1,2)}+\tfrac12\delta_{(2,2)}+\delta_{(3,3)},
 \label{eq:configuration-field-green}\\
 \Lambda(0.9,0.55)
 &=\delta_{(6,0)}+\delta_{(1,1)}
   +\tfrac45\delta_{(2,2)}+\tfrac15\delta_{(3,2)}+\delta_{(4,3)}.
 \label{eq:configuration-field-orange}
\end{align}
The occupied columns are pairwise distinct in each configuration.

\begin{figure}[htbp]
\centering
\begin{tikzpicture}[x=4.95cm,y=4.95cm,>=Latex,scale=1.3,transform shape]
  \definecolor{AnchorBlue}{RGB}{0,58,135}
  \definecolor{AnchorOrange}{RGB}{170,82,0}
  \definecolor{AnchorGreen}{RGB}{0,100,55}
  \definecolor{AnchorOrangePt}{RGB}{185,110,0}
  \draw[->,line width=.7pt] (-0.2,0) -- (1.55,0) node[right] {$x$};
  \draw[->,line width=.7pt] (0,-0.2) -- (0,1.55) node[above] {$y$};
  \foreach \t/\lab in {0/0,0.5/{0.5},1/1}{
    \draw[line width=.5pt] (\t,0.018)--(\t,-0.018)
      node[below=2pt,font=\scriptsize] {$\lab$};
    \draw[line width=.5pt] (0.018,\t)--(-0.018,\t)
      node[left=2pt,font=\scriptsize] {$\lab$};
  }

  \node[anchor=south west,inner sep=0] at (0,0)
    {\MiniConfigurationTable{AnchorBlue}{
       \MiniConfigTL{1}{1}{1}{AnchorBlue}%
       \MiniConfigTL{2}{2}{1}{AnchorBlue}%
       \MiniConfigTL{3}{3}{1}{AnchorBlue}%
       \MiniConfigTL{4}{4}{1}{AnchorBlue}}};
  \node[anchor=south west,inner sep=0] at (0,1)
    {\MiniConfigurationTable{AnchorBlue}{
       \MiniConfigBL{1}{4}{1}{AnchorBlue}%
       \MiniConfigTL{2}{1}{1}{AnchorBlue}%
       \MiniConfigTL{3}{2}{1}{AnchorBlue}%
       \MiniConfigTL{4}{3}{1}{AnchorBlue}}};
  \node[anchor=south west,inner sep=0] at (1,0)
    {\MiniConfigurationTable{AnchorBlue}{
       \MiniConfigTL{1}{2}{1}{AnchorBlue}%
       \MiniConfigTL{2}{3}{1}{AnchorBlue}%
       \MiniConfigTL{3}{4}{1}{AnchorBlue}%
       \MiniConfigTL{4}{5}{1}{AnchorBlue}}};
  \node[anchor=south west,inner sep=0] at (1,1)
    {\MiniConfigurationTable{AnchorBlue}{
       \MiniConfigBL{1}{5}{1}{AnchorBlue}%
       \MiniConfigTL{2}{2}{1}{AnchorBlue}%
       \MiniConfigTL{3}{3}{1}{AnchorBlue}%
       \MiniConfigTL{4}{4}{1}{AnchorBlue}}};
  \node[anchor=south west,inner sep=0] at (0,.5)
    {\MiniConfigurationTable{AnchorGreen}{
       \MiniConfigTL{1}{6}{1}{AnchorGreen}%
       \MiniConfigTL{2}{1}{1}{AnchorGreen}%
       \MiniConfigTL{3}{2}{.5}{AnchorGreen}%
       \MiniConfigTL{3}{3}{.5}{AnchorGreen}%
       \MiniConfigTL{4}{4}{1}{AnchorGreen}}};
  \node[anchor=south west,inner sep=0] at (.9,.55)
    {\MiniConfigurationTable{AnchorOrangePt}{
       \MiniConfigTR{1}{1}{1}{AnchorOrangePt}%
       \MiniConfigTL{2}{2}{1}{AnchorOrangePt}%
       \MiniConfigTL{3}{3}{.8}{AnchorOrangePt}%
       \MiniConfigTL{3}{4}{.2}{AnchorOrangePt}%
       \MiniConfigTL{4}{5}{1}{AnchorOrangePt}}};

  \fill[AnchorBlue] (0,0) circle (1.7pt);
  \fill[AnchorBlue] (0,1) circle (1.7pt);
  \fill[AnchorBlue] (1,0) circle (1.7pt);
  \fill[AnchorBlue] (1,1) circle (1.7pt);
  \fill[AnchorGreen] (0,.5) circle (1.7pt);
  \fill[AnchorOrangePt] (.9,.55) circle (1.7pt);
\end{tikzpicture}
\caption{The legal configuration field $\Lambda$ on one fundamental
square for $T=\Z\times4\Z$ and $P=6\Z\times\Z$.  The blue tables are the
four equivariant corner values.  The green and orange tables are the values
in~\eqref{eq:configuration-field-green} and~\eqref{eq:configuration-field-orange}; in
each table the positive entries occupy distinct columns.  This is the same
explicit field used in the second example in Subsection~\ref{subsec:two-further-equality-examples}, whose tiling
and buffered packing appear in Figure~\ref{fig:four-six-tiling-packing}.}
\label{fig:lambda-plane-configurations}
\end{figure}
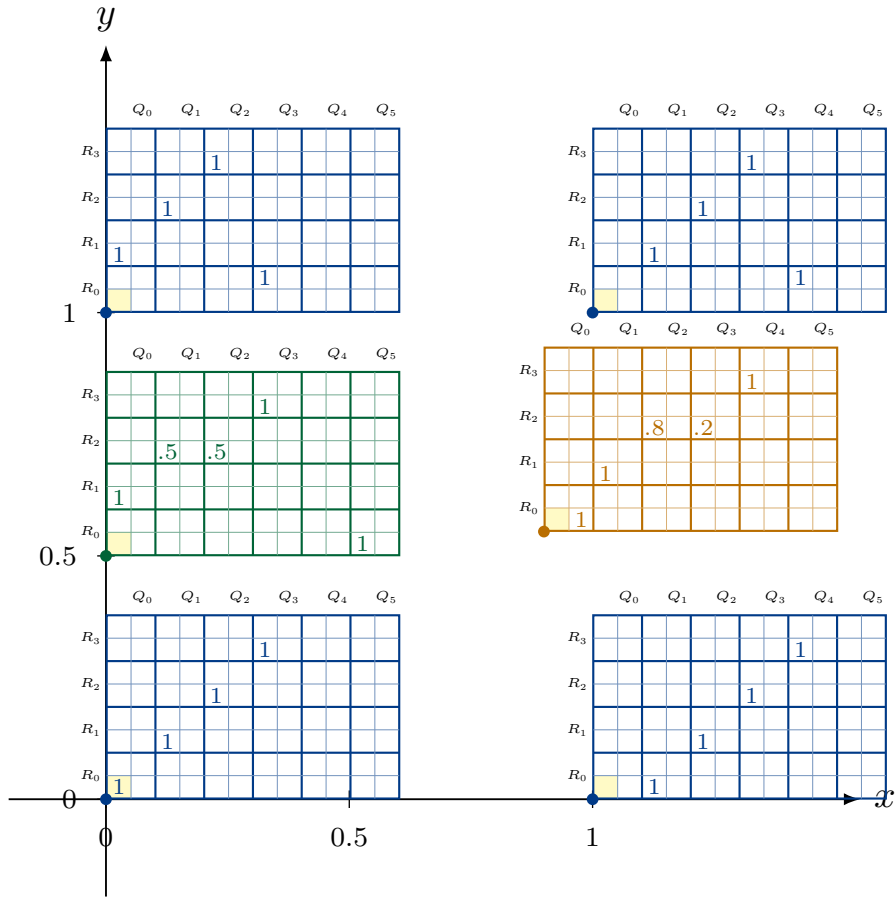

The first extension step already explains why interpolation through hard configurations alone is insufficient.
Even on an edge, the straight segment between the endpoint hard
configurations may leave the soft-configuration space: two rows can acquire
positive mass in the same column.  We therefore subdivide the edge, choose a
polygonal path of soft configurations, and interpolate affinely along its
segments.  On an opposite edge, the path is defined by the appropriate
translate $\rho_{e_i}$.  This is the case $j=1$ of
Theorem~\ref{thm:soft-filling}.

We now fill the higher-dimensional simplices of $F$ inductively.  Suppose
that $\Lambda_F$ has already been defined on the $(j-1)$-skeleton and obeys
the opposite-face relations there.  A simplex in a boundary face cannot be
filled independently of its translated copy on the opposite face.  More
precisely, if $\sigma\subset F_i^-$, then $\sigma+e_i\subset F_i^+$ must be
filled by the translated rule $\Lambda_F(x+e_i)=\rho_{e_i}\Lambda_F(x)$.
If $\sigma$ lies in several coordinate faces, the same rule applies to every
copy obtained by adding the corresponding sums of coordinate vectors.
Thus the boundary identifications partition the finitely many $j$-simplices
into finitely many classes.  Interior simplices form singleton classes.  We
choose one representative from each class, fill only that representative,
and obtain every other member of the class by equivariance.  This guarantees
compatibility on all identified faces by construction.

Fix one representative $\sigma$.  Its boundary has already been filled and carries the
simplexwise-affine map
\begin{equation}\label{eq:cell-boundary-map}
 f_\sigma=\Lambda_F|_{\partial\sigma}:
 \partial\sigma\longrightarrow\mathcal C(T,P).
\end{equation}
The boundary triangulation has finitely many vertices, and every vertex
value has finite support, so only finitely many addresses occur there; call
their union $E_\sigma$.  Each row already occurs, since every boundary
configuration has row sum one.  Nevertheless, some row--column boxes may
contain no address from $E_\sigma$.  For each such missing box choose one
arbitrary address and adjoin it to $E_\sigma$.  The resulting family
$A_\sigma$ is finite and contains at least one address in every box.  The
boundary map still lands in the corresponding finite subcomplex
$\mathcal C_{m,n}(A_\sigma)$.  We therefore apply the finite abstract
Theorem~\ref{thm:soft-filling-strong} to this prescribed address system.

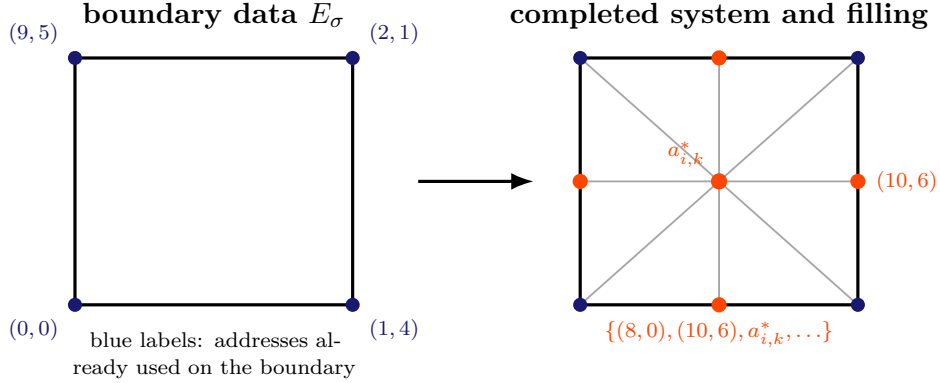
\begin{figure}[htbp]
\centering
\begin{tikzpicture}[>=Latex,scale=1.02]
  \begin{scope}[shift={(-5.2,0)}]
    \node[font=\bfseries] at (1.8,3.75) {boundary data $E_\sigma$};
    \draw[very thick] (0,0) rectangle (3.6,3.2);
    \fill[MidnightBlue] (0,0) circle (2.6pt);
    \fill[MidnightBlue] (3.6,0) circle (2.6pt);
    \fill[MidnightBlue] (3.6,3.2) circle (2.6pt);
    \fill[MidnightBlue] (0,3.2) circle (2.6pt);
    \node[MidnightBlue,font=\scriptsize,anchor=north east] at (-.08,-.05)
      {$ (0,0)$};
    \node[MidnightBlue,font=\scriptsize,anchor=north west] at (3.68,-.05)
      {$ (1,4)$};
    \node[MidnightBlue,font=\scriptsize,anchor=south west] at (3.68,3.25)
      {$ (2,1)$};
    \node[MidnightBlue,font=\scriptsize,anchor=south east] at (-.08,3.25)
      {$ (9,5)$};
    \node[align=center,font=\scriptsize,text width=4.5cm] at (1.8,-.65)
      {blue labels: addresses already used on the boundary};
  \end{scope}

  \draw[-{Latex[length=3mm]},very thick] (-.75,1.6)--(.75,1.6);

  \begin{scope}[shift={(1.35,0)}]
    \node[font=\bfseries] at (1.8,3.75)
      {completed system and filling};
    \draw[very thick] (0,0) rectangle (3.6,3.2);
    \coordinate (c) at (1.8,1.6);
    \coordinate (m1) at (1.8,0);
    \coordinate (m2) at (3.6,1.6);
    \coordinate (m3) at (1.8,3.2);
    \coordinate (m4) at (0,1.6);
    \draw[gray!70,line width=.7pt] (c)--(0,0) (c)--(3.6,0)
      (c)--(3.6,3.2) (c)--(0,3.2);
    \draw[gray!70,line width=.7pt] (m1)--(c)--(m3) (m4)--(c)--(m2);
    \foreach \p in {(0,0),(3.6,0),(3.6,3.2),(0,3.2)}
      \fill[MidnightBlue] \p circle (2.6pt);
    \fill[OrangeRed] (m1) circle (2.8pt);
    \fill[OrangeRed] (m2) circle (2.8pt)
      node[right=3pt,font=\scriptsize] {$(10,6)$};
    \fill[OrangeRed] (c) circle (3pt)
      node[above left=2pt,font=\scriptsize] {$a^*_{i,k}$};
    \fill[OrangeRed] (m3) circle (2.8pt);
    \fill[OrangeRed] (m4) circle (2.8pt);
    \node[OrangeRed,font=\scriptsize] at (1.8,-.38)
      {$\{(8,0),(10,6),a^*_{i,k},\ldots\}$};
  \end{scope}
\end{tikzpicture}
\caption{Finite bookkeeping before filling a two-dimensional parameter cell,
shown schematically as a square.  The boundary uses the finite address
set $E_\sigma$ shown in blue.  We adjoin addresses from every missing
row--column class and then triangulate the interior; new data are shown in
orange.  Rows are already nonempty, so the completion is needed for missing
boxes, not for missing rows.}
\label{fig:soft-finite-bookkeeping}
\end{figure}

Since $j\le d\le n-m$, Theorem~\ref{thm:soft-filling-strong}, applied to
$\mathcal C_{m,n}(A_\sigma)$, extends $f_\sigma$ simplexwise affinely over
$\sigma$, using only addresses from the prescribed finite system $A_\sigma$
and leaving the existing boundary triangulation unchanged.  If
$\sigma+a$ is another simplex in the same opposite-face identification
class, define its filling by
\[
 \Lambda_F(x+a)=\rho_a\Lambda_F(x),
 \qquad x\in\sigma.
\]
The relation $\rho_a\rho_b=\rho_{a+b}$ makes this prescription consistent
on intersections of several faces.  Adjacent fillings have the same vertex
values on their common face, so their affine formulas agree there.  Repeating
this for $j=1,\ldots,d$ constructs a simplexwise affine map
\[
 \Lambda_F:F\longrightarrow\mathcal C(T,P)
\]
satisfying~\eqref{eq:cube-face-compatibility}.  Because $F$ has only
finitely many simplices after the construction, only finitely many addresses
occur on the entire cube.

After the cube has been filled, we extend the map equivariantly to all of $\R^d$.
Every $x\in\R^d$ has a unique representation
\[
 x=a+u,
 \qquad a\in\Z^d,
 \quad u\in[0,1)^d.
\]
Define
\begin{equation}\label{eq:global-field-from-cube}
 \Lambda(x)=\rho_a\Lambda_F(u).
\end{equation}
Then
\[
 \Lambda(x+b)=\rho_b\Lambda(x),
 \qquad b\in\Z^d,
\]
so the desired equivariance holds.  It remains only to check continuity
across the boundaries of adjacent cubes.  If, for example, $u_i=0$, the
limit from the neighboring cube on the negative side is
\[
 \rho_{a-e_i}\Lambda_F(u+e_i)
 =\rho_{a-e_i}\rho_{e_i}\Lambda_F(u)
 =\rho_a\Lambda_F(u),
\]
by~\eqref{eq:cube-face-compatibility}.  This equals the value in
\eqref{eq:global-field-from-cube}.  The same calculation works at
intersections of several coordinate faces.  Hence $\Lambda$ is continuous
on all of $\R^d$ and is simplexwise affine on the translated cube
triangulations.

\subsection{The scalar partition of unity}

The final scalar function is the coordinate of the zero address.  We now show
that it is a compactly supported piecewise-affine $T$-partition of unity:
\begin{equation}\label{eq:scalar-zero-coordinate}
 \varphi(x)=\Lambda(x)(0),
 \qquad x\in\R^d.
\end{equation}
Since coordinate evaluation is continuous, $\varphi$ is continuous.  For
every $z\in\Z^d$, equivariance gives the especially simple identity
\begin{equation}\label{eq:scalar-all-coordinates}
 \varphi(x-z)
 =\Lambda(x-z)(0)
 =(\rho_{-z}\Lambda(x))(0)
 =\Lambda(x)(z).
\end{equation}
Thus the integer translates of the one scalar function are exactly the
coordinates of the configuration $\Lambda(x)$.

Since $-T=T=R_0$, the row-normalization condition gives the
partition-of-unity identity
\begin{equation}\label{eq:phi-partition}
 \sum_{t\in T}\varphi(x+t)
 =\sum_{t\in T}\Lambda(x)(-t)=1,
 \qquad x\in\R^d.
\end{equation}
If $p\in P\setminus\{0\}$ and both $\varphi(x)$ and $\varphi(x-p)$ were
positive, then \eqref{eq:scalar-all-coordinates} implies that $\Lambda(x)$ is
positive at the distinct addresses $0$ and $p$, which lie in the same column
$Q_0=P$.  This contradicts \eqref{eq:soft-column}.  Hence
\begin{equation}\label{eq:phi-separation}
 \varphi(x)\varphi(x-p)=0,
 \qquad x\in\R^d,
 \quad p\in P\setminus\{0\}.
\end{equation}

Only finitely many addresses occur while $x$ ranges over one closed unit
cube; denote their set by $E$.  If $u\in[0,1)^d$ and
$\varphi(u-z)>0$, then \eqref{eq:scalar-all-coordinates} implies $z\in E$.
Therefore
\[
 \supp\varphi\subset[0,1]^d-E,
\]
which is compact.  Hence $\varphi\in C_c(\R^d)$.  Since $\Lambda$ is simplexwise affine,
$\varphi$ is piecewise affine on the translated periodic triangulation.
Together with~\eqref{eq:phi-partition}, this proves that $\varphi$ is a
compactly supported piecewise-affine $T$-partition of unity.

\subsection{Thresholding}

For fixed $x$, the nonzero terms $\varphi(x+t)$, $t\in T$, are the positive
coordinates in one row of $\Lambda(x)$.  Each of the other $m-1$ rows must
use at least one column, and these columns are distinct by the
column-exclusion rule.  The chosen row can therefore use at most
\[
 D=n-m+1
\]
columns.  Hence at most $D$ members of the $T$-partition of unity
\eqref{eq:phi-partition} are nonzero at any point, and one of them is at
least $1/D$.
Choose
\[
 0<\tau<\frac1D
\]
and set
\[
 U=\{x:\varphi(x)>\tau\},
 \qquad
 K=\{x:\varphi(x)\ge\tau\}.
\]
Then $U$ is open, $K$ is compact, and $U+T=\R^d$.  Moreover,
\[
 K\cap(K+p)=\varnothing,
 \qquad p\in P\setminus\{0\},
\]
by \eqref{eq:phi-separation}.  Since $K$ is compact and $P$ is discrete,
\begin{equation}\label{eq:positive-K-separation}
 \delta=
 \inf_{p\in P\setminus\{0\}}\dist(K,K+p)>0.
\end{equation}
To see the strict positivity, choose $R>0$ so large that
$\dist(K,K+p)>1$ whenever $\norm{p}_2>R$.  Only finitely many nonzero
$p\in P$ satisfy $\norm{p}_2\le R$, and each of the corresponding compact
sets $K$ and $K+p$ is disjoint.  The minimum of their positive distances is
therefore positive.

\subsection{The first-hit selector}\label{sec:first-hit-selector}

Let $F_T$ be a bounded half-open fundamental parallelepiped for $T$ and enumerate $T=\{t_1,t_2,\ldots\}$.
For $x\in F_T$, define
\[
 j(x)=\min\{j:x+t_j\in U\}.
\]
This is defined because $U+T=\R^d$.  Set
\begin{equation}\label{eq:first-hit-Omega}
 \Omega=\{x+t_{j(x)}:x\in F_T\}.
\end{equation}
For each $j$, the set
\[
 E_j=\{x\in F_T:x+t_j\in U,\ x+t_i\notin U\text{ for }i<j\}
\]
is Borel and
\[
 \Omega=\bigcup_j(E_j+t_j).
\]
The function $\varphi$ is affine on every simplex of a periodic
triangulation and has compact support.  It is therefore affine on the cells
of a finite polyhedral subdivision of its support.  Consequently, $U$ is a
finite union of relatively open polyhedral sets.  Since both $F_T$ and $U$
are bounded, only finitely many translates $F_T+t_j$ can meet $U$.  Hence
only finitely many sets $E_j$ are nonempty, and each nonempty $E_j$ is a
finite union of half-open polyhedral pieces.  It follows that $\Omega$ is a
finite union of bounded half-open polyhedral pieces.

The set $\Omega$ contains exactly one representative of every
$T$-coset, so it is a $T$-fundamental domain.  Moreover,
$\Omega\subset U\subset K$, and therefore
\[
 \dist(\Omega,\Omega+p)\ge\delta,
 \qquad p\in P\setminus\{0\}.
\]
For any $0<\eps<\delta/2$, the set $\Omega+B_\eps$ packs with $P$.  The inclusion
$\Omega\subset K$ also proves boundedness.

Undoing the initial common linear normalization proves commensurable
sufficiency.  Indeed, if $S\in GL_d(\R)$ is the normalizing map, then
$S^{-1}(B_\eps)$ contains a Euclidean ball of some positive radius.  Thus a
positive packing buffer in the normalized coordinates yields a positive
Euclidean packing buffer for the original lattices.

\section{Examples}\label{sec:commensurable-examples}

We now record four explicit constructions. In each case we list the active
addresses that occur in the construction. The first example is written out in
detail, including the simplexwise formulas for $\Lambda$, the induced scalar
function $\varphi$, and the resulting fundamental domain $\Omega$.

\subsection{A fully worked example: $T=\boldsymbol{\Z^2}$ and $P=\boldsymbol{3\Z\times\Z}$}
\label{subsec:worked-example-Z2-3ZxZ}

We now carry out the construction explicitly for
\[
 T=\Z^2,
 \qquad
 P=3\Z\times\Z.
\]
Then
\[
 T+P=\Z^2,
 \qquad
 m=[\Z^2:T]=1,
 \qquad
 n=[\Z^2:P]=3,
\]
so the sharp equality case $n=m+d$ holds when $d=2$.
The three columns are the congruence classes of the first coordinate modulo~$3$,
and there is only one row.

We work on the unit square with vertices
\[
 v_0=(0,0),\quad v_1=(1,0),\quad v_2=\Bigl(1,\frac12\Bigr),\quad
 v_3=(1,1),\quad v_4=(0,1),\quad v_5=\Bigl(0,\frac12\Bigr),
\]
and triangulate it by
\[
 \Delta_1=[v_0,v_1,v_5],
 \quad
 \Delta_2=[v_1,v_2,v_5],
 \quad
 \Delta_3=[v_2,v_3,v_5],
 \quad
 \Delta_4=[v_3,v_4,v_5].
\]
Assign the hard configurations
\[
 v_0\mapsto\delta_{(0,0)},
 \quad
 v_1\mapsto\delta_{(1,0)},
 \quad
 v_2\mapsto\delta_{(0,0)},
 \quad
 v_3\mapsto\delta_{(1,1)},
 \quad
 v_4\mapsto\delta_{(0,1)},
 \quad
 v_5\mapsto\delta_{(-1,0)}.
\]
On each triangle we interpolate affinely.  Relative to $[0,1]^2$, we have
\[
 \begin{aligned}
 \Delta_1&=\{x+2y\le1\},
 &\qquad
 \Delta_2&=\{y\le\tfrac12,\ x+2y\ge1\},\\
 \Delta_3&=\{y\ge\tfrac12,\ 2y\le1+x\},
 &\qquad
 \Delta_4&=\{2y\ge1+x\},
 \end{aligned}
\]
and obtain the configuration field
\begin{equation}\label{eq:worked-example-Lambda}
 \Lambda(x,y)=
 \begin{cases}
  (1-x-2y)\delta_{(0,0)}+x\delta_{(1,0)}+2y\delta_{(-1,0)},
   &(x,y)\in\Delta_1,\\[1mm]
  (x+2y-1)\delta_{(0,0)}+(1-2y)\delta_{(1,0)}+(1-x)\delta_{(-1,0)},
   &(x,y)\in\Delta_2,\\[1mm]
  (x-2y+1)\delta_{(0,0)}+(2y-1)\delta_{(1,1)}+(1-x)\delta_{(-1,0)},
   &(x,y)\in\Delta_3,\\[1mm]
  (2y-1-x)\delta_{(0,1)}+x\delta_{(1,1)}+(2-2y)\delta_{(-1,0)},
   &(x,y)\in\Delta_4.
 \end{cases}
\end{equation}
These formulas agree on common edges.  Moreover,
\[
 \Lambda(1,y)=\rho_{(1,0)}\Lambda(0,y),
 \qquad
 \Lambda(x,1)=\rho_{(0,1)}\Lambda(x,0),
\]
so $\Lambda$ extends to all of $\R^2$ by
\[
 \Lambda(u+z)=\rho_z\Lambda(u),
 \qquad u\in[0,1)^2,\ z\in\Z^2.
\]

The associated scalar function is determined by
\[
 \varphi(u-z)=\Lambda(u)(z),
 \qquad u\in[0,1)^2,\ z\in\Z^2.
\]
Only the five addresses
\[
 (0,0),\ (1,0),\ (-1,0),\ (0,1),\ (1,1)
\]
occur.  On the fundamental square the nonzero coordinate functions are
\begin{equation}\label{eq:worked-example-phi-table}
\begin{array}{c|cccc}
 z&\Delta_1&\Delta_2&\Delta_3&\Delta_4\\ \hline
 (0,0)
 &1-x-2y&x+2y-1&x-2y+1&0\\
 (1,0)
 &x&1-2y&0&0\\
 (-1,0)
 &2y&1-x&1-x&2-2y\\
 (0,1)
 &0&0&0&2y-1-x\\
 (1,1)
 &0&0&2y-1&x.
\end{array}
\end{equation}
Hence $\varphi$ is a nonnegative compactly supported continuous piecewise-affine function.
By construction,
\[
 \sum_{t\in\Z^2}\varphi(x+t)=1,
 \qquad
 \varphi(x)\varphi(x-p)=0\quad (p\in(3\Z\times\Z)\setminus\{0\}).
\]

Choose $\tau=0.3$.  In this detailed example and in Figure~\ref{fig:worked-example-tiling-packing}, the first-hit selector uses the active addresses
$((0,0),(1,0),(-1,0),(0,1),(1,1))$, equivalently the enumeration
\begin{equation}\label{eq:standard-first-hit-enumeration}
 (t_1,t_2,t_3,t_4,t_5)=((0,0),(-1,0),(1,0),(0,-1),(-1,-1))
\end{equation}
of the corresponding $T$-translates, where $t_j=-a_j$.
For $u\in[0,1)^2$, let $j(u)$ be the first index for which
$\Lambda(u)(a_j)>\tau$.  Since on every triangle three nonnegative coefficients add to $1$,
some coefficient is at least $1/3$, so $j(u)$ is well defined.  We then set
\begin{equation}\label{eq:worked-example-Omega}
 \Omega=\{u-a_{j(u)}:u\in[0,1)^2\}.
\end{equation}
Equivalently, if
\[
 E_j=\{u\in[0,1)^2:\Lambda(u)(a_j)>\tau\text{ and }\Lambda(u)(a_i)\le \tau\text{ for }i<j\},
\]
then
\[
 \Omega=\bigcup_{j=1}^5(E_j-a_j).
\]
Thus $\Omega$ is a finite union of bounded half-open polygons and a $\Z^2$-fundamental domain.  Since
$\Omega\subset\{\varphi\ge 0.3\}$ and the latter set has positive distance from all its
nontrivial translates by $3\Z\times\Z$, the set $\Omega+B_\eps$ packs with
$3\Z\times\Z$ for all sufficiently small $\eps>0$.

\begin{figure}[htbp]
\centering
\begin{minipage}[b]{.48\textwidth}
  \centering
  \RasterWithZTwo{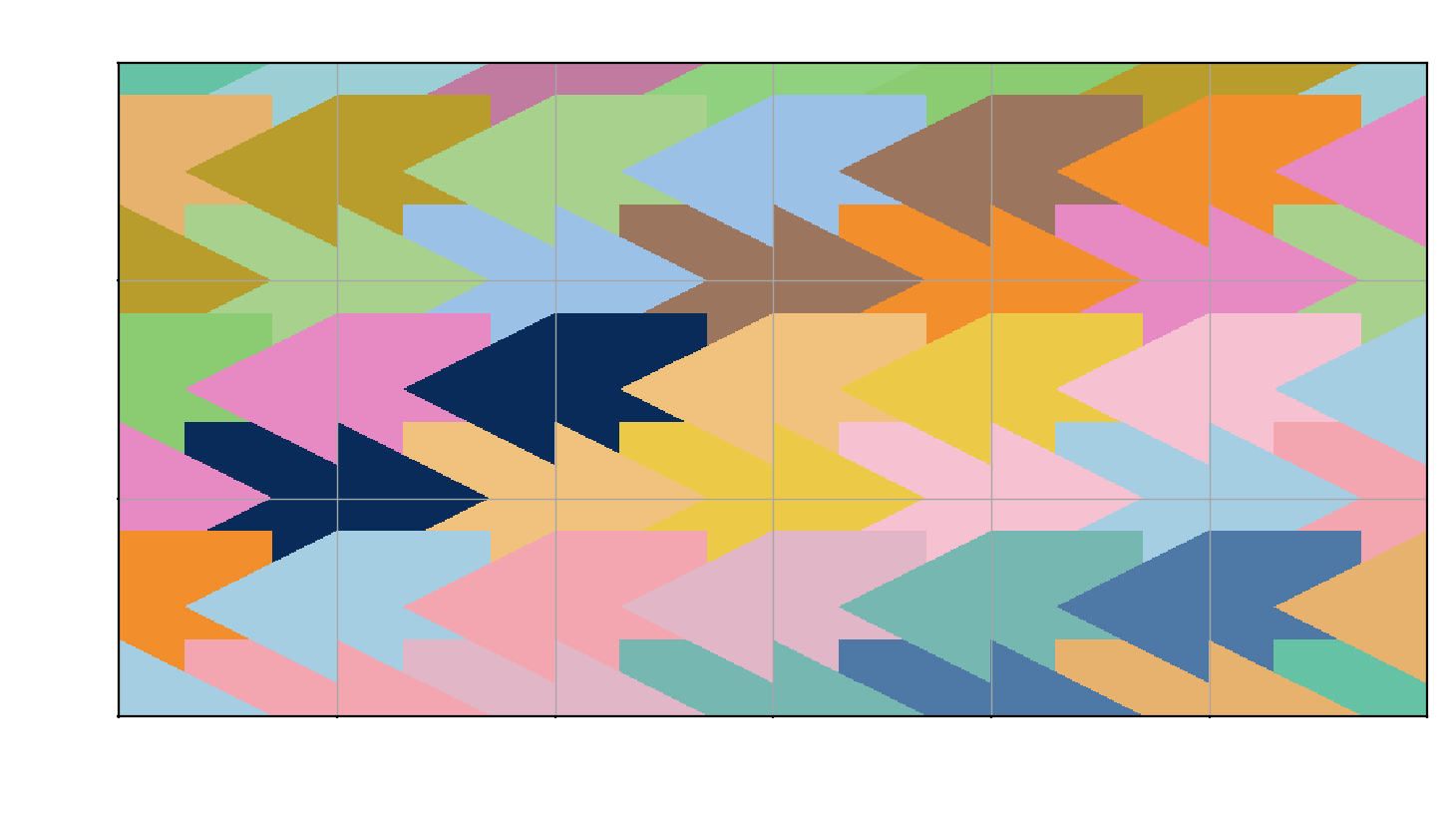}{.230936}{.394949}{.149772}{.264646}{-1}{5}{-1}{2}
\end{minipage}\hfill
\begin{minipage}[b]{.48\textwidth}
  \centering
  \RasterWithThreeZxZ{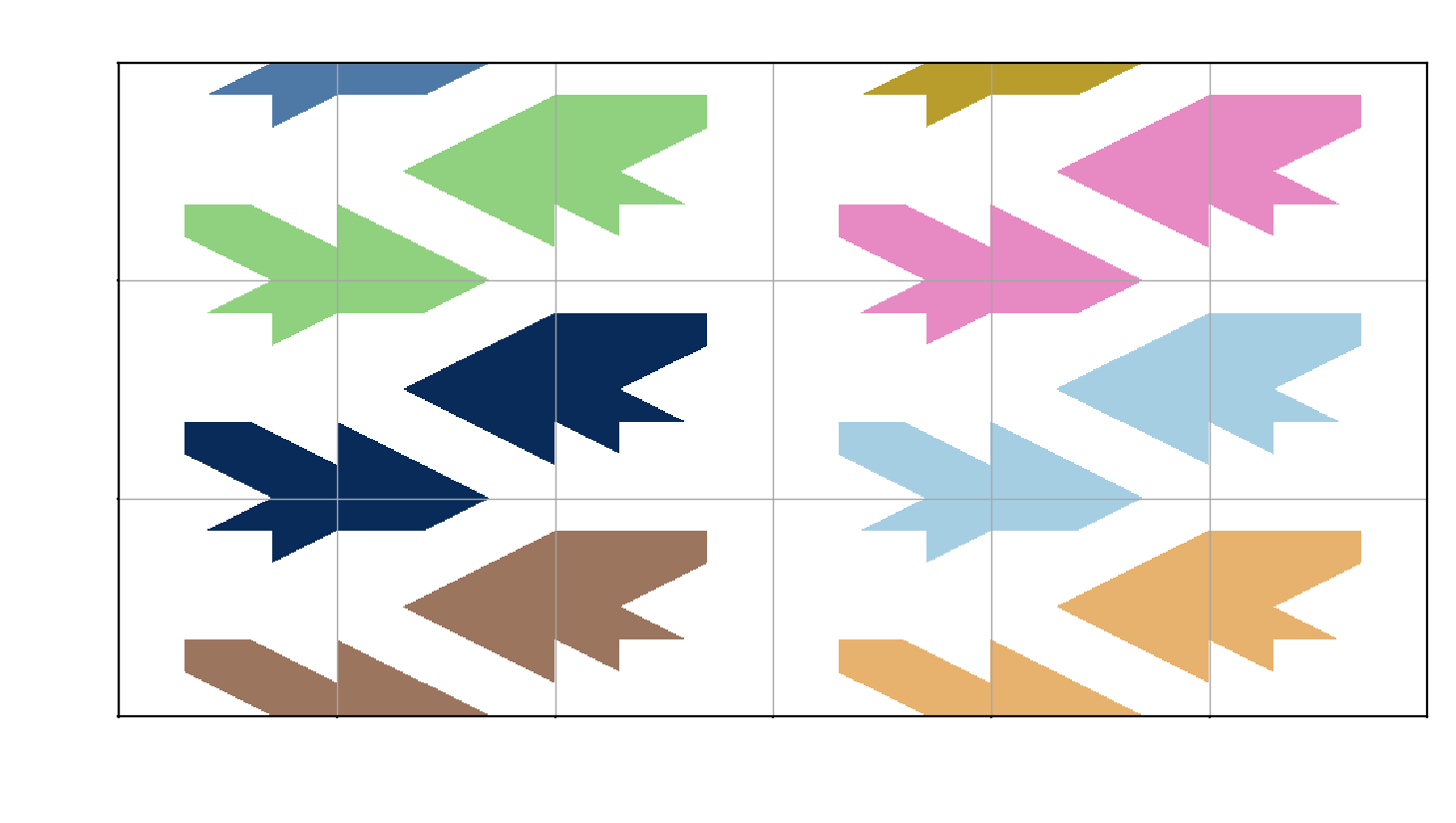}{.230936}{.394949}{.149772}{.264646}{-1}{5}{-1}{2}
\end{minipage}
\caption{The domain obtained from the first-hit enumeration~\eqref{eq:standard-first-hit-enumeration}.  Left: the tiling $\Omega+\Z^2$ on $[-1,5]\times[-1,2]$.  Right: the corresponding translates by $3\Z\times\Z$ on the same rectangle.  The unshifted domain is dark blue, and the visible gaps give the positive packing buffer.}
\label{fig:worked-example-tiling-packing}
\end{figure}

\begin{figure}[htbp]
\centering
\begin{minipage}[b]{.48\textwidth}
  \centering
  \includegraphics[width=\linewidth]{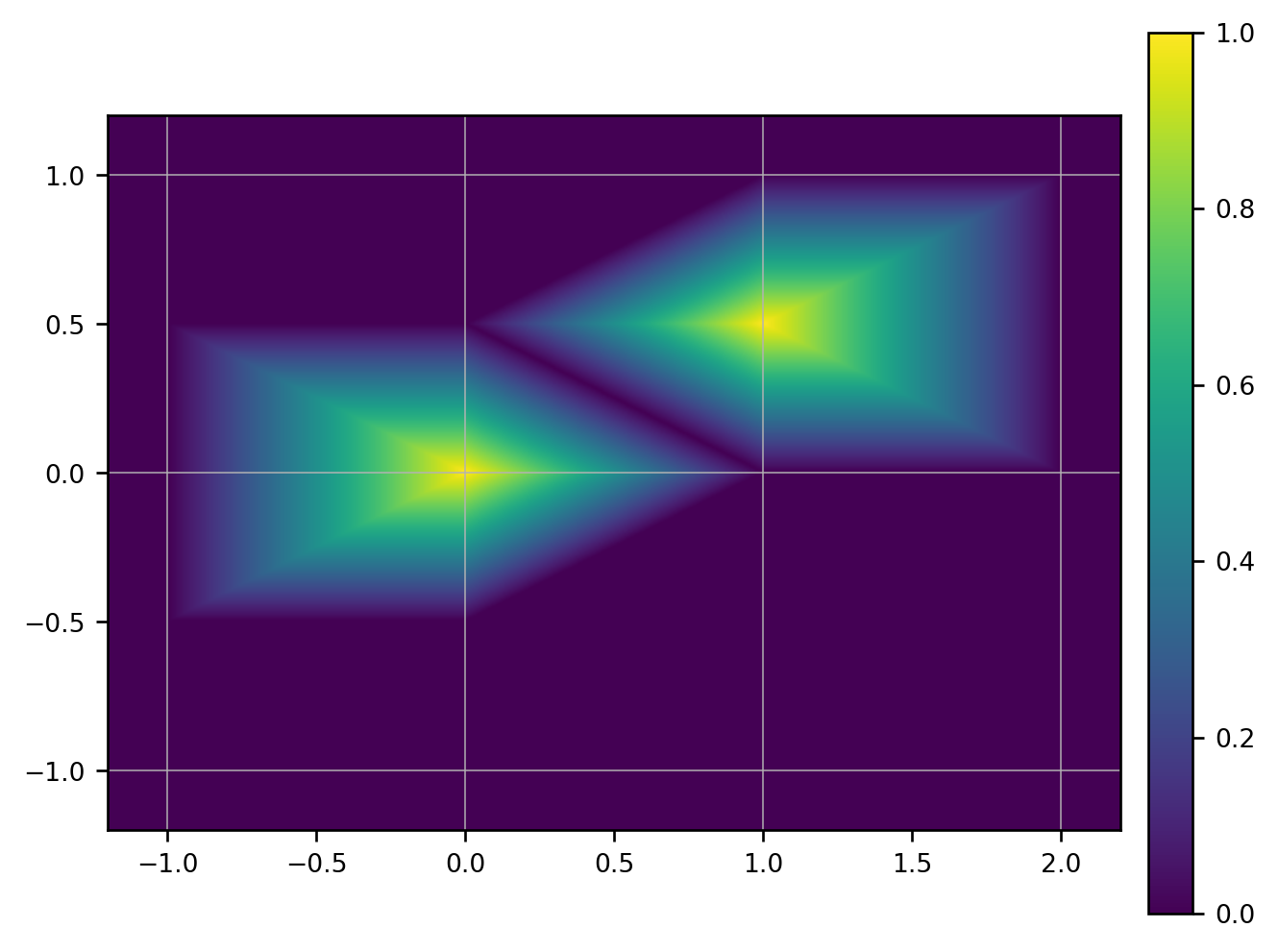}
\end{minipage}\hfill
\begin{minipage}[b]{.48\textwidth}
  \centering
  \includegraphics[width=\linewidth]{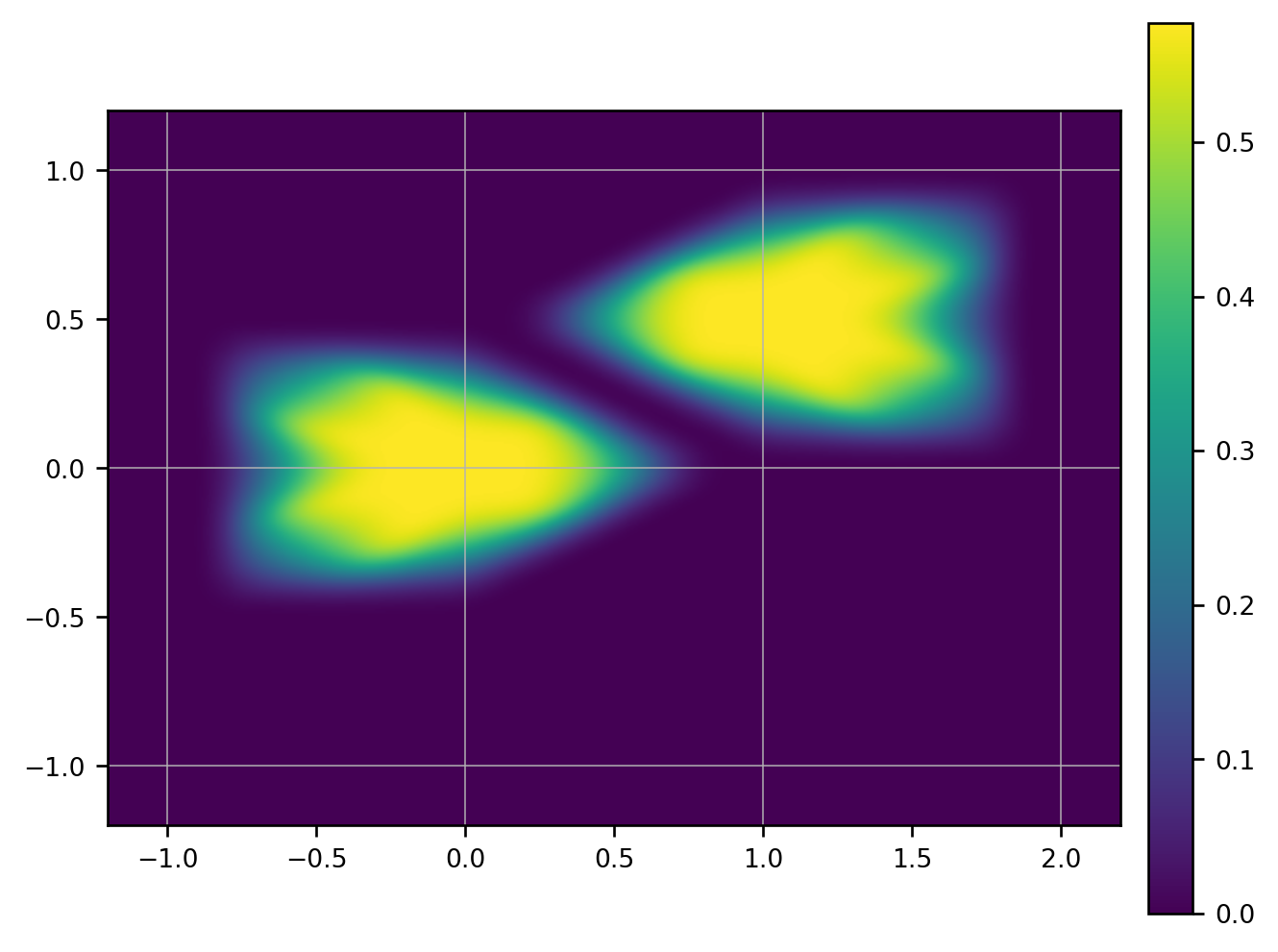}
\end{minipage}
\caption{Left: a two-dimensional color plot of the compactly supported continuous piecewise-affine function $\varphi$ produced from the soft-configuration field.
Right: a two-dimensional color plot of the smooth window $g$ obtained in Part~B by
cutting $\varphi$ below a positive level, mollifying the result, and then
normalizing the periodic energy.}
\label{fig:worked-example-phi-g}
\end{figure}

\subsection{Two further equality-case examples}\label{subsec:two-further-equality-examples}

We record two explicit constructions at the planar equality threshold.  For
$T=7\Z\times\Z$ and $P=3\Z\times3\Z$, one has $m=7$ and $n=9=m+2$; the positive addresses used by the construction are
\begin{equation}\label{eq:seven-nine-address-set}
\begin{aligned}
E_{7,9}=\{&(1,1),(1,0),(0,1),(0,0),(-1,1),(-1,0),(-2,1),\\
&(-2,0),(-2,-1),(-3,0),(-3,-1),(-4,0),(-4,-1),\\
&(-5,0),(-5,-1),(-5,-2),(-6,-1),(-6,-2),(-7,-2)\}.
\end{aligned}
\end{equation}

\begin{figure}[htbp]
\centering
\begin{minipage}[b]{.48\textwidth}\centering
  \RasterWithSevenZxZ{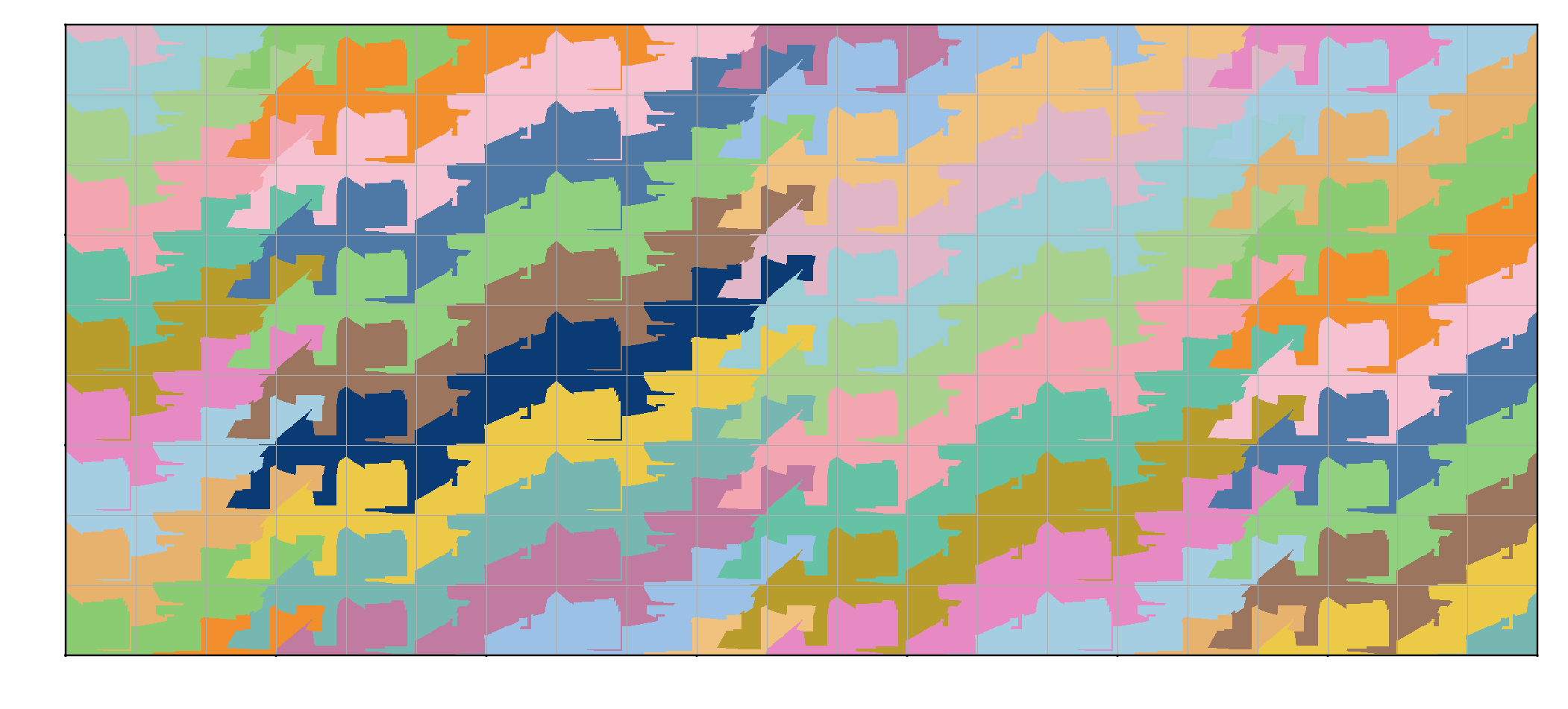}{.175717}{.372766}{.044697}{.098843}{-3}{18}{-3}{6}
\end{minipage}\hfill
\begin{minipage}[b]{.48\textwidth}\centering
  \RasterWithThreeZxThreeZ{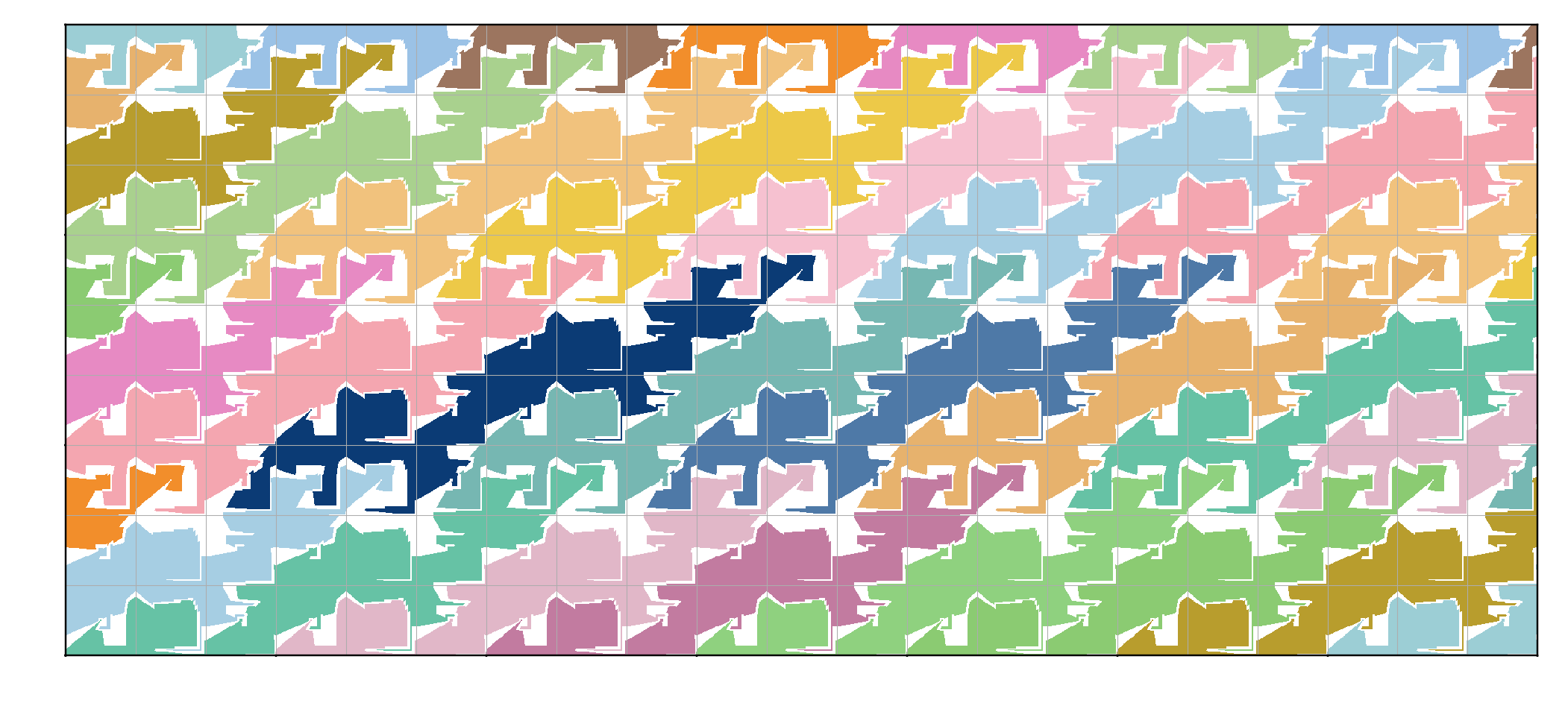}{.175717}{.372766}{.044697}{.098843}{-3}{18}{-3}{6}
\end{minipage}
\caption{The nineteen-address construction for $T=7\Z\times\Z$ and $P=3\Z\times3\Z$.  From left to right: the $T$-tiling and the buffered $P$-packing.}
\label{fig:seven-nine-tiling-packing}
\end{figure}

For $T=\Z\times4\Z$ and $P=6\Z\times\Z$, one has $m=4$ and $n=6=m+2$; the positive addresses are
\begin{equation}\label{eq:four-six-address-set}
\begin{aligned}
E_{4,6}=\{&(0,0),(1,0),(2,1),(3,2),(4,3),(5,0),(6,0),(0,1),\\
&(1,1),(1,2),(2,2),(2,3),(3,3),(3,4),(4,4)\}.
\end{aligned}
\end{equation}

\begin{figure}[htbp]
\centering
\begin{minipage}[b]{.48\textwidth}\centering
  \RasterWithZxFourZ{example_4x6_tiling_upright.png}{.213277}{.206178}{.075742}{.077047}{-2}{10}{-2}{10}
\end{minipage}\hfill
\begin{minipage}[b]{.48\textwidth}\centering
  \RasterWithSixZxZ{example_4x6_packing_upright.png}{.213277}{.206178}{.075742}{.077047}{-2}{10}{-2}{10}
\end{minipage}
\caption{The fifteen-address construction for $T=\Z\times4\Z$ and $P=6\Z\times\Z$.  From left to right: the $T$-tiling and the buffered $P$-packing.}
\label{fig:four-six-tiling-packing}
\end{figure}

\section{Abstract row--column complexes and proof of Theorem~\ref{thm:soft-filling}}
\label{sec:abstract-row-column}
\label{sec:proof-soft-filling}

The extension problem used in Section~\ref{sec:soft-configurations} depends only on the row--column
incidence structure, not on the ambient lattice groups.  We now forget the
group origin of the address classes and prove a stronger abstract theorem.

Fix integers $1\le m\le n$.  For every row $i=1,\ldots,m$ and column
$j=1,\ldots,n$, let $A_{i,j}$ be a nonempty set, not necessarily finite.
The sets are regarded as disjoint labeled sets, and
\[
 A_i=\bigsqcup_{j=1}^nA_{i,j}.
\]
An abstract soft configuration is a tuple
$\lambda=(\lambda_1,\ldots,\lambda_m)$ such that every $\lambda_i$ is a
finitely supported probability vector on $A_i$ and distinct positive
coordinates never use the same column.  Explicitly,
\[
 \lambda_i(a)\ge0,
 \qquad
 \#\supp\lambda_i<\infty,
 \qquad
 \sum_{a\in A_i}\lambda_i(a)=1,
\]
and, whenever $a\in A_{i,j}$ and $b\in A_{i',j}$ are distinct labeled
addresses,
\[
 \lambda_i(a)\lambda_{i'}(b)=0.
\]
The resulting, possibly infinite, polyhedral complex is denoted by
$\mathcal C_{m,n}(A)$.  We regard it as a union of finite-dimensional cells
inside the coordinate space
\[
 \prod_{i=1}^m\mathbb R^{(A_i)},
\]
where $\mathbb R^{(A_i)}$ denotes the finitely supported real functions on
$A_i$.  If $S=(S_1,\ldots,S_m)$ is an admissible finite support pattern, its
cell is
\[
 C_S=\prod_{i=1}^m\Delta(S_i),
 \qquad
 \mathcal C_{m,n}(A)=\bigcup_{S\text{ admissible}}C_S.
\]
When every $A_{i,j}$ is finite, this is the finite polyhedral complex used in
the proof below.  Simplexwise affine maps into $\mathcal C_{m,n}(A)$ are
defined by the same vertex-interpolation rule as in Section~\ref{sec:hard-configurations}.

\begin{theorem}\label{thm:soft-filling-strong}
Let every $A_{i,j}$ be a nonempty set, finite or infinite, and let
$\mathcal C_{m,n}(A)$ consist of the finitely supported abstract soft
configurations defined above.  For every integer
\[
 1\le k\le n-m,
\]
every simplexwise affine map on a finitely triangulated $S^{k-1}$ extends
simplexwise affinely over a finitely triangulated $B^k$ without changing the
boundary triangulation or its vertex values.  Moreover, for each boundary
map there are finite nonempty subsets
\[
 B_{i,j}\subset A_{i,j}
\]
containing every address that occurs in a boundary vertex value, and the
extension may be chosen to take values in the finite subcomplex
$\mathcal C_{m,n}(B)\subset\mathcal C_{m,n}(A)$.  If $n=m$, the space
$\mathcal C_{m,m}(A)$ is nonempty.
\end{theorem}

\begin{theorem}\label{thm:continuous-soft-filling-main}
Let every $A_{i,j}$ be finite and nonempty.  For every integer
$1\le k\le n-m$, every continuous map
\[
 f:S^{k-1}\longrightarrow\mathcal C_{m,n}(A)
\]
extends to a continuous map
\[
 F:B^k\longrightarrow\mathcal C_{m,n}(A).
\]
If $n=m$, the space $\mathcal C_{m,m}(A)$ is nonempty.  Consequently,
$\mathcal C_{m,n}(A)$ is $(n-m-1)$-connected.
\end{theorem}

Theorem~\ref{thm:continuous-soft-filling-main} is proved in
Appendix~\ref{app:continuous-soft-filling}.  The appendix starts from the singleton-color theorem of
Abrams--Gay--Hower, adds colors one column at a time, and uses a gluing lemma
to preserve the wedge-of-spheres homotopy type.  A padding-and-retraction
argument then gives the general finite-address statement, followed by the
classical deleted-product, Hom-complex, join, and matching interpretations.

\subsection{The finite-address reduction}

We first show that every finite simplexwise-affine problem lives inside a finite
subcomplex, even when the ambient address classes are infinite.

\begin{lemma}\label{lem:finite-address-reduction}
Let every $A_{i,j}$ be nonempty, possibly infinite.  Let $K$ be a finite
simplicial complex and let
\[
 f:|K|\longrightarrow\mathcal C_{m,n}(A)
\]
be simplexwise affine.  Then there are finite nonempty sets
$B_{i,j}\subset A_{i,j}$ such that
\[
 f(|K|)\subseteq\mathcal C_{m,n}(B).
\]
The sets $B_{i,j}$ may be chosen to contain every address occurring in the
support of a vertex value $f(v)$, $v\in K^{(0)}$.
\end{lemma}

\begin{proof}
For every row and column put
\[
 E_{i,j}
 =\bigcup_{v\in K^{(0)}}
   \bigl(\supp f(v)_i\cap A_{i,j}\bigr).
\]
The triangulation has finitely many vertices, and each vertex value is a
finitely supported configuration.  Hence every $E_{i,j}$ is finite.  If
$E_{i,j}$ is nonempty, set $B_{i,j}=E_{i,j}$; if it is empty, choose one
address $b_{i,j}\in A_{i,j}$ and set $B_{i,j}=\{b_{i,j}\}$.  Thus $B$ is a
finite complete row--column subsystem.

Let $\sigma=[v_0,\ldots,v_p]$ be a simplex of $K$ and write
$x=\sum_{\ell=0}^p t_\ell v_\ell$ in barycentric coordinates.  Since $f$ is
affine on $\sigma$,
\[
 f(x)=\sum_{\ell=0}^p t_\ell f(v_\ell).
\]
Every coordinate outside the union of the supports of the vertex values is
therefore zero.  All those supports are contained in the sets $B_{i,j}$, so
$f(\sigma)\subseteq\mathcal C_{m,n}(B)$.  Taking the union over the finitely
many simplices proves the assertion.
\end{proof}

This is a finite-address reduction, not a local-finiteness assertion about
the entire infinite complex.  Indeed, when some $A_{i,j}$ is infinite,
infinitely many cells can meet at one hard configuration.  What is used here
is only that a simplexwise affine problem on a finite triangulation activates
finitely many addresses.

To prove Theorem~\ref{thm:soft-filling-strong}, it now suffices to establish
the theorem when all $A_{i,j}$ are finite.  Once that finite case is known,
Lemma~\ref{lem:finite-address-reduction} places any boundary map in a finite
complete subcomplex $\mathcal C_{m,n}(B)$, and its extension there is also an
extension in $\mathcal C_{m,n}(A)$.  For the finite-case proof below, we
therefore assume that every $A_{i,j}$ is finite.  This assumption is used in
the last-column decomposition to make the cover by the sets $U_v$ finite.

\subsection{A finite-piece extension lemma}

The next lemma is a convenient device for the last-column decomposition.
All polyhedral subcomplexes below are unions of cells of one fixed finite
polyhedral structure.

\begin{lemma}\label{lem:soft-pieces}
Let $I$ be a finite nonempty index set and let
\[
 X=\bigcup_{\alpha\in I}X_\alpha
\]
be a finite polyhedral complex written as a union of subcomplexes.  Suppose
that a nonempty subcomplex $A$ satisfies
\[
 A\subseteq X_\alpha,\qquad \alpha\in I,
\]
and
\[
 X_\alpha\cap X_\beta=A,
 \qquad \alpha,\beta\in I,\quad \alpha\ne\beta.
\]
Fix $k\ge1$.  Assume that, for every $\alpha\in I$ and every
$0\le p\le k$, every simplexwise affine map from a triangulated $S^p$ into
$X_\alpha$ extends simplexwise affinely over $B^{p+1}$ without changing the
boundary triangulation.  Assume also that the same statement holds with
$A$ as target for every $0\le p\le k-1$.  Then every simplexwise affine map
from a triangulated $S^k$ into $X$ extends simplexwise affinely over
$B^{k+1}$.
\end{lemma}

\begin{proof}
Choose a finite triangulation $K$ of $S^k$ with respect to which the
boundary map $f:|K|\to X$ is simplexwise affine.  Thus $f$ is affine on
every simplex of $K$, and the image of each simplex lies in one cell of the
fixed polyhedral structure on $X$.
Choose $a_0\in A$, add an interior vertex $c$, and identify the cone
$c*K$ with $B^{k+1}$.  Set $F(c)=a_0$.  For every simplex $\sigma\in K$ put
\[
 \widehat\sigma=\operatorname{conv}(\{c\}\cup\sigma).
\]
We fill these conical simplices in increasing order of $\dim\sigma$; every
filling is allowed to subdivide the interior of $\widehat\sigma$ but must
leave its already triangulated boundary unchanged.

Let $C_\sigma$ be the unique smallest target cell containing
$f(\sigma)$.  If $C_\sigma\not\subseteq A$, then it is contained in a unique
piece $X_{\alpha(\sigma)}$: if it belonged to two distinct pieces, it would
be contained in their intersection $A$.  Fix $\alpha_0\in I$ and define
\[
 Z_\sigma=
 \begin{cases}
 X_{\alpha(\sigma)},&C_\sigma\not\subseteq A,\\
 A,&C_\sigma\subseteq A\text{ and }\dim\sigma<k,\\
 X_{\alpha_0},&C_\sigma\subseteq A\text{ and }\dim\sigma=k.
 \end{cases}
\]
The assignments are nested along faces:
\begin{equation}\label{eq:carrier-nesting}
 \eta\subsetneq\sigma\quad\Longrightarrow\quad Z_\eta\subseteq Z_\sigma.
\end{equation}
Indeed, $f(\eta)\subseteq f(\sigma)$ implies that the smallest cell
$C_\eta$ is a face of $C_\sigma$.  If $C_\sigma\not\subseteq A$, both cells
lie in the same unique piece unless $C_\eta\subseteq A$, in which case
$Z_\eta=A\subset X_{\alpha(\sigma)}$.  If $C_\sigma\subseteq A$, then
$C_\eta\subseteq A$ as well.

Suppose that the regions $\widehat\eta$ have already been filled for all
proper faces $\eta$ of $\sigma$.  The top face $\sigma$ maps into
$Z_\sigma$, and each side face $\widehat\eta$ maps into
$Z_\eta\subseteq Z_\sigma$.  Thus the complete triangulated boundary of the
$(\dim\sigma+1)$-simplex $\widehat\sigma$ maps into $Z_\sigma$.  If
$Z_\sigma=A$, then $\dim\sigma\le k-1$, so the extension hypothesis for $A$
applies with $p=\dim\sigma$.  Otherwise the hypothesis for the appropriate
piece $X_\alpha$ applies, since $\dim\sigma\le k$.  We therefore triangulate
$\widehat\sigma$ relative to its boundary and fill it inside $Z_\sigma$.

Proceeding through all simplices of $K$ fills the cone $c*K$.  Whenever two
filled regions meet, their common face had already been triangulated and
assigned vertex values before either interior was filled.  The affine
formulas therefore agree on that face.  Hence the resulting map is
simplexwise affine and continuous on $B^{k+1}$.
\end{proof}

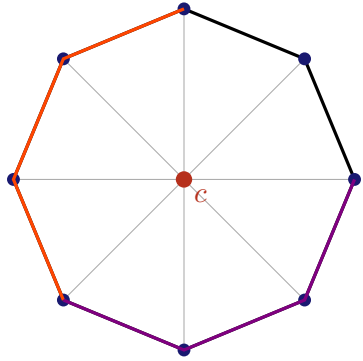
\begin{figure}[htbp]
\centering
\begin{tikzpicture}[scale=1.1]
  \coordinate (c) at (0,0);
  \foreach \a in {0,45,...,315}{\coordinate (p\a) at (\a:2.05);}
  \draw[very thick] (p0)--(p45)--(p90)--(p135)--(p180)--(p225)--(p270)--(p315)--cycle;
  \foreach \a in {0,45,...,315}{\draw[gray!65] (c)--(p\a);}
  \fill[BrickRed] (c) circle (2.8pt) node[below right] {$c$};
  \foreach \a in {0,45,...,315}{\fill[MidnightBlue] (p\a) circle (2.2pt);}
  \draw[very thick,OrangeRed] (p90)--(p135)--(p180)--(p225);
  \draw[very thick,Purple] (p225)--(p270)--(p315)--(p0);
\end{tikzpicture}
\caption{The ball is filled from the interior vertex toward the boundary,
processing the regions associated with lower-dimensional boundary simplices
before those associated with higher-dimensional ones.}
\label{fig:radial-filling}
\end{figure}

\subsection{The last-column decomposition}

Assume throughout this subsection that $n>m$.  Let
\[
 A_{\bullet,n}=\bigsqcup_{i=1}^m A_{i,n}
\]
be the labeled addresses in the last column.  For
$v\in A_{\bullet,n}$, define
\[
 U_v=\Bigl\{\lambda\in\mathcal C_{m,n}(A):
   \lambda_i(a)=0\text{ for every }i\text{ and every }
   a\in A_{i,n}\setminus\{v\}\Bigr\}.
\]
Thus the last column is either unused, or $v$ is its only positive address.
Also put
\[
 A_0=\Bigl\{\lambda\in\mathcal C_{m,n}(A):
   \lambda_i(a)=0\text{ for every }i\text{ and every }a\in A_{i,n}\Bigr\}.
\]
Every configuration either leaves the last column unused or uses a unique
address $v$ there.  Consequently,
\begin{equation}\label{eq:last-column-cover}
 \mathcal C_{m,n}(A)=\bigcup_{v\in A_{\bullet,n}}U_v,
 \qquad U_v\cap U_w=A_0,
 \quad v,w\in A_{\bullet,n},\quad v\ne w.
\end{equation}

\begin{figure}[htbp]
\centering
\begin{tikzpicture}[>=Latex]
  \node[draw,rounded corners,fill=MidnightBlue!8,
        minimum width=2.4cm,minimum height=.8cm] (Uv) at (-5,1.65) {$U_v$};
  \node[draw,rounded corners,fill=MidnightBlue!8,
        minimum width=2.4cm,minimum height=.8cm] (Uw) at (0,1.65) {$U_w$};
  \node[draw,rounded corners,fill=MidnightBlue!8,
        minimum width=2.4cm,minimum height=.8cm] (Uz) at (5,1.65) {$U_z$};
  \node[draw,rounded corners,fill=ForestGreen!10,
        minimum width=4.1cm,minimum height=.9cm] (A0) at (0,-1)
        {$A_0$: last column unused};
  \draw[-{Latex[length=2.3mm]},thick] (A0.north west) -- (Uv.south);
  \draw[-{Latex[length=2.3mm]},thick] (A0.north) -- (Uw.south);
  \draw[-{Latex[length=2.3mm]},thick] (A0.north east) -- (Uz.south);
  \node[font=\scriptsize,fill=white,inner sep=2pt] at (-2.5,2.65)
        {$U_v\cap U_w=A_0$};
  \node[font=\scriptsize,fill=white,inner sep=2pt] at (2.5,2.65)
        {$U_w\cap U_z=A_0$};
\end{tikzpicture}
\caption{The last-column decomposition.  Each piece permits at most one
specified address in the last column, and all distinct pieces meet in the
common subspace $A_0$ where that column is unused.}
\label{fig:last-column-decomposition}
\end{figure}
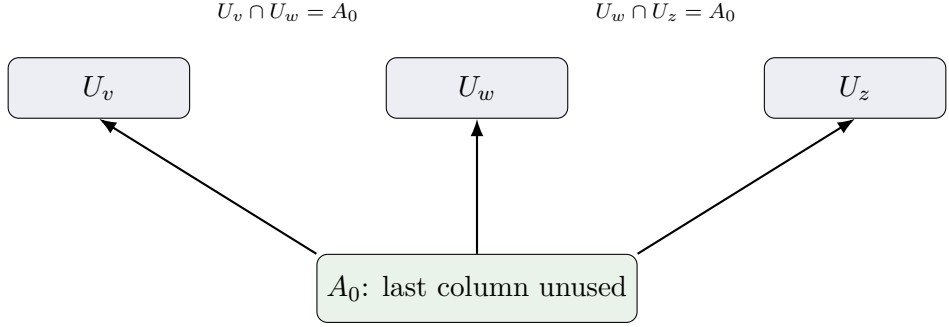

We now define the reduced address families explicitly.  Deleting the last
column gives
\[
 A^-=(A^-_{i,j})_{1\le i\le m,\,1\le j\le n-1},
 \qquad A^-_{i,j}=A_{i,j}.
\]
Consequently,
\begin{equation}\label{eq:A0-smaller}
 A_0\cong\mathcal C_{m,n-1}(A^-).
\end{equation}
If $v\in A_{i_0,n}$, let
\[
 \iota_v:\{1,\ldots,m-1\}\longrightarrow
 \{1,\ldots,m\}\setminus\{i_0\}
\]
be the increasing bijection, and define
\[
 A^v_{\ell,j}=A_{\iota_v(\ell),j},
 \qquad 1\le\ell\le m-1,
 \quad 1\le j\le n-1.
\]
Thus $A^v$ is obtained by deleting row $i_0$ and column $n$ and relabeling
the remaining rows.

Fix $v\in A_{i_0,n}$.  Define
\[
 (P_v\lambda)_{i_0}(v)=1,
 \qquad
 (P_v\lambda)_{i_0}(a)=0,
 \quad a\in A_{i_0}\setminus\{v\},
\]
and set $(P_v\lambda)_i=\lambda_i$ for $i\ne i_0$.  Let $L$ be the linear map that leaves every row other than $i_0$
unchanged and replaces the $i_0$-th row by zero, and let $e_v$ be the vector
whose only nonzero coordinate is the value $1$ at $v$.  Then
\[
 P_v\lambda=L\lambda+e_v.
\]
Thus $P_v$ is the restriction of an affine map on the ambient coordinate
space and is therefore continuous.  It leaves the other rows unchanged and
replaces the $i_0$-th row by the fixed point mass at $v$.  Its image is
\[
 Y_v=\{\lambda\in U_v:\lambda_{i_0}(v)=1\}.
\]
For $m\ge2$, deleting the fixed row and the unused last column gives
\begin{equation}\label{eq:Yv-smaller}
 Y_v\cong\mathcal C_{m-1,n-1}(A^v),
\end{equation}
while for $m=1$ the set $Y_v$ is one point.

The vertexwise definition makes compositions transparent.  Fix
$v\in A_{i_0,n}$, and suppose that $f$ is affine on a simplex
$\sigma=[v_0,\ldots,v_p]$ and that
\[
 f(v_\ell)\in U_v,
 \qquad \ell=0,\ldots,p.
\]
In the application below, the whole map $f$ takes values in $U_v$.  If
$x=\sum_{\ell=0}^p t_\ell v_\ell$ is written in barycentric coordinates, then, for
every row $i$ and every address $a\in A_{i,n}\setminus\{v\}$,
\[
 f(x)_i(a)
 =\sum_{\ell=0}^p t_\ell f(v_\ell)_i(a)
 =0.
\]
Hence $f(\sigma)\subset U_v$.  Let $C_S$ be the smallest
soft-configuration cell containing the vertex values.  It follows that the
support pattern $S$ uses no last-column address other than possibly $v$.
The map $P_v$ is affine on $C_S$ and sends it into the cell obtained by
replacing the support of row $i_0$ by $\{v\}$.  Therefore,
\[
 P_v\!\left(\sum_{\ell=0}^p t_\ell f(v_\ell)\right)
 =\sum_{\ell=0}^p t_\ell P_v(f(v_\ell)).
\]
Thus $P_v\circ f$ is simplexwise affine on the same triangulation; no target
subdivision is required.

\begin{figure}[htbp]
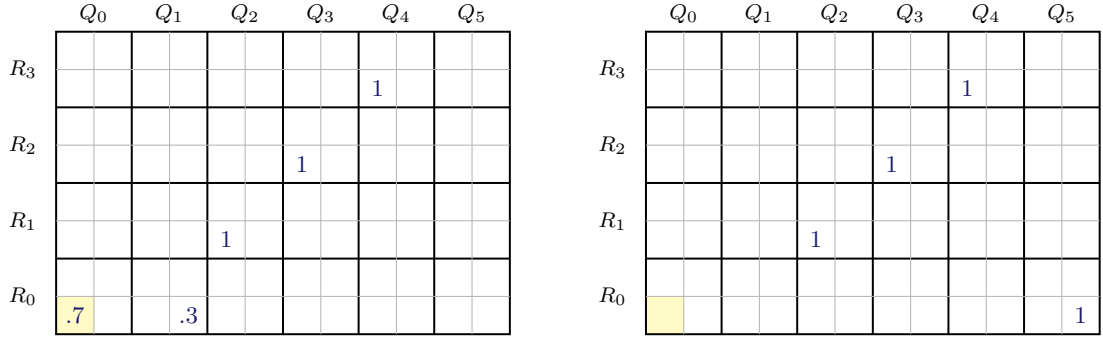

\centering
\begin{minipage}[b]{.48\textwidth}
\centering
\ConfigurationTable{%
  \ConfigTL{1}{1}{.7}%
  \ConfigTR{1}{2}{.3}%
  \ConfigTL{2}{3}{1}%
  \ConfigTL{3}{4}{1}%
  \ConfigTL{4}{5}{1}%
}
\end{minipage}\hfill
\begin{minipage}[b]{.48\textwidth}
\centering
\ConfigurationTable{%
  \ConfigTR{1}{6}{1}%
  \ConfigTL{2}{3}{1}%
  \ConfigTL{3}{4}{1}%
  \ConfigTL{4}{5}{1}%
}
\end{minipage}
\caption{The affine map $P_v$ in one example.  On the left, the last
column is unused and the first row splits its mass as $0.7+0.3$.  On the
right, $P_v$ transfers the whole first-row mass to the chosen address
$v\in A_{1,6}$ and leaves the other rows unchanged.}
\label{fig:Pv-mass-transfer}
\end{figure}

\begin{proposition}\label{prop:Uv-filling}
Fix $v\in A_{i_0,n}$.  If $m=1$, every simplexwise affine map
$S^p\to U_v$ extends simplexwise affinely over $B^{p+1}$ for every $p\ge0$.
If $m\ge2$ and Theorem~\ref{thm:soft-filling-strong} is known for
$\mathcal C_{m-1,n-1}(A^v)$, the same conclusion holds for
$0\le p\le n-m-1$.
\end{proposition}

\begin{proof}
Let $K$ triangulate $S^p$ and let $f:|K|\to U_v$ be simplexwise affine.
Take an outer copy of $K$ and a concentric inner copy.  For each
$p'$-simplex $\sigma=[z_0,\ldots,z_{p'}]$ of $K$, triangulate the region
between its two copies, identified with $\sigma\times[0,1]$, by the standard
staircase simplices
\[
 \operatorname{conv}\{(z_0,0),\ldots,(z_i,0),
                      (z_i,1),\ldots,(z_{p'},1)\},
 \qquad i=0,\ldots,p'.
\]
Assign the outer vertex $(z,0)$ the value $f(z)$ and the inner vertex
$(z,1)$ the value $P_v(f(z))$.

For a fixed $\sigma$, all outer and inner vertex values lie in one common
cell: start with a cell containing the outer values and replace the support
of row $i_0$ by its union with $\{v\}$.  This remains admissible because $U_v$ permits no
other positive last-column address.  Affine interpolation on every staircase
simplex therefore stays in $U_v$.  The triangulations agree on common prism
faces, and the formulas there use the same vertex values, so the resulting
map on the region between the two spheres is continuous.  It equals $f$ on
the outer sphere and $P_v\circ f$ on the inner sphere.

If $m=1$, the inner map is constant and fills the inner ball.  If $m\ge2$,
then
\[
 p\le n-m-1=(n-1)-(m-1)-1,
\]
so the induction hypothesis fills the inner map inside
$Y_v\cong\mathcal C_{m-1,n-1}(A^v)$.  The inner filling keeps its boundary
triangulation and vertex values, hence it joins the prism interpolation
without any further compatibility argument.  Together the two parts form a
simplexwise affine filling of $f$ over a ball.
\end{proof}

\subsection{Proof of Theorem~\ref{thm:soft-filling-strong}}

All ingredients are now in place, and we combine them in an induction on the total
number of rows and columns.

\begin{proof}
By Lemma~\ref{lem:finite-address-reduction}, it is enough to prove the
finite-address case.  We now prove that case by induction on $m+n$.  If
$n=m$, choose one address $a_i\in A_{i,i}$ in every row and put weight one
at $a_i$.  The chosen columns are distinct, so
$\mathcal C_{m,m}(A)$ is nonempty.

Assume $n>m$ and that the finite-address theorem is known whenever the sum of
the row and column numbers is smaller.  Let $K$ triangulate $S^k$, where
$0\le k\le n-m-1$, and let $f:|K|\to\mathcal C_{m,n}(A)$ be simplexwise
affine.  Every piece $U_v$ has the required filling property through
dimension $n-m-1$ by Proposition~\ref{prop:Uv-filling}, and
$A_0\cong\mathcal C_{m,n-1}(A^-)$.

If $k=0$, write the two values of $f$ as $x$ and $y$.  Choose pieces
$U_v$ and $U_w$ containing them.  Since $n-1\ge m$, the set $A_0$ is
nonempty; choose $a_0\in A_0$.  Proposition~\ref{prop:Uv-filling} gives a
polygonal path from $x$ to $a_0$ inside $U_v$ and another from $a_0$ to $y$
inside $U_w$.  Their concatenation fills $S^0$ over $B^1$.

Now let $k\ge1$.  For every $v\in A_{\bullet,n}$ and every $0\le p\le k$, maps into
$U_v$ extend.  For $0\le p\le k-1$,
\[
 p\le k-1\le n-m-2=(n-1)-m-1,
\]
so the induction hypothesis gives the corresponding extensions into $A_0$.
The finite-piece Lemma~\ref{lem:soft-pieces}, applied to
\eqref{eq:last-column-cover}, fills $f$ over $B^{k+1}$.  Since
$j=k+1$ ranges from $1$ to $n-m$, this proves the finite-address case.  The
general case follows from Lemma~\ref{lem:finite-address-reduction}, and the
extension produced there lies in the finite subcomplex
$\mathcal C_{m,n}(B)$ asserted in the theorem.
\end{proof}

Under the hypothesis $n\ge m+d$, apply
Theorem~\ref{thm:soft-filling-strong} to the possibly infinite address sets
\[
 A_{i+1,j+1}=\mathcal A_{i,j},
 \qquad 0\le i\le m-1,\quad 0\le j\le n-1,
\]
and restrict to $1\le k\le d$.  This gives
Theorem~\ref{thm:soft-filling}.  Its finite-address conclusion is precisely
the conclusion of Lemma~\ref{lem:finite-address-reduction} in the lattice
setting.

\section{The fully dense case: partitions of unity and slots}
\label{sec:fully-dense-case}

Assume throughout this section that
\begin{equation}\label{eq:fully-dense-assumption}
 \overline{T+P}=\R^d,
 \qquad
 \covol(T)<\covol(P).
\end{equation}
The group $T+P$ is countable; ``fully dense'' refers to the closure in
\eqref{eq:fully-dense-assumption}.  In this case there is no discrete
quotient and hence no parameter torus on which a configuration field must be
filled.  The construction reduces to two elementary finite objects in the
connected Euclidean space: small functions forming a periodic partition of
unity and separated open sets into which their supports can be moved.

The geometric idea is the packet--slot matching used by Grepstad,
Kolountzakis, and Spyridakis~\cite{GrepstadKolountzakisSpyridakis2025};
related common-domain constructions appear in
\cite{GrepstadKolountzakis2026}.  At the level of characteristic functions,
one cuts a source fundamental domain into many small pieces and assigns each
piece wholly to a distinct target slot.  This is the connected analogue of a
hard configuration: every source packet chooses one slot, and no slot is used
twice.  We keep the same counting and density mechanism, but soften the source
side by replacing characteristic functions with nonnegative elements of a
partition of unity.  The argument below is self-contained and does not import
an existence theorem from those papers.

We develop the two finite objects in a form that will be reused in
Section~\ref{sec:nondiscrete-construction}.  Let $V$ be an
$\ell$-dimensional Euclidean space and let $T_0,P_0\subset V$ be full-rank
lattices.  A finite family $h_1,\ldots,h_N$ is a
\emph{finite $T_0$-partition of unity} if
\[
 \sum_{i=1}^N\sum_{\tau\in T_0}h_i(v+\tau)=1,
 \qquad v\in V.
\]
The functions $h_i$ are called the partition elements.  A
\emph{target slot family} is a finite family of bounded open sets
$\widetilde S_j$, $j\in\{1,\ldots,M\}$, such that
$\overline{\widetilde S_j}\cap(\overline{\widetilde S_{j'}}+p_0)=\varnothing$
for all $j,j'\in\{1,\ldots,M\}$ and $p_0\in P_0$, unless
$j=j'$ and $p_0=0$.  A partition element fits slot $j$ with margin
$\eta>0$ if a translate of $\supp h_i+B_\eta^V$ is contained in
$\widetilde S_j$, where
\[
 B_\eta^V=\{v\in V:\norm{v}_2<\eta\}.
\]

The construction will be carried out in either of two concrete vector
spaces.  Since a nonzero globally affine function cannot have compact
support, the ``affine'' alternative means continuous compactly supported
piecewise-affine functions.  We write $C_{c,\mathrm{PA}}(V)$ for the
functions that are affine on the cells of a finite polyhedral subdivision;
the second choice will be $C_c^\infty(V)$.

\subsection{Piecewise-affine and smooth partitions of unity}

We first prove the required fine partition-of-unity statement separately in
the two regularity classes.

\begin{lemma}[Piecewise-affine lattice partition of unity]
\label{lem:pa-source-packets}
For every full-rank lattice $T_0\subset V$, every $\lambda_1>1$, and all
sufficiently large integers $s$, there are $N=s^\ell$ nonzero functions
\[
 h_1,\ldots,h_N\in C_{c,\mathrm{PA}}(V),
 \qquad h_i\ge0,
\]
whose periodizations define a $T_0$-periodic continuous piecewise-affine map
\[
 H=(H_1,\ldots,H_N):V\longrightarrow\Delta^{N-1},
 \qquad
 H_i(v)=\sum_{\tau\in T_0}h_i(v+\tau),
\]
and satisfy
\begin{equation}\label{eq:pa-packet-partition}
 \sum_{i=1}^NH_i(v)=1.
\end{equation}
After a fixed linear normalization carrying $T_0$ to $\Z^\ell$, every
$\supp h_i$ is contained in a cube of side length less than
$\lambda_1/s$.
\end{lemma}

\begin{proof}
Apply a linear map carrying $T_0$ to $\Z^\ell$.  Partition the half-open
unit cube into
\[
 C_a=\frac1s\bigl(a+[0,1)^\ell\bigr),
 \qquad a\in\{0,\ldots,s-1\}^\ell.
\]
Choose
\[
 0<r<\frac{\lambda_1-1}{2s}.
\]
On the torus $\R^\ell/\Z^\ell$, choose one point in the interior of every
$C_a$ and triangulate so that all these points are vertices and every closed
vertex star has diameter less than $r$.  This is obtained by starting from a
triangulation containing the finitely many prescribed points and taking
sufficiently many subdivisions.  Lift it to a $\Z^\ell$-periodic
triangulation of $\R^\ell$.

Let $\vartheta_v$ be the nodal hat function at a vertex $v$ of the lifted
triangulation.  Choose one representative of every vertex orbit modulo
$\Z^\ell$.  First assign the prescribed interior vertex of each $C_a$ to
that cube; assign every remaining representative to the unique half-open
cube containing it.  Define
\[
 h_a=\sum_{v\text{ assigned to }C_a}\vartheta_v,
 \qquad
 H_a(x)=\sum_{q\in\Z^\ell}h_a(x+q).
\]
Every $h_a$ is nonzero, nonnegative, compactly supported, and piecewise
affine.  The periodic triangulation makes each $H_a$ periodic and affine on
every simplex.  The nodal hat functions over all vertices of the periodic
triangulation sum to one.  Since every vertex orbit was assigned exactly
once,
\[
 \sum_aH_a(x)=\sum_a\sum_{q\in\Z^\ell}h_a(x+q)=1.
\]
Thus $H=(H_a)_a$ maps into the probability simplex.  Moreover, the star of a
vertex assigned to $C_a$ lies in $C_a+B_r$, so
\[
 \supp h_a\subset C_a+B_r,
\]
which is contained in a cube of side length
$1/s+2r<\lambda_1/s$.  Undo the initial linear normalization.
\end{proof}

\begin{lemma}[Smooth lattice partition of unity]
\label{lem:smooth-source-partition}
For every full-rank lattice $T_0\subset V$, every $\lambda_2>1$, and all
sufficiently large integers $s$, there are $N=s^\ell$ nonzero functions
\[
 h_1^\infty,\ldots,h_N^\infty\in C_c^\infty(V),
 \qquad h_i^\infty\ge0,
\]
whose periodizations form a smooth map
$H^\infty:V\to\Delta^{N-1}$ and satisfy
\begin{equation}\label{eq:smooth-source-partition}
 \sum_{i=1}^N\sum_{\tau\in T_0}h_i^\infty(v+\tau)=1,
 \qquad v\in V.
\end{equation}
After the same normalization carrying $T_0$ to $\Z^\ell$, every support is
contained in a cube of side length less than $\lambda_2/s$.
\end{lemma}

\begin{proof}
Choose $1<\lambda_1<\lambda_2$ and apply
Lemma~\ref{lem:pa-source-packets}.  Choose a nonnegative mollifier
$\rho_s\in C_c^\infty(V)$ of integral one with
\[
 \supp\rho_s\subset B_{\varepsilon_s}^V,
 \qquad
 0<\varepsilon_s<\frac{\lambda_2-\lambda_1}{2s}.
\]
Set $h_i^\infty=h_i*\rho_s$.  The functions are nonzero, nonnegative,
smooth, and compactly supported, and convolution enlarges the side length of
every support cube by less than $(\lambda_2-\lambda_1)/s$.  Periodization
commutes with convolution and the sums are locally finite, so
\begin{align*}
 \sum_{i=1}^N\sum_{\tau\in T_0}h_i^\infty(v+\tau)
 &=\int_V\rho_s(y)
   \sum_{i=1}^N\sum_{\tau\in T_0}h_i(v-y+\tau)\,dy\\
 &=\int_V\rho_s(y)\,dy=1.
\end{align*}
The coordinate periodizations therefore define the asserted smooth
simplex-valued map.
\end{proof}

For the remainder of Sections~\ref{sec:fully-dense-case} and
\ref{sec:nondiscrete-construction}, fix one of the two choices
\begin{equation}\label{eq:function-space-X}
 \mathcal X(E)=C_{c,\mathrm{PA}}(E)
 \qquad\text{or}\qquad
 \mathcal X(E)=C_c^\infty(E)
\end{equation}
for every Euclidean space $E$, and use the same choice throughout the
argument.  Both choices are real vector spaces, stable under translations
and finite sums.  Those are the only vector-space properties used after the
partition elements have been constructed; no norm, topology, or
multiplicative structure is needed.  The preceding two lemmas supply the
additional existence statement that is essential here: arbitrarily fine
nonnegative lattice partitions of unity with compact support.

\subsection{Separated target slots}

We next count small cubes that sit strictly inside one $P_0$-fundamental
parallelepiped; after shrinking, their periodic lifts form a target slot
family with a uniform margin.

\begin{lemma}\label{lem:elementary-slot-count}
Let $P_0\subset\R^\ell$ be a full-rank lattice and let $F_P$ be a closed
fundamental parallelepiped.  Fix $\lambda>0$ and put $\rho_s=\lambda/s$.
For all sufficiently large $s$, the interior of $F_P$ contains
$M_s$ closed cubes of side length $\rho_s$, with pairwise disjoint
interiors, where
\begin{equation}\label{eq:elementary-slot-count}
 M_s=\frac{\covol(P_0)}{\lambda^\ell}s^\ell+O(s^{\ell-1}).
\end{equation}
\end{lemma}

\begin{proof}
Use the grid of half-open cubes
\[
 \rho_s\bigl(a+[0,1)^\ell\bigr),
 \qquad a\in\Z^\ell,
\]
and retain those whose closures lie in the interior of $F_P$.  Their
interiors are pairwise disjoint.  Let $U_s$ be their union.  A grid cube has
diameter $\sqrt\ell\,\rho_s$, so every point of $F_P$ at distance more than
$\sqrt\ell\,\rho_s$ from $\partial F_P$ belongs to one of the retained
cubes.

Because $F_P$ is a parallelepiped, the volume of its boundary strip of
width $r$ is at most $C r$ for $0<r<1$, with a constant depending only on
$F_P$.  Indeed, after applying the inverse of the linear map defining
$F_P$, this strip is contained in the union of $2\ell$ slabs of thickness
comparable to $r$ along the faces of the unit cube.  Consequently,
\[
 \covol(P_0)-C\sqrt\ell\,\rho_s
 \le |U_s|=M_s\rho_s^\ell
 \le \covol(P_0).
\]
Divide by $\rho_s^\ell=(\lambda/s)^\ell$ to obtain
\eqref{eq:elementary-slot-count}.
\end{proof}

The next lemma combines the partition-of-unity and target-slot
constructions.  The extra
integers $m,n,k$ are included now because Section~\ref{sec:nondiscrete-construction} will use the same
statement after the discrete rows and columns have been introduced.  The
fully dense case is the specialization $m=n=1$ and $k=0$.

\begin{lemma}[Partition-of-unity and slot supply]\label{lem:packet-slot-supply}
Fix one of the choices $\mathcal X(V)$ in~\eqref{eq:function-space-X}.
Let $T_0,P_0\subset V$ be full-rank lattices, let $m,n\ge1$, and let
$k\ge0$.  Put
\[
 a_T=\covol_V(T_0),
 \qquad
 a_P=\covol_V(P_0).
\]
If
\begin{equation}\label{eq:packet-slot-volume-gap}
 m a_T<n a_P,
\end{equation}
then there exist positive integers $N,M$, nonzero nonnegative functions
$h_1,\ldots,h_N\in\mathcal X(V)$, bounded open slots
$\widetilde S_1,\ldots,\widetilde S_M\subset V$, and $\eta>0$ such that
\begin{enumerate}[label=(\roman*)]
\item the functions generate a $T_0$-partition of unity:
\begin{equation}\label{eq:packet-slot-pou}
 \sum_{i=1}^N\sum_{\tau\in T_0}h_i(v+\tau)=1,
 \qquad v\in V;
\end{equation}
\item for every $i\in\{1,\ldots,N\}$ and
$j\in\{1,\ldots,M\}$ there is $y_{i,j}\in V$ with
\begin{equation}\label{eq:packet-fits-slot}
 (\supp h_i+B_\eta^V)+y_{i,j}\subset\widetilde S_j;
\end{equation}
\item $\overline{\widetilde S_j}\cap
(\overline{\widetilde S_{j'}}+p_0)=\varnothing$ for all
$j,j'\in\{1,\ldots,M\}$ and $p_0\in P_0$, unless
$j=j'$ and $p_0=0$;
\item
\begin{equation}\label{eq:packet-slot-surplus}
 nM\ge mN+k.
\end{equation}
\end{enumerate}
\end{lemma}

\begin{proof}
Fix a linear isomorphism $A:V\to\R^\ell$ with
$A(T_0)=\Z^\ell$.  In these coordinates $\alpha:=\covol(A(P_0))=a_P/a_T$, so the hypothesis is $m<n\alpha$.  Choose $1<\lambda_1<\lambda_2<\lambda$ such that $n\alpha/\lambda^\ell>m$.
For every sufficiently large $s$, apply
Lemma~\ref{lem:pa-source-packets} if
$\mathcal X(V)=C_{c,\mathrm{PA}}(V)$ and
Lemma~\ref{lem:smooth-source-partition} if
$\mathcal X(V)=C_c^\infty(V)$.  In either case we obtain
$N=s^\ell$ nonnegative partition elements in $\mathcal X(V)$ whose
$A$-images are supported in cubes of side less than $\lambda_1/s$.

Apply Lemma~\ref{lem:elementary-slot-count} to the lattice $A(P_0)$ and a
closed fundamental parallelepiped for it.  We obtain $M=M_s$ outer cubes of
side $\lambda/s$.  Inside each outer cube take the concentric open cube of
side $\lambda_2/s$, and let $\widetilde S_j$ be its inverse image under
$A$.  The outer cubes lie strictly inside one fundamental parallelepiped and
have disjoint interiors.  Their shrinking therefore gives, for all
$j,j'\in\{1,\ldots,M\}$ and $p_0\in P_0$ with
$(j,p_0)\ne(j',0)$,
\[
 \overline{\widetilde S_j}\cap
 (\overline{\widetilde S_{j'}}+p_0)=\varnothing.
\]
This proves~(iii).

For every $i$, the set $A(\supp h_i)$ lies in a cube strictly smaller than the inner
slot cube.  Aligning the two centers gives a vector $y_{i,j}\in V$, and the
strict difference of side lengths leaves a positive margin.  Since only
finitely many pairs $(i,j)$ occur, one common $\eta>0$ works in~(ii).
Item~(i) is the partition identity supplied by the corresponding
partition-of-unity lemma.

Finally, Lemma~\ref{lem:elementary-slot-count} gives
\[
 nM-mN
 =\left(\frac{n\alpha}{\lambda^\ell}-m\right)s^\ell
   +O(s^{\ell-1}).
\]
The leading coefficient is positive, so the right-hand side is at least
$k$ for all sufficiently large $s$.  This proves~(iv).
\end{proof}

\subsection{Direct construction when $\overline{T+P}=\R^d$}

We now combine the fixed function space $\mathcal X(\R^d)$ with the density of $T$ modulo
$P$.  No configuration field is needed in the fully dense case.

\begin{proposition}[Fully dense construction in $\mathcal X$]
\label{prop:fully-dense-pa}
Fix one of the choices $\mathcal X(\R^d)$ in~\eqref{eq:function-space-X}.
Under
\eqref{eq:fully-dense-assumption}, there are a nonnegative function
$\varphi_{\mathcal X}\in\mathcal X(\R^d)$ and a number
$\eta_{\mathcal X}>0$ such that $\varphi_{\mathcal X}$ is a
$T$-partition of unity and
\begin{equation}\label{eq:dense-pa-buffer}
 (\supp\varphi_{\mathcal X}+B_{\eta_{\mathcal X}})+P
 \quad\text{is a packing.}
\end{equation}
In particular, one may choose either
\[
 \varphi_{\rm pa}\in C_{c,\mathrm{PA}}(\R^d)
 \qquad\text{or}\qquad
 \varphi_\infty\in C_c^\infty(\R^d).
\]
For the piecewise-affine choice, the fundamental domain in
Theorem~\ref{thm:main-geometric} may be chosen as a finite union of bounded
half-open polyhedra.
\end{proposition}

\begin{proof}
Apply Lemma~\ref{lem:packet-slot-supply} with
\[
 V=\R^d,
 \qquad T_0=T,
 \qquad P_0=P,
 \qquad m=n=1,
 \qquad k=0,
\]
and with the chosen function space $\mathcal X(\R^d)$.  The strict
covolume inequality gives the hypothesis, and
\eqref{eq:packet-slot-surplus} gives $M\ge N$.  Choose distinct slot labels
$j(1),\ldots,j(N)$.

Let $\rho_P:\R^d\to\R^d/P$ be the quotient map.  The density assumption
implies
\[
 \overline{\rho_P(T)}
 =\rho_P(\overline{T+P})
 =\R^d/P.
\]
For each $i$, Lemma~\ref{lem:packet-slot-supply}(ii) gives a geometric
placement vector $y_{i,j(i)}$.  The inclusion in
\eqref{eq:packet-fits-slot} is strict and therefore remains valid after a
sufficiently small perturbation.  Density of $\rho_P(T)$ gives
$t_i\in T$ and $p_i\in P$ such that
\[
 (\supp h_i+B_\eta)+t_i
 \subset \widetilde S_{j(i)}+p_i.
\]
Define
\[
 \varphi_{\mathcal X}(x)=\sum_{i=1}^N h_i(x-t_i).
\]
Translation invariance and the vector-space property of
$\mathcal X(\R^d)$ imply that
$\varphi_{\mathcal X}\in\mathcal X(\R^d)$.  Since $t_i\in T$ and the
$h_i$ generate a $T$-partition of unity,
\[
 \sum_{t\in T}\varphi_{\mathcal X}(x+t)
 =\sum_{i=1}^N\sum_{t\in T}h_i(x+t)=1.
\]
The slot labels $j(i)$ are distinct, and the periodic slot closures satisfy
the disjointness condition in Lemma~\ref{lem:packet-slot-supply}(iii).
Hence \eqref{eq:dense-pa-buffer} holds with
$\eta_{\mathcal X}=\eta$.

Taking $\mathcal X(\R^d)=C_{c,\mathrm{PA}}(\R^d)$ and then
$\mathcal X(\R^d)=C_c^\infty(\R^d)$ gives the two asserted regularities.
For the polyhedral conclusion, use the piecewise-affine choice.  Let $D$ be
a uniform bound for the number of nonzero members of its $T$-partition of
unity and choose $0<\tau<1/D$.  Then
$K=\{\varphi_{\rm pa}\ge\tau\}$ is a compact finite union of polyhedra and
$K+T=\R^d$.  The first-hit construction of
Subsection~\ref{sec:first-hit-selector} selects one representative from each
$T$-coset inside $K$.  Its pieces are half-open polyhedra, and
\eqref{eq:dense-pa-buffer} supplies the positive $P$-packing buffer.
\end{proof}

\section{The general partially dense case}
\label{sec:nondiscrete-construction}

We now assume that
\[
 C=\overline{T+P}
\]
is neither discrete nor all of $\R^d$.  Thus the closed group $C$ has both a
nonzero connected factor and a nonzero discrete quotient.  We give the full
argument in five steps.  First, a linear normal form separates those two
parts.  Second, the partition-of-unity and slot construction of Section~\ref{sec:fully-dense-case} is
applied in the connected factor.  Third, the discrete quotient supplies rows
and columns, and the partition-element and slot labels refine them into a
larger finite table.  Fourth, the extension theorem of Section~\ref{sec:abstract-row-column} fills an
equivariant field of soft configurations over the quotient torus.  Finally,
the field is unrolled to a scalar partition of unity and then replaced by
piecewise-affine and smooth partitions of unity.

The connected-factor step starts from the same geometric packet--slot idea
used by Grepstad, Kolountzakis, and Spyridakis
\cite{GrepstadKolountzakisSpyridakis2025}.  In the hard version, each small
source packet is assigned wholly to one separated target slot, exactly as a
hard configuration assigns one address to one column.  In the present
partially dense setting these assignments must vary over the quotient torus,
so a hard choice is too rigid.  We soften it twice: the source packets are
replaced by functions forming a partition of unity, and the varying slot
assignments are encoded by soft configurations.  The support margin inside
the target slots retains the geometric separation of the hard construction.

The reason for formulating Lemma~\ref{lem:packet-slot-supply} with the
auxiliary integers $m,n,k$ is precisely this application.  Here $m$ and $n$
will be the numbers of discrete rows and columns, while $k$ is the dimension
of the quotient torus.  If the connected factor uses $N$ compactly supported
partition elements and $M$ target slots, then the refined table has
\[
 \widetilde m=mN,
 \qquad
 \widetilde n=nM,
\]
and the inequality $nM\ge mN+k$ supplies the connectivity required to fill
that $k$-dimensional torus.

Fix henceforth one of the function spaces $\mathcal X(V)$ in
\eqref{eq:function-space-X}.  All partition elements in the connected
factor will belong to this space.  The final replacement argument is made
with the same notation and may then be run once for each of the two choices.

\subsection{Closed-subgroup normal form}

Let $V=C^\circ$ be the connected component of $0$ in $C$.  A connected
closed subgroup of a Euclidean space is a vector subspace, so $V$ is a
linear subspace.  Put
\[
 \ell=\dim V,
 \qquad
 k=d-\ell.
\]
Let
\[
 q:\R^d\longrightarrow \R^d/V
\]
be the quotient map by the subspace $V$.  The group $q(C)$ is a closed
subgroup of the Euclidean space $\R^d/V$ and has trivial identity component;
hence it is discrete.  Since $C$ spans $\R^d$, the group $q(C)$ has rank
$k$.

Choose a linear complement $W$ of $V$.  The restriction
$q|_W:W\to\R^d/V$ is a linear isomorphism.  Define
\[
 L=(q|_W)^{-1}(q(C)).
\]
Then $L$ is a full-rank lattice in $W$, every $c\in C$ has a unique
decomposition $c=\ell_0+v$ with $\ell_0\in L$ and $v\in V$, and
\begin{equation}\label{eq:closed-subgroup-sum}
 \R^d=W\oplus V,
 \qquad
 C=L\oplus V.
\end{equation}
After one common invertible linear change of variables, we may and shall
assume
\[
 W=\R^k,
 \qquad
 L=\Z^k,
 \qquad
 V=\R^\ell.
\]

We now define the projection used below without suppressing the quotient
space.  Let
\[
 q_C:C\longrightarrow C/V
\]
be the quotient map of the closed group $C$ by its connected subgroup $V$.
The decomposition~\eqref{eq:closed-subgroup-sum} identifies $C/V$ with the
lattice $L$ by
\[
 (\ell_0+v)+V\longmapsto\ell_0.
\]
The map
\begin{equation}\label{eq:pi-explicit}
 \pi:C\longrightarrow L,
 \qquad
 \pi(\ell_0+v)=\ell_0,
\end{equation}
is the composition of $q_C$ with this identification.  Set
\[
 T'=\pi(T),
 \qquad
 P'=\pi(P),
 \qquad
 T_0=T\cap V,
 \qquad
 P_0=P\cap V.
\]

For completeness, we verify the finiteness statements.  Because $T$ is a
lattice in $\R^d$ and $C$ is closed with $T\subset C$, the quotient $C/T$ is
a closed subset of the compact torus $\R^d/T$, and is therefore compact.
The connected component of the identity in $C/T$ is the image of $V$; its
kernel is $T_0$.  Thus $V/T_0$ is compact, so $T_0$ is a full-rank lattice
in $V$.  The map
\[
 C/T\longrightarrow L/T',
 \qquad
 c+T\longmapsto\pi(c)+T',
\]
is continuous and surjective.  Hence $L/T'$ is compact.  Since it is also
discrete, it is finite.  Therefore $T'$ has finite index in $L$.  Repeating
the same argument with $P$ shows that $P_0$ is a full-rank lattice in $V$
and that $P'$ has finite index in $L$.  Moreover,
\[
 T'+P'=L,
\]
because $\pi(T+P)$ is dense in the discrete group $L$.

As in Section~\ref{sec:hard-configurations}, enumerate the quotient groups themselves by
\begin{equation}\label{eq:nondiscrete-rows-columns}
 L/T'=\{R_0=T',R_1,\ldots,R_{m-1}\},
 \qquad
 L/P'=\{Q_0=P',Q_1,\ldots,Q_{n-1}\}.
\end{equation}
Thus $R_r$ are the discrete rows and $Q_q$ are the discrete columns.  Put
\[
 m=[L:T'],
 \qquad
 n=[L:P'],
 \qquad
 a_T=\covol_V(T_0),
 \qquad
 a_P=\covol_V(P_0),
\]
where covolumes with subscript $V$ are computed using Lebesgue measure on
$V$.

\begin{lemma}\label{lem:product-graph-normal}
After applying a shear that fixes $V$ pointwise, one may assume
\[
 T=T'\times T_0.
\]
In the same coordinates there is a linear map $B:W\to V$ such that
\[
 P=\{(p',Bp'+p_0):p'\in P',\ p_0\in P_0\}.
\]
If Lebesgue measure on $W$ is normalized by $\covol_W(L)=1$, then
\[
 \covol(T)=m a_T,
 \qquad
 \covol(P)=n a_P.
\]
\end{lemma}

\begin{proof}
Choose a basis $t'_1,\ldots,t'_k$ of the rank-$k$ lattice $T'$.  For each
$i$, choose a lift $t_i\in T$ with $\pi(t_i)=t'_i$, and write
\[
 t_i=(t'_i,v_i),
 \qquad v_i\in V.
\]
Every $t\in T$ has $\pi(t)=\sum_i z_i t'_i$ for uniquely determined
integers $z_i$.  Therefore
\[
 t-\sum_i z_i t_i\in T\cap V=T_0.
\]
Thus the chosen lifts together with $T_0$ generate all of $T$, and their
intersection with $T_0$ is trivial.

Define the linear map $A:W\to V$ by $A(t'_i)=v_i$ and extend linearly from
the basis $t'_1,\ldots,t'_k$ of $W$.  The shear
\[
 S_A:W\oplus V\longrightarrow W\oplus V,
 \qquad
 S_A(w,v)=(w,v-Aw),
\]
fixes every vector of $V$ and sends each lift $t_i$ to $(t'_i,0)$.  Hence,
after applying this common change of coordinates,
\[
 T=T'\times T_0.
\]

Now choose a basis $p'_1,\ldots,p'_k$ of $P'$ and lifts
$p_i=(p'_i,w_i)\in P$.  As above, every element of $P$ differs from an
integer combination of the $p_i$ by an element of $P_0$.  Define the unique
linear map $B:W\to V$ by $B(p'_i)=w_i$.  It follows directly that every
$p\in P$ has a unique representation
\[
 p=(p',Bp'+p_0),
 \qquad p'\in P',\quad p_0\in P_0,
\]
which is the asserted graph form.

Finally choose lattice bases of $T'$, $T_0$, $P'$, and $P_0$.  Relative to
the decomposition $W\oplus V$, the corresponding basis matrices for $T$
and $P$ are block triangular.  Their diagonal blocks give
\[
 \covol(T)=\covol_W(T')\covol_V(T_0),
 \qquad
 \covol(P)=\covol_W(P')\covol_V(P_0).
\]
Because $\covol_W(L)=1$, the finite-index formulas give
$\covol_W(T')=[L:T']=m$ and $\covol_W(P')=[L:P']=n$, proving the result.
\end{proof}

The strict covolume hypothesis is now
\begin{equation}\label{eq:nondiscrete-volume}
 m a_T<n a_P.
\end{equation}

\subsection{Partition of unity in the connected space}

We now apply the abstract packet--slot result in the connected space $V$.
The only regularity used at this stage is that the partition elements belong
to the fixed function space $\mathcal X(V)\subset C_c(V)$.

\begin{lemma}\label{lem:packet-slot-refinement}
Assume $\ell\ge1$ and $m a_T<n a_P$.  There exist positive integers $N,M$,
nonzero nonnegative functions $h_1,\ldots,h_N\in\mathcal X(V)$, bounded
open slot lifts $\widetilde S_1,\ldots,\widetilde S_M\subset V$, and
$\eta>0$ such that
\begin{enumerate}[label=(\roman*)]
\item the functions generate a $T_0$-partition of unity:
\[
 \sum_{i=1}^N\sum_{\tau\in T_0}h_i(v+\tau)=1,
 \qquad v\in V;
\]
\item for every $i\in\{1,\ldots,N\}$ and
$j\in\{1,\ldots,M\}$ there is $y_{i,j}\in V$ such that
\[
 (\supp h_i+B_\eta^V)+y_{i,j}\subset\widetilde S_j;
\]
\item $\overline{\widetilde S_j}\cap
(\overline{\widetilde S_{j'}}+p_0)=\varnothing$ for all
$j,j'\in\{1,\ldots,M\}$ and $p_0\in P_0$, unless
$j=j'$ and $p_0=0$;
\item
\[
 nM\ge mN+k.
\]
\end{enumerate}
\end{lemma}

\begin{proof}
Apply Lemma~\ref{lem:packet-slot-supply} in the Euclidean space $V$ to the
fixed function space $\mathcal X(V)$, the lattices $T_0,P_0$, and the integers
$m,n,k$ defined above.  Its hypothesis is exactly
$m a_T<n a_P$, and its four conclusions give the assertions here.
\end{proof}

\subsection{Rows, columns, and admissible addresses}

We now keep the row--column notation of Sections~\ref{sec:hard-configurations} and~\ref{sec:soft-configurations}.  The discrete rows
and columns are the cosets $R_r\in L/T'$ and $Q_q\in L/P'$ from
\eqref{eq:nondiscrete-rows-columns}.  Refinement adds a partition-element label to a row
and a slot label to a column:
\[
 \widetilde R_{r,i}=(R_r,i),
 \qquad
 0\le r\le m-1,
 \quad 1\le i\le N,
\]
and
\[
 \widetilde Q_{q,j}=(Q_q,j),
 \qquad
 0\le q\le n-1,
 \quad 1\le j\le M.
\]
Thus the refined table has
\[
 \widetilde m=mN
 \qquad\text{rows and}\qquad
 \widetilde n=nM
\]
columns, and Lemma~\ref{lem:packet-slot-refinement} gives
\begin{equation}\label{eq:refined-gap}
 \widetilde n\ge\widetilde m+k.
\end{equation}

It is useful to spell out the moving address sets explicitly.  The graph
form of $P$ gives a canonical homeomorphism
\begin{equation}\label{eq:graph-quotient-trivialization}
 \Theta:(W\oplus V)/P\longrightarrow (W/P')\times(V/P_0),
 \qquad
 \Theta\bigl((w,v)+P\bigr)=\bigl(w+P',v-Bw+P_0\bigr).
\end{equation}
Indeed, replacing $(w,v)$ by $(w+p',v+Bp'+p_0)$ leaves the right-hand
side unchanged.

For a base point $u\in W$, a refined row $(R_r,i)$, and a refined column
$(Q_q,j)$, an address is a labeled quadruple
\[
 a=(\ell,\tau;i,j),
 \qquad
 \ell\in R_r\cap Q_q,
 \quad \tau\in T_0.
\]
The sign convention is chosen to match the unrolling formula below: the
address $\ell$ at the base point $u$ contributes on the affine fiber with
$W$-coordinate $u-\ell$.  Define
\begin{equation}\label{eq:moving-address-definition}
\begin{split}
 \mathcal A_{(r,i),(q,j)}(u)
 =\bigl\{(\ell,\tau;i,j):{}&\ \ell\in R_r\cap Q_q,
 \ \tau\in T_0,\ \text{and for some }p_0\in P_0,\\[-1mm]
 & (\supp h_i+\overline{B_\eta^V})+\tau-B(u-\ell)
    \subset \widetilde S_j+p_0\bigr\}.
\end{split}
\end{equation}
After decreasing $\eta$ once, if necessary, the closed-ball version of the
packet inclusion follows from the strict containments in
Lemma~\ref{lem:packet-slot-refinement}(ii).  Lemma~\ref{lem:packet-slot-refinement}(iii)
shows that the element $p_0$ in~\eqref{eq:moving-address-definition} is
unique; denote it by $\kappa(a,u)$.  Formula~\eqref{eq:graph-quotient-trivialization}
explains the expression in~\eqref{eq:moving-address-definition}: for a point
$(u-\ell,v)$ in the translated packet, its connected coordinate modulo $P$
is $v-B(u-\ell)+P_0$.

The definition now makes all dependence on $u$ explicit.  Since the compact
set in~\eqref{eq:moving-address-definition} lies strictly inside an open
slot, admissibility is open in $u$.  If $\gamma\in L$, put
\[
 a+\gamma=(\ell+\gamma,\tau;i,j).
\]
Then $u+\gamma-(\ell+\gamma)=u-\ell$, and hence
\begin{equation}\label{eq:moving-address-equivariance}
 \mathcal A_{(R_r+\gamma,i),(Q_q+\gamma,j)}(u+\gamma)
 =\mathcal A_{(R_r,i),(Q_q,j)}(u)+\gamma.
\end{equation}
Thus the moving address table is $L$-equivariant.

\begin{lemma}\label{lem:dense-row-orbits}
For every $u\in W$, every refined row $\widetilde R_{r,i}$, and every
refined column $\widetilde Q_{q,j}$, the address set
$\mathcal A_{(r,i),(q,j)}(u)$ is nonempty.  More precisely, for any
$\ell_0\in R_r\cap Q_q$, the set
\begin{equation}\label{eq:dense-connected-coordinates}
 \bigl\{\tau-B(u-\ell)+P_0:
   \ell\in R_r\cap Q_q,\ \tau\in T_0\bigr\}
\end{equation}
is dense in $V/P_0$.
\end{lemma}

\begin{proof}
Put $H=T'\cap P'$.  Since $R_r\cap Q_q=\ell_0+H$, the set in
\eqref{eq:dense-connected-coordinates} is a translate of the image of
$T_0+BH$ in $V/P_0$.  We claim that this image is dense.  The elements of
$T+P$ with zero discrete component are precisely
\[
 (0,\tau+p_0+Bh),
 \qquad \tau\in T_0,\quad p_0\in P_0,\quad h\in H.
\]
Because $\overline{T+P}=L\oplus V$, their connected coordinates are dense
in $V$.  Equivalently, $T_0+BH$ is dense modulo $P_0$, proving the claim and
hence~\eqref{eq:dense-connected-coordinates}.

For fixed $i,j$, Lemma~\ref{lem:packet-slot-refinement}(ii) gives a
nonempty open set of translation classes $y+P_0\in V/P_0$ for which
$(\supp h_i+\overline{B_\eta^V})+y$ lies in a translate of
$\widetilde S_j$.  The dense set~\eqref{eq:dense-connected-coordinates}
meets this open set.  The corresponding $\ell$ and $\tau$ satisfy
\eqref{eq:moving-address-definition}, and therefore give an address in
$\mathcal A_{(r,i),(q,j)}(u)$.
\end{proof}

\subsection{Triangulating the moving address data}

For $u\in W$, let
\[
 \mathcal A(u)
 =\bigl(\mathcal A_{(r,i),(q,j)}(u)\bigr)_{
   \substack{0\le r<m,\ 1\le i\le N\\
             0\le q<n,\ 1\le j\le M}}
\]
be the refined address table.  Lemma~\ref{lem:dense-row-orbits} says that
every entry of this table is nonempty.  An individual address $a$ has an open
admissibility set
\[
 U_a=\{u\in W:a\text{ is valid at }u\}.
\]
The $L$-action relabels addresses, rows, and columns.  We use the same sign
convention as in Section~\ref{sec:soft-configurations}:
\[
 (\rho_\gamma\lambda)(a)=\lambda(a-\gamma),
 \qquad \gamma\in L.
\]
Then $U_{a+\gamma}=U_a+\gamma$.

The following lemma is the moving-table counterpart of the cube construction
in Section~\ref{sec:soft-configurations}.  Its proof is included to make clear where finiteness and
continuity enter.

\begin{lemma}\label{lem:moving-address-field}
Assume that the refined table $\mathcal A(u)$ has
$\widetilde m$ rows and $\widetilde n$ columns, that every entry is nonempty,
and that address admissibility is open and $L$-equivariant as above.  If
\[
 \widetilde n\ge\widetilde m+k,
\]
then there is an $L$-equivariant continuous simplexwise affine field
$\Lambda$ on $W$ such that, for every $u\in W$, the value $\Lambda(u)$ is a
soft configuration with $\widetilde m$ rows and $\widetilde n$ columns and
uses only addresses belonging to $\mathcal A(u)$.  The field uses only finitely many $L$-orbits of actual addresses.
\end{lemma}

\begin{proof}
The quotient $W/L$ is a compact $k$-torus.  At a point $\bar u\in W/L$,
choose one address from every entry of the finite
$\widetilde m\times\widetilde n$ table.  This gives a finite complete table
$A_{\bar u}$.  Since all chosen admissibility sets are open, there is an open
neighborhood $O_{\bar u}$ on which every address of $A_{\bar u}$ remains
valid.  Compactness provides finitely many such neighborhoods.

Choose a finite triangulation $K$ of $W/L$ sufficiently fine that the closed
star of every vertex lies in one of the neighborhoods $O_{\bar u}$.  One way
to obtain this is to start from any triangulation, subdivide until its mesh is
smaller than a Lebesgue number of the finite cover, and then take one
barycentric subdivision.  Lift $K$ to an $L$-periodic triangulation
$\widetilde K$ of $W$.

For one representative of every vertex orbit, choose a complete finite table
that is valid on the whole closed star of that vertex, and translate these
tables by $L$.  Denote the table at a lifted vertex $v$ by $A_v$.  If
$\sigma$ is a simplex, put
\[
 A_\sigma=\bigcup_{v\in\sigma^{(0)}}A_v.
\]
Every address in $A_\sigma$ is valid at every point of $\sigma$.  Moreover,
if $\tau$ is a face of $\sigma$, then
\[
 \mathcal C_{\widetilde m,\widetilde n}(A_\tau)
 \subset
 \mathcal C_{\widetilde m,\widetilde n}(A_\sigma).
\]

We now fill the lifted triangulation skeleton by skeleton.  At each vertex
choose a hard configuration from its complete table.  Make these choices on
one representative of every vertex orbit and extend them $L$-equivariantly.
Assume inductively that the field has been defined simplexwise affinely on
the $(s-1)$-skeleton.  For an $s$-simplex $\sigma$, the already defined
boundary map takes values in the finite complex
\[
 \mathcal C_{\widetilde m,\widetilde n}(A_\sigma).
\]
Since
\[
 s\le k\le\widetilde n-\widetilde m,
\]
Theorem~\ref{thm:soft-filling-strong} extends this boundary map over
$\sigma$, after subdividing only its interior and without changing its
boundary triangulation or boundary vertex values.  Fill one representative
of every $L$-orbit of $s$-simplices and translate the filling by the action
$\rho_\gamma$.

The relative boundary condition ensures that fillings of adjacent simplices
agree on their common face.  Each point of $W$ has a neighborhood meeting
only finitely many simplices of the periodic triangulation, so the finite
pasting lemma gives continuity.  The quotient triangulation has finitely
many simplices and each filling uses finitely many addresses.  Therefore
only finitely many address orbits occur.
\end{proof}

\subsection{Constructing the function $\varphi$}

We now convert the soft-configuration field into a scalar partition of unity, in exact
analogy with Subsection~\ref{sec:commensurable-construction}.  Write a point
of $\R^d=W\oplus V$ uniquely as
\[
 x=(u-\ell,v),
 \qquad
 u\in[0,1)^k,
 \quad \ell\in L,
 \quad v\in V.
\]
An address in refined row $(R_r,i)$ and refined column $(Q_q,j)$ is a
labeled quadruple $a=(\ell,\tau;i,j)$ belonging to
$\mathcal A_{(r,i),(q,j)}(u)$.  We write
$\Lambda(u)_{r,i}(\ell,\tau,j)$ for its probability weight and let
$r(\ell)$ be determined by $\ell+T'=R_{r(\ell)}$.  Define
\begin{equation}\label{eq:nondiscrete-scalar}
 \varphi(u-\ell,v)
 =\sum_{i=1}^N\sum_{\tau\in T_0}\sum_{j=1}^M
 \Lambda(u)_{r(\ell),i}(\ell,\tau,j)h_i(v-\tau).
\end{equation}
The sum is locally finite.  Indeed, on one fundamental cube in $W$ the
field uses only finitely many $L$-orbits of addresses, and for $v$ in a
bounded set only finitely many $T_0$-translates of the compact supports of
the $h_i$ can occur.

The equivariance relation
\[
 \Lambda(u+\gamma)=\rho_\gamma\Lambda(u),
 \qquad \gamma\in L,
\]
shows that~\eqref{eq:nondiscrete-scalar} is independent of the half-open
representative.  Indeed, the same point can be written as
$(u-\ell,v)=((u+\gamma)-(\ell+\gamma),v)$, and equivariance gives equal
weights at the addresses $\ell$ and $\ell+\gamma$.  The formulas therefore
agree on opposite faces, and $\varphi$ is continuous.  On a fixed
fundamental cube only finitely many discrete address orbits occur; together
with the compact supports of the partition elements, this shows that
$\varphi\in C_c(\R^d)$.

We verify the $T$-partition-of-unity identity.  Let $x=(u-\ell,v)$ and write
an element of $T=T'\times T_0$ as $(t',\sigma)$.  Then
\begin{align*}
 \sum_{(t',\sigma)\in T'\times T_0}
 \varphi(u-\ell+t',v+\sigma)
 &=\sum_{i=1}^N
   \sum_{\ell'\in R_{r(\ell)}}
   \sum_{\tau\in T_0}\sum_{j=1}^M
   \Lambda(u)_{r(\ell),i}(\ell',\tau,j)
   \sum_{\sigma\in T_0}h_i(v+\sigma-\tau)\\
 &=\sum_{i=1}^N\sum_{\sigma\in T_0}h_i(v+\sigma)=1.
\end{align*}
Here we used $\ell'=\ell-t'$ in the first equality, the row-normalization
of $\Lambda(u)$ in the second, and the connected-factor partition of unity
from Lemma~\ref{lem:packet-slot-refinement}(i) in the last.  Thus
\begin{equation}\label{eq:nondiscrete-partition}
 \sum_{t\in T}\varphi(x+t)=1,
 \qquad x\in\R^d.
\end{equation}

It remains to prove the packing identity.  Let
$p=(p',Bp'+p_0)\in P$ and suppose, for contradiction, that both
$\varphi(x)$ and $\varphi(x-p)$ are positive.  Choose positive summands in
\eqref{eq:nondiscrete-scalar}.  They correspond to labeled addresses
\[
 a=(\ell,\tau;i,j),
 \qquad
 a'=(\ell+p',\tau';i',j'),
\]
because
\[
 x-p=(u-(\ell+p'),v-Bp'-p_0).
\]
Set $y=v-B(u-\ell)$.  The definition of admissibility and positivity of the
partition elements give
\[
 y\in\widetilde S_j+\kappa(a,u).
\]
For the second summand,
\[
 (v-Bp'-p_0)-B\bigl(u-(\ell+p')\bigr)=y-p_0,
\]
so
\[
 y-p_0\in\widetilde S_{j'}+\kappa(a',u),
 \qquad\text{or equivalently}\qquad
 y\in\widetilde S_{j'}+\kappa(a',u)+p_0.
\]
The disjointness of the closed periodic slot translates in
Lemma~\ref{lem:packet-slot-refinement}(iii) therefore forces
\begin{equation}\label{eq:slot-label-comparison}
 j=j',
 \qquad
 \kappa(a,u)=\kappa(a',u)+p_0.
\end{equation}
The discrete addresses $\ell$ and $\ell+p'$ lie in the same class modulo
$P'$, and~\eqref{eq:slot-label-comparison} says that their slot labels also
agree.  Hence $a$ and $a'$ lie in the same refined column.  If they are
distinct labeled addresses, this contradicts column exclusion.  If they are
the same labeled address, then $p'=0$ and
$\kappa(a,u)=\kappa(a',u)$;~\eqref{eq:slot-label-comparison} then gives
$p_0=0$.  In either case a nonzero $p$ is impossible.  Therefore
\begin{equation}\label{eq:nondiscrete-separation}
 \varphi(x)\varphi(x-p)=0,
 \qquad p\in P\setminus\{0\}.
\end{equation}

The function in~\eqref{eq:nondiscrete-scalar} is continuous, but it need
not be piecewise affine: even when $\Lambda$ and the partition elements are
piecewise affine separately, their products in the base and fiber variables
are generally piecewise bilinear.  Choosing the partition elements smooth
does not make the result smooth either, because the configuration field is
only simplexwise affine in the base variable.  The following replacement
lemma produces both regularities after the exact partition-of-unity and
separation identities have been established.

\subsection{Piecewise-affine and smooth replacement}
\label{sec:pou-replacement}

The intermediate product construction is only continuous, so we now replace it by
partitions of unity with the two regularities needed in the applications.

\begin{lemma}[Regularity replacement in $\mathcal X$]
\label{lem:pou-replacement}
Fix one of the choices $\mathcal X(\R^d)$ in~\eqref{eq:function-space-X}.
Let $T,P\subset\R^d$ be full-rank lattices.  Suppose that
$\varphi\in C_c(\R^d)$ is a nonnegative $T$-partition of unity and satisfies
\begin{equation}\label{eq:replacement-input-separation}
 \varphi(x)\varphi(x-p)=0,
 \qquad x\in\R^d,\quad p\in P\setminus\{0\}.
\end{equation}
Then there exist a nonnegative function
$\psi_{\mathcal X}\in\mathcal X(\R^d)$ and a number
$\eta_{\mathcal X}>0$ such that $\psi_{\mathcal X}$ is a
$T$-partition of unity and
\begin{equation}\label{eq:replacement-output-buffer}
 (\supp\psi_{\mathcal X}+B_{\eta_{\mathcal X}})+P
 \quad\text{is a packing.}
\end{equation}
\end{lemma}

\begin{proof}
Compactness of $\supp\varphi$ gives a finite number
\[
 D=\sup_{x\in\R^d}\#\{t\in T:\varphi(x+t)>0\}.
\]
Choose $0<\tau<1/D$ and set
\[
 K=\{x:\varphi(x)\ge\tau\}.
\]
Because $\varphi$ is a $T$-partition of unity, at least one member of the
partition is at least $1/D$ at every point.  Hence
\begin{equation}\label{eq:replacement-cover}
 K+T=\R^d.
\end{equation}
By~\eqref{eq:replacement-input-separation}, the compact sets $K$ and $K+p$
are disjoint for every $p\in P\setminus\{0\}$.  Since $P$ is discrete and
$K$ is bounded,
\begin{equation}\label{eq:replacement-delta}
 \delta=\inf_{p\in P\setminus\{0\}}
 \operatorname{dist}(K,K+p)>0.
\end{equation}
Choose $0<r<\delta/3$ and a $T$-periodic locally finite triangulation of
$\R^d$ whose closed vertex stars have diameter less than $r$.  There are
only finitely many vertex orbits modulo $T$; choose representatives
$v_1,\ldots,v_s$.  By~\eqref{eq:replacement-cover}, for every $j$ there is
$t_j\in T$ such that
\[
 v_j+t_j\in K.
\]
Assign the value $1$ to the vertex $v_j+t_j$ and the value $0$ to every
other vertex in the orbit $v_j+T$.  Extend these values affinely over every
simplex.  The resulting function $\psi_{\rm pa}$ is nonnegative and
continuous.  Only the closed stars of the finitely many selected vertices
can meet its support; hence
\[
 \supp\psi_{\rm pa}\subset K+B_r,
\]
so it is compactly supported and piecewise affine.

Let $\sigma=[w_0,\ldots,w_q]$ be a simplex and write
$x=\sum_i\alpha_iw_i$ in barycentric coordinates.  Periodicity of the
triangulation gives the same barycentric coordinates for $x+t$ in
$\sigma+t$.  Therefore
\[
 \sum_{t\in T}\psi_{\rm pa}(x+t)
 =\sum_i\alpha_i\sum_{t\in T}\psi_{\rm pa}(w_i+t)
 =\sum_i\alpha_i=1,
\]
because exactly one vertex in every $T$-orbit was assigned the value $1$.
Thus $\psi_{\rm pa}$ is a $T$-partition of unity.  Moreover,
\[
 \operatorname{dist}(\supp\psi_{\rm pa},\supp\psi_{\rm pa}+p)
 \ge\delta-2r>0,
 \qquad p\in P\setminus\{0\}.
\]
Choose $0<\eta_{\rm pa}<(\delta-2r)/2$.

If $\mathcal X(\R^d)=C_{c,\mathrm{PA}}(\R^d)$, set
$\psi_{\mathcal X}=\psi_{\rm pa}$ and
$\eta_{\mathcal X}=\eta_{\rm pa}$.  If
$\mathcal X(\R^d)=C_c^\infty(\R^d)$, choose a nonnegative mollifier
$\rho\in C_c^\infty(B_\varepsilon)$ of integral one, where
$0<\varepsilon<\eta_{\rm pa}/3$, and set
\[
 \psi_{\mathcal X}=\psi_{\rm pa}*\rho.
\]
Then $\psi_{\mathcal X}$ is nonnegative, smooth, and compactly supported.
Since periodization commutes with convolution,
\begin{align*}
 \sum_{t\in T}\psi_{\mathcal X}(x+t)
 &=\int_{\R^d}\rho(y)
   \sum_{t\in T}\psi_{\rm pa}(x-y+t)\,dy\\
 &=1.
\end{align*}
Moreover,
\[
 \supp\psi_{\mathcal X}\subset\supp\psi_{\rm pa}+B_\varepsilon.
\]
For any $0<\eta_{\mathcal X}<\eta_{\rm pa}-\varepsilon$, we obtain the required
$P$-packing buffer.
\end{proof}

Apply Lemma~\ref{lem:pou-replacement} twice to the continuous function
constructed in~\eqref{eq:nondiscrete-scalar}, using
\eqref{eq:nondiscrete-partition} and
\eqref{eq:nondiscrete-separation}: once with
$\mathcal X(\R^d)=C_{c,\mathrm{PA}}(\R^d)$ and once with
$\mathcal X(\R^d)=C_c^\infty(\R^d)$.  Denote the two outputs by
$\psi_{\rm pa}$ and $\psi_\infty$, and their buffers by
$\eta_{\rm pa}$ and $\eta_\infty$.  The polyhedral selector uses
$\psi_{\rm pa}$.

Let
\[
 D_\psi=\sup_x\#\{t\in T:\psi_{\rm pa}(x+t)>0\}
\]
and choose $0<\tau_\psi<1/D_\psi$.  Then
\[
 K_\psi=\{x:\psi_{\rm pa}(x)\ge\tau_\psi\}
\]
is a compact finite union of polyhedra and $K_\psi+T=\R^d$.  Choose a
half-open $T$-fundamental parallelepiped $F$.  Only finitely many translates
$F+t_1,\ldots,F+t_N$ meet $K_\psi$.  Put
\[
 D_j=\{x\in F:x+t_j\in K_\psi\},
 \qquad
 E_1=D_1,
 \qquad
 E_j=D_j\setminus\bigcup_{i<j}D_i,
\]
and define
\begin{equation}\label{eq:nondiscrete-polyhedral-selector}
 \Omega=\bigcup_{j=1}^N(E_j+t_j).
\end{equation}
The sets $E_j$ form a disjoint partition of $F$, so $\Omega$ contains
exactly one representative of every $T$-coset and $\Omega+T$ is an exact
tiling.  Since $K_\psi$ is polyhedral, the $E_j$ and hence $\Omega$ are
finite unions of bounded half-open polyhedra.  Moreover,
$\overline\Omega\subset K_\psi\subset\supp\psi_{\rm pa}$, so any sufficiently small
$0<\varepsilon<\eta_{\rm pa}$ makes $\Omega+B_\varepsilon$ pack with
$P$ by~\eqref{eq:replacement-output-buffer}, applied with
$\mathcal X(\R^d)=C_{c,\mathrm{PA}}(\R^d)$.

\begin{corollary}\label{cor:global-partitions}
Whenever the condition~\eqref{eq:TeP} holds, there exist both a nonnegative
continuous compactly supported piecewise-affine $T$-partition of unity
$\varphi_{\rm pa}$ and a nonnegative compactly supported smooth
$T$-partition of unity $\varphi_\infty$.  Moreover, there are
$\eta_{\rm pa},\eta_\infty>0$ such that
\[
 (\supp\varphi_\star+B_{\eta_\star})+P
 \quad\text{is a packing},
 \qquad \star\in\{\mathrm{pa},\infty\}.
\]
\end{corollary}

\begin{proof}
If $T+P$ is discrete, Section~\ref{sec:soft-configurations} constructs a
piecewise-affine $T$-partition of unity directly; apply
Lemma~\ref{lem:pou-replacement} once for each choice in
\eqref{eq:function-space-X}.  If $T+P$ is dense in $\R^d$, apply
Proposition~\ref{prop:fully-dense-pa} once for each choice.  In the remaining
non-discrete case, the conclusion is the pair of applications of
Lemma~\ref{lem:pou-replacement} made above.
\end{proof}

This completes the intermediate non-discrete construction.

\section{Necessity in the Tiling--$\varepsilon$--Packing theorem}\label{sec:geometric-necessity}

The necessity argument has two parts: the strict covolume obstruction and,
when $T+P$ is discrete, the finite-index obstruction.

\subsection{The strict covolume obstruction}

If a bounded $T$-fundamental domain $\Omega$ satisfies that
$\Omega+B_\eps$ packs with $P$, then volume comparison gives
\[
 \covol(T)=|\Omega|<|\Omega+B_\eps|\le \covol(P).
\]
This is the strict covolume obstruction stated after
Theorem~\ref{thm:main-geometric}.

\subsection{The discrete index obstruction}

Assume that $H=T+P$ is a lattice.  After a common invertible linear change
of variables, we may suppose $H=\Z^d$.  For this proof only, enumerate the
cosets as
\[
 \Z^d/T=\{R_0=T,R_1,\ldots,R_{m-1}\},
 \qquad
 \Z^d/P=\{Q_0=P,Q_1,\ldots,Q_{n-1}\},
\]
where
\[
 m=[\Z^d:T],
 \qquad
 n=[\Z^d:P].
\]
No configuration-space terminology is needed in the necessity argument.

\begin{theorem}\label{thm:discrete-necessity}
If $n\le m+d-1$, then there is no Borel $T$-fundamental domain
$\Omega$ and no $\eps>0$ such that $\Omega+B_\eps$ packs with $P$.
\end{theorem}

The proof uses the classical covering-dimension theorem of Lebesgue: a finite
closed cover of a $d$-cube by sets of sufficiently small mesh must have
multiplicity at least $d+1$.  We use the following standard form; see
\cite[Chapter~IV]{HurewiczWallman1941}.

\begin{theorem}[Lebesgue~\protect\cite{HurewiczWallman1941}]\label{thm:lebesgue-covering}
Let $Q\subset\R^d$ be a cube.  If a finite closed cover of $Q$ has mesh
strictly smaller than the side length of $Q$, then some point of $Q$ belongs
to at least $d+1$ members of the cover.
\end{theorem}

\begin{proof}[Proof of Theorem~\ref{thm:discrete-necessity}]
Assume, to obtain a contradiction, that such $\Omega$ and $\eps$ exist.
We may assume that $\Omega$ is bounded: otherwise one can choose infinitely
many points of $\Omega$ separated by more than $2\eps$, so
$\Omega+B_\eps$ contains infinitely many disjoint balls of radius
$\eps/2$ and has infinite measure, which is impossible for a $P$-packing.
Choose representatives
\[
 u_i\in R_i,
 \qquad i=0,\ldots,m-1,
 \qquad u_0=0.
\]
Because $\Omega+T$ is an exact tiling of $\R^d$, the closed translates
\[
 \{\cl\Omega+t:t\in T\}
\]
cover $\R^d$.  Their diameters are uniformly bounded by
$\operatorname{diam}(\cl\Omega)$.  Every bounded cube meets only finitely
many of them: if a cube $Q$ meets $\cl\Omega+t$, then
$t\in Q-\cl\Omega$, which is bounded, and the lattice $T$ has only finitely
many points in a bounded set.

Choose a cube $Q$ whose side length is strictly larger than
$\operatorname{diam}(\cl\Omega)$, and cover it by the finitely many closed
sets $Q\cap(\cl\Omega+t)$ that meet it.  The cover has mesh smaller than the
side length of $Q$.  Theorem~\ref{thm:lebesgue-covering} gives distinct
$t_0,\ldots,t_d\in T$ and $x_0\in Q$ such that
\[
 x_0\in\bigcap_{j=0}^d(\cl\Omega+t_j).
\]

For each $i=1,\ldots,m-1$, the family
\[
 \{\cl\Omega+t+u_i:t\in T\}
\]
is a translate of the preceding closed cover.  Choose
$\widetilde t_i\in T$ such that
\[
 x_0\in\cl\Omega+\widetilde t_i+u_i.
\]
The elements
\[
 t_0,\ldots,t_d,
 \widetilde t_1+u_1,\ldots,\widetilde t_{m-1}+u_{m-1}
\]
are $m+d$ distinct elements of $\Z^d$: the first $d+1$ lie in the zero
class modulo $T$ and are distinct, while the remaining elements lie in the
pairwise distinct nonzero cosets $R_i$.

There are only $n\le m+d-1$ classes modulo $P$.  Hence two of the displayed
elements, say $z\ne z'$, belong to the same coset $Q_j$ modulo $P$.  Therefore
\[
 p_0=z-z'\in P\setminus\{0\}.
\]
By construction,
\[
 x_0\in(\cl\Omega+z)\cap(\cl\Omega+z').
\]
Subtracting $z'$ gives a point in
$(\cl\Omega+p_0)\cap\cl\Omega$.  Since
$\cl\Omega\subset\Omega+B_\eps$, this contradicts the assumed
$P$-packing.
\end{proof}

This completes Part~A.

\displaypart{Existence of Schwartz-Class and $C^\infty_c$ Gabor Windows}

\section{Gabor systems, duality, and the rational Zak matrix model}
\label{sec:gabor-background}

This section develops the common Gabor-theoretic and rational Zak framework
used throughout Part~B.  The same finite Zak matrix will later be studied in
two regularity regimes: continuity leads to the arithmetic gap in~\eqref{eq:SG},
whereas Sobolev regularity permits finite-uncertainty windows under the sole
strict-density condition.  The definitions of time shifts, frequency shifts,
Gabor systems, frames, tight frames, and Parseval frames are those introduced
in Subsection~\ref{subsec:intro-gabor}.  Tools used later, such as the
symplectic normal form or metaplectic reduction, are introduced only after the simplest case of
diagonal lattices has been fully resolved.

\subsection{Frames, duality, and density}\label{subsec:frames-duality-density}

The frame operator of a Gabor frame $\G(g,\Lambda)$ is
\[
 S_{g,\Lambda}f
   =\sum_{\lambda\in\Lambda}
      \ip{f}{\pi(\lambda)g}\,\pi(\lambda)g.
\]
The series converges unconditionally in \(L^2(\R^d)\).  The frame is
Parseval precisely when \(S_{g,\Lambda}=I\).

Let \(S=S_{g,\Lambda}\).  Since \(S\) is positive and invertible, its positive square root \(S^{1/2}\) and inverse
square root \(S^{-1/2}\) are well defined.  The canonical tight window
\[
                         \widetilde g=S^{-1/2}g
\]
generates a Parseval Gabor frame on the same lattice.  If
\(g\in\Sclass(\R^d)\), respectively \(g\in\Szero(\R^d)\), then
\(S^{-1/2}g\) belongs to the same space; see
\cite{Luef2009Projections,JakobsenLuef2020} and
\cite[Chapters~12--13]{Grochenig2001}.  Conversely, every Parseval frame is
tight and every tight frame is a frame, while a tight frame with bound
\(A>0\) becomes Parseval after replacing \(g\) by \(A^{-1/2}g\).
Consequently, on a fixed lattice, frame, tight-frame, and Parseval-frame
existence are equivalent for windows in either the Schwartz class or the
Feichtinger algebra.

For a finite family \(\mathbf g=(g_1,\ldots,g_q)\), the multiwindow frame
operator is obtained by summing the corresponding operators over
\(j=1,\ldots,q\).  We refer to
\cite[Chapters~5--7]{Grochenig2001} for the standard frame-operator theory,
canonical dual and tight windows, and unconditional Gabor expansions.

Recall from Subsection~\ref{subsec:intro-gabor} that the adjoint lattice of a full-rank phase-space lattice $\Lambda$ is
\[
 \Lambda^\circ=\{z\in\R^{2d}:\sigma(z,\lambda)\in\Z
 \text{ for all }\lambda\in\Lambda\}.
\]
It satisfies
\(
 \covol(\Lambda^\circ)=\covol(\Lambda)^{-1}.
\)
For a separable lattice $\Lambda=\Gamma\times\Phi$, one has
\(
                         \Lambda^\circ=\Phi^*\times\Gamma^*.
\)
We use the following normalized form of the duality principle; see
\cite{Janssen1995,RonShen1997,JakobsenLuef2020} and
\cite[Chapter~7]{Grochenig2001}.

\begin{theorem}[Janssen--Ron--Shen~\protect\cite{Janssen1995,RonShen1997}]\label{thm:gabor-duality}
The system $\G(g,\Lambda)$ is a Parseval frame if and only if
$\displaystyle \G\bigl(\covol(\Lambda)^{-1/2}g,\Lambda^\circ\bigr)$
is an orthonormal system.
\end{theorem}

The density theorem gives the basic restriction: if $\G(g,\Lambda)$ is a
frame, then $\covol(\Lambda)\le1$; see
\cite[Chapters~7--8]{Grochenig2001}.  At critical density,
$\covol(\Lambda)=1$, the Balian--Low phenomenon excludes well-localized
windows.  In particular, a frame window cannot belong to
$\Szero(\R^d)$, and hence cannot belong to $\Sclass(\R^d)$; see
\cite{Balian1981,Low1985,Battle1988,GrochenigHanHeilKutyniok2002,CabrelliMolterPfander2016,FeichtingerKaiblinger2004,GrochenigOrtegaRomero2015}.

\subsection{The scalar Zak transform and its regularity classes}

For general Zak-transform background and its role in Gabor analysis, see
\cite[Chapter~8]{Grochenig2001}.  We use the convention
\begin{equation}\label{eq:zak-definition}
 Zf(x,\omega)
   =\sum_{k\in\Z^d}f(x-k)e^{2\pi i\ip{k}{\omega}},
 \qquad x,\omega\in\R^d.
\end{equation}

\begin{proposition}\label{prop:zak-standard}
Let $f\in L^2(\R^d)$.
\begin{enumerate}[label=\textup{(\roman*)}]
\item The Zak transform extends to a unitary map
\[
 Z:L^2(\R^d)\longrightarrow L^2([0,1)^{2d}).
\]
\item For $n,m\in\Z^d$,
\begin{align}
 Zf(x+n,\omega)&=e^{2\pi i\ip{n}{\omega}}Zf(x,\omega),
                                                        \label{eq:zak-x-quasi}\\
 Zf(x,\omega+m)&=Zf(x,\omega).          \label{eq:zak-w-periodic}
\end{align}
\item For $n\in\Z^d$ and $y\in\R^d$,
\begin{equation}\label{eq:zak-covariance}
 Z(M_nT_yf)(x,\omega)
   =e^{2\pi i\ip{n}{x}}Zf(x-y,\omega).
\end{equation}
\item If $f\in\Sclass(\R^d)$, then $Zf$ is smooth on $\R^{2d}$, and
all partial derivatives may be obtained by differentiating the defining
series term by term.
\item One has
\[
 \Szero(\R^d)
 \subset \Wzero(\R^d)
 \subset \Wiener(\R^d)
 \subset \CZ(\R^d).
\]
More precisely, if $f\in\Wiener(\R^d)$, then the Zak series converges
absolutely and uniformly on $[0,1]^d\times[0,1]^d$, and its quasiperiodic
extension is continuous on $\R^{2d}$.
\item If $f\in\CZ(\R^d)$, then the identities
\eqref{eq:zak-x-quasi}--\eqref{eq:zak-w-periodic} hold everywhere for the
chosen continuous representative.
\item With the Fourier-transform convention $\mathcal Ff(\omega)=\widehat f(\omega)=\int_{\R^d}f(x)e^{-2\pi i\ip{x}{\omega}}\,dx$, one has
\begin{equation}\label{eq:zak-fourier-covariance}
 Z(\widehat f)(x,\omega)
 =e^{2\pi i\ip{x}{\omega}}Zf(-\omega,x)
\end{equation}
in the $L^2$ sense, and hence pointwise whenever the two sides have
continuous representatives.  In particular, the Fourier transform maps
$\CZ(\R^d)$ onto itself.
\item The finite-uncertainty space is represented by first-order Sobolev
regularity on the Zak cube.  More precisely,
$f\in\mathbb H^1(\R^d)$ if and only if the restriction of $Zf$ to
$[0,1]^{2d}$ belongs to $W^{1,2}$, with the quasiperiodic traces prescribed
by~\eqref{eq:zak-x-quasi}--\eqref{eq:zak-w-periodic}.  In the distributional
sense,
\begin{equation}\label{eq:zak-H1-identities}
 Z(\partial_j f)=\partial_{x_j}Zf,
 \qquad
 Z(x_jf)=x_jZf-\frac1{2\pi i}\partial_{\omega_j}Zf.
\end{equation}
\end{enumerate}
\end{proposition}

These are standard Zak-transform and Wiener-amalgam facts.  The Sobolev
identities follow first for Schwartz functions by differentiation and then in
the distributional sense by density.  Proofs of the remaining unitarity,
quasiperiodicity, covariance, smoothness, Fourier covariance, and
uniform-convergence statements may be found in
\cite[Chapter~8 and Chapters~11--12]{Grochenig2001}.  The notation
$\Wzero$ for the Fourier-invariant Wiener class is the one used in
\cite{deDiosLiehrTaylor2026}.

\subsection{The rational Zibulski--Zeevi matrix model}
\label{sec:rational-zak-matrix-model}

We now record the finite matrix model once, so that both the continuous and
the Sobolev arguments can use it without repeating the Zak bookkeeping.
Fix
\begin{equation}\label{eq:zak-D-data}
 \begin{gathered}
 D=\mathsf B\mathsf A^{-1},\qquad
 \mathsf A=\operatorname{diag}(a_1,\ldots,a_d),\qquad
 \mathsf B=\operatorname{diag}(b_1,\ldots,b_d),\\
 N=\det\mathsf A,\qquad R=\det\mathsf B.
 \end{gathered}
\end{equation}
where the pairs $(a_i,b_i)$ are coprime positive integers.  Set
\begin{equation}\label{eq:zak-index-sets}
 \mathcal I_{\mathsf A}=\prod_{i=1}^d\{0,\ldots,a_i-1\},\qquad
 \mathcal I_{\mathsf B}=\prod_{i=1}^d\{0,\ldots,b_i-1\},
\end{equation}
so that $|\mathcal I_{\mathsf A}|=N$ and
$|\mathcal I_{\mathsf B}|=R$, and put
\begin{equation}\label{eq:zak-base-rectangle}
 \mathcal Q_{\mathsf B}=[0,1)^d\times
       \prod_{i=1}^d[0,b_i^{-1}).
\end{equation}
Following Zibulski--Zeevi~\cite{ZibulskiZeevi1997}, define
\begin{equation}\label{eq:zz-vector-transform}
 \mathcal Zf(x,\eta)
 =\bigl(Zf(x,\eta+\mathsf B^{-1}s)\bigr)_{s\in\mathcal I_{\mathsf B}}
 \in\C^R
\end{equation}
and the $R\times N$ matrix
\begin{equation}\label{eq:zak-H-definition}
 \mathcal Z_g(x,\eta)
 =\bigl[\mathcal Z(T_{Dt}g)(x,\eta)\bigr]_{t\in\mathcal I_{\mathsf A}},
 \qquad
 (\mathcal Z_g(x,\eta))_{s,t}
 =Zg(x-Dt,\eta+\mathsf B^{-1}s).
\end{equation}
The map $f\mapsto\mathcal Zf$ is unitary from $L^2(\R^d)$ onto
$L^2(\mathcal Q_{\mathsf B};\C^R)$, and
\begin{equation}\label{eq:zak-fiber-norm}
 \|f\|_2^2=\int_{\mathcal Q_{\mathsf B}}
        \|\mathcal Zf(x,\eta)\|_{\C^R}^2\,dx\,d\eta.
\end{equation}
Every $k\in\Z^d$ has a unique decomposition
\begin{equation}\label{eq:zak-k-decomposition}
 k=\mathsf A\ell+t,
 \qquad \ell\in\Z^d,
 \quad t\in\mathcal I_{\mathsf A}.
\end{equation}
A direct Fourier-series calculation on $\mathcal Q_{\mathsf B}$ gives
\begin{equation}\label{eq:zak-coefficient-identity}
 \sum_{n,k\in\Z^d}|\ip{f}{M_nT_{Dk}g}|^2
 =\frac1R\int_{\mathcal Q_{\mathsf B}}
      \|\mathcal Z_g(x,\eta)^*\mathcal Zf(x,\eta)\|_{\C^N}^2
      \,dx\,d\eta.
\end{equation}
Initially the identity follows for Schwartz test functions; whenever
$\mathcal Z_g$ is essentially bounded it extends to all $f\in L^2$ by
density.

\begin{proposition}[Rational Zak matrix criterion]
\label{prop:zz-matrix-criterion}
Let $g\in L^2(\R^d)$ and assume that $\mathcal Z_g$ is essentially bounded.
Then $\G(g,D\Z^d\times\Z^d)$ is a frame with bounds $A_0,B_0$ if and only if
\begin{equation}\label{eq:zak-matrix-frame-bounds}
 R A_0 I_R\le \mathcal Z_g(x,\eta)\mathcal Z_g(x,\eta)^*
 \le R B_0 I_R
\end{equation}
for almost every $(x,\eta)\in\mathcal Q_{\mathsf B}$.  In particular, the
frame is Parseval if and only if
$\mathcal Z_g\mathcal Z_g^*=R I_R$ almost everywhere.
\end{proposition}

\begin{proof}
Combine~\eqref{eq:zak-fiber-norm} and
\eqref{eq:zak-coefficient-identity}.  The implication from the matrix bounds
to the frame inequalities is immediate.  Conversely, test the frame
inequalities on vector-valued functions of the form $\mathbf1_Ez$, divide by
$|E|$, and let $E$ shrink to a Lebesgue point.  A countable dense set of
vectors $z\in\C^R$ yields the asserted quadratic-form inequalities.
\end{proof}

The matrix carries simple face relations.  For $1\le i\le d$, define
unitaries on $\C^{\mathcal I_{\mathsf B}}$ by
\begin{equation}\label{eq:zak-Delta-C}
 (\Delta_i z)_s=e^{2\pi i s_i/b_i}z_s,
 \qquad
 (C_i z)_s=z_{s+e_i\ ({\rm mod}\ b_i)}.
\end{equation}
Then
\begin{align}
 \mathcal Z_g(x+e_i,\eta)
  &=e^{2\pi i\eta_i}\Delta_i\mathcal Z_g(x,\eta),
                                                        \label{eq:zak-fine-x}\\
 \mathcal Z_g(x,\eta+b_i^{-1}e_i)
  &=C_i\mathcal Z_g(x,\eta),                   \label{eq:zak-fine-eta}
\end{align}
and hence
\begin{align}
 \mathcal Z_g(x+b_i e_i,\eta)
  &=e^{2\pi i b_i\eta_i}\mathcal Z_g(x,\eta),
                                                        \label{eq:zak-coarse-x}\\
 \mathcal Z_g(x,\eta+e_i)
  &=\mathcal Z_g(x,\eta).                       \label{eq:zak-coarse-eta}
\end{align}
All these transformations are unitary row operations and unimodular scalar
multiplications; in particular they preserve rank and singular values.

For the Sobolev construction it is convenient to use the still finer closed
rectangle
\begin{equation}\label{eq:H1-fine-Zak-box}
 \mathcal Q=\prod_{i=1}^d[0,a_i^{-1}]
  \times\prod_{i=1}^d[0,b_i^{-1}].
\end{equation}
If $F$ is scalar Zak-quasiperiodic, define
\begin{equation}\label{eq:H1-fine-matrix}
 \mathcal M_F(u,\eta)_{s,t}
 =F(u-Dt,\eta+\mathsf B^{-1}s),
 \qquad s\in\mathcal I_{\mathsf B},\quad
        t\in\mathcal I_{\mathsf A}.
\end{equation}
The $RN$ translated rectangles on the right partition the unit Zak cube,
up to their boundaries.  Thus the entries of $\mathcal M_F$ are precisely
the restrictions of one scalar Zak function.  Conversely, they reassemble
into such a function provided that the paired faces satisfy the following
rules.  Let $\tau_i(t)$ agree with $t$ away from the $i$-th coordinate and be
determined there by
\[
 b_i\tau_i(t)_i\equiv b_it_i-1\pmod{a_i},
 \qquad 0\le\tau_i(t)_i<a_i,
\]
and put
\[
 q_i(t)=\frac{1-b_it_i+b_i\tau_i(t)_i}{a_i}\in\Z.
\]
On the faces $u_i=0$ and $\eta_i=0$ one requires
\begin{align}
 M_{s,t}(u+a_i^{-1}e_i,\eta)
 &=e^{2\pi i q_i(t)(\eta_i+s_i/b_i)}
   M_{s,\tau_i(t)}(u,\eta),
                                                        \label{eq:H1-fine-u-face}\\
 M_{s,t}(u,\eta+b_i^{-1}e_i)
 &=M_{s+e_i\ ({\rm mod}\ b_i),t}(u,\eta).
                                                        \label{eq:H1-fine-eta-face}
\end{align}
The first operation is a diagonal unitary followed by a column permutation,
and the second is a row permutation.  The corner compatibility follows from
scalar Zak quasiperiodicity.  If the matrix entries belong to
$W^{1,2}(\mathcal Q)$ and have matching traces under these rules, the
reassembled scalar function belongs to $W^{1,2}([0,1]^{2d})$.  Proposition
\ref{prop:zak-standard}(viii) then converts it into a window in
$\mathbb H^1(\R^d)$.

This is the common framework for the two arguments below.  A continuous
matrix satisfying the frame bounds has full rank everywhere and is subject
to the topological gap $N\ge R+d$.  A $W^{1,2}$ matrix is required to have
full rank only almost everywhere; the elementary cell-by-cell construction
in Section~\ref{sec:full-multivariate-balian-low} works under the strictly
weaker condition $N>R$.

\subsection{Continuous quasiperiodic matrices and the rectangular rank gap}
\label{sec:rectangular-common-zero}

Let $V$ be a complex vector space.  A continuous function
$s:\R^q\times\R^q\to V$ is called \emph{Zak-quasiperiodic} if
$s(u+n,v+m)=e^{2\pi i\ip{n}{v}}s(u,v)$ for all $n,m\in\Z^q$.
We use the following common-zero theorem of de Dios Pont, Liehr, and Taylor.

\begin{theorem}[de Dios Pont--Liehr--Taylor~\protect\cite{deDiosLiehrTaylor2026}]
\label{thm:ddlt-common-zero}
Let $s:\R^q\times\R^q\to\C^M$ be continuous and Zak-quasiperiodic.  If $M\le q$, then $s$ has a zero.  If $M>q$, there are smooth such maps with no zero.
\end{theorem}

Appendix~\ref{app:pfaffian-common-zero} gives an independent
first-principles proof of the implication $M\le q$.  We now use both parts of
Theorem~\ref{thm:ddlt-common-zero} to derive the rectangular rank gap needed
below.  For completeness, we first give the short explicit construction for
the second assertion.
Put
\[
 \rho(t)=\begin{cases}0,&t\le0,\\ e^{-1/t},&t>0,\end{cases}
 \qquad
 \chi(t)=\frac{\rho(t)}{\rho(t)+\rho(1-t)}.
\]
Then $\chi\in C^\infty(\R)$, $\chi=0$ on $(-\infty,0]$, $\chi=1$ on
$[1,\infty)$, every derivative of $\chi$ vanishes at $0$ and $1$, and
$\chi(1-t)=1-\chi(t)$.  On $0\le t\le1$ set
$h_*(t,v)=1-\chi(t)+\chi(t)e^{2\pi iv}$.  If $u=n+t$ with $n\in\Z$ and
$0\le t<1$, define
\[
 h(u,v)=e^{2\pi inv}\bigl(1-\chi(t)+\chi(t)e^{2\pi iv}\bigr).
\]
Flatness at the endpoints makes this quasiperiodic extension smooth, and its
only zero modulo $\Z^2$ is $(1/2,1/2)$.  For $\beta\in\R$, put
$h_\beta(u,v)=e^{2\pi i\beta u}h(u,v-\beta)$; its zero is
$(1/2,1/2+\beta)$.  Taking $\beta_r=r/(q+1)$ and
\[
 s_r(u,v)=\prod_{j=1}^q h_{\beta_r}(u_j,v_j),\qquad r=0,\ldots,q,
\]
gives $q+1$ smooth Zak-quasiperiodic components without a common zero,
since each coordinate can annihilate at most one component.  Appending zero components yields examples for every $M>q+1$.

For the corresponding continuous example one may use the identity cutoff
$\chi(t)=t$, which gives the transparent formula
$h_*(t,v)=1-t+te^{2\pi iv}$ and the same zero.  Its quasiperiodic extension
has a derivative jump at the integers, however, so the identity cannot be
used for the smooth assertion of the theorem; the flat cutoff above is used
throughout the smooth construction.

\begin{theorem}
\label{thm:rectangular-common-zero-gap}
Let
\[
 S:\R^d\times\R^d\longrightarrow M_{N,R}(\C)
\]
be a continuous Zak-quasiperiodic map.  If $S(u,v)$ has full column rank
$R$ for every $(u,v)$, then
\begin{equation}\label{eq:rectangular-gap-conclusion}
                              N\ge R+d.
\end{equation}
Conversely, whenever $N\ge R+d$, there exists a smooth
Zak-quasiperiodic map having full column rank everywhere.
Thus the bound is sharp.
\end{theorem}

\begin{proof}
If $R=1$, the $N$ entries of the single column are continuous
Zak-quasiperiodic functions in dimension $d$ with no common zero.  The first
part of Theorem~\ref{thm:ddlt-common-zero} gives $N>d$, hence
$N\ge d+1=R+d$.

Assume $R\ge2$.  Set $q=R-1$ and apply the second part of
Theorem~\ref{thm:ddlt-common-zero} with $M=R$ in dimension $q$.  It provides
smooth Zak-quasiperiodic functions $\rho_1,\ldots,\rho_R$ with no common
zero.  Thus the vector
\begin{equation}\label{eq:rho-nowhere-zero}
 \rho(x,y):=(\rho_1(x,y),\ldots,\rho_R(x,y))^{\mathsf T}
 \ne0,
 \qquad (x,y)\in\R^q\times\R^q.
\end{equation}
For each row index $\ell=1,\ldots,N$, define a scalar function in $d+q$
Zak dimensions by
\begin{equation}\label{eq:rectangular-amplified-functions}
 F_\ell\bigl((u,x),(v,y)\bigr)
   =\sum_{k=1}^R S_{\ell k}(u,v)\rho_k(x,y).
\end{equation}
In vector form,
\begin{equation}\label{eq:rectangular-vector-form}
 \bigl(F_1,\ldots,F_N\bigr)^{\mathsf T}
       =S(u,v)\rho(x,y).
\end{equation}

Let $n,m\in\Z^d$ and $p,r\in\Z^q$.  The quasiperiodicity of $S$ and $\rho$
gives
\begin{align*}
 F_\ell\bigl((u+n,x+p),(v+m,y+r)\bigr)
  &=e^{2\pi i\ip{n}{v}}
           e^{2\pi i\ip{p}{y}}
           F_\ell\bigl((u,x),(v,y)\bigr)\\
  &=e^{2\pi i\ip{(n,p)}{(v,y)}}
           F_\ell\bigl((u,x),(v,y)\bigr).
\end{align*}
Thus $F_1,\ldots,F_N$ are continuous Zak-quasiperiodic functions in
$d+R-1$ dimensions.

They have no common zero.  Indeed, if they all vanished at some point, then
\eqref{eq:rectangular-vector-form} would give
\[
                         S(u,v)\rho(x,y)=0.
\]
Full column rank means that the kernel of $S(u,v)$ is zero, so this would
force $\rho(x,y)=0$, contrary to
\eqref{eq:rho-nowhere-zero}.

The first half of Theorem~\ref{thm:ddlt-common-zero}, applied in dimension
$q'=d+q=d+R-1$, therefore forces
\(
                            N>d+R-1.
\)
All quantities are integers, hence $N\ge R+d$, which is
\eqref{eq:rectangular-gap-conclusion}.

For the converse, take smooth Zak-quasiperiodic functions
$s_0,\ldots,s_d$ with no common zero from the explicit construction after
Theorem~\ref{thm:ddlt-common-zero}.  First suppose $N=R+d$.  Index the rows
by $0,\ldots,R+d-1$ and the columns by $0,\ldots,R-1$, and define
\[
 (S_0(u,v))_{\ell k}=
 \begin{cases}
  s_{\ell-k}(u,v),&0\le \ell-k\le d,\\
  0,&\text{otherwise}.
 \end{cases}
\]
Every entry has the required quasiperiodicity.  To check the rank, fix
$(u,v)$ and suppose $S_0(u,v)c=0$, where
$c=(c_0,\ldots,c_{R-1})^{\mathsf T}$.  Form the two polynomials
\[
 a(z)=\sum_{j=0}^d s_j(u,v)z^j,
 \qquad
 c(z)=\sum_{k=0}^{R-1}c_kz^k.
\]
The entries of $S_0(u,v)c$ are exactly the coefficients of the product
$a(z)c(z)$.  Since the $s_j$ have no common zero, $a$ is not the zero
polynomial.  The ring $\C[z]$ has no zero divisors, so $c(z)=0$ and hence
$c=0$.  Thus $S_0$ has full column rank everywhere.  If $N>R+d$, append
$N-R-d$ zero rows.  This completes the sharp converse.  Notice that when
\eqref{eq:rectangular-gap-conclusion} fails, the first part of the theorem
rules out an everywhere full-rank map; the existence assertion is precisely
for the range in which the inequality holds.
\end{proof}

\section{The continuous and smooth diagonal rational regime}
\label{sec:diagonal-model}

With the general Gabor and Zak-transform tools in place, we now prove the
one-window result for the diagonal lattice
\begin{equation}\label{eq:diagonal-data}
 \Lambda_D=D\Z^d\times\Z^d,
 \qquad
 D=\operatorname{diag}\left(\frac{b_1}{a_1},\ldots,
                             \frac{b_d}{a_d}\right),
 \qquad (a_i,b_i)=1.
\end{equation}
All $a_i,b_i$ are positive integers.  Put
\(
 N=\prod_{i=1}^d a_i,
 \qquad
 R=\prod_{i=1}^d b_i.
\)
For the sufficiency direction we work in the nontrivial density range
\begin{equation}\label{eq:diagonal-density-range}
 \covol(\Lambda_D)=\det D=\frac{R}{N}<1,
 \qquad
 \operatorname{density}(\Lambda_D)=\frac{N}{R}>1.
\end{equation}

\begin{lemma}
\label{lem:diagonal-arithmetic}
For the rational lattice  $\Lambda_D$ in \eqref{eq:diagonal-data},
\begin{equation}\label{eq:diagonal-invariants}
 \nu(\Lambda_D)=N,
 \qquad
 \nu(\Lambda_D^\circ)=R,
 \qquad
 \covol(\Lambda_D)=\frac{R}{N}.
\end{equation}
Consequently, the condition~\eqref{eq:SG} for $\Lambda_D$ is equivalent to
\begin{equation}\label{eq:diagonal-SG-gap}
                              N\ge R+d.
\end{equation}
\end{lemma}

\begin{proof}
Write a point of $\R^{2d}$ as $(x,\xi)$.  By the definition of the adjoint
lattice, $(x,\xi)\in\Lambda_D^\circ$ if and only if
\[
 \ip{n}{x}\in\Z\quad\text{for every }n\in\Z^d,
 \qquad
 \ip{\xi}{Dk}\in\Z\quad\text{for every }k\in\Z^d.
\]
Hence
\begin{equation}\label{eq:diagonal-adjoint-explicit}
 \Lambda_D^\circ=\Z^d\times D^{-1}\Z^d.
\end{equation}

Let $\mathsf A=\operatorname{diag}(a_1,\ldots,a_d)$ and
$\mathsf B=\operatorname{diag}(b_1,\ldots,b_d)$.  Coprimality gives, in
each coordinate,
\[
 \frac{b_i}{a_i}\Z\cap\Z=b_i\Z,
 \qquad
 \Z\cap\frac{a_i}{b_i}\Z=a_i\Z.
\]
Therefore
\begin{equation}\label{eq:diagonal-integral-subgroup-explicit}
 \Lambda_{D,\mathrm{int}}
 =\Lambda_D\cap\Lambda_D^\circ
 =\mathsf B\Z^d\times\mathsf A\Z^d.
\end{equation}
The index factors over the two coordinates:
\[
 [\Lambda_D:\Lambda_{D,\mathrm{int}}]
 =[D\Z^d:\mathsf B\Z^d]
  [\Z^d:\mathsf A\Z^d]
 =N^2,
\]
so $\nu(\Lambda_D)=N$.  The same intersection is the integral subgroup of
$\Lambda_D^\circ$, and
\[
 [\Lambda_D^\circ:\Lambda_{D,\mathrm{int}}]
 =[\Z^d:\mathsf B\Z^d]
  [D^{-1}\Z^d:\mathsf A\Z^d]
 =R^2.
\]
Thus $\nu(\Lambda_D^\circ)=R$.  Finally,
$\covol(\Lambda_D)=\det D=R/N$, and substituting these identities into
\eqref{eq:SG} yields $N\ge R+d$.
\end{proof}

\subsection{Compactly supported Parseval windows in the diagonal model (proof of Theorem~\ref{thm:strengthened-one-window}\textup{(a)} in diagonal form)}
\label{sec:gabor-sufficiency}

Assume that the condition~\eqref{eq:SG} holds for $\Lambda_D$.  By
Lemma~\ref{lem:diagonal-arithmetic}, the condition~\eqref{eq:TeP} holds
for the pair $(D\Z^d,\Z^d)$.  The following theorem is the diagonal existence result.

\begin{theorem}
\label{thm:diagonal-positive}
Let $D$ satisfy \eqref{eq:diagonal-data}.  If $N\ge R+d$, then there exists
a nonnegative $g\in C_c^\infty(\R^d)$ such that
$\G(g,D\Z^d\times\Z^d)$ is a Parseval frame.  With
$\Gamma=D\Z^d$ and $\Phi=\Z^d$, the window can be chosen to satisfy
\begin{align}
 \supp g\cap(\supp g+\xi^*)&=\varnothing,
 &&\xi^*\in\Phi^*\setminus\{0\},
 \label{eq:g-support-separation}\\
 \sum_{\gamma\in\Gamma}|g(x+\gamma)|^2&=\covol(\Phi),
 &&x\in\R^d.
 \label{eq:g-periodization}
\end{align}
Equivalently, $\covol(\Phi)^{-1}|g|^2$ is a smooth
$\Gamma$-partition of unity.
Alternatively, there is a nonnegative $g_{\rm pa}\in C_c(\R^d)$ for which
$\covol(\Phi)^{-1}g_{\rm pa}^2$ is a continuous piecewise-affine
$\Gamma$-partition of unity and
\eqref{eq:g-support-separation}--\eqref{eq:g-periodization} hold with
$g_{\rm pa}$ in place of $g$.
\end{theorem}

\begin{proof}
Corollary~\ref{cor:global-partitions} supplies a nonnegative compactly
supported piecewise-affine $\Gamma$-partition of unity $\varphi$ satisfying
\begin{align}
 \sum_{\gamma\in\Gamma}\varphi(x+\gamma)&=1,
 \label{eq:phi-gabor-partition}\\
 (\supp\varphi+B_\eta)+\Phi^*&\text{ is a packing}
 \label{eq:phi-gabor-separation}
\end{align}
for some $\eta>0$.  The unsmoothed alternative is immediate: put
\begin{equation}\label{eq:g-pa-from-phi}
 g_{\rm pa}(x)=\bigl(\covol(\Phi)\,\varphi(x)\bigr)^{1/2}.
\end{equation}
Then $\covol(\Phi)^{-1}g_{\rm pa}^2=\varphi$ is the
piecewise-affine $\Gamma$-partition of unity,
$\supp g_{\rm pa}=\supp\varphi$, and
\eqref{eq:g-support-separation}--\eqref{eq:g-periodization} hold.  The
standard adjoint-lattice argument below therefore shows that
$\G(g_{\rm pa},\Gamma\times\Phi)$ is Parseval.

To obtain a smooth window, we mollify a positive part of the same
piecewise-affine partition of unity while
preserving its positive support separation, and then normalize its energy.

Because $\varphi$ has compact support, there is a finite integer $D_0$ such
that at most $D_0$ terms in \eqref{eq:phi-gabor-partition} are nonzero at
any point.  Consequently, a single translate of $\varphi$ must contribute at least $1/D_0$ to the constant sum. Choose
\[
                         0<a<b<c<\frac1{D_0}
\]
and set
\[
 K_b=\{\varphi\ge b\},
 \qquad
 V_a=\{\varphi>a\}.
\]
Then $K_b$ is compact and $K_b\Subset V_a$, where $\Subset$ denotes compact
containment, and $\{\varphi\ge c\}+\Gamma=\R^d$.  For a real number $r$,
write $r_+=\max\{r,0\}$, and set
\[
                         h=(\varphi-b)_+.
\]
Choose a standard nonnegative mollifier $\rho\in C_c^\infty(B_1)$ with
$\int\rho=1$, put $\rho_\delta(x)=\delta^{-d}\rho(x/\delta)$, and define
\[
                         \psi=h*\rho_\delta.
\]
For sufficiently small $\delta>0$,
\[
 \psi\in C_c^\infty(\R^d),
 \qquad
 \supp\psi\subset V_a,
\]
and $\psi>0$ on $\{\varphi\ge c\}$.  Hence
$\{\psi>0\}+\Gamma=\R^d$.  The support remains $\Phi^*$-separated, because
$\supp\psi\subset V_a$ and \eqref{eq:phi-gabor-separation} excludes positive
values at points differing by a nonzero shift of $\Phi^*$.

Define
\[
 D_\psi(x)=\sum_{\gamma\in\Gamma}\psi(x+\gamma)^2.
\]
This function is smooth, strictly positive, and $\Gamma$-periodic.  Put
\[
                         g(x)=\sqrt{\covol(\Phi)}\,
                         \frac{\psi(x)}{\sqrt{D_\psi(x)}}.
\]
Then $g\in C_c^\infty(\R^d)$, $g\ge0$, and
\eqref{eq:g-support-separation}--\eqref{eq:g-periodization} hold.

It remains to verify the Parseval property.  Let
$
                         h_0=\covol(\Gamma\times\Phi)^{-1/2}g.
$
Then
\[
 \sum_{\gamma\in\Gamma}|h_0(x+\gamma)|^2
 =\frac1{\covol(\Gamma)},
\]
and the $\Phi^*$-support separation is unchanged.  We show that
$\G(h_0,\Phi^*\times\Gamma^*)$ is orthonormal.  Distinct translations by
$\Phi^*$ have disjoint supports.  For a fixed translation and
$\gamma^*,\gamma^{*\prime}\in\Gamma^*$, periodization over a fundamental
domain $F_\Gamma$ gives
\[
 \begin{aligned}
 \ip{h_0}{M_{\gamma^{*\prime}-\gamma^*}h_0}
 &=\int_{F_\Gamma}\sum_{\gamma\in\Gamma}|h_0(x+\gamma)|^2
   e^{2\pi i\ip{\gamma^{*\prime}-\gamma^*}{x+\gamma}}\,dx\\
 &=\frac1{\covol(\Gamma)}
   \int_{F_\Gamma}
   e^{2\pi i\ip{\gamma^{*\prime}-\gamma^*}{x}}\,dx
 =\delta_{\gamma^*,\gamma^{*\prime}}.
 \end{aligned}
\]
Theorem~\ref{thm:gabor-duality} now gives that
$\G(g,\Gamma\times\Phi)=\G(g,D\Z^d\times\Z^d)$ is Parseval.
\end{proof}

\subsection{Necessity of the $\SG$ gap in the diagonal model (proof of Theorem~\ref{thm:intro-diagonal-zak-obstruction})}
\label{sec:rational-zak}

We now assume only that the Zak transform of the window is continuous.  The
next theorem proves directly that a one-window frame on the diagonal lattice
must satisfy the algebraic gap.

\begin{theorem}
\label{thm:diagonal-one-window}
Let $D$ be as in \eqref{eq:diagonal-data}, and let $g\in\CZ(\R^d)$.  If
\(
                         \G(g,D\Z^d\times\Z^d)
\)
is a frame, then
\(
                              N\ge R+d,
\)
that is, the condition~\eqref{eq:SG} is satisfied.
\end{theorem}

\begin{proof}
Use the notation and matrix model of
Subsection~\ref{sec:rational-zak-matrix-model}.  By
Proposition~\ref{prop:zz-matrix-criterion}, the Zibulski--Zeevi matrix
$\mathcal Z_g$ satisfies the frame bounds
\eqref{eq:zak-matrix-frame-bounds} almost everywhere.  Since $Zg$ is
continuous, $\mathcal Z_g$ is continuous on the compact closure of the base
rectangle.  Its boundary behavior is given by
\eqref{eq:zak-fine-x}--\eqref{eq:zak-coarse-eta}.

The smallest eigenvalue of $\mathcal Z_g\mathcal Z_g^*$ is continuous.  By
\eqref{eq:zak-matrix-frame-bounds}, it is bounded below by $RA_0$ almost
everywhere on the base rectangle.  If it were smaller at one point,
continuity would give an open set of positive measure on which the same
strict inequality held, a contradiction.  The boundary relations propagate
the conclusion throughout $\R^{2d}$.  Hence
\begin{equation}\label{eq:zak-everywhere-rank}
                      \rank \mathcal Z_g(x,\eta)=R,
 \qquad (x,\eta)\in\R^{2d}.
\end{equation}

We now introduce a diagonal gauge and derive the common-zero contradiction.
For $s\in\mathcal I_{\mathsf B}$ define the diagonal unitary matrix
\begin{equation}\label{eq:zak-Gamma-definition}
 \Gamma(x)_{s,s}
  =\exp\left(-2\pi i\sum_{j=1}^d\frac{s_j}{b_j}x_j\right).
\end{equation}
Then $\Gamma(x+e_i)=\Delta_i^{-1}\Gamma(x)$.  Define the
$N\times R$ matrix
\begin{equation}\label{eq:zak-S-definition}
                    S_g(x,\eta)=\mathcal Z_g(x,\eta)^{\mathsf T}\Gamma(x).
\end{equation}
It is continuous and has full column rank $R$ everywhere.  Using
\eqref{eq:zak-fine-x}, the identity $\Delta_i^{\mathsf T}=\Delta_i$, and
the boundary rule for $\Gamma$, we obtain
\begin{align}
 S_g(x+e_i,\eta)
  &=e^{2\pi i\eta_i}S_g(x,\eta),
                                                        \label{eq:zak-S-u}\\
 S_g(x,\eta+e_i)
  &=S_g(x,\eta).                         \label{eq:zak-S-v}
\end{align}
Equivalently,
\begin{equation}\label{eq:zak-S-full-quasi}
 S_g(x+n,\eta+m)=e^{2\pi i\ip{n}{\eta}}S_g(x,\eta),
 \qquad n,m\in\Z^d.
\end{equation}
Thus $S_g$ satisfies the hypotheses of
Theorem~\ref{thm:rectangular-common-zero-gap}, which gives $N\ge R+d$.
\end{proof}

This proves the diagonal continuous-Zak obstruction stated in
Theorem~\ref{thm:intro-diagonal-zak-obstruction}.

\section{Symplectic and metaplectic extension machinery}
\label{sec:lattice-theory}

We now prepare the passage from the diagonal model to arbitrary phase-space
lattices, using the symplectic form, adjoint lattice, symplectic rationality,
and the invariant $\nu$ introduced in Subsection~\ref{subsec:intro-gabor}.  Let
\[
 J=\begin{pmatrix}0&I_d\\-I_d&0\end{pmatrix},
\]
so that $\sigma(z,w)=z^TJw$.  The symplectic group is
\[
 \Sp(2d,\R)
 =\{S\in GL_{2d}(\R):S^TJS=J\}.
\]
Thus $S$ is symplectic exactly when
$\sigma(Sz,Sw)=\sigma(z,w)$ for all $z,w\in\R^{2d}$.  Two phase-space
lattices $\Lambda$ and $\Lambda'$ are \emph{symplectically equivalent} if
$\Lambda'=S\Lambda$ for some $S\in\Sp(2d,\R)$.  In concrete terms, the
phase-space variable $z=(x,\omega)$ is replaced by
$Sz=(x',\omega')$, and every lattice point $\lambda$ is replaced by
$S\lambda$.  Symplectic equivalence preserves covolume and satisfies
\[
                         (S\Lambda)^\circ=S\Lambda^\circ.
\]

Write a real $2d\times2d$ matrix in blocks as
\[
 S=\begin{pmatrix}A&B\\ C&D\end{pmatrix},
 \qquad A,B,C,D\in M_d(\R).
\]
The block criterion for symplectic matrices
\cite[Lemma~9.4.1(c)]{Grochenig2001} states that $S\in\Sp(2d,\R)$ if and
only if
\begin{equation}\label{eq:symplectic-block-criterion}
 A^TC=C^TA,
 \qquad B^TD=D^TB,
 \qquad A^TD-C^TB=I_d.
\end{equation}
Equivalently,
\begin{equation}\label{eq:symplectic-block-criterion-dual}
 AB^T=BA^T,
 \qquad CD^T=DC^T,
 \qquad AD^T-BC^T=I_d.
\end{equation}
For example, if $A$ is invertible, these identities imply that
$CA^{-1}$ and $A^{-1}B$ are symmetric.  These elementary block relations
will be used below to read support properties directly from a symplectic
matrix.

The group $\Sp(2d,\R)$ is generated by matrices of the following three
forms:
\begin{align*}
 S_A&=\begin{pmatrix}A&0\\0&A^{-T}\end{pmatrix},
       &&A\in GL_d(\R),\\
 V_B&=\begin{pmatrix}I_d&0\\B&I_d\end{pmatrix},
       &&B\in M_d(\R),\quad B=B^T,\\
 J&=\begin{pmatrix}0&I_d\\-I_d&0\end{pmatrix}.
\end{align*}
The matrix $S_A$ is called the \emph{symplectic dilation associated with
$A$}; it acts by
\[
                         (x,\omega)\longmapsto(Ax,A^{-T}\omega).
\]

The metaplectic group is a double cover of $\Sp(2d,\R)$ by unitary
operators on $L^2(\R^d)$.  In fact, for each $S\in\Sp(2d,\R)$ there exists a unitary  \emph{metaplectic lift} $\mu(S)$.  Up to the sign ambiguity arising from the double cover, we have $\mu(S_1S_2)=\mu(S_1)\mu(S_2)$. With the time--frequency convention used here, a
metaplectic lift satisfies
\[
 \mu(S)\pi(z)\mu(S)^{-1}=c(S,z)\pi(Sz),
 \qquad |c(S,z)|=1.
\]
Up to constants of modulus one, convenient lifts of the three generators
are
\[
 (\mu(S_A)f)(t)=|\det A|^{-1/2}f(A^{-1}t),
 \qquad
 (\mu(V_B)f)(t)=e^{\pi i t^TBt}f(t),
\]
and, with the Fourier convention used in this paper,
\[
 (\mu(J)f)(\omega)
 =\int_{\R^d}f(t)e^{-2\pi i\ip{t}{\omega}}\,dt=\widehat f(\omega).
\]
Consequently, $\mu(S)$ is unitary on $L^2(\R^d)$, maps
$\Sclass(\R^d)$ and $\Szero(\R^d)$ onto themselves, and preserves frame
bounds when the lattice and all windows are transformed together.  In
particular,
\[
 \G(g,\Lambda)\text{ is Parseval}
 \quad\Longleftrightarrow\quad
 \G(\mu(S)g,S\Lambda)\text{ is Parseval}.
\]
These standard facts may be found in \cite{Folland1989},
\cite[Section~9.4]{Grochenig2001}, and \cite{GjertsenLuef2024}.

We next isolate the arithmetic notion relevant to the normal form.  For a
full-rank phase-space lattice $\Lambda$, recall
\[
 \Lambda_{\rm int}=\Lambda\cap\Lambda^\circ.
\]
The lattice is \emph{symplectically rational} if
$[\Lambda:\Lambda_{\rm int}]<\infty$.  A real matrix is called
\emph{entrywise rational} if all of its entries lie in $\Q$.

\begin{unnumberedlemma}
Let $\Lambda=M\Z^{2d}$ and put
$\Theta_\Lambda=M^TJM$.  Then $\Lambda$ is symplectically rational if and
only if $\Theta_\Lambda$ is entrywise rational.  In particular, let
\[
 \Lambda=\Gamma\times\Phi,
 \qquad
 \Gamma=A_0\Z^d,
 \qquad
 \Phi=B_0\Z^d,
\]
with $A_0,B_0\in GL_d(\R)$.  The following conditions are equivalent:
\begin{enumerate}[label=\textup{(\roman*)}]
\item $\Gamma\times\Phi$ is symplectically rational;
\item the pairing matrix $A_0^TB_0$ is entrywise rational;
\item the Euclidean lattices $\Gamma$ and $\Phi^*$ are commensurable.
\end{enumerate}
\end{unnumberedlemma}

\begin{proof}
Suppose first that $\Theta_\Lambda$ is entrywise rational.  Choose
$q\in\N$ such that $q\Theta_\Lambda$ is integral.  If
$\lambda=Mk$ and $\mu=M\ell$, with $k,\ell\in\Z^{2d}$, then
\[
 \sigma(q\lambda,\mu)=qk^T\Theta_\Lambda\ell\in\Z.
\]
Thus $q\Lambda\subset\Lambda_{\rm int}$, so
$[\Lambda:\Lambda_{\rm int}]<\infty$.

Conversely, suppose $\Lambda/\Lambda_{\rm int}$ is finite, and let $q$ be
its exponent.  Then $q\Lambda\subset\Lambda_{\rm int}$.  Applying this to
the basis vectors $Me_j$ gives
\[
 q(\Theta_\Lambda)_{jk}
 =\sigma(qMe_j,Me_k)\in\Z,
 \qquad 1\le j,k\le2d.
\]
Hence $\Theta_\Lambda$ is entrywise rational.

For the separable lattice $\Gamma\times\Phi$, the symplectic Gram matrix of
the basis $\operatorname{diag}(A_0,B_0)$ is
\[
 \begin{pmatrix}
 0&A_0^TB_0\\
 -B_0^TA_0&0
 \end{pmatrix}.
\]
This proves the equivalence of (i) and (ii).  Under the common linear change
of Euclidean coordinates $A_0^{-1}$, the pair $(\Gamma,\Phi^*)$ becomes
\[
 \bigl(\Z^d,(A_0^TB_0)^{-T}\Z^d\bigr).
\]
A lattice $C\Z^d$ is commensurable with $\Z^d$ exactly when
$C\in GL_d(\Q)$.  This proves the equivalence with (iii).
\end{proof}

We use the ordinary Smith normal form and its alternating, or skew,
analogue in the proof of the normal-form theorem.

\begin{unnumberedtheorem}[Smith--Newman~\protect\cite{Newman1972}]
Let $K\in M_d(\Z)$ have nonzero determinant.  There exist
$U,V\in GL_d(\Z)$ and positive integers
$s_1,\ldots,s_d$, with $s_j$ dividing $s_{j+1}$ for
$j=1,\ldots,d-1$, such that
\[
 U^T K V=\operatorname{diag}(s_1,\ldots,s_d).
\]
Let $\Theta\in M_{2d}(\Z)$ be nonsingular and alternating, that is $\Theta^T=-\Theta$.  There exist
$W\in GL_{2d}(\Z)$ and positive integers
$h_1,\ldots,h_d$, with $h_j$ dividing $h_{j+1}$ for
$j=1,\ldots,d-1$, such that
\[
 W^T\Theta W
 =\begin{pmatrix}0&H\\-H&0\end{pmatrix},
 \qquad H=\operatorname{diag}(h_1,\ldots,h_d).
\]
See \cite[Chapters~II and~IV]{Newman1972}; the alternating statement is the
skew Smith normal form.
\end{unnumberedtheorem}

\begin{theorem}
\label{thm:diagonal-symplectic-normal-form}
Every symplectically rational full-rank phase-space lattice
$\Lambda\subset\R^{2d}$ is symplectically equivalent to
\begin{equation}\label{eq:normal-form-lattice}
 \Lambda_D=D\Z^d\times\Z^d,
 \qquad
 D=\operatorname{diag}\left(\frac{b_1}{a_1},\ldots,
                             \frac{b_d}{a_d}\right),
 \qquad (a_i,b_i)=1,
\end{equation}
where all $a_i,b_i$ are positive integers.  Thus there exists
$S\in\Sp(2d,\R)$ such that $\Lambda=S\Lambda_D$.

If
\[
 N=\prod_{i=1}^d a_i,
 \qquad
 R=\prod_{i=1}^d b_i,
\]
then
\begin{equation}\label{eq:normal-form-invariants}
 \nu(\Lambda)=\nu(\Lambda_D)=N,
 \qquad
 \nu(\Lambda^\circ)=\nu(\Lambda_D^\circ)=R,
 \qquad
 \covol(\Lambda)=\covol(\Lambda_D)=\frac{R}{N}.
\end{equation}
Consequently, the symplectically rational conditions $\SG$ and $\mSG_q$
become, respectively,
\[
                             N\ge R+d,
 \qquad
                             qN\ge R+d.
\]

If $\Lambda=\Gamma\times\Phi$ is separable, the normalizing symplectic
matrix may be chosen to be a symplectic dilation.  If $\Gamma+\Phi^*$ is a
lattice, then
\[
 \nu(\Gamma\times\Phi)=[\Gamma+\Phi^*:\Phi^*],
 \qquad
 \nu((\Gamma\times\Phi)^\circ)=[\Gamma+\Phi^*:\Gamma].
\]
\end{theorem}

\begin{proof}
We begin with the separable case.  Write
\[
 \Gamma=A_0\Z^d,
 \qquad
 \Phi=B_0\Z^d,
 \qquad
 C_{\Gamma,\Phi}=A_0^TB_0.
\]
By the preceding lemma, $C_{\Gamma,\Phi}$ is entrywise rational.  Choose
$m\in\N$ so that $mC_{\Gamma,\Phi}$ is integral.  The first part of the preceding Smith--Newman theorem gives
$U,V\in GL_d(\Z)$ and positive integers
$s_1,\ldots,s_d$, with $s_j$ dividing $s_{j+1}$ for
$j=1,\ldots,d-1$, such that
\[
 U^T(mC_{\Gamma,\Phi})V
 =\operatorname{diag}(s_1,\ldots,s_d).
\]
Put
\[
 D_0=\operatorname{diag}(s_1/m,\ldots,s_d/m).
\]
The symplectic dilation associated with $A_0^{-1}$ sends
$\Gamma\times\Phi$ to
$\Z^d\times C_{\Gamma,\Phi}\Z^d$.  The symplectic dilation associated with
$U^{-1}$ then sends this lattice to
\[
 \Z^d\times U^TC_{\Gamma,\Phi}\Z^d
 =\Z^d\times D_0\Z^d,
\]
because $V\Z^d=\Z^d$.  Finally, the symplectic dilation associated with
$D_0$ sends this lattice to
\[
                         D_0\Z^d\times\Z^d.
\]
Writing each diagonal entry of $D_0$ in lowest terms as $b_i/a_i$ gives
\eqref{eq:normal-form-lattice}.  The product of the three symplectic
dilations is again a symplectic dilation.  Thus no chirp or Fourier
transform is needed in the separable reduction.

We now treat an arbitrary symplectically rational phase-space lattice.
Choose a basis matrix $M\in GL_{2d}(\R)$ with
$\Lambda=M\Z^{2d}$ and set $\Theta_\Lambda=M^TJM$.  By the preceding
lemma, $\Theta_\Lambda$ is entrywise rational.  Choose $q\in\N$ such that
$q\Theta_\Lambda$ is integral.  The second part of the preceding Smith--Newman theorem, applied to the invertible and alternating matrix 
$q\Theta_\Lambda$, gives
$U\in GL_{2d}(\Z)$ and positive integers
$h_1,\ldots,h_d$, with $h_j$ dividing $h_{j+1}$ for
$j=1,\ldots,d-1$, such that
\[
 U^T\Theta_\Lambda U
 =\begin{pmatrix}0&D_0\\-D_0&0\end{pmatrix},
 \qquad
 D_0=\operatorname{diag}\left(\frac{h_1}{q},\ldots,
                               \frac{h_d}{q}\right).
\]
Let
\[
 C_0=\begin{pmatrix}D_0&0\\0&I_d\end{pmatrix},
 \qquad
 S=MUC_0^{-1}.
\]
Since
\[
 C_0^TJC_0
 =\begin{pmatrix}0&D_0\\-D_0&0\end{pmatrix}
 =U^TM^TJMU,
\]
we obtain
\[
 S^TJS=C_0^{-T}U^TM^TJMU C_0^{-1}=J.
\]
Thus $S$ is symplectic and
\[
 \Lambda=M\Z^{2d}
 =MU\Z^{2d}
 =SC_0\Z^{2d}
 =S(D_0\Z^d\times\Z^d).
\]
Reducing the diagonal entries of $D_0$ gives the coprime pairs
$(a_i,b_i)$ in \eqref{eq:normal-form-lattice}.

For the diagonal lattice, the three formulas in
\eqref{eq:normal-form-invariants} were computed explicitly in
Lemma~\ref{lem:diagonal-arithmetic}.  If $\Lambda=S\Lambda_D$, then
$(S\Lambda_D)^\circ=S\Lambda_D^\circ$,
$S(\Lambda_D\cap\Lambda_D^\circ)
 =\Lambda\cap\Lambda^\circ$, and $\det S=1$.  Therefore the two indices and
the covolume are unchanged, proving all three equalities in
\eqref{eq:normal-form-invariants}.

In the separable case, the common linear changes above carry both $\Gamma$
and $\Phi^*$ by the same invertible matrix, so the indices in their sum are
unchanged.  On the diagonal representative the two indices are $N$ and $R$,
which gives
\[
 \nu(\Gamma\times\Phi)=[\Gamma+\Phi^*:\Phi^*],
 \qquad
 \nu((\Gamma\times\Phi)^\circ)=[\Gamma+\Phi^*:\Gamma].
\]
\end{proof}

\begin{proof}[Proof of Proposition~\ref{prop:TeP-SG-relation}]
By the separable part of
Theorem~\ref{thm:diagonal-symplectic-normal-form},
$\Gamma\times\Phi$ is symplectically rational exactly when
$\Gamma+\Phi^*$ is a lattice.  If $\Gamma+\Phi^*$ is not discrete, then
\[
 \covol(\Gamma\times\Phi)<1
 \quad\Longleftrightarrow\quad
 \covol(\Gamma)<\covol(\Phi^*),
\]
which is the non-discrete branch of $\TeP$.  If $\Gamma+\Phi^*$ is a lattice, the index
formulas in the theorem show that the symplectically rational branch of
$\SG$ is precisely
\[
                         [\Gamma+\Phi^*:\Phi^*]\ge[\Gamma+\Phi^*:\Gamma]+d,
\]
which is the commensurable branch of $\TeP$.
\end{proof}

\subsection{Effect on the support of the window}
\label{subsec:metaplectic-implementation}

The distinction between separable and general lattices is important for the
effect on the support of the window function $g$.  The block criterion gives
a precise support-preserving case.  Let
\[
 S=\begin{pmatrix}A&B\\ C&D\end{pmatrix}\in\Sp(2d,\R).
\]
If $B=0$, then~\eqref{eq:symplectic-block-criterion} gives
$D=A^{-T}$ and shows that
\[
 Q=CA^{-1}
\]
is symmetric.  Hence
\[
 S=
 \begin{pmatrix}I_d&0\\ Q&I_d\end{pmatrix}
 \begin{pmatrix}A&0\\0&A^{-T}\end{pmatrix},
\]
and, up to a constant of modulus one,
\begin{equation}\label{eq:metaplectic-lower-triangular-formula}
 (\mu(S)f)(t)
 =|\det A|^{-1/2}e^{\pi i t^TQt}f(A^{-1}t).
\end{equation}
Consequently,
\begin{equation}\label{eq:metaplectic-support-image}
 \supp\mu(S)f=A(\supp f).
\end{equation}
Thus $\mu(S)$ maps $C_c^\infty(\R^d)$ onto itself whenever $B=0$.
Conversely, suppose $B\ne0$ and let $r=\rank B\ge1$.  In the standard
Bruhat factorization of $S$, all factors except one preserve compact support:
they are invertible linear changes of variables or multiplication by chirps.
The remaining factor is an $r$-dimensional partial Fourier transform.  Write
the variables for that factor as $(x,y)\in\R^r\times\R^{d-r}$.  If both its
input $f_0$ and output were compactly supported, then Fubini's theorem would
give, for almost every fixed $y$, a compactly supported slice
$x\mapsto f_0(x,y)$ whose Fourier transform in $x$ is also compactly
supported.  That Fourier transform extends to an entire function of
$\zeta\in\C^r$ by the Paley--Wiener theorem.  It vanishes on an open
rectangle in the complement of its real support; applying the one-variable
identity theorem successively in $\zeta_1,\ldots,\zeta_r$ shows that it is
identically zero.  Hence $f_0(\,\cdot\,,y)=0$ for almost every $y$, and
therefore $f_0=0$.  Undoing the support-preserving factors gives the same
conclusion for the original input.  Thus a nonzero compactly supported
function cannot be carried to a compactly supported function when $B\ne0$,
and $B=0$ is the intrinsic block criterion for a metaplectic operator to
preserve compact support.

The Fourier transform of $\mu(S)f$ is, up to a constant of modulus one,
$\mu(JS)f$.  Since
\[
 JS=\begin{pmatrix}C&D\\-A&-B\end{pmatrix},
\]
Equation~\eqref{eq:metaplectic-lower-triangular-formula} applies to $JS$
exactly when $D=0$.  Thus $f\in C_c^\infty(\R^d)$ implies
$\widehat{\mu(S)f}\in C_c^\infty(\R^d)$ when $D=0$.

In the separable part of
Theorem~\ref{thm:diagonal-symplectic-normal-form}, the normalizing matrix is
a symplectic dilation, so its upper-right block is zero.  The reduction of a
separable lattice to diagonal form therefore does not destroy compact
support.  For an arbitrary symplectically rational phase-space lattice, the
normal-form theorem produces a general symplectic matrix, and its upper-right
block need not vanish.  This is why transfer to a general lattice naturally
guarantees a Schwartz or Feichtinger-algebra window, but not a compactly
supported one.

\section{The Sobolev Zak regime: finite uncertainty and strict density}
\label{sec:full-multivariate-balian-low}

We now prove Theorem~\ref{thm:full-multivariate-balian-low}.  The necessity
at critical density combines the weak arbitrary-lattice Balian--Low theorem
with regularity of the canonical dual.  For strict density, the
non-symplectically-rational case follows immediately from
Theorem~\ref{thm:main-gabor}.  In the rational case we return to the common
Zak matrix model of Subsection~\ref{sec:rational-zak-matrix-model}, but work
in the weaker Sobolev class: full rank is needed only almost everywhere, and
the threshold drops from $N\ge R+d$ to $N>R$.

We first prove \textup{(a)}$\Rightarrow$\textup{(b)}.

Suppose that $\G(g,\Lambda)$ is a frame with
$g\in\mathbb H^1(\R^d)$.  The density theorem gives
$\covol(\Lambda)\le1$.  We first exclude equality.  Let $S$ be the frame
operator and put $\gamma=S^{-1}g$.  At critical density the frame is a Riesz
basis, so $\gamma$ is its biorthogonal dual window.  By the canonical-dual
regularity theorem of Lee--Philipp--Voigtlaender,
$\gamma\in\mathbb H^1(\R^d)$ whenever $g\in\mathbb H^1(\R^d)$
\cite{LeePhilippVoigtlaender2023}.  This contradicts the weak
arbitrary-lattice Balian--Low theorem of
Gr\"ochenig--Han--Heil--Kutyniok, which excludes simultaneous membership
of a critical Riesz-basis window and its biorthogonal dual in
$\mathbb H^1(\R^d)$ \cite{GrochenigHanHeilKutyniok2002}.  Hence
$\covol(\Lambda)\ne1$, and therefore $\covol(\Lambda)<1$.

By Plancherel, $\partial_jg\in L^2$ if and only if
$\xi_j\widehat g\in L^2$.  Since a frame window is nonzero, membership in
$\mathbb H^1(\R^d)$ is equivalent to finiteness of the uncertainty product
in Theorem~\ref{thm:full-multivariate-balian-low}.

We next prove \textup{(b)}$\Rightarrow$\textup{(a)}.  We begin with
lattices that are not symplectically rational.

Assume that $\covol(\Lambda)<1$ and that $\Lambda$ is not symplectically
rational.  In this case the condition~\eqref{eq:SG} is precisely the strict
density inequality.  Theorem~\ref{thm:main-gabor} therefore gives a
Schwartz-class Gabor-frame window.  Since
$\Sclass(\R^d)\subset\mathbb H^1(\R^d)$, this proves the implication in the
non-symplectically-rational case.

We now move to constructing a window in $\mathbb{H}^1(\R^d)$ for lattices 
with rational diagonal coordinates. Note that this includes the case when 
the arithmetic condition in \eqref{eq:SG} fails, in contrast to the Schwartz 
or Feichtinger algebra case. The construction is based on building a matrix 
field on the unit cube that satisfies the Zak quasiperiodicity conditions 
and has $W^{1,2}$ regularity, after which the desired window can be found 
by taking the inverse Zak transform. 

\begin{lemma}[Finite-uncertainty window in the diagonal model]
\label{lem:finite-uncertainty-diagonal}
Let
\[
 \Lambda_D=D\Z^d\times\Z^d,
 \qquad
 D=\operatorname{diag}(b_1/a_1,\ldots,b_d/a_d),
 \qquad (a_i,b_i)=1,
\]
and put $N=\prod_i a_i$ and $R=\prod_i b_i$.  If $N>R$, then there exists
$g\in\mathbb H^1(\R^d)$ such that $\G(g,\Lambda_D)$ is a Parseval frame.
\end{lemma}

\begin{proof}
We construct the required Zak matrix on a fixed periodic simplicial complex.
The proof has five steps, which we now summarize.

First, we rewrite the two basic Zak translation identities as left and right
unitary operations on the finite matrix.  This gives the precise rule for
transporting matrix values from one translate of the fine Zak rectangle to
another and verifies that the transport is consistent around its corners.
At the same time we fix a periodic triangulation of the Zak parameter space.
Second, we analyze the set of matrices that fail to have full row rank and
show, by elementary block elimination, that it is covered by countably many
Lipschitz images of dimension at most $2RN-4$.  Third, using this codimension
estimate, we assign full-rank matrices to the vertices of the triangulation,
join them along its edges, and fill its triangular faces without meeting the
rank-deficient set.  This constructs the matrix field on the two-skeleton and
also gives a uniform positive lower bound for its smallest singular value.
Fourth, we extend successively across every simplex of dimension at least
three by keeping the boundary value constant along radial rays.  A direct
polar-coordinate calculation shows that these extensions have square-
integrable first derivatives.  Fifth, we orthonormalize the rows, glue the
matrix entries into one scalar Zak function, and take the inverse Zak
transform.

The simplicial complex used throughout is obtained by affinely rescaling
the closed fine Zak rectangle
\[
 \overline{\mathcal Q}=
 \prod_{i=1}^d[0,a_i^{-1}]\times
 \prod_{i=1}^d[0,b_i^{-1}]
\]
to the unit cube in $\R^{2d}$, using the standard staircase triangulation of
that cube, extending it periodically to $\R^{2d}$, and then scaling back.  
Denote the resulting $\Gamma$-periodic simplicial complex by $\mathcal T$, where
\[
 \Gamma=\mathsf A^{-1}\Z^d\times\mathsf B^{-1}\Z^d.
\]
The triangulations induced on opposite sides of $\overline{\mathcal Q}$
agree under the corresponding translation.  The $k$-skeleton
$\mathcal T^{(k)}$ is the union of all simplices of dimension at most $k$;
in particular, the vertices are the $0$-simplices, the edges are the
$1$-simplices, and the triangles are the $2$-simplices.  The quotient
$\mathcal T/\Gamma$ is finite.  We therefore make each choice on one
representative of a simplex orbit and define the values on all its translates
by the translation rules established in Step~1.

\medskip
\noindent\textbf{Step 1: the spatial and frequency translation rules.}
Use the notation of Subsection~\ref{sec:rational-zak-matrix-model}.  Thus
\[
 \mathsf A=\operatorname{diag}(a_1,\ldots,a_d),\qquad
 \mathsf B=\operatorname{diag}(b_1,\ldots,b_d),\qquad
 D=\mathsf B\mathsf A^{-1},
\]
$|\mathcal I_{\mathsf A}|=N$, $|\mathcal I_{\mathsf B}|=R$, and
\[
 \mathcal Q=
 \prod_{i=1}^d[0,a_i^{-1}]\times
 \prod_{i=1}^d[0,b_i^{-1}].
\]
The two identities
\eqref{eq:H1-fine-u-face}--\eqref{eq:H1-fine-eta-face} describe what happens
to the matrix when the spatial Zak variable is translated by
$a_i^{-1}e_i$ or the frequency Zak variable is translated by
$b_i^{-1}e_i$.  We now put these identities into a form that makes their
unitarity and compatibility transparent.

Fix $i$.  Choose $c_i\in\{0,\ldots,a_i-1\}$ such that
$b_ic_i\equiv1\pmod{a_i}$, and put
\[
 h_i=\frac{b_ic_i-1}{a_i}\in\Z.
\]
The definition of the column permutation $\tau_i$ says that, for a uniquely
determined integer $\rho_i(t)$,
\[
 \tau_i(t)_i=t_i-c_i+a_i\rho_i(t).
\]
Consequently,
\begin{equation}\label{eq:H1-q-congruence}
 q_i(t)
 =\frac{1-b_it_i+b_i\tau_i(t)_i}{a_i}
 =-h_i+b_i\rho_i(t),
 \qquad
 q_i(t)\equiv-h_i\pmod{b_i}.
\end{equation}
In particular, the residue of $q_i(t)$ modulo $b_i$ is independent of the
column index $t$.  Therefore
\[
 e^{2\pi iq_i(t)(\eta_i+s_i/b_i)}
 =e^{-2\pi ih_is_i/b_i}\,e^{2\pi iq_i(t)\eta_i}.
\]
It follows that translation of the spatial Zak variable by
$a_i^{-1}e_i$ acts on the matrix by
\begin{equation}\label{eq:H1-face-factorization}
                         M\longmapsto L_i M K_i(\eta),
\end{equation}
where $L_i$ is the diagonal row unitary
\[
 (L_i z)_s=e^{-2\pi ih_is_i/b_i}z_s,
\]
and $K_i(\eta)$ is the column permutation $t\mapsto\tau_i(t)$ followed by
the unimodular column phases $e^{2\pi iq_i(t)\eta_i}$.  Translation of the
frequency Zak variable by $b_i^{-1}e_i$ acts by the row permutation
\[
 (\mathcal F_iM)_{s,t}=M_{s+e_i\ ({\rm mod}\ b_i),t}.
\]
Let $\mathcal S_i(\eta)$ denote the spatial translation operation in
\eqref{eq:H1-face-factorization}.  Both $\mathcal S_i(\eta)$ and
$\mathcal F_i$ are left-right unitary operations.  They preserve rank, the
operator norm, and every singular value.

Moreover, note that translating in different orders gives the same
answer.  Translations belonging to different coordinates commute because
they change different coordinates of $s$ and $t$, and their scalar phases
depend on different variables.  For the two translations in the same
coordinate, the required cocycle identity is
\begin{equation}\label{eq:H1-face-cocycle}
 \mathcal S_i(\eta+b_i^{-1}e_i)\,\mathcal F_i
 =\mathcal F_i\,\mathcal S_i(\eta).
\end{equation}
Indeed, if $s_i<b_i-1$, both sides multiply the same entry by
$e^{2\pi iq_i(t)(\eta_i+(s_i+1)/b_i)}$.  If $s_i=b_i-1$, one side contains
$e^{2\pi iq_i(t)(\eta_i+1)}$ and the other contains
$e^{2\pi iq_i(t)\eta_i}$; these are equal because $q_i(t)\in\Z$.
Thus the translation rules are consistent on every intersection of
translated sides of the fine rectangle.  Hence they define an unambiguous
$\Gamma$-equivariant transport of matrix values on the periodic complex
$\mathcal T$.

\medskip
\noindent\textbf{Step 2: the rank-deficient set has codimension at least four.}
Let
\[
 V=\C^{R\times N}\cong\R^{2RN},
 \qquad
 \Sigma=\{A\in V:\rank A<R\}.
\]
We first note that $\Sigma$ is covered by low-dimensional
Lipschitz parametrizations.  Fix $j\in\{1,\ldots,R\}$ and consider a matrix
of rank exactly $R-j$.  Some $(R-j)\times(R-j)$ minor is invertible.  There
are only finitely many choices of its rows and columns.  After the
corresponding row and column permutations, write the matrix as
\[
 A=\begin{pmatrix}P&Q\\ S&T\end{pmatrix},
 \qquad
 P\in\C^{(R-j)\times(R-j)},\quad P\text{ invertible}.
\]
Block Gaussian elimination gives
\[
 \rank A=R-j
 \quad\Longleftrightarrow\quad
                         T=SP^{-1}Q.
\]
Thus $P,Q,S$ are free and $T$ is determined.  Their total complex dimension
is
\[
 (R-j)^2+(R-j)(N-R+j)+j(R-j)
   =RN-j(N-R+j).
\]
The corresponding real dimension is therefore
$2RN-2j(N-R+j)$.  The largest such dimension occurs for $j=1$; since
$N>R$,
\begin{equation}\label{eq:H1-singular-codimension}
 \dim_{\R}\Sigma
 \le 2RN-2(N-R+1)
 \le 2RN-4.
\end{equation}

Now, for each choice of the invertible minor and each integer $m\ge1$, restrict the free
blocks to
\[
 \|P\|,\|Q\|,\|S\|\le m,
 \qquad
 \|P^{-1}\|\le m.
\]
On this bounded set the map
\[
 (P,Q,S)\longmapsto
 \begin{pmatrix}P&Q\\ S&SP^{-1}Q\end{pmatrix}
\]
is Lipschitz.  To see this, use
\[
 P^{-1}-{P'}^{-1}=P^{-1}(P'-P){P'}^{-1},
\]
so inversion is Lipschitz when both inverse norms are bounded, and then use
the usual product estimates for $SP^{-1}Q$.  Every matrix of rank $R-j$
belongs to one of these sets: choose an invertible minor and then choose
$m$ larger than the norms of all free blocks and of the inverse minor.
There are finitely many minors, finitely many possible values of $j$, and
countably many $m$.  Hence $\Sigma$ is covered by countably many Lipschitz
images of bounded subsets of Euclidean spaces of dimension at most
$2RN-4$.

We shall use the following elementary null-set fact.  If $E\subset\R^q$ is
bounded, $f:E\to\R^m$ is Lipschitz, and $q<m$, then $f(E)$ has zero
$m$-dimensional Lebesgue measure.  To see this, enclose $E$ in a
fixed cube and subdivide that cube into cubes of side $\delta$.  At most
$C\delta^{-q}$ of them meet $E$.  If $f$ has Lipschitz constant $L$, the
image of each such small cube has diameter at most $L\sqrt q\,\delta$ and
is contained in an $m$-cube of side $2L\sqrt q\,\delta$.  Hence the
$m$-dimensional outer measure of $f(E)$ is at most
\[
 C\delta^{-q}(2L\sqrt q\,\delta)^m
 =C'\delta^{m-q}.
\]
Letting $\delta\downarrow0$ gives measure zero.  A countable union of such
images is again null.

\medskip
\noindent\textbf{Step 3: construction on the vertices, edges, and triangles.}
Recall that $\mathcal T^{(0)}$, $\mathcal T^{(1)}$, and
$\mathcal T^{(2)}$ denote respectively the vertices, the union of the
vertices and edges, and the union of all vertices, edges, and triangles of
the periodic triangulation.  We define the matrix field orbit by orbit.
Choose an arbitrary full-row-rank matrix at one representative of each
vertex orbit and transport it to every translated vertex by the translation
rules from Step~1.

We next fill one representative of each edge orbit.  Let
$A,B\in V\setminus\Sigma$ be its endpoint values.  We claim that there is a
matrix $C\in V$ such that both segments $[A,C]$ and $[C,B]$ avoid $\Sigma$.
For the first segment, a bad choice of $C$ satisfies
\[
 (1-t)A+tC=Y\in\Sigma
 \quad\text{for some }0<t\le1,
\]
and hence
\begin{equation}\label{eq:H1-bad-edge-center}
 C=\frac{Y-(1-t)A}{t}.
\end{equation}
Take one of the bounded Lipschitz parametrizations of $\Sigma$ from Step~2
and restrict $t$ to $[1/q,1]$.  The right-hand side of
\eqref{eq:H1-bad-edge-center} is then a Lipschitz image of a set of real
dimension at most
\[
 (2RN-4)+1=2RN-3<2RN.
\]
It is therefore a null set in $V$.  Taking the countable union over all
parametrizations and all $q\in\N$ still gives a null set.  Reversing the
parameter on the second segment gives the same conclusion for the bad
choices of $C$ for $[C,B]$.  We choose $C$ outside the union of the two null
sets.  The broken segment $A\to C\to B$ is then a piecewise-linear path of
full-row-rank matrices.  Transport it to every translated edge and repeat
this for the finitely many edge orbits.

Now let $\sigma$ be a representative triangle.  The already defined values
on its three edges form a piecewise-linear loop
\[
 L:S^1\longrightarrow V\setminus\Sigma.
\]
For $C\in V$, form the cone
\begin{equation}\label{eq:H1-two-cell-cone}
 H_C(r,\theta)=(1-r)C+rL(\theta),
 \qquad 0\le r\le1.
\end{equation}
If this cone meets $\Sigma$ at a point with $r<1$, then
\begin{equation}\label{eq:H1-bad-cone-center}
 C=\frac{Y-rL(\theta)}{1-r}
 \qquad\text{for some }Y\in\Sigma.
\end{equation}
On a bounded Lipschitz parametrization of $\Sigma$, on one of the finitely
many linear arcs of $L$, and for $0\le r\le1-1/q$, the right-hand side is a
Lipschitz image of a set of dimension at most
\[
 (2RN-4)+1+1=2RN-2<2RN.
\]
Thus the bad cone centers form a countable union of null sets.  Choose $C$
outside this union.  Then the cone avoids $\Sigma$ for $r<1$, and at $r=1$
it equals the already full-rank loop $L$.  Hence it is a Lipschitz
full-row-rank filling of the triangle.  Transport this filling to all
translated triangles and repeat for the finitely many triangle orbits.

Now let $\mathcal T_{\mathcal Q}^{(2)}$ be the finite union of
all closed vertices, edges, and triangles of $\mathcal T$ that meet the
closed fundamental rectangle $\overline{\mathcal Q}$.  The matrix field is
continuous on this finite compact complex: it is continuous on every closed
simplex, and the definitions agree on common faces.  It has full row rank at
every point.  Therefore the continuous function
$z\mapsto\sigma_{\min}(M(z))$ has a strictly positive minimum
$\mu>0$ on $\mathcal T_{\mathcal Q}^{(2)}$, while
$z\mapsto\|M(z)\|$ has a finite maximum $K$.  Since the eigenvalues of
$MM^*$ are the squares of the singular values of $M$, we obtain
\begin{equation}\label{eq:H1-two-skeleton-bounds}
 \mu^2 I_R\le MM^*\le K^2 I_R
 \quad\text{on }\mathcal T_{\mathcal Q}^{(2)}.
\end{equation}
Set $\alpha=\mu^2$ and $\beta=K^2$.  Every translated simplex has the same
bounds, because the translation operations from Step~1 are unitary and
therefore leave all singular values unchanged.

\medskip
\noindent\textbf{Step 4: radial extension through the higher skeleta.}
Suppose that, for some $k\ge3$, the field has been constructed on the
$(k-1)$-skeleton, belongs to $W^{1,2}$ on every face, has matching traces on
common faces, and satisfies the bounds
$\alpha I_R\le MM^*\le\beta I_R$ almost everywhere.  Let $\sigma$ be a
representative $k$-simplex.  Its boundary is a finite union of
$(k-1)$-simplices.  Because the traces agree on their common subfaces, the
piecewise boundary map is a single element of
$W^{1,2}(\partial\sigma;V)$; this follows, for example, by integrating by
parts on the faces, where the contributions from each common subface cancel
in pairs.

Choose a bi-Lipschitz map from the unit ball $B^k$ onto $\sigma$ that maps
$S^{k-1}$ onto $\partial\sigma$, and let
$h\in W^{1,2}(S^{k-1};V)$ denote the transported boundary map.  We first
assume that $h$ is smooth.  For $x=r\theta$, with $0<r<1$ and
$\theta\in S^{k-1}$, define
\begin{equation}\label{eq:H1-radial-extension}
 H(r\theta)=h(\theta).
\end{equation}
The derivative in the radial direction is zero.  If
$\tau\in T_\theta S^{k-1}$ is a unit tangent vector, then the corresponding
unit tangent direction at $r\theta$ is obtained by scaling the spherical
direction by $r$, and the chain rule gives
\[
 D_\tau H(r\theta)=r^{-1}D_\tau h(\theta).
\]
Consequently,
\[
 |\nabla H(r\theta)|^2
 =r^{-2}|\nabla_{S^{k-1}}h(\theta)|^2.
\]
Since the polar-coordinate volume element is
$r^{k-1}\,dr\,d\theta$, we obtain the exact energy calculation
\begin{equation}\label{eq:H1-radial-energy}
 \begin{aligned}
 \int_{B^k}|\nabla H(x)|^2\,dx
 &=\int_0^1\int_{S^{k-1}}
      r^{k-1}r^{-2}|\nabla_{S^{k-1}}h(\theta)|^2
      \,d\theta\,dr\\
 &=\left(\int_0^1r^{k-3}\,dr\right)
      \int_{S^{k-1}}|\nabla_{S^{k-1}}h|^2\,d\theta\\
 &=\frac1{k-2}
      \int_{S^{k-1}}|\nabla_{S^{k-1}}h|^2\,d\theta<\infty.
 \end{aligned}
\end{equation}
Also,
\[
 \int_{B^k}|H|^2
 =\frac1k\int_{S^{k-1}}|h|^2,
\]
so the radial extension is bounded from
$W^{1,2}(S^{k-1};V)$ to $W^{1,2}(B^k;V)$ when $k\ge3$.  The same conclusion
for general $h\in W^{1,2}$ follows by approximating $h$ in
$W^{1,2}(S^{k-1})$ by smooth maps and using the two displayed bounds.

For clarity, we also verify directly that the weak derivative has no extra
term supported at the center.  For a smooth compactly supported test
function $\phi$ and $0<\delta<1$, integrate by parts on
$B^k\setminus\overline{B_\delta}$.  Besides the integral containing the
classical derivative of $H$, the only new term is the inner boundary term
\[
 \int_{|x|=\delta}H(x)\phi(x)n_j\,dS(x).
\]
The constructed boundary maps, and hence their radial extensions, are
bounded.  Therefore the absolute value of this term is at most
\[
 \|H\|_\infty\|\phi\|_\infty |S^{k-1}|\delta^{k-1},
\]
which tends to zero as $\delta\downarrow0$.  Thus the weak derivative on the
whole ball is exactly the $L^2$ function computed above; there is no Dirac
mass or other distribution supported at the center.

For every $r>0$, $H(r\theta)$ is exactly a boundary value.  Hence the
matrix bounds $\alpha I_R\le HH^*\le\beta I_R$ hold almost everywhere in
the ball.  The value assigned at the center is immaterial.  Transport the
extension back to $\sigma$ and then to all its translates.  The translation
operations are unitary and have smooth phase factors with bounded
derivatives on the compact fundamental rectangle, so they preserve the
singular-value bounds and local $W^{1,2}$ regularity.  Proceeding through all
orbits of $k$-simplices and then through $k=3,\ldots,2d$ gives
\begin{equation}\label{eq:H1-matrix-bounds}
 M\in W^{1,2}(\mathcal Q;\C^{R\times N})\cap L^\infty(\mathcal Q),
 \qquad
 \alpha I_R\le MM^*\le\beta I_R
 \quad\text{a.e. on }\mathcal Q.
\end{equation}
The simplexwise functions form one global $W^{1,2}$ function because their
traces agree.  Equivalently, an integration-by-parts calculation over all
simplices has cancelling contributions on every interior face.

\medskip
\noindent\textbf{Step 5: row normalization, scalar reassembly, and the window.}
For a full-row-rank matrix $A$, define
\begin{equation}\label{eq:H1-row-normalization}
 \mathcal P(A)=\sqrt R\,(AA^*)^{-1/2}A.
\end{equation}
On the set $\alpha I_R\le AA^*\le\beta I_R$, the inverse square-root is a
smooth matrix function with bounded derivative.  The Sobolev chain rule
therefore gives
\[
 \widetilde M:=\mathcal P(M)
 \in W^{1,2}(\mathcal Q;\C^{R\times N})\cap L^\infty(\mathcal Q),
\]
and
\begin{equation}\label{eq:H1-normalized-identity}
                  \widetilde M\widetilde M^*=RI_R
                  \quad\text{a.e. on }\mathcal Q.
\end{equation}
This normalization commutes with every unitary
translation operation.  If $U$ and $V$ are unitary matrices of the
appropriate sizes, then
\[
 (UAV)(UAV)^*=UAA^*U^*.
\]
For every positive definite matrix $B$, unitary invariance of the functional
calculus gives
\[
 (UBU^*)^{-1/2}=UB^{-1/2}U^*.
\]
Applying this with $B=AA^*$ yields
\[
 \begin{aligned}
 \mathcal P(UAV)
 &=\sqrt R\,(UAA^*U^*)^{-1/2}UAV\\
 &=\sqrt R\,U(AA^*)^{-1/2}U^*UAV\\
 &=U\mathcal P(A)V.
 \end{aligned}
\]
Thus $\widetilde M$ obeys exactly the same spatial and frequency translation
rules as $M$.

We conclude by making explicit the reassembly of Zak image.  Put
\[
 \mathcal Q_x=\prod_i[0,a_i^{-1}),
 \qquad
 \mathcal Q_\eta=\prod_i[0,b_i^{-1}),
\]
and, for $s\in\mathcal I_{\mathsf B}$ and
$t\in\mathcal I_{\mathsf A}$, consider
\[
 \mathcal R_{s,t}
   =(-Dt+\mathcal Q_x)\times
     (\mathsf B^{-1}s+\mathcal Q_\eta).
\]
On $\mathcal R_{s,t}$ define
\begin{equation}\label{eq:H1-local-scalar-reassembly}
 F(x,\omega)
 =\widetilde M_{s,t}
    (x+Dt,\omega-\mathsf B^{-1}s).
\end{equation}
For each coordinate $i$, multiplication by $b_i$ permutes the residue
classes modulo $a_i$.  Therefore the intervals
\[
 -\frac{b_it_i}{a_i}+[0,a_i^{-1}),
 \qquad 0\le t_i<a_i,
\]
are, modulo integer translation, exactly the $a_i$ consecutive subintervals
of a unit interval.  Likewise,
$s_i/b_i+[0,b_i^{-1})$, $0\le s_i<b_i$, partition a unit interval.  Hence
the interiors of the $RN$ rectangles $\mathcal R_{s,t}$ are disjoint modulo
$\Z^{2d}$ and their integer translates cover $\R^{2d}$.

We check that the local definitions have matching traces.  Across a common
boundary between adjacent frequency subrectangles, the row index increases
by one.  If it wraps from $b_i-1$ to $0$, the frequency argument changes by
the integer vector $e_i$, under which a scalar Zak function is periodic.
This is exactly the frequency translation identity from Step~1.  Across a
common boundary between adjacent spatial subrectangles, the neighboring
column is $\tau_i(t)$, and
\begin{equation}\label{eq:H1-reassembly-identity}
 u+a_i^{-1}e_i-Dt
   =u-D\tau_i(t)+q_i(t)e_i.
\end{equation}
Thus the two scalar arguments differ by the integer spatial translation
$q_i(t)e_i$.  The phase in the spatial translation identity from Step~1 is
exactly the corresponding scalar Zak phase
$e^{2\pi iq_i(t)\omega_i}$.  Hence all traces match with the required Zak
quasiperiodicity.  Since the local pieces belong to $W^{1,2}$ and their
traces agree, they glue to a scalar function
\[
 F\in W^{1,2}([0,1]^{2d})\cap L^\infty([0,1]^{2d})
\]
whose extension satisfies
\[
 F(x+n,\omega+m)=e^{2\pi i\langle n,\omega\rangle}F(x,\omega),
 \qquad n,m\in\Z^d.
\]
The fine matrix associated with $F$ is $\widetilde M$.  Iterating the
translation identities extends
\eqref{eq:H1-normalized-identity} from $\mathcal Q_x$ to the full spatial
range $[0,1)^d$; every intervening operation is unitary.  Consequently,
\[
 \mathcal Z_g(x,\eta)\mathcal Z_g(x,\eta)^*=RI_R
 \quad\text{for almost every }(x,\eta)\in\mathcal Q_{\mathsf B},
\]
where $g=Z^{-1}F$.

Finally, Proposition~\ref{prop:zak-standard}(viii) and the identities
\eqref{eq:zak-H1-identities} imply
$g\in\mathbb H^1(\R^d)$.  Proposition~\ref{prop:zz-matrix-criterion} and
the last matrix identity show that $\G(g,\Lambda_D)$ is a Parseval frame.
This proves the lemma.
\end{proof}

\section{Transfer from the diagonal model to arbitrary lattices}
\label{sec:gabor-necessity}

We now transfer the diagonal conclusions to arbitrary phase-space lattices.
The finite-uncertainty theorem uses the same metaplectic reduction, so we
record that transfer first and then turn to the continuous and smooth regimes.

\subsection{Metaplectic transfer in the rational case: proof of Theorem~\ref{thm:full-multivariate-balian-low}, \textup{(b)}$\Rightarrow$\textup{(a)}}

Let $\Lambda$ now be symplectically rational with
$\covol(\Lambda)<1$.  By
Theorem~\ref{thm:diagonal-symplectic-normal-form}, there are
$S\in\Sp(2d,\R)$ and a diagonal rational lattice $\Lambda_D$ with
$\Lambda=S\Lambda_D$ and the same covolume.  Thus $N>R$ for the diagonal
data, and Lemma~\ref{lem:finite-uncertainty-diagonal} gives a Parseval-frame
window $g_D\in\mathbb H^1(\R^d)$ on $\Lambda_D$.

The space $\mathbb H^1(\R^d)$ is invariant under every metaplectic operator.
Indeed, let $X_jf=x_jf$ and $P_jf=(2\pi i)^{-1}\partial_jf$.  Differentiating
the metaplectic covariance relation shows that each of
$X_j\mu(S)f$ and $P_j\mu(S)f$ is a finite linear combination of the
functions $\mu(S)X_kf$ and $\mu(S)P_kf$.  Hence
$f\in\mathbb H^1$ implies $\mu(S)f\in\mathbb H^1$.  Metaplectic covariance
also preserves the Parseval property.  Therefore
\[
                         g=\mu(S)g_D
\]
belongs to $\mathbb H^1(\R^d)$ and generates a Parseval frame on $\Lambda$.
This completes the proof of
Theorem~\ref{thm:full-multivariate-balian-low}.

\subsection{Extension of the necessity result (proof of Theorem~\ref{thm:main-gabor}, \textup{(b)}$\Rightarrow$\textup{(c)})}

Assume first that $\Lambda$ is symplectically rational and that
$\G(g,\Lambda)$ is a frame for some $g\in\Szero(\R^d)$.  By
Theorem~\ref{thm:diagonal-symplectic-normal-form}, there are
$S\in\Sp(2d,\R)$ and a diagonal $D$ as in
\eqref{eq:normal-form-lattice} such that
\begin{equation}\label{eq:rational-diagonal-reduction}
                  \Lambda=S(D\Z^d\times\Z^d).
\end{equation}
The reduced window $g_0=\mu(S)^{-1}g$ belongs to $\Szero(\R^d)$ by
metaplectic invariance, hence belongs to $\CZ(\R^d)$, and
\[
                 \G(g_0,D\Z^d\times\Z^d)
\]
is a frame.  Theorem~\ref{thm:diagonal-one-window} gives
\begin{equation}\label{eq:rational-NR-gap}
                              N\ge R+d.
\end{equation}
By \eqref{eq:normal-form-invariants} and symplectic invariance, this is
\begin{equation}\label{eq:rational-SG-necessary}
                   \nu(\Lambda)\ge\nu(\Lambda^\circ)+d.
\end{equation}
Thus the symplectically rational inequality in~\eqref{eq:SG} is necessary.

\begin{theorem}
\label{thm:rational-zak-obstruction}
Let $\Lambda\subset\R^{2d}$ be symplectically rational and choose a
diagonalization
\[
                         \Lambda=S(D\Z^d\times\Z^d).
\]
If $\G(g,\Lambda)$ is a frame and
$Z(\mu(S)^{-1}g)$ is continuous, then
\[
                   \nu(\Lambda)\ge\nu(\Lambda^\circ)+d.
\]
For the diagonal representative this is
\[
                   \prod_{i=1}^d a_i
                   \ge \prod_{i=1}^d b_i+d.
\]
\end{theorem}

\begin{proof}
The transformed window belongs to $\CZ(\R^d)$ and generates a frame on the
diagonal representative.  Apply Theorem~\ref{thm:diagonal-one-window} and
then Theorem~\ref{thm:diagonal-symplectic-normal-form}.
\end{proof}

\begin{corollary}
\label{cor:rational-S0-obstruction}
Let $\Lambda\subset\R^{2d}$ be symplectically rational.  If
$g\in\Szero(\R^d)$ and $\G(g,\Lambda)$ is a frame, then
\[
 \covol(\Lambda)<1,
 \qquad
 \nu(\Lambda)\ge\nu(\Lambda^\circ)+d.
\]
\end{corollary}

\begin{proof}
Choose a diagonalization
\[
             \Lambda=S(D\Z^d\times\Z^d).
\]
Metaplectic invariance of the Feichtinger algebra gives
$\mu(S)^{-1}g\in\Szero(\R^d)$.  By Proposition~\ref{prop:zak-standard}(v), its Zak transform is continuous.
Theorem~\ref{thm:rational-zak-obstruction} gives the $\SG$ inequality.
\end{proof}

\begin{remark}
Theorem~\ref{thm:rational-zak-obstruction} concerns the Zak transform in the
coordinates adapted to a chosen diagonal representative.  Its conclusion is
that every frame window has discontinuous \emph{adapted} Zak transform when
the symplectically rational $\SG$ inequality fails.  This does not imply that
the standard Zak transform $Zg$ is discontinuous: continuity of
$Z(\mu(S)^{-1}g)$ and continuity of $Zg$ are not equivalent for a general
metaplectic change of coordinates.  The diagonal statement in
Theorem~\ref{thm:intro-diagonal-zak-obstruction} therefore uses the larger
class $\CZ$.  The coordinate-free necessity implication
\textup{(b)}$\Rightarrow$\textup{(c)} in
Theorem~\ref{thm:main-gabor} uses the smaller but metaplectically invariant
Wiener amalgam $W(\mathcal F L^1,L^1)=\Szero$.
\end{remark}

Suppose next that $\Lambda$ is not symplectically rational and that
$\G(g,\Lambda)$ is a frame with $g\in\Szero(\R^d)$.  The density theorem
gives $\covol(\Lambda)\le1$.  At equality the frame is a Riesz basis, and
the critical-density amalgam Balian--Low theorem excludes windows in the
Feichtinger algebra; see
\cite{BenedettoHeilWalnut1995,NitzanOlsen2013}.  Hence
\begin{equation}\label{eq:nonrational-necessity}
                         \covol(\Lambda)<1.
\end{equation}
This proves the necessity of the remaining inequality in~\eqref{eq:SG} for
$S_0$ windows, and hence in particular for Schwartz windows.

\subsection{Existence under the condition~\eqref{eq:SG} (proof of Theorem~\ref{thm:main-gabor}, \textup{(c)}$\Rightarrow$\textup{(a)}, and Theorem~\ref{thm:strengthened-one-window}\textup{(a)--(b)})}
\label{sec:general-existence}

Assume first that $\Lambda$ is symplectically rational and that the inequalities in~\eqref{eq:SG} hold.
Write
\[
                         \Lambda=S(D\Z^d\times\Z^d)
\]
as in Theorem~\ref{thm:diagonal-symplectic-normal-form}.  Then
$N\ge R+d$, so Theorem~\ref{thm:diagonal-positive} provides a
$g_0\in C_c^\infty(\R^d)$ whose Gabor system on the diagonal lattice is
Parseval.  Metaplectic covariance shows that
\[
                         g=\mu(S)g_0\in\Sclass(\R^d)
\]
generates a Parseval frame on $\Lambda$.

For a separable lattice $\Lambda=\Gamma\times\Phi$, one obtains more.
Proposition~\ref{prop:TeP-SG-relation} identifies the inequalities in~\eqref{eq:SG} with those in~\eqref{eq:TeP} for the
pair $(\Gamma,\Phi^*)$.  The construction of Part~A and the mollification
and energy normalization carried out in
Section~\ref{sec:gabor-sufficiency} use only this pair.  Therefore the same
argument gives the following statement, whether or not the separable lattice
is symplectically rational.

\begin{theorem}
\label{thm:smooth-gabor-positive}
Let $\Gamma,\Phi\subset\R^d$ be full-rank lattices.  Suppose
\[
 \covol(\Gamma)<\covol(\Phi^*)=\covol(\Phi)^{-1}
\]
and, if $\Gamma+\Phi^*$ is discrete, suppose also that
\[
 [\Gamma+\Phi^*:\Phi^*]\ge[\Gamma+\Phi^*:\Gamma]+d.
\]
Then there exists a nonnegative $g\in C_c^\infty(\R^d)$ such that
$\G(g,\Gamma\times\Phi)$ is a Parseval frame and
\begin{align}
 \supp g\cap(\supp g+\xi^*)&=\varnothing,
 &&\xi^*\in\Phi^*\setminus\{0\},
 \label{eq:general-separable-support}\\
 \sum_{\gamma\in\Gamma}|g(x+\gamma)|^2&=\covol(\Phi),
 &&x\in\R^d.
 \label{eq:general-separable-periodization}
\end{align}
Equivalently, $\covol(\Phi)^{-1}|g|^2$ is a smooth
$\Gamma$-partition of unity.  In particular,
\[
 g(x)g(x-\xi^*)=0,
 \qquad x\in\R^d,\quad \xi^*\in\Phi^*\setminus\{0\}.
\]
Alternatively, one may choose $g\in\Sclass(\R^d)$ so that
$\widehat g\in C_c^\infty(\R^d)$ is nonnegative and satisfies
\begin{align}
 \supp\widehat g\cap(\supp\widehat g+\gamma^*)&=\varnothing,
 &&\gamma^*\in\Gamma^*\setminus\{0\},
 \label{eq:general-fourier-support}\\
 \sum_{\phi\in\Phi}|\widehat g(\omega+\phi)|^2&=\covol(\Gamma),
 &&\omega\in\R^d.
 \label{eq:general-fourier-periodization}
\end{align}
Equivalently, $\covol(\Gamma)^{-1}|\widehat g|^2$ is a smooth
$\Phi$-partition of unity.
In the time-side statement, one may alternatively choose a nonnegative
$g_{\rm pa}\in C_c(\R^d)$ such that
$\covol(\Phi)^{-1}g_{\rm pa}^2$ is a continuous piecewise-affine
$\Gamma$-partition of unity.  In the Fourier-side statement one may
alternatively arrange that $\widehat g\in C_c(\R^d)$ and
$\covol(\Gamma)^{-1}|\widehat g|^2$ is a continuous piecewise-affine
$\Phi$-partition of unity.
\end{theorem}

\begin{proof}
Proposition~\ref{prop:TeP-SG-relation} shows that the pair
$(\Gamma,\Phi^*)$ satisfies the two inequalities in~\eqref{eq:TeP}.
Corollary~\ref{cor:global-partitions} gives a nonnegative compactly
supported piecewise-affine $\Gamma$-partition of unity $\varphi$ with an
actual $\Phi^*$-support buffer.  Hence
\[
 g_{\rm pa}=\sqrt{\covol(\Phi)\,\varphi}
\]
is compactly supported and
$\covol(\Phi)^{-1}g_{\rm pa}^2=\varphi$ is the piecewise-affine
$\Gamma$-partition of unity.  It satisfies
\eqref{eq:general-separable-support}--\eqref{eq:general-separable-periodization}.
The adjoint-lattice argument below proves that its Gabor system is Parseval.
Mollification and energy normalization as in
Section~\ref{sec:gabor-sufficiency} instead give a nonnegative
$g\in C_c^\infty(\R^d)$ satisfying the same identities.

Let
\[
 h_0=\covol(\Gamma\times\Phi)^{-1/2}g.
\]
Then
\[
 \sum_{\gamma\in\Gamma}|h_0(x+\gamma)|^2
 =\frac{1}{\covol(\Gamma)},
 \qquad x\in\R^d.
\]
The adjoint lattice is
$(\Gamma\times\Phi)^\circ=\Phi^*\times\Gamma^*$.
Distinct translations of $h_0$ by elements of $\Phi^*$ have disjoint
supports.  For a fixed translation and
$\gamma^*,\gamma^{*\prime}\in\Gamma^*$, put
$\eta=\gamma^{*\prime}-\gamma^*$.  Periodization over a fundamental
domain $F_\Gamma$ gives
\[
 \begin{aligned}
 \ip{h_0}{M_\eta h_0}
 &=\int_{F_\Gamma}
   \sum_{\gamma\in\Gamma}|h_0(x+\gamma)|^2
   e^{2\pi i\ip{\eta}{x+\gamma}}\,dx\\
 &=\frac1{\covol(\Gamma)}
   \int_{F_\Gamma}e^{2\pi i\ip{\eta}{x}}\,dx
 =\delta_{\gamma^*,\gamma^{*\prime}}.
 \end{aligned}
\]
Thus
$\G(h_0,\Phi^*\times\Gamma^*)$ is an orthonormal system.  The normalized
duality principle, Theorem~\ref{thm:gabor-duality}, now implies that
$\G(g,\Gamma\times\Phi)$ is a Parseval frame.

For the Fourier-side alternative, apply the result just proved to the
separable lattice
\[
 J(\Gamma\times\Phi)=\Phi\times(-\Gamma).
\]
The hypotheses are unchanged because $J$ is symplectic.  We obtain a
nonnegative $h\in C_c^\infty(\R^d)$ such that
$\G(h,\Phi\times(-\Gamma))$ is Parseval and
\begin{align*}
 \supp h\cap(\supp h+\gamma^*)&=\varnothing,
 &&\gamma^*\in\Gamma^*\setminus\{0\},\\
 \sum_{\phi\in\Phi}|h(\omega+\phi)|^2&=\covol(\Gamma),
 &&\omega\in\R^d.
\end{align*}
Set $g=\mu(-J)h$.  Then $\widehat g=\mu(J)g=h$, and metaplectic
covariance transfers the Parseval frame from
$\Phi\times(-\Gamma)$ back to $\Gamma\times\Phi$.  This proves
\eqref{eq:general-fourier-support}--\eqref{eq:general-fourier-periodization}.
If the unsmoothed precursor is used on the rotated lattice instead, the same
argument gives the stated alternative with $|\widehat g|^2$ piecewise
affine.
\end{proof}

\begin{proposition}\label{prop:symplectic-lattice-support-refinement}
Let $a>0$, let
\[
 S=\begin{pmatrix}A&B\\C&D\end{pmatrix}\in\Sp(2d,\R),
 \qquad \Lambda=aS\Z^{2d},
\]
and assume the two inequalities in~\eqref{eq:SG} hold for $\Lambda$.  If
$A=0$ or $B=0$, then $\Lambda$ admits a Parseval frame window
$g\in C_c^\infty(\R^d)$; alternatively, it admits a compactly supported
continuous Parseval window with $|g|^2$ piecewise affine.  If $C=0$ or
$D=0$, then $\Lambda$ admits a Parseval frame window $g$ such that
$\widehat g\in C_c^\infty(\R^d)$; alternatively, one may arrange that
$\widehat g$ is compactly supported and continuous and
$|\widehat g|^2$ is piecewise affine.
\end{proposition}

\begin{proof}
The two inequalities in~\eqref{eq:SG} are symplectically invariant, so the
separable lattice $a\Z^d\times a\Z^d$ satisfies the hypotheses of
Theorem~\ref{thm:smooth-gabor-positive}.  Choose
$h\in C_c^\infty(\R^d)$ so that $\G(h,a\Z^{2d})$ is Parseval.  Independently,
choose a compactly supported continuous Parseval window $h_{\rm pa}$ such
that $|h_{\rm pa}|^2$ is piecewise affine.

If $B=0$, put $g=\mu(S)h$.  Metaplectic covariance gives a Parseval frame
on $\Lambda$, and
Equation~\eqref{eq:metaplectic-lower-triangular-formula} gives
$g\in C_c^\infty(\R^d)$.  Applying the same operator to $h_{\rm pa}$ gives a
compactly supported Parseval window $g_{\rm pa}$ with
\[
 |g_{\rm pa}(t)|^2
 =|\det A|^{-1}|h_{\rm pa}(A^{-1}t)|^2,
\]
which is piecewise affine.  If $A=0$, use instead
\[
 \Lambda=aSJ^{-1}\Z^{2d}.
\]
The upper-right block of $SJ^{-1}$ is $-A=0$, so
$g=\mu(SJ^{-1})h$ is compactly supported and generates a Parseval frame on
$\Lambda$; applying $\mu(SJ^{-1})$ to $h_{\rm pa}$ gives the corresponding
modulus-square piecewise-affine alternative.

If $D=0$, put $g=\mu(S)h$.  Since the upper-right block of $JS$ is
$D=0$,
\[
 \widehat g=\mu(J)\mu(S)h=c\,\mu(JS)h,
 \qquad |c|=1,
\]
belongs to $C_c^\infty(\R^d)$.  Applying the same expression to
$h_{\rm pa}$ shows that $|\widehat g|^2$ is piecewise affine.  Finally, if
$C=0$, use the representation $\Lambda=aSJ^{-1}\Z^{2d}$ and put
$g=\mu(SJ^{-1})h$.  The upper-right block of $JSJ^{-1}$ is $-C=0$, and the
same argument gives $\widehat g\in C_c^\infty(\R^d)$, with the analogous
modulus-square piecewise-affine alternative obtained from $h_{\rm pa}$.
\end{proof}

Finally, assume that $\Lambda$ is not symplectically rational and
$\covol(\Lambda)<1$.  In the notation of Enstad--Thiel--Vilalta the
correction term is zero, so their Theorem~C
\cite{EnstadThielVilalta2025} gives a one-window Schwartz Gabor frame on
$\Lambda$.  If $S_g$ is its frame operator, the canonical Parseval window is
$S_g^{-1/2}g$.  Spectral invariance of the smooth twisted group algebra keeps
this window in the Schwartz class; see
\cite{Luef2009Projections,JakobsenLuef2020}.

The necessity and existence arguments prove
Theorem~\ref{thm:main-gabor}: the implication
\textup{(c)}$\Rightarrow$\textup{(a)} is the existence result just
established, \textup{(a)}$\Rightarrow$\textup{(b)} follows from
$\Sclass\subset\Szero$, and
\textup{(b)}$\Rightarrow$\textup{(c)} follows from
Corollary~\ref{cor:rational-S0-obstruction} together with the strict-density
argument for lattices that are not symplectically rational.  They also prove
Theorem~\ref{thm:intro-diagonal-zak-obstruction},
Theorem~\ref{thm:strengthened-one-window}\textup{(a)} through
Theorem~\ref{thm:smooth-gabor-positive}, and
Theorem~\ref{thm:strengthened-one-window}\textup{(b)} through
Proposition~\ref{prop:symplectic-lattice-support-refinement}.

\section{The multiple-window extension}
\label{subsec:multiwindow-adjustment}

The paper remains focused on one-window frames.  We record the multiple-window
extension after completing the one-window theorem because its necessity proof
uses the same diagonal matrix argument with a larger number of columns.

Multiple windows only increase the number of Zibulski--Zeevi columns.  The following is the continuous-Zak necessity statement in the diagonal coordinates.

\begin{proposition}
\label{prop:continuous-zak-multiwindow}
Retain the diagonal data in \eqref{eq:zak-D-data}, and let
$\mathbf g=(g_1,\ldots,g_q)$ with $g_j\in\CZ(\R^d)$.  If
\[
       \G(\mathbf g,D\Z^d\times\Z^d)
\]
is a frame, then
\begin{equation}\label{eq:multiwindow-diagonal-gap}
                              qN\ge R+d.
\end{equation}
\end{proposition}

\begin{proof}
For each $j$, let $\mathcal Z_{g_j}(x,\eta)\in M_{R,N}(\C)$ be the Zibulski--Zeevi matrix from
\eqref{eq:zak-H-definition}, and form the horizontal concatenation
\begin{equation}\label{eq:multiwindow-H-definition}
 \mathcal Z_{\mathbf g}(x,\eta)
   =\begin{bmatrix}
       \mathcal Z_{g_1}(x,\eta)&\cdots&\mathcal Z_{g_q}(x,\eta)
     \end{bmatrix}
   \in M_{R,qN}(\C).
\end{equation}
Summing \eqref{eq:zak-coefficient-identity} over the windows gives
\begin{align}
 &\sum_{j=1}^q\sum_{n,k\in\Z^d}
   \bigl|\ip{f}{M_nT_{Dk}g_j}\bigr|^2 \notag\\
 &\qquad=\frac1R\int_{\mathcal Q_{\mathsf B}}
       \bigl\|\mathcal Z_{\mathbf g}(x,\eta)^*\mathcal Zf(x,\eta)\bigr\|_{\C^{qN}}^2
       \,dx\,d\eta.
 \label{eq:multiwindow-coefficient-identity}
\end{align}
Consequently, a multiwindow frame with bounds $A_0,B_0$ is equivalent to
\begin{equation}\label{eq:multiwindow-matrix-bounds}
 R A_0 I_R
 \ \le\ \mathcal Z_{\mathbf g}(x,\eta)\mathcal Z_{\mathbf g}(x,\eta)^*
 \ \le\ R B_0 I_R
\end{equation}
for almost every point of the base rectangle.  Since all $Zg_j$ are continuous, the smallest eigenvalue in
\eqref{eq:multiwindow-matrix-bounds} is continuous.  The open-set argument in the proof of
Theorem~\ref{thm:diagonal-one-window} yields
\begin{equation}\label{eq:multiwindow-full-row-rank}
             \rank \mathcal Z_{\mathbf g}(x,\eta)=R
             \qquad\text{for every }(x,\eta).
\end{equation}

Every block satisfies the same boundary relations.  With the diagonal gauge
$\Gamma(x)$ from \eqref{eq:zak-Gamma-definition}, set
\begin{equation}\label{eq:multiwindow-S-definition}
 S_{\mathbf g}(x,\eta)
   =\mathcal Z_{\mathbf g}(x,\eta)^{\mathsf T}\Gamma(x)
   \in M_{qN,R}(\C).
\end{equation}
Then
\begin{equation}\label{eq:multiwindow-standard-quasi}
 S_{\mathbf g}(x+n,\eta+m)
   =e^{2\pi i\ip{n}{\eta}}S_{\mathbf g}(x,\eta),
 \qquad n,m\in\Z^d,
\end{equation}
and \eqref{eq:multiwindow-full-row-rank} says that this matrix has full
column rank everywhere.  Theorem~\ref{thm:rectangular-common-zero-gap}, with
$N$ replaced by $qN$, gives $qN\ge R+d$.
\end{proof}

\begin{proof}[Proof of Theorem~\ref{thm:main-multiwindow}]
Since $\Sclass(\R^d)\subset\Szero(\R^d)$, the implication
\textup{(a)}$\Rightarrow$\textup{(b)} is immediate.  We prove
\textup{(b)}$\Rightarrow$\textup{(c)} and
\textup{(c)}$\Rightarrow$\textup{(a)}.

We first treat symplectically rational lattices.  Suppose that
$\G(\mathbf g,\Lambda)$ is a $q$-window frame with
$g_1,\ldots,g_q\in\Szero(\R^d)$.  By
Theorem~\ref{thm:diagonal-symplectic-normal-form} and metaplectic covariance,
we may work on $D\Z^d\times\Z^d$.  Metaplectic invariance gives reduced
windows in $\Szero(\R^d)$, whose Zak transforms are continuous by Proposition~\ref{prop:zak-standard}(v), and
\[
                 N=\nu(\Lambda),
 \qquad
                 R=\nu(\Lambda^\circ).
\]
Proposition~\ref{prop:continuous-zak-multiwindow} gives
\[
             qN\ge R+d,
\]
which is equivalent to
\[
             q\,\nu(\Lambda)\ge\nu(\Lambda^\circ)+d.
\]
This proves necessity in the symplectically rational case for
Feichtinger-algebra windows, including the stronger continuous-Zak assertion in diagonal coordinates.

For sufficiency in the symplectically rational case, put
$n=\nu(\Lambda)$ and $m=\nu(\Lambda^\circ)$.  Then
\[
                         \covol(\Lambda)=\frac mn.
\]
Theorem~D of Enstad--Thiel--Vilalta
\cite{EnstadThielVilalta2025} supplies a Schwartz multiwindow frame with
\begin{align}
 k
 &=\left\lfloor
       \covol(\Lambda)+\frac{d-1}{\nu(\Lambda)}
    \right\rfloor+1 \notag\\
 &=\left\lfloor\frac{m+d-1}{n}\right\rfloor+1
  =\left\lceil\frac{m+d}{n}\right\rceil
 \label{eq:multiwindow-k-min}
\end{align}
windows.  The symplectically rational condition $\mSG_q$ is exactly $q\ge k$.

Now let $\Lambda$ be a lattice that is not symplectically rational.
Necessity begins with the standard
multiwindow density theorem,
\[
                         \covol(\Lambda)\le q.
\]
Equality cannot occur for Feichtinger-algebra windows by the multiwindow
amalgam Balian--Low theorem; see Enstad
\cite[Theorems~5.4--5.6]{Enstad2022Density} for multiwindow density,
Gerhold--Lamando--Luef
\cite[Theorem~3.38 and Corollary~3.39]{GerholdLamandoLuef2026} for the
multiwindow deformation theorem and the critical-density obstruction, and
Zibulski--Zeevi~\cite{ZibulskiZeevi1997} for the classical rational
multiwindow framework.
Hence
\[
                         \covol(\Lambda)<q,
\]
which is the corresponding branch of $\mSG_q$.

Conversely, when $\Lambda$ is not symplectically rational, Theorem~D of
Enstad--Thiel--Vilalta gives a Schwartz frame with
\[
                         k=\lfloor\covol(\Lambda)\rfloor+1
\]
windows.  The inequality $\covol(\Lambda)<q$ is equivalent to $q\ge k$.
Thus in both the symplectically rational case and the case of lattices
that are not symplectically rational, we obtain a Schwartz frame
with at most $q$ windows.  If the construction produces $k<q$ windows, set
$\ell=q-k+1$ and replace one nonzero window $h$ by $\ell$ copies of
$h/\sqrt{\ell}$.  The total number of windows is then
$(k-1)+\ell=q$, and the sum of the copied windows' frame-operator
contributions equals that of $h$.  Finally apply the inverse square root of the common frame operator to
every window.  Spectral invariance keeps all canonical Parseval windows in
the Schwartz class.  This proves sufficiency and produces exactly $q$
windows.

The necessity arguments prove \textup{(b)}$\Rightarrow$\textup{(c)}, and
the existence construction proves \textup{(c)}$\Rightarrow$\textup{(a)}.
The common frame-operator normalization above gives the tight and Parseval
versions.  The formulas for the minimum number of windows are the least
integers satisfying the two branches of $\mSG_q$.
\end{proof}

The proof shows explicitly how the single-window argument is adjusted:
replace the $N$ Zibulski--Zeevi columns by $qN$ columns and make no other change in the
symplectically rational obstruction.  All geometric constructions and all detailed Zak
computations in the body of the paper remain single-window arguments.

\section*{Acknowledgements}
Part~A is the culmination of a long-running project on the sufficiency of the
sharp condition $n\ge m+d$.  Before the present argument was found, the
authors had obtained partial sufficiency results and a complete proof of the
necessity of this condition.  A principal motivation for the project was the
construction of compactly supported smooth Gabor windows for separable
time--frequency lattices.

At an earlier stage, repeated discussions with ChatGPT~5.5~Pro did not
resolve either the smallest outstanding example
\[
 T=7\mathbb Z\times\mathbb Z,
 \qquad
 P=3\mathbb Z\times 3\mathbb Z,
\]
or the general sufficiency problem.  A later discussion with
ChatGPT~5.6~Pro led to the key idea of replacing hard configurations by soft
configurations and studying the resulting extension problem.  This
soft-configuration viewpoint was the breakthrough that led to the proof of
the sufficiency of the $(\mathrm{T}\epsilon\mathrm{P})$ condition in
Theorem~\ref{thm:main-geometric}.

After that point, the authors used ChatGPT as an exploratory and drafting
tool in the further development, checking, reorganization, and presentation
of the remaining arguments, including parts of the Gabor-frame analysis.
The authors determined the mathematical objectives and proof strategy,
directed the system toward the intended statements, checked the
computations and references, and revised the resulting text.  The authors
take full responsibility for every statement and proof in the paper; no
artificial-intelligence system is an author.

A. Caragea and G. Pfander are supported by the German Research Foundation
(DFG) Grant CA 3683/1-1.

\appendix
\section{Topology of finite soft-configuration spaces}
\label{app:continuous-soft-filling}

This appendix proves Theorem~\ref{thm:continuous-soft-filling-main}.  We
reduce the colored case to the singleton-color case by adding colors one
column at a time; a gluing lemma identifies the homotopy type after each
addition.  We then pass from equal address multiplicities to arbitrary finite
nonempty address sets.

\subsection{The uniform colored complex}

Let $A,B,C$ be finite nonempty sets, and write
\[
 a=|A|,\qquad b=|B|,\qquad c=|C|.
\]
Let $\mathcal F(A,B,C)$ be the space of all nonnegative arrays
$(x_{i,j,k})_{i\in A,j\in B,k\in C}$ satisfying
\begin{align}
 \sum_{j\in B}\sum_{k\in C}x_{i,j,k}&=1
 &&(i\in A),\label{eq:appendix-colored-row-normalization}\\
 \#\{(i,k)\in A\times C:x_{i,j,k}>0\}&\le1
 &&(j\in B).\label{eq:appendix-colored-column-exclusion}
\end{align}
Thus every row has total mass one, while every column contains at most one
positive row-color pair.  We also write
$\mathcal F_{a,b,c}$ instead of $\mathcal{F}(A, B, C)$ when we only need the cardinalities of the sets $A, B, C$.


\begin{theorem}[Colored wedge theorem]\label{thm:colored-wedge}
If $b<a$, then $\mathcal F_{a,b,c}$ is empty.  If $b=a$, then it is a
discrete space finitely many points.  If $b>a$, then
\[
 \mathcal F_{a,b,c}\simeq
 \bigvee^{N_{a,b,c}}S^{b-a},
\]
where $N_{a,b,c}<\infty$. In particular, $\mathcal F_{a,b,c}$ is $(b-a-1)$-connected when $b>a$.
\end{theorem}

For a configuration $x$ and a
row $i$, put
\[
 S_i(x)=\{j\in B:x_{i,j,k}>0\text{ for some }k\in C\}.
\]
The sets $S_i(x)$ are nonempty and pairwise disjoint. For every
$j\in S_i(x)$ there is a unique color $\kappa_i(j)\in C$ with
$x_{i,j,\kappa_i(j)}>0$.  Conversely, pairwise disjoint nonempty subsets
$S_i\subset B$ and color maps $\kappa_i:S_i\to C$ determine an open cell
\[
 e(S,\kappa)\cong
 \prod_{i\in A}\mathring\Delta(S_i).
\]
Its dimension is
\[
 \sum_{i\in A}(|S_i|-1)
 =\left|\bigcup_{i\in A}S_i\right|-a\le b-a.
\]
Hence
$\mathcal F_{a,b,c}$ is a finite regular CW complex of dimension $b-a$ when
$b\ge a$.  This also makes the first two assertions of the theorem clear:
if $b<a$, there are too few columns for the $a$ nonempty rows, while for
$b=a$ every column is used exactly once, giving finitely many hard
configurations.

\subsection{Partially colored spaces}

Order $B=\{1,\ldots,b\}$ and choose a distinguished color $k_0\in C$.
For $0\le m\le b$, let $\mathcal F_m(a,b)$ be the subspace obtained by
imposing, in addition to
\eqref{eq:appendix-colored-row-normalization}--\eqref{eq:appendix-colored-column-exclusion}, the restriction
\begin{equation}\label{eq:partial-colors}
 x_{i,j,k}=0\qquad\text{whenever }j>m\text{ and }k\ne k_0.
\end{equation}
Thus all colors are allowed in the first $m$ columns, whereas only $k_0$ is
allowed in the remaining columns.  In particular,
\[
 \mathcal F_0(a,b)
 \text{ is the singleton-color space},
 \qquad
 \mathcal F_b(a,b)=\mathcal F(A,B,C).
\]
Each $\mathcal F_m(a,b)$ is a CW subcomplex of dimension at most $b-a$.

Abrams, Gay, and Hower proved that the singleton-color space is a wedge of finitely many spheres~\cite{AbramsGayHower2013}:
\begin{equation}\label{eq:singleton-color-wedge}
 \mathcal F_0(a,b)\simeq\bigvee S^{b-a}
 \qquad(b>a).
\end{equation}
We use this result as a black box and add the colors one column at a time.
We first induct on $b$; for fixed $b$, we
then induct on $m$.  The case $b=a$ is the discrete case already discussed.
Assume now that $b>a$, that the assertion is known for all partially colored
spaces with $b-1$ columns, and that $m\ge1$.  The inner induction hypothesis
says that $\mathcal F_{m-1}(a,b)$ is a wedge of $(b-a)$-spheres.

\subsection{Adding the colors at $j=m$}

For each $k\in C$, let $X_k\subset\mathcal F_m(a,b)$ be the subspace
\begin{equation}\label{eq:Xk-colored-cover}
 X_k=\{x:x_{i,m,\ell}=0
       \text{ for every }i\in A\text{ and every }\ell\ne k\}.
\end{equation}
Thus, if column $m$ is used in a configuration belonging to $X_k$, its color
must be $k$.  If column $m$ is unused, the configuration lies in every
$X_k$.  Condition~\eqref{eq:appendix-colored-column-exclusion} therefore
gives the cover
\begin{equation}\label{eq:colored-cover}
 \mathcal F_m(a,b)=\bigcup_{k\in C}X_k.
\end{equation}

Each $X_k$ is homeomorphic to $\mathcal F_{m-1}(a,b)$: all coordinates with
$j\ne m$ are left unchanged, and at $j=m$ one merely renames the unique
permitted color from $k_0$ to $k$.  Consequently, with
$d=b-a$,
\begin{equation}\label{eq:Xk-wedge}
 X_k\simeq\bigvee S^d.
\end{equation}
If $k\ne\ell$, membership in both $X_k$ and $X_\ell$ forces column $m$ to
be unused.  Thus all pairwise intersections are the same CW subcomplex
\begin{equation}\label{eq:common-Y}
 Y=X_k\cap X_\ell
  =\{x:x_{i,m,r}=0\text{ for all }i\in A,
                         r\in C\}.
\end{equation}
Deleting the unused column $m$ gives a homeomorphism
\begin{equation}\label{eq:Y-homeomorphism}
 Y\cong\mathcal F_{m-1}(a,b-1).
\end{equation}
By the outer induction hypothesis, $Y$ is a wedge of $(d-1)$-spheres (with
the usual discrete interpretation when $d=1$), and its natural CW structure
has dimension at most $d-1$.  On the other hand, every $X_k$ is
$(d-1)$-connected.  Therefore every inclusion $Y\hookrightarrow X_k$ is
null-homotopic.  This follows, for example, by extending a null-homotopy over
the cells of $Y$ one dimension at a time.

\subsection{The gluing lemma}

The preceding cover has a particularly simple homotopy type because all of
its pairwise intersections coincide.

\begin{lemma}[Gluing lemma]\label{lem:colored-gluing}
Let $X_1,\ldots,X_r$ be path-connected CW complexes with the same pairwise
intersection $Y$, and suppose that $Y$ is a CW subcomplex of every $X_q$.
If every inclusion $Y\hookrightarrow X_q$ is null-homotopic, then
\begin{equation}\label{eq:gluing-lemma}
 \bigcup_{q=1}^{r}X_q
 \simeq
 \left(\bigvee_{q=1}^{r}X_q\right)
 \vee
 \left(\bigvee^{r-1}\Sigma Y\right),
\end{equation}
where $\Sigma Y$ is the reduced suspension after a basepoint of $Y$ has been
chosen.
\end{lemma}

\begin{proof}
We first take $r=2$.  Since $Y$ is a CW subcomplex, both inclusions
$Y\hookrightarrow X_q$ are cofibrations.  The ordinary pushout
$X_1\cup_YX_2$ therefore has the homotopy type of the corresponding homotopy
pushout, represented by the double mapping cylinder
\[
 M=X_1\sqcup(Y\times[0,1])\sqcup X_2,
 \qquad
 (y,0)\sim y\in X_1,
 \quad
 (y,1)\sim y\in X_2.
\]
Both attaching maps are null-homotopic.  By homotopy invariance of the
homotopy pushout, we may replace them by constant maps.  The two ends of the
cylinder are then attached at one point of $X_1$ and one point of $X_2$.
Choose $y_0\in Y$.  Collapsing the contractible path
$\{y_0\}\times[0,1]$ identifies those two attachment points and leaves the
reduced suspension
\[
 \Sigma Y=
 \frac{Y\times[0,1]}
 {Y\times\{0\}\cup Y\times\{1\}
  \cup\{y_0\}\times[0,1]}.
\]
Hence
\[
 X_1\cup_YX_2\simeq X_1\vee X_2\vee\Sigma Y.
\]

For more than two spaces, put $U_q=X_1\cup\cdots\cup X_q$.  Since every
pairwise intersection is $Y$, one has $U_{q-1}\cap X_q=Y$.  The inclusion
$Y\hookrightarrow U_{q-1}$ is null-homotopic because it factors through the
null-homotopic inclusion $Y\hookrightarrow X_1$.  Repeated application of
the two-space case adds one new copy of $\Sigma Y$ whenever a new $X_q$ is
attached, proving~\eqref{eq:gluing-lemma}.  This is the standard homotopy
pushout gluing argument; see, for example,
\cite[Section~2.1]{MayPonto2012} and \cite[Chapter~6]{Arkowitz2011}.
\end{proof}

\subsection{Completion of the induction}

Apply Lemma~\ref{lem:colored-gluing} to the family $(X_k)_{k\in C}$ in
\eqref{eq:colored-cover}.  Each $X_k$ is a wedge of $d$-spheres, while
$Y$ is a wedge of $(d-1)$-spheres.  Reduced suspension raises every sphere
dimension by one, so
\[
 \Sigma Y\simeq\bigvee S^d.
\]
Formula~\eqref{eq:gluing-lemma} therefore expresses
$\mathcal F_m(a,b)$ as a wedge of $d$-spheres.  This completes the inner
induction on $m$ and then the outer induction on $b$.  Taking $m=b$ proves
that $\mathcal F(A,B,C)$ is a wedge of $(b-a)$-spheres when $b>a$.
Thus the color index changes the number of spheres but not their dimension.

This completes the proof of Theorem~\ref{thm:colored-wedge}.



\subsection{Padding, retraction, and continuous extension}

We now pass from the uniform color multiplicity to arbitrary finite nonempty
address sets.  This step is needed because the cardinalities of the sets
$A_{i,j}$ in Theorem~\ref{thm:continuous-soft-filling-main} need not agree.

\begin{lemma}[Padding retraction]\label{lem:padding-retraction}
Let every $A_{i,j}$ be finite and nonempty, and choose an integer
$c\ge\max_{i,j}|A_{i,j}|$.  Enlarge each $A_{i,j}$ to a disjoint set
$\widetilde A_{i,j}$ of cardinality $c$.  Then
$\mathcal C_{m,n}(A)$ is a retract of
$\mathcal C_{m,n}(\widetilde A)\cong\mathcal F_{m,n,c}$.
\end{lemma}

\begin{proof}
Choose one distinguished address $a^0_{i,j}\in A_{i,j}$.  Inclusion of the
original coordinates defines
$\iota:\mathcal C_{m,n}(A)\to\mathcal C_{m,n}(\widetilde A)$.  Define a
linear coordinate-merging map $r$ by retaining all original coordinates,
except that the coordinate at $a^0_{i,j}$ is replaced by
\[
 \lambda_i(a^0_{i,j})+
 \sum_{b\in\widetilde A_{i,j}\setminus A_{i,j}}\lambda_i(b),
\]
and all added coordinates are then set equal to zero.  A soft configuration
has at most one positive coordinate in a fixed column, so this merger
preserves both row normalization and column exclusion.  Hence $r$ maps the
enlarged complex continuously to the original one, and
$r\circ\iota=\operatorname{id}$.
\end{proof}

\begin{proof}[Proof of Theorem~\ref{thm:continuous-soft-filling-main}]
Choose $c$, $\widetilde A$, $\iota$, and $r$ as in the padding lemma.  Given
a continuous map
$f:S^{k-1}\to\mathcal C_{m,n}(A)$ with $1\le k\le n-m$, compose it with
$\iota$.  By Theorem~\ref{thm:colored-wedge}, the enlarged complex is
$(n-m-1)$-connected, so $\iota\circ f$ extends to a continuous map
$G:B^k\to\mathcal C_{m,n}(\widetilde A)$.  Then $r\circ G$ extends $f$ and
has values in the original complex.  When $n=m$, a hard configuration is
obtained by choosing one address in the cells of any permutation matching.
This proves the theorem.
\end{proof}





\section{A first-principles Pfaffian proof of the common-zero theorem}
\label{app:pfaffian-common-zero}

This appendix proves the common-zero implication in
Theorem~\ref{thm:ddlt-common-zero} without vector bundles, Chern classes,
differential forms, Sard's theorem, regular-value theory, tangent spaces,
orientations, or zero-set manifolds.  The argument is entirely coordinate
based.  Its ingredients are a covariant smoothing operator, ordinary
partial derivatives, a real skew-symmetric matrix, the Pfaffian written as
an explicit finite sum, finite-dimensional linear algebra, and repeated
applications of the fundamental theorem of calculus on a cube.

A continuous function
\[
 s:\R^d\times\R^d\longrightarrow\C
\]
is called \emph{Zak-quasiperiodic} if
\begin{equation}\label{pf-eq:scalar-qp}
 s(u+n,v+m)=e^{2\pi i n\cdot v}s(u,v),
 \qquad u,v\in\R^d,\quad n,m\in\Z^d.
\end{equation}
A vector-valued function
$S=(s_1,\ldots,s_N):\R^{2d}\to\C^N$ is Zak-quasiperiodic when
\begin{equation}\label{pf-eq:vector-qp}
 S(u+n,v+m)=e^{2\pi i n\cdot v}S(u,v).
\end{equation}

\begin{theorem}\label{pf-thm:common-zero}
Let $1\le N\le d$, and let
$s_1,\ldots,s_N:\R^d\times\R^d\to\C$ be continuous
Zak-quasiperiodic functions.  Then the functions have a common zero.
Equivalently, the map
\[
 S=(s_1,\ldots,s_N):\R^{2d}\longrightarrow\C^N
\]
cannot be everywhere nonzero.
\end{theorem}

The proof has four steps.
\begin{enumerate}[label=(\arabic*)]
\item A continuous nowhere-zero quasiperiodic map can be smoothed without
      changing its quasiperiodicity or creating a zero.
\item After normalizing a smooth map to have length one, we form scalar
      coefficients $a_r$ from its first derivatives and a skew-symmetric
      matrix $F=(\partial_r a_s-\partial_s a_r)$.
\item A pointwise rank calculation forces $\Pf F=0$ everywhere.
\item The boundary increments imposed by Zak quasiperiodicity force
      $\int_{[0,1]^{2N}}\Pf F=1$.
\end{enumerate}
The last two conclusions contradict one another.

It suffices first to prove the result when $d=N$.  Indeed, if $d>N$,
we fix the final $d-N$ coordinate pairs and apply the $d=N$ case to the
remaining variables.  We therefore work from Section~\ref{pf-sec:normalize}
onward with $N$ coordinate pairs.

\subsection{Covariant smoothing of continuous quasiperiodic maps}

We record the smoothing step because the Pfaffian calculation uses ordinary
partial derivatives.

Choose a nonnegative function
$\rho\in C_c^\infty(\R^{2N})$ with
\[
 \int_{\R^{2N}}\rho(x,y)\,dx\,dy=1,
\]
and put
\[
 \rho_\varepsilon(x,y)=\varepsilon^{-2N}
 \rho(x/\varepsilon,y/\varepsilon).
\]
For a continuous Zak-quasiperiodic map
$S:\R^{2N}\to\C^N$, define
\begin{equation}\label{pf-eq:smoothing}
 (\mathcal M_\varepsilon S)(u,v)
 =\int_{\R^{2N}}
 \rho_\varepsilon(x,y)\,
 e^{2\pi i u\cdot y}
 S(u-x,v-y)\,dx\,dy.
\end{equation}

\begin{lemma}\label{pf-lem:smoothing}
The function $\mathcal M_\varepsilon S$ is smooth and satisfies the same
Zak boundary rule as $S$.  Moreover,
\[
 \mathcal M_\varepsilon S\longrightarrow S
\]
uniformly on $[0,1]^{2N}$ as $\varepsilon\to0$.  Consequently, if $S$ is
nowhere zero, then $\mathcal M_\varepsilon S$ is nowhere zero for all
sufficiently small $\varepsilon>0$.
\end{lemma}

\begin{proof}
Let $n,m\in\Z^N$.  Using quasiperiodicity of $S$,
\begin{align*}
 &(\mathcal M_\varepsilon S)(u+n,v+m)\\
 &\quad=\int \rho_\varepsilon(x,y)
 e^{2\pi i(u+n)\cdot y}
 S(u+n-x,v+m-y)\,dx\,dy\\
 &\quad=\int \rho_\varepsilon(x,y)
 e^{2\pi i u\cdot y}e^{2\pi i n\cdot y}
 e^{2\pi i n\cdot(v-y)}S(u-x,v-y)\,dx\,dy\\
 &\quad=e^{2\pi i n\cdot v}
 (\mathcal M_\varepsilon S)(u,v).
\end{align*}
Thus the boundary rule is preserved exactly.

To see smoothness without differentiating the original continuous function,
make the change of variables
\[
 p=u-x,\qquad q=v-y.
\]
Then
\begin{equation}\label{pf-eq:smoothing-kernel}
 (\mathcal M_\varepsilon S)(u,v)
 =\int_{\R^{2N}}
 \rho_\varepsilon(u-p,v-q)
 e^{2\pi i u\cdot(v-q)}S(p,q)\,dp\,dq.
\end{equation}
For $(u,v)$ in a compact set, the factor
$\rho_\varepsilon(u-p,v-q)$ restricts $(p,q)$ to another compact set.
Every derivative in $u$ and $v$ falls on the smooth kernel
\[
 \rho_\varepsilon(u-p,v-q)e^{2\pi i u\cdot(v-q)},
\]
so differentiation under the integral is legitimate to every order.

For uniform convergence, rewrite \eqref{pf-eq:smoothing} as
\begin{align*}
 (\mathcal M_\varepsilon S)(u,v)-S(u,v)
  =\int \rho_\varepsilon(x,y)
  \bigl(e^{2\pi i u\cdot y}S(u-x,v-y)-S(u,v)\bigr)\,dx\,dy.
\end{align*}
On a slightly enlarged compact fundamental cube, $S$ is uniformly
continuous and bounded.  The expression in parentheses tends uniformly to
zero as $(x,y)\to(0,0)$.  The usual approximate-identity estimate therefore
gives uniform convergence on $[0,1]^{2N}$.

Finally, $\norm{S(u,v)}_2$ is $\Z^{2N}$-periodic.  If $S$ has no zero, then
\[
 \delta=\min_{[0,1]^{2N}}\norm{S(u,v)}_2>0.
\]
For sufficiently small $\varepsilon$ the uniform approximation error is
less than $\delta/2$, and then
$\norm{\mathcal M_\varepsilon S}_2\ge\delta/2$ everywhere.
\end{proof}

\subsection{Normalization and the derivative matrix}\label{pf-sec:normalize}

Assume for contradiction that a smooth Zak-quasiperiodic map
\[
 S:\R^N\times\R^N\longrightarrow\C^N
\]
is nowhere zero.  Normalize it by
\begin{equation}\label{pf-eq:normalize}
 U(x,y)=\frac{S(x,y)}{\norm{S(x,y)}_2}.
\end{equation}
Then
\begin{equation}\label{pf-eq:unit}
 U(x,y)^*U(x,y)=1
\end{equation}
and
\begin{equation}\label{pf-eq:U-boundary}
 U(x+n,y+m)=e^{2\pi i n\cdot y}U(x,y).
\end{equation}

We order the real coordinates as
\begin{equation}\label{pf-eq:coordinate-order}
 t_1=x_1,\quad t_2=y_1,\quad t_3=x_2,\quad t_4=y_2,
 \quad\ldots,\quad t_{2N-1}=x_N,\quad t_{2N}=y_N,
\end{equation}
and write $\partial_r=\partial/\partial t_r$.

For $r=1,\ldots,2N$, define
\begin{equation}\label{pf-eq:a-r}
 a_r(x,y)=\frac{1}{2\pi i}\,U(x,y)^*\partial_rU(x,y).
\end{equation}
These coefficients are real-valued.  Indeed, differentiating
\eqref{pf-eq:unit} gives
\[
 (\partial_rU)^*U+U^*\partial_rU=0,
\]
so
\[
 \overline{U^*\partial_rU}
 = (\partial_rU)^*U
 =-U^*\partial_rU.
\]
Thus $U^*\partial_rU$ is purely imaginary.

For clarity, write
\[
 a_{x_k}=\frac{1}{2\pi i}U^*\partial_{x_k}U,
 \qquad
 a_{y_k}=\frac{1}{2\pi i}U^*\partial_{y_k}U.
\]

\subsubsection*{Boundary relations for the coefficients}
From
\[
 U(x+e_j,y)=e^{2\pi i y_j}U(x,y)
\]
we obtain
\begin{equation}\label{pf-eq:ax-boundary}
 a_{x_k}(x+e_j,y)=a_{x_k}(x,y).
\end{equation}
For the $y_k$ derivative,
\begin{align*}
 \partial_{y_k}U(x+e_j,y)
 &=\partial_{y_k}\bigl(e^{2\pi i y_j}U(x,y)\bigr)\\
 &=e^{2\pi i y_j}
 \bigl(2\pi i\delta_{jk}U(x,y)+\partial_{y_k}U(x,y)\bigr).
\end{align*}
Multiplying on the left by $U(x+e_j,y)^*=e^{-2\pi i y_j}U(x,y)^*$ gives
\begin{equation}\label{pf-eq:ay-boundary}
 a_{y_k}(x+e_j,y)=a_{y_k}(x,y)+\delta_{jk}.
\end{equation}
Since $U$ is periodic in every $y_j$ variable,
\begin{equation}\label{pf-eq:y-periodicity-a}
 a_r(x,y+e_j)=a_r(x,y).
\end{equation}

Define the real skew-symmetric matrix
\begin{equation}\label{pf-eq:F-def}
 F=(F_{rs})_{r,s=1}^{2N},
 \qquad
 F_{rs}=\partial_ra_s-\partial_sa_r.
\end{equation}
The additive constants in \eqref{pf-eq:ay-boundary} disappear after
differentiation, so every entry $F_{rs}$ is $\Z^{2N}$-periodic.

\subsection{The Pfaffian and pointwise degeneracy}

For a real skew-symmetric $2N\times2N$ matrix $B=(B_{rs})$, define
\begin{equation}\label{pf-eq:pf-def}
 \Pf B
 =\frac{1}{2^NN!}
 \sum_{i_1,\ldots,i_{2N}=1}^{2N}
 \varepsilon_{i_1\ldots i_{2N}}
 B_{i_1i_2}B_{i_3i_4}\cdots B_{i_{2N-1}i_{2N}},
\end{equation}
where $\varepsilon_{i_1\ldots i_{2N}}$ is the alternating symbol.
We use the elementary algebraic identity
\begin{equation}\label{pf-eq:det-pf}
 \det B=(\Pf B)^2.
\end{equation}
It follows directly by expanding both sides as alternating polynomials in the
entries of $B$; it is also proved by induction using the expansion of the
Pfaffian along its first row.

\begin{proposition}\label{pf-prop:pointwise-zero}
For the matrix $F$ in \eqref{pf-eq:F-def},
\[
 \Pf F(x,y)=0
\]
at every point $(x,y)$.
\end{proposition}

\begin{proof}
Differentiate \eqref{pf-eq:a-r}.  The product rule gives
\[
 \partial_r(U^*\partial_sU)
 = (\partial_rU)^*\partial_sU+U^*\partial_r\partial_sU,
\]
and the same formula with $r$ and $s$ interchanged.  Since $U$ is smooth,
$\partial_r\partial_sU=\partial_s\partial_rU$, and the second-derivative
terms cancel.  Hence
\begin{align}
 F_{rs}
 &=\frac{1}{2\pi i}
 \left[
 \partial_r(U^*\partial_sU)-\partial_s(U^*\partial_rU)
 \right]\notag\\
 &=\frac{1}{2\pi i}
 \left[
 (\partial_rU)^*\partial_sU-(\partial_sU)^*\partial_rU
 \right].
 \label{pf-eq:F-derivative}
\end{align}

Fix a point $(x,y)$ for the remainder of the argument and abbreviate
$U=U(x,y)$.  Put
\begin{equation}\label{pf-eq:Q-def}
 Q=\Id_N-UU^*.
\end{equation}
We verify explicitly that $Q$ is the orthogonal projection onto the complex
orthogonal complement of $U$.  Since $U^*U=1$,
\[
 Q^*=\Id_N-(UU^*)^*=Q
\]
and
\[
 Q^2=(\Id_N-UU^*)^2
 =\Id_N-2UU^*+U(U^*U)U^*
 =\Id_N-UU^*=Q.
\]
Moreover,
\[
 QU=U-U(U^*U)=0,
 \qquad
 U^*Q=U^*-(U^*U)U^*=0.
\]
Thus $\operatorname{ran}Q\subseteq U^\perp$.  Conversely, if $z\in U^\perp$,
then $U^*z=0$ and therefore $Qz=z$.  It follows that
\[
 \operatorname{ran}Q=U^\perp
 =\{z\in\C^N:U^*z=0\}.
\]
Because $U\ne0$, this space has complex dimension $N-1$ and real dimension
$2N-2$.

For each $r=1,\ldots,2N$, set
\begin{equation}\label{pf-eq:w-r}
 w_r=Q\partial_rU,
 \qquad
 \alpha_r=U^*\partial_rU.
\end{equation}
The vector $w_r$ belongs to $U^\perp$, because $U^*Q=0$.  The identity
$\Id_N=UU^*+Q$ gives the orthogonal decomposition
\[
 \partial_rU
 =(UU^*+Q)\partial_rU
 =U(U^*\partial_rU)+Q\partial_rU
 =U\alpha_r+w_r.
\]
Differentiating $U^*U=1$ in the $t_r$ direction gives
\[
 (\partial_rU)^*U+U^*\partial_rU=0,
\]
so $\overline{\alpha_r}=-\alpha_r$.  Thus there is a real number
$\beta_r$ with $\alpha_r=i\beta_r$.  In particular,
\[
 \overline{\alpha_r}\alpha_s
 =(-i\beta_r)(i\beta_s)=\beta_r\beta_s
 =\overline{\alpha_s}\alpha_r.
\]
Also $U^*w_s=0$ and $w_r^*U=0$.  Consequently,
\begin{align*}
 (\partial_rU)^*\partial_sU
 &=(\overline{\alpha_r}U^*+w_r^*)
   (U\alpha_s+w_s)
   =\overline{\alpha_r}\alpha_s+w_r^*w_s,\\
 (\partial_sU)^*\partial_rU
 &=\overline{\alpha_s}\alpha_r+w_s^*w_r.
\end{align*}
The scalar terms are equal and cancel in \eqref{pf-eq:F-derivative}.  We
therefore obtain
\begin{equation}\label{pf-eq:F-w}
 F_{rs}
 =\frac{1}{2\pi i}\bigl(w_r^*w_s-w_s^*w_r\bigr)
 =\frac{1}{\pi}\operatorname{Im}(w_r^*w_s).
\end{equation}
The second equality follows from
$w_s^*w_r=\overline{w_r^*w_s}$.

We now turn \eqref{pf-eq:F-w} into an explicit real matrix factorization.
Choose a complex orthonormal basis
$e_1,\ldots,e_{N-1}$ of $U^\perp$.  Write
\[
 w_r=\sum_{\ell=1}^{N-1}(p_{\ell r}+iq_{\ell r})e_\ell,
 \qquad p_{\ell r},q_{\ell r}\in\R.
\]
Let $W$ be the real $(2N-2)\times2N$ matrix whose $r$th column is
\[
 W_{\cdot r}
 =(p_{1r},q_{1r},p_{2r},q_{2r},\ldots,
   p_{N-1,r},q_{N-1,r})^T,
\]
and let
\[
 \Omega=\frac1\pi\operatorname{diag}
 \left(
 \begin{pmatrix}0&1\\-1&0\end{pmatrix},\ldots,
 \begin{pmatrix}0&1\\-1&0\end{pmatrix}
 \right)
 \in M_{2N-2}(\R).
\]
A direct calculation gives
\[
 (W_{\cdot r})^T\Omega W_{\cdot s}
 =\frac1\pi\sum_{\ell=1}^{N-1}
   (p_{\ell r}q_{\ell s}-q_{\ell r}p_{\ell s})
 =\frac1\pi\operatorname{Im}(w_r^*w_s)
 =F_{rs}.
\]
Thus, entry by entry,
\begin{equation}\label{pf-eq:F-factor}
 F=W^T\Omega W.
\end{equation}
When $N=1$, the space $U^\perp$ is zero-dimensional, $W$ has no rows, and
this formula simply says $F=0$.

Finally,
\[
 \rank F=\rank(W^T\Omega W)
 \le \rank W\le2N-2<2N.
\]
Hence the $2N\times2N$ matrix $F$ is singular, and so $\det F=0$.  The
identity \eqref{pf-eq:det-pf} now yields $(\Pf F)^2=0$, and therefore
$\Pf F=0$ at the chosen point.  Since the point was arbitrary, the
conclusion holds everywhere.
\end{proof}

\subsection{The boundary-Pfaffian identity and the common-zero contradiction}
\label{app:boundary-pfaffian}

We now isolate the boundary identity and complete the contradiction argument.

\subsubsection{An abstract boundary-Pfaffian identity}

The next lemma is independent of the map $U$.  It uses only smooth real
functions with the boundary increments found in
\eqref{pf-eq:ax-boundary}--\eqref{pf-eq:y-periodicity-a}.  We keep the
coordinate order
\[
 x_1,y_1,x_2,y_2,\ldots,x_N,y_N.
\]
Thus, inside an alternating symbol, the notation $x_j$ means the index
$2j-1$ and the notation $y_j$ means the index $2j$.

\begin{lemma}\label{pf-lem:boundary-pfaffian}
For $j=1,\ldots,N$, let $a_{x_j},a_{y_j}$ be smooth real functions on
$\R^{2N}$ satisfying, for every $j,k=1,\ldots,N$,
\begin{align}
 a_{x_k}(x+e_j,y)&=a_{x_k}(x,y),
 \label{pf-eq:abstract-bc-1}\\
 a_{y_k}(x+e_j,y)&=a_{y_k}(x,y)+\delta_{jk},
 \label{pf-eq:abstract-bc-2}\\
 a_r(x,y+e_j)&=a_r(x,y).
 \label{pf-eq:abstract-bc-3}
\end{align}
Order the coordinates as in \eqref{pf-eq:coordinate-order}, and define
$F_{rs}=\partial_ra_s-\partial_sa_r$.  Then
\begin{equation}\label{pf-eq:pf-integral-one}
 \int_{[0,1]^{2N}}\Pf F\,dt_1\cdots dt_{2N}=1.
\end{equation}
\end{lemma}

Before proving the lemma, note that every entry of $F$ is periodic in all
$2N$ variables.  Indeed, across an $x_j$-face each coefficient $a_s$
changes by a constant, namely either $0$ or $1$.  Differentiating the
boundary relation therefore gives
\[
 \partial_ra_s(x+e_j,y)=\partial_ra_s(x,y).
\]
Across a $y_j$-face the coefficients themselves are periodic.  Hence
$F_{rs}$ has identical values on every pair of opposite faces of the unit
cube.

The proof of the lemma is a repeated boundary calculation.  We first derive
an explicit divergence identity.

\paragraph{The divergence identity.}
For $r=1,\ldots,2N$, define
\begin{align}
 V_r
 &=\frac{1}{2^{N-1}N!}
 \sum_{i_2,\ldots,i_{2N}=1}^{2N}
 \varepsilon_{r i_2\ldots i_{2N}}\,
 a_{i_2}
 F_{i_3i_4}F_{i_5i_6}\cdots F_{i_{2N-1}i_{2N}}.
 \label{pf-eq:V-r}
\end{align}
Here and below a term is automatically zero when two indices coincide,
because the alternating symbol then vanishes.

\begin{proposition}\label{pf-prop:divergence}
The functions in \eqref{pf-eq:V-r} satisfy
\begin{equation}\label{pf-eq:divergence}
 \sum_{r=1}^{2N}\partial_rV_r=\Pf F.
\end{equation}
\end{proposition}

\begin{proof}
Replace the outer index $r$ by $i_1$ and apply the product rule.  There are
two kinds of terms: the derivative may fall on $a_{i_2}$, or it may fall
on one of the $N-1$ factors
$F_{i_3i_4},\ldots,F_{i_{2N-1}i_{2N}}$.  Thus
\[
 \sum_{r=1}^{2N}\partial_rV_r=A+\sum_{p=2}^N B_p,
\]
where
\begin{align}
 A
 &=\frac{1}{2^{N-1}N!}
 \sum_{i_1,\ldots,i_{2N}}
 \varepsilon_{i_1i_2\ldots i_{2N}}
 (\partial_{i_1}a_{i_2})
 F_{i_3i_4}\cdots F_{i_{2N-1}i_{2N}},
 \label{pf-eq:derivative-a}
\end{align}
and, for $p=2,\ldots,N$,
\begin{align*}
 B_p
 &=\frac{1}{2^{N-1}N!}
 \sum_{i_1,\ldots,i_{2N}}
 \varepsilon_{i_1i_2\ldots i_{2N}}a_{i_2}
 (\partial_{i_1}F_{i_{2p-1}i_{2p}})\\
 &\hspace{28mm}\times
 \prod_{\substack{q=2\\q\ne p}}^N
 F_{i_{2q-1}i_{2q}}.
\end{align*}

We first compute $A$.  Let
\[
 \Sigma=
 \sum_{i_1,\ldots,i_{2N}}
 \varepsilon_{i_1i_2\ldots i_{2N}}
 (\partial_{i_1}a_{i_2})
 F_{i_3i_4}\cdots F_{i_{2N-1}i_{2N}}.
\]
Interchange the two dummy indices $i_1$ and $i_2$.  The alternating symbol
changes sign, while the remaining factors are unchanged.  Hence the same
sum also equals
\[
 -\sum_{i_1,\ldots,i_{2N}}
 \varepsilon_{i_1i_2\ldots i_{2N}}
 (\partial_{i_2}a_{i_1})
 F_{i_3i_4}\cdots F_{i_{2N-1}i_{2N}}.
\]
Adding these two representations of $\Sigma$ gives
\begin{align*}
 2\Sigma
 &=\sum_{i_1,\ldots,i_{2N}}
 \varepsilon_{i_1i_2\ldots i_{2N}}
 (\partial_{i_1}a_{i_2}-\partial_{i_2}a_{i_1})
 F_{i_3i_4}\cdots F_{i_{2N-1}i_{2N}}\\
 &=\sum_{i_1,\ldots,i_{2N}}
 \varepsilon_{i_1i_2\ldots i_{2N}}
 F_{i_1i_2}F_{i_3i_4}\cdots F_{i_{2N-1}i_{2N}}.
\end{align*}
Therefore the coefficient in \eqref{pf-eq:derivative-a} gives
\[
 A=\frac{1}{2^NN!}
 \sum_{i_1,\ldots,i_{2N}}
 \varepsilon_{i_1\ldots i_{2N}}
 F_{i_1i_2}\cdots F_{i_{2N-1}i_{2N}}
 =\Pf F.
\]

It remains to prove $B_p=0$ for each $p$.  The elementary identity used for
this cancellation is
\begin{equation}\label{pf-eq:bianchi}
 \partial_rF_{ab}+\partial_aF_{br}+\partial_bF_{ra}=0.
\end{equation}
To verify it from first principles, substitute
$F_{ab}=\partial_aa_b-\partial_ba_a$:
\begin{align*}
 &\partial_rF_{ab}+\partial_aF_{br}+\partial_bF_{ra}\\
 &\quad=
 \partial_r\partial_aa_b-\partial_r\partial_ba_a
 +\partial_a\partial_ba_r-\partial_a\partial_ra_b
 +\partial_b\partial_ra_a-\partial_b\partial_aa_r=0.
\end{align*}
The first and fourth terms cancel, the second and fifth terms cancel, and
the third and sixth terms cancel, because mixed partial derivatives
commute.

Fix $p$ and abbreviate
$a=i_{2p-1}$ and $b=i_{2p}$.  In the sum defining $B_p$, cyclically rename
the three dummy indices
\[
 (i_1,a,b)\longmapsto(a,b,i_1).
\]
A three-cycle is an even permutation, so the alternating symbol is
unchanged.  The value of the sum is therefore unchanged if
$\partial_{i_1}F_{ab}$ is replaced by $\partial_aF_{bi_1}$, and it is also
unchanged if it is replaced by $\partial_bF_{i_1a}$.  Averaging these three
equal representations yields
\begin{align*}
 B_p
 &=\frac{1}{3\cdot2^{N-1}N!}
 \sum_{i_1,\ldots,i_{2N}}
 \varepsilon_{i_1i_2\ldots i_{2N}}a_{i_2}\\
 &\quad\times
 \bigl(
 \partial_{i_1}F_{ab}+\partial_aF_{bi_1}
 +\partial_bF_{i_1a}
 \bigr)
 \prod_{\substack{q=2\\q\ne p}}^N
 F_{i_{2q-1}i_{2q}}=0
\end{align*}
by \eqref{pf-eq:bianchi}.  Thus every $B_p$ vanishes, while $A=\Pf F$,
which proves \eqref{pf-eq:divergence}.
\end{proof}

\paragraph{Proof of the boundary identity.}

We now verify the boundary identity by expanding the divergence formula.

\begin{proof}[Proof of Lemma~\ref{pf-lem:boundary-pfaffian}]
We argue by induction on $N$.  Throughout, an integral over a face means
integration with respect to all coordinates that remain free on that face.

For $N=1$, the coordinate order is $(x_1,y_1)$ and
\[
 \Pf F=F_{x_1y_1}
 =\partial_{x_1}a_{y_1}-\partial_{y_1}a_{x_1}.
\]
Applying the one-variable fundamental theorem of calculus first in $x_1$
and then in $y_1$ gives
\begin{align*}
 \int_0^1\int_0^1\Pf F\,dx_1\,dy_1
 &=\int_0^1
 \bigl(a_{y_1}(1,y_1)-a_{y_1}(0,y_1)\bigr)\,dy_1\\
 &\quad-
 \int_0^1
 \bigl(a_{x_1}(x_1,1)-a_{x_1}(x_1,0)\bigr)\,dx_1.
\end{align*}
The first bracket equals $1$ by
\eqref{pf-eq:abstract-bc-2}, and the second bracket equals $0$ by
\eqref{pf-eq:abstract-bc-3}.  Hence the integral equals $1$.

Assume now that the assertion has been proved for $N-1$ coordinate pairs.
Integrate \eqref{pf-eq:divergence} over the unit cube.  For a coordinate
$t_r$, write $dt_{\widehat r}$ for integration in all variables except
$t_r$.  The fundamental theorem of calculus gives
\begin{equation}\label{pf-eq:boundary-sum}
 \int_{[0,1]^{2N}}\Pf F
 =\sum_{r=1}^{2N}
 \int_{[0,1]^{2N-1}}
 \left(V_r\big|_{t_r=1}-V_r\big|_{t_r=0}\right)
 \,dt_{\widehat r}.
\end{equation}

We first examine a pair of $y_j$-faces.  By
\eqref{pf-eq:abstract-bc-3}, every coefficient $a_s$ has the same value at
$y_j=0$ and $y_j=1$.  As observed before the divergence calculation, every
entry of $F$ is also periodic.  Since $V_{y_j}$ is a polynomial expression
in the $a_s$ and $F_{rs}$, it follows pointwise that
\[
 V_{y_j}\big|_{y_j=1}=V_{y_j}\big|_{y_j=0}.
\]
Thus every pair of $y_j$-faces contributes zero to
\eqref{pf-eq:boundary-sum}.

Now fix $j$ and consider the two $x_j$-faces.  In the notation of
\eqref{pf-eq:V-r}, take $r=x_j$.  All $F$-factors agree on the two faces.
For the remaining factor, the boundary relations say
\[
 a_{i_2}(x+e_j,y)-a_{i_2}(x,y)
 =\begin{cases}
 1,&i_2=y_j,\\
 0,&i_2\ne y_j.
 \end{cases}
\]
Therefore only the terms with $i_2=y_j$ survive in the difference, and
\begin{align}
 &V_{x_j}\big|_{x_j=1}-V_{x_j}\big|_{x_j=0}\notag\\
 &\quad=\frac{1}{2^{N-1}N!}
 \sum_{i_3,\ldots,i_{2N}}
 \varepsilon_{x_jy_ji_3\ldots i_{2N}}
 F_{i_3i_4}\cdots F_{i_{2N-1}i_{2N}}.
 \label{pf-eq:x-face-jump}
\end{align}
A nonzero term in this sum must use each of the remaining $2N-2$ indices
exactly once.

Let
\[
 I_j=\{x_1,y_1,\ldots,x_N,y_N\}\setminus\{x_j,y_j\},
\]
and let $F^{(j)}$ be the submatrix of $F$ with rows and columns indexed by
$I_j$, evaluated on the face $x_j=0$.  The same submatrix is obtained on
$x_j=1$ because $F$ is periodic.  To compare alternating symbols, move the
ordered pair $(x_j,y_j)$ from its positions $(2j-1,2j)$ to the first two
positions.  Moving $x_j$ requires $2j-2$ transpositions, and after that
moving $y_j$ requires another $2j-2$ transpositions.  The total number
$4j-4$ is even, so no sign is introduced.  Hence the alternating symbol in
\eqref{pf-eq:x-face-jump} is exactly the alternating symbol associated with
the inherited order on $I_j$.

The Pfaffian of the $(2N-2)\times(2N-2)$ matrix $F^{(j)}$ is
\[
 \Pf F^{(j)}
 =\frac{1}{2^{N-1}(N-1)!}
 \sum_{i_3,\ldots,i_{2N}\in I_j}
 \varepsilon^{(j)}_{i_3\ldots i_{2N}}
 F_{i_3i_4}\cdots F_{i_{2N-1}i_{2N}},
\]
where $\varepsilon^{(j)}$ is the alternating symbol for the inherited
coordinate order.  Since $N!=N(N-1)!$, comparison with
\eqref{pf-eq:x-face-jump} gives
\begin{equation}\label{pf-eq:x-face-pf-minor}
 V_{x_j}\big|_{x_j=1}-V_{x_j}\big|_{x_j=0}
 =\frac1N\Pf F^{(j)}.
\end{equation}

It remains to integrate this minor.  Fix $y_j\in[0,1]$ and set $x_j=0$.
For every $k\ne j$, restrict the functions $a_{x_k}$ and $a_{y_k}$ to this
slice and regard them as functions of the remaining $2N-2$ variables.
For shifts in a remaining $x_\ell$ variable, $\ell\ne j$, they satisfy
\[
 a_{x_k}(x+e_\ell,y)=a_{x_k}(x,y),
 \qquad
 a_{y_k}(x+e_\ell,y)=a_{y_k}(x,y)+\delta_{\ell k},
\]
and they are periodic in every remaining $y_\ell$ variable.  Thus these
restricted functions satisfy exactly the hypotheses of the lemma with
$N-1$ coordinate pairs.  Their skew derivative matrix, taken only with
respect to the remaining variables, is precisely $F^{(j)}$.  The induction
hypothesis therefore gives
\begin{equation}\label{pf-eq:minor-integral}
 \int_{[0,1]^{2N-2}}\Pf F^{(j)}=1
\end{equation}
for every fixed value of $y_j$.

Using \eqref{pf-eq:x-face-pf-minor}, the total contribution of the two
$x_j$-faces to \eqref{pf-eq:boundary-sum} is therefore
\[
 \frac1N\int_0^1
 \left(\int_{[0,1]^{2N-2}}\Pf F^{(j)}\right)dy_j
 =\frac1N\int_0^1 1\,dy_j
 =\frac1N.
\]
There are $N$ choices of $j$.  All $y_j$-face contributions vanish, and
each pair of $x_j$-faces contributes $1/N$.  Hence
\[
 \int_{[0,1]^{2N}}\Pf F
 =\sum_{j=1}^N\frac1N=1,
\]
which completes the induction.

For example, when $N=2$ and the coordinates are
$(x_1,y_1,x_2,y_2)$,
\[
 \Pf F
 =F_{x_1y_1}F_{x_2y_2}
  -F_{x_1x_2}F_{y_1y_2}
  +F_{x_1y_2}F_{y_1x_2}.
\]
The argument above says that the $y_1$- and $y_2$-faces cancel, while the
$x_1$-faces and the $x_2$-faces each contribute $1/2$.
\end{proof}

\subsubsection{The contradiction and the continuous theorem}

The remaining step rules out a smooth nowhere-zero quasiperiodic map and then returns
to the original continuous function by approximation.

\begin{proposition}\label{pf-prop:smooth-no-nonzero}
There is no smooth nowhere-zero Zak-quasiperiodic map
\[
 S:\R^N\times\R^N\longrightarrow\C^N.
\]
\end{proposition}

\begin{proof}
Assume that such an $S$ exists.  Normalize it to $U$ as in
\eqref{pf-eq:normalize}, and define the coefficients $a_r$ and matrix $F$ as in
\eqref{pf-eq:a-r} and \eqref{pf-eq:F-def}.  Proposition~\ref{pf-prop:pointwise-zero}
gives
\[
 \Pf F(x,y)=0
\]
for every $(x,y)$.  Therefore
\[
 \int_{[0,1]^{2N}}\Pf F=0.
\]
On the other hand, the boundary relations
\eqref{pf-eq:ax-boundary}--\eqref{pf-eq:y-periodicity-a} satisfy the hypotheses of
Lemma~\ref{pf-lem:boundary-pfaffian}, which gives
\[
 \int_{[0,1]^{2N}}\Pf F=1.
\]
This is impossible.
\end{proof}

\begin{proof}[Proof of Theorem~\ref{pf-thm:common-zero}]
First suppose $d=N$ and assume that the continuous map
$S=(s_1,\ldots,s_N)$ is nowhere zero.  Lemma~\ref{pf-lem:smoothing} produces,
for sufficiently small $\varepsilon>0$, a smooth nowhere-zero
Zak-quasiperiodic map $\mathcal M_\varepsilon S$.  This contradicts
Proposition~\ref{pf-prop:smooth-no-nonzero}.

Now let $d>N$.  Fix the final $d-N$ coordinate pairs, for example at zero:
\[
 u_{N+1}=\cdots=u_d=0,
 \qquad
 v_{N+1}=\cdots=v_d=0.
\]
The restricted functions retain the Zak boundary rule in the first $N$
coordinate pairs.  The already proved case $d=N$ gives a common zero of the
restricted functions, and therefore a common zero of the original
functions.
\end{proof}

\subsubsection{The first two dimensions}

The formulas become especially transparent for small $N$.

For $N=1$,
\[
 F=
 \begin{pmatrix}
 0&F_{x_1y_1}\\
 -F_{x_1y_1}&0
 \end{pmatrix},
 \qquad
 \Pf F=F_{x_1y_1}.
\]
The boundary calculation is simply
\[
 \int_{[0,1]^2}
 (\partial_{x_1}a_{y_1}-\partial_{y_1}a_{x_1})=1,
\]
which is the ordinary winding-number increment.

For $N=2$, in the coordinate order
$(x_1,y_1,x_2,y_2)$,
\begin{align*}
 \Pf F
 &=F_{x_1y_1}F_{x_2y_2}
   -F_{x_1x_2}F_{y_1y_2}
   +F_{x_1y_2}F_{y_1x_2}.
\end{align*}
The pointwise rank estimate gives $\rank F\le2$, so this expression is zero
at every point.  The boundary calculation nevertheless gives
\[
 \int_{[0,1]^4}\Pf F=1.
\]
Thus the higher-dimensional contradiction is an iterated version of the
one-dimensional winding computation.

\subsubsection{Concluding comment}

The proof is entirely coordinate based.  Nonvanishing allows us to
normalize the tuple.  The normalized tuple produces a skew derivative
matrix $F$.  Its columns are generated by vectors in a real space of
dimension $2N-2$, so its top Pfaffian vanishes pointwise.  The Zak boundary
phase, however, contributes one unit of winding in each coordinate pair,
and the explicit divergence calculation packages those increments into
\[
 \int_{[0,1]^{2N}}\Pf F=1.
\]
The contradiction proves the common-zero theorem from ordinary calculus
and finite-dimensional linear algebra.

\section{A compactly supported Parseval frame outside the block-zero cases}
\label{app:compact-parseval-example}

This appendix proves the example stated after
Theorem~\ref{thm:strengthened-one-window}.  The construction is one-dimensional
and completely explicit.

\begin{proposition}\label{prop:compact-parseval-example}
Let
\[
 a=\frac{71}{100},
 \qquad
 \Gamma_0=a\begin{pmatrix}1&1\\-1&1\end{pmatrix}\Z^2.
\]
There exists a nonzero $g_0\in C_c^\infty(\R)$ such that
$\G(g_0,\Gamma_0)$ is an orthonormal system.  Consequently, if
\[
 \Lambda_0=\Gamma_0^\circ
 =\frac1{2a}\begin{pmatrix}1&-1\\1&1\end{pmatrix}\Z^2,
\]
then
\[
 g=\sqrt{\covol(\Lambda_0)}\,g_0
\]
belongs to $C_c^\infty(\R)$ and $\G(g,\Lambda_0)$ is a Parseval frame.
Moreover,
\[
 \Lambda_0=cS\Z^2,
 \qquad
 c=\frac1{\sqrt2a},
 \qquad
 S=\frac1{\sqrt2}\begin{pmatrix}1&-1\\1&1\end{pmatrix}\in\Sp(2,\R),
\]
and every entry of $S$ is nonzero.
\end{proposition}

\begin{proof}
Put
\[
 L=\frac1{2a}=\frac{50}{71}.
\]
The numerical inequality needed below is
\[
 L<a,
 \qquad
 a-L=\frac{41}{7100}>0.
\]
Choose
\[
 0<\delta<\frac{a-L}{2};
\]
for example, $\delta=(a-L)/4$.  Then
\begin{equation}\label{eq:compact-example-diameter}
 L+2\delta<a.
\end{equation}

We first construct a smooth function whose squared modulus has constant
$L$-periodization.  Define
\[
 \rho(s)=
 \begin{cases}
  0,&s\le0,\\
  e^{-1/s},&s>0,
 \end{cases}
 \qquad
 \eta(s)=\frac{\rho(s)}{\rho(s)+\rho(1-s)}.
\]
Then $\eta\in C^\infty(\R)$, $\eta(s)=0$ for $s\le0$, and
$\eta(s)=1$ for $s\ge1$.  Set
\[
 \vartheta(t)=\frac\pi2\,
 \eta\!\left(\frac{t+\delta}{2\delta}\right).
\]
Thus $\vartheta(t)=0$ for $t\le-\delta$,
$\vartheta(t)=\pi/2$ for $t\ge\delta$, and all derivatives of
$\vartheta$ vanish at the two transition endpoints.  Define
\begin{equation}\label{eq:compact-example-window}
 g_0(t)=\frac1{\sqrt L}
 \begin{cases}
  \sin\vartheta(t),&-\delta\le t\le\delta,\\
  1,&\delta\le t\le L-\delta,\\
  \cos\vartheta(t-L),&L-\delta\le t\le L+\delta,\\
  0,&\text{otherwise}.
 \end{cases}
\end{equation}
The flatness of the transition function shows that the four pieces and all
their derivatives agree at the endpoints.  Hence
$g_0\in C_c^\infty(\R)$, and
\[
 \supp g_0\subset[-\delta,L+\delta].
\]
By \eqref{eq:compact-example-diameter}, the diameter of this support is
strictly smaller than $a$.

We claim that
\begin{equation}\label{eq:compact-example-periodization}
 \sum_{j\in\Z}|g_0(x+jL)|^2=\frac1L,
 \qquad x\in\R.
\end{equation}
It is enough to consider $0\le x<L$.  If
$\delta\le x\le L-\delta$, only the term $j=0$ is nonzero, and its
squared modulus is $1/L$.  If $0\le x<\delta$, the only nonzero terms are
$j=0$ and $j=1$, and
\[
 |g_0(x)|^2+|g_0(x+L)|^2
 =\frac1L\bigl(\sin^2\vartheta(x)+\cos^2\vartheta(x)\bigr)
 =\frac1L.
\]
If $L-\delta<x<L$, the only nonzero terms are $j=0$ and $j=-1$, and the
same identity, with argument $x-L$, applies.  This proves
\eqref{eq:compact-example-periodization}.

Taking the Fourier coefficients of the $L$-periodic function in
\eqref{eq:compact-example-periodization} gives, for every $k\in\Z$,
\begin{align}
 \int_{\R}|g_0(t)|^2e^{-2\pi i kt/L}\,dt
 &=\int_0^L\sum_{j\in\Z}|g_0(x+jL)|^2e^{-2\pi i kx/L}\,dx \notag\\
 &=\frac1L\int_0^L e^{-2\pi i kx/L}\,dx
 =\delta_{k,0}.
 \label{eq:compact-example-fourier-coefficients}
\end{align}
In particular, $\norm{g_0}_2=1$.

We now verify orthogonality on $\Gamma_0$.  For $m,n\in\Z$,
\[
 a\begin{pmatrix}1&1\\-1&1\end{pmatrix}
 \begin{pmatrix}m\\n\end{pmatrix}
 =\bigl(a(m+n),a(-m+n)\bigr).
\]
Writing $p=m+n$ and $q=-m+n$, we obtain
\begin{equation}\label{eq:compact-example-lattice-parity}
 \Gamma_0
 =\{(ap,aq):p,q\in\Z,\ p\equiv q\pmod2\}.
\end{equation}
Let $(ap,aq)\in\Gamma_0\setminus\{0\}$.  If $p\ne0$, then
$|ap|\ge a$, while the support diameter of $g_0$ is smaller than $a$.
Therefore
\[
 \supp g_0\cap(\supp g_0+ap)=\varnothing,
\]
and hence
\[
 \ip{g_0}{M_{aq}T_{ap}g_0}=0.
\]
If $p=0$, the parity condition in
\eqref{eq:compact-example-lattice-parity} gives $q=2k$ for some
$k\in\Z\setminus\{0\}$.  Since $2a=1/L$, we have $aq=k/L$, and
\eqref{eq:compact-example-fourier-coefficients} yields
\[
 \ip{g_0}{M_{aq}g_0}
 =\int_{\R}|g_0(t)|^2e^{2\pi i kt/L}\,dt=0.
\]
Thus $g_0$ is orthogonal to every nontrivial time--frequency translate of
itself by a point of $\Gamma_0$.  Since $\norm{g_0}_2=1$ and the inner
product of two lattice atoms differs, up to a constant of modulus one, from
an inner product with their lattice difference, $\G(g_0,\Gamma_0)$ is an
orthonormal system.

Let
\[
 M_0=a\begin{pmatrix}1&1\\-1&1\end{pmatrix}.
\]
For the convention used in this paper, the adjoint lattice of
$M_0\Z^2$ is $JM_0^{-T}\Z^2$.  A direct calculation gives
\[
 JM_0^{-T}\Z^2
 =\frac1{2a}\begin{pmatrix}-1&1\\-1&-1\end{pmatrix}\Z^2
 =\frac1{2a}\begin{pmatrix}1&-1\\1&1\end{pmatrix}\Z^2,
\]
where the final equality is only a change of the integer lattice basis.
This proves the displayed formula for $\Lambda_0$.

Write
\[
 c=\frac1{\sqrt2a}=\frac{100}{71\sqrt2},
 \qquad
 S=\frac1{\sqrt2}\begin{pmatrix}1&-1\\1&1\end{pmatrix}.
\]
Then $\det S=1$, so $S\in\Sp(2,\R)$, and $\Lambda_0=cS\Z^2$.
All four entries, which are the four scalar blocks in dimension one, are
nonzero.  Furthermore,
\[
 \covol(\Lambda_0)=c^2=\frac1{2a^2}=\frac{5000}{5041}.
\]
Its symplectic Gram matrix is
\[
 (cS)^TJ(cS)=c^2J=\frac{5000}{5041}J.
\]
Since $5000$ and $5041$ are coprime, an integer coordinate vector
$k\in\mathbb Z^2$ represents an element of the integral symplectic subgroup
exactly when both coordinates of $k$ are divisible by $5041$.  Hence
\[
 [\Lambda_0:(\Lambda_0)_{\rm int}]=5041^2,
 \qquad
 \nu(\Lambda_0)=5041.
\]
The reciprocal symplectic Gram matrix for the adjoint lattice has denominator
$5000$, and therefore
\[
 \nu(\Lambda_0^\circ)=5000.
\]
In particular, $\Lambda_0$ satisfies \eqref{eq:SG} in dimension one.

Finally, put
\[
 g=\sqrt{\covol(\Lambda_0)}\,g_0=cg_0.
\]
Then
$\covol(\Lambda_0)^{-1/2}g=g_0$, and
$\Lambda_0^\circ=\Gamma_0$.  The normalized Gabor duality principle,
Theorem~\ref{thm:gabor-duality}, therefore shows that
$\G(g,\Lambda_0)$ is a Parseval frame.  Since scalar multiplication does
not change support or smoothness, $g\in C_c^\infty(\R)$.
\end{proof}

\end{document}